\documentclass[a4paper,11pt]{article}

\usepackage{amsmath,amssymb,amsthm}
\usepackage[utf8]{inputenc}
\usepackage{amssymb,amsmath,mathrsfs,amsthm,bbm,xcolor}
\usepackage{verbatim}%
\usepackage{indentfirst}
\usepackage{geometry}
\usepackage{enumerate} 
\usepackage{graphicx} 
\usepackage{tikz}
\usepackage{tikz-cd}
\usetikzlibrary{calc,decorations.pathreplacing,angles,arrows.meta,bending}
\usepackage{float}
\usepackage{booktabs,caption}
\usepackage[flushleft]{threeparttable}
\usepackage{hyperref} %
\hypersetup{colorlinks=true,allcolors=blue}
\usepackage{hypcap}
\usepackage[shortlabels]{enumitem}

\usepackage{todonotes}

\usepackage{soul} %

\usepackage[capitalize]{cleveref} %

\usepackage{autonum}

\usepackage{enumitem}
\newlist{properties}{enumerate}{1}
\setlist[properties]{label=\textup{(P\arabic*)},ref=\textup{(P\arabic*)},leftmargin=*}
\newlist{hypotheses}{enumerate}{1}
\setlist[hypotheses]{label=\textup{(H\arabic*)},ref=\textup{(H\arabic*)},leftmargin=*}

\usepackage{longtable}
\usepackage{array} %

\makeatletter
\def\cref@thmoptarg[#1]#2#3#4{%
    \ifhmode\unskip\unskip\par\fi%
    \normalfont%
    \trivlist%
    \let\thmheadnl\relax%
    \let\thm@swap\@gobble%
    \thm@notefont{\fontseries\mddefault\upshape}%
    \thm@headpunct{.}%
    \thm@headsep 5\p@ plus\p@ minus\p@\relax%
    \thm@space@setup%
    #2%
    \@topsep \thm@preskip               %
    \@topsepadd \thm@postskip           %
    \def\@tempa{#3}\ifx\@empty\@tempa%
      \def\@tempa{\@oparg{\@begintheorem{#4}{}}[]}%
    \else%
      \refstepcounter[#1]{#3}%
      \@namedef{cref@#3@alias}{#1}%
      \def\@tempa{\@oparg{\@begintheorem{#4}{\csname the#3\endcsname}}[]}%
    \fi%
    \@tempa}%
\makeatother

\crefname{ptheorem}{THEOREM}{THEOREMS}
\crefname{plemma}{lemma}{lemmas}

\newtheorem{thm}{Theorem}[section]
\newtheorem{prop}[thm]{Proposition}

\newtheorem{thmx}{Theorem}

\crefname{thmx}{Theorem}{Theorems}  \Crefname{thmx}{Theorem}{Theorems}

\newtheorem{cor}[thm]{Corollary}

\newtheorem{lem}[thm]{Lemma}

\newtheorem*{lem*}{Lemma}

\newtheorem*{claim*}{Claim}
\newtheorem{claim}{Claim}

\theoremstyle{remark}
\newtheorem{rem}[thm]{Remark}
\newtheorem*{rem*}{Remark}

\theoremstyle{definition}
\newtheorem{defi}[thm]{Definition}

\makeatletter
\newtheorem*{rep@theorem}{\rep@title}
\newcommand{\newreptheorem}[2]{%
\newenvironment{rep#1}[1]{%
 \def\rep@title{#2 \ref{##1}}%
 \begin{rep@theorem}}%
 {\end{rep@theorem}}}
\makeatother
\newreptheorem{thm}{Theorem}
\newreptheorem{lem}{Lemma}

\makeatletter
\def\symmetricbbox{%
  \pgfpointanchor{current bounding box}{south west}\edef\bbxmin{\the\pgf@x}\edef\bbymin{\the\pgf@y}%
  \pgfpointanchor{current bounding box}{north east}\edef\bbxmax{\the\pgf@x}\edef\bbymax{\the\pgf@y}%
  \pgfmathsetlength\dimen@{max(0pt-\bbxmin,\bbxmax-\panelshift)}%
  \useasboundingbox (0pt-\dimen@,\bbymin) rectangle (\panelshift+\dimen@,\bbymax);}
\makeatother

\usepackage{color}

\RequirePackage{extarrows}
\RequirePackage{mathtools}
\RequirePackage{textcomp}
\RequirePackage{wasysym}
\RequirePackage{tikz}
\RequirePackage{graphicx}
\usetikzlibrary{arrows}
\usetikzlibrary{calc}
\usepackage{esint}

\numberwithin{equation}{section}

\DeclareMathOperator{\supp}{\mathrm{supp}}

\DeclareMathOperator{\diam}{\mathrm{diam}}

\newcommand{\cA}{\mathcal A}
\newcommand{\cB}{\mathcal B}

\newcommand{\cD}{\mathcal D}

\newcommand{\cF}{\mathcal F}

\newcommand{\cN}{\mathcal N}

\newcommand{\cR}{\mathcal R}

\newcommand{\xx}{\underline{x}}
\newcommand{\yy}{\underline{y}}

\renewcommand{\leq}{\leqslant}
\renewcommand{\geq}{\geqslant}

\renewcommand{\bf}[1]{\mathbf{#1}}
\renewcommand{\rm}[1]{\mathrm{#1}}
\renewcommand{\cal}[1]{\mathcal{#1}}
\newcommand{\bb}[1]{\mathbb{#1}}

\newcommand{\jialun}[1]{{\color{olive}#1}}

\def\R{\bb R}

\def\Z{\bb Z}

\def\N{\bb N}

\def\calA{\cal A}

\def\Id{\rm{Id}}

\def\bi{\bf i}

\newcommand{\bfi}{\mathbf{i}}

\newcommand{\ptm}{\mathbb{P}TM}
\newcommand{\wloc}{W^{su}_{\mathrm{loc}}}

 \newcommand{\bfj}{{\mathbf{j}}}
\newcommand{\bfk}{{\mathbf{k}}}

\renewcommand{\epsilon}{\varepsilon}

\begin{document}
	\bibliographystyle{alpha}
	\title{Dimensions of surface repellers and attractors of non-linear planar IFSs}
\author{Jialun Li\thanks{School of Mathematical Sciences, Fudan University, No 220 Handan Road, Shanghai 200433, China. Email: \texttt{jialunli@fudan.edu.cn}.},\;
Wenyu Pan\thanks{Department of Mathematics, University of Toronto, 40 St George St, Toronto ON,	M5S 2E4, Canada. Email: \texttt{wenyup.pan@utoronto.ca}.},\;
Yao Tong\thanks{ School of Mathematical Sciences, Peking University, No 5 Yiheyuan Road, Beijing 100871, China. Email: \texttt{2401110024@stu.pku.edu.cn}.}\; and
Disheng Xu\thanks{School of Science, Great Bay University and Great bay institute for advanced study, Songshan Lake International Innovation Entrepreneurship Community A5, Dongguan, Guangdong, 523000, Email: \texttt{xudisheng@gbu.edu.cn}.}}
\date{}
 \maketitle
	\begin{abstract}

    We establish that for a $C^r$-generic surface repeller ($1 < r \leq \infty$), its Hausdorff and box dimensions are exactly the unique zero of the sub-additive topological pressure function. As an application of our framework, we show that the attractor of a uniformly non-conformal and weakly irreducible planar non-linear iterated function system (IFS) satisfying the strong separation condition (SSC) attains its expected Hausdorff and box dimensions. Furthermore, as a direct consequence of our generic repeller theorem, we deduce that this dimension formula also holds for a generic $C^r$ planar IFS satisfying the SSC. Using this approach, we also extend the dichotomy result for graphs of Weierstrass-type functions of Ren and Shen \cite{ren_dichotomy_2021} by weakening their real-analytic requirement to arbitrary $C^r$ regularity for $r > 1$.

    \end{abstract}
\setcounter{tocdepth}{2}
\tableofcontents

\section{Introduction}

\subsection{Background and main results}

Dimension is a primary invariant for quantifying the geometric complexity of fractals. Within the study of dynamically generated fractals, a cornerstone result is the work of Bowen \cite{bowen_hausdorff_1979} and Ruelle \cite{ruelle_repellers_1982} on conformal expanding maps: if $\Lambda$ is a repeller for a conformal expanding map $f$, its Hausdorff dimension is the unique root $s$ of Bowen's equation
\[
P(f|_{\Lambda},-s\log\|Df\|)=0,
\]
where $P(f|_\Lambda,\cdot)$ denotes the topological pressure (see also \cite{patterson_limit_1976, sullivan_density_1979}). This forged a deep connection between dimension theory and thermodynamic formalism.

\iffalse
The dimension theory of non-conformal repellers is substantially more complex. As remarked by Chen and Pesin \cite{chen_dimension_2010}, \emph{it has long been considered a daunting task to establish a Bowen-type formula for non-conformal repellers} (see also \cite{barreira_pesin_schmeling_1999}). The expected value of the dimension is the unique root $s_0$ of the sub-additive pressure equation
\[
P_{\mathrm{sub}}(s)=0,
\]
where $P_{\mathrm{sub}}$ is the sub-additive topological pressure of the singular value potentials; we defer the precise definition to \cref{subsubsec:sub-additive pressures}. Following Falconer's introduction of the sub-additive pressure \cite{falconer_bounded_1994}, who obtained the upper bound under a bunching condition, Zhang \cite{zhang_dynamical_1997} proved $\dim_{\mathrm H}\Lambda\leq s_0$ for every $C^1$ repeller, and Ban, Cao and Hu \cite{BanCaoHu} showed that the same bound holds for the upper box dimension:
\[
\dim_{\mathrm H}\Lambda\leq\overline{\dim}_{\mathrm B}\Lambda\leq s_0 .
\]
\fi

The dimension theory of non-conformal repellers is substantially more
complex. As remarked by Chen and Pesin \cite{chen_dimension_2010}, \emph{it has long been considered a daunting task to establish a Bowen-type formula for non-conformal repellers} (see also \cite{barreira_pesin_schmeling_1999}). The natural analog of Bowen's
equation is the sub-additive pressure equation
\[
P_{\mathrm{sub}}(s)=0,
\]
where $P_{\mathrm{sub}}$ is the sub-additive topological pressure of
the singular value potentials; it has a unique root, denoted $s_0$
throughout, and we defer the precise definitions to
\cref{subsubsec:sub-additive pressures}. The pressure
$P_{\mathrm{sub}}$ was introduced by Falconer
\cite{falconer_subadditive_1988}, who obtained the upper bound
$\dim_{\mathrm H}\Lambda\leq s_0$ under a bunching condition \cite{falconer_bounded_1994}; Zhang
\cite{zhang_dynamical_1997} proved it for every $C^1$ repeller, and
Ban, Cao and Hu \cite{BanCaoHu} extended it to the upper box
dimension:
\[
\dim_{\mathrm H}\Lambda\leq\overline{\dim}_{\mathrm B}\Lambda\leq s_0 .
\]
The expected value of the dimension is therefore $s_0$.

Toward the reverse inequality $\dim_{\mathrm H}\Lambda\geq s_0$, the
known results are the deterministic lower bound of Cao, Pesin, and
Zhao \cite{cao_dimension_2019}, given by the zero of a super-additive
pressure and, in general, strictly smaller than $s_0$; exact formulas,
due to Feng and Simon \cite{Feng_C1_2022}, are valid for almost every
parameter in certain families of non-linear repellers and iterated function
systems satisfying a transversality condition; and formulas for
restricted classes of systems
\cite{luzia_hausdorff_2006,ren_dichotomy_2021}. On the other hand,
reducible or integrable examples show that the upper bound need not
be attained; see, for example, \cite{ren_dichotomy_2021}. Such
dimension-drop phenomena are nevertheless regarded as exceptional:
the formula $\dim_{\mathrm H}\Lambda=s_0$ should hold for generic
systems.

In this paper, we confirm this expectation for surface repellers. We also prove the formula for every surface repeller satisfying two explicit dynamical conditions: dominated splitting and $su$-non-integrability.

To state the results, let $M$ be a smooth Riemannian manifold without
boundary, not necessarily compact, and let $f\colon M\to M$ be a
$C^{r}$ map, $r\in(1,\infty]$. Let $\Lambda:=\Lambda_f$ be a nonempty
compact subset of $M$. We call $\Lambda$ a \emph{repeller} for $f$ if
the following conditions hold:
\begin{enumerate}[label=(\roman*), ref=(\roman*)]
\item\label{item:isolated} there exists an open neighborhood $U$ of
$\Lambda$, called an \emph{isolating neighborhood}, such that
\[
\Lambda=\{x\in U : f^{n}(x)\in U \text{ for all } n\geq 0\};
\]
\item\label{item:expanding} there exists a constant $\kappa>1$ such that
\[
\|D_xf\,v\| \geq \kappa \|v\|
\qquad\text{for all } x\in\Lambda \text{ and } v\in T_xM;
\]
\item\label{item:nonwandering} every point of $\Lambda$ is
non-wandering for the restriction $f|_{\Lambda}$, that is,
$\Omega(f|_{\Lambda})=\Lambda$.
\end{enumerate}
When $M$ is two-dimensional, we speak of a \emph{surface repeller}.
Throughout, we exclude the trivial case that $(\Lambda_f,f)$ is a transitive repeller and $f|_{\Lambda_f}$ is injective; hence, $\Lambda_f$ is a single periodic orbit. We assume that the restriction of $f$ to each transitive
component of $\Lambda_f$ is not injective. 

Repellers are structurally stable (see \cref{subsec:continuation}): for $r>1$, there is a $C^r$-open neighborhood $U_f\subseteq C^r(M,M)$ of $f$ such that every $g\in U_f$ has a repeller $\Lambda_g$, the \emph{continuation} of $\Lambda_f$, with $(\Lambda_g,g)$ topologically conjugate to $(\Lambda_f,f)$; here $C^r(M,M)$ carries the topology of uniform $C^r$ convergence on compact sets, and a subset of $U_f$ is \emph{residual} if it contains a countable intersection of open dense subsets of $U_f$. For $g\in U_f$, we write $s_0(g)$ for the unique zero of the sub-additive pressure $P_{\mathrm{sub}}$ of $(\Lambda_g,g)$. 

\begin{thmx}\label{thm:A}
Let $r\in(1,\infty]$. For every $C^r$ surface repeller
$(\Lambda_f,f)$, there exist a $C^r$-open neighborhood
$U_f\subseteq C^r(M,M)$ of $f$ and a subset $V\subseteq U_f$, residual in the $C^r$ topology of $U_f$, such that for every $g\in V$
the continuation $(\Lambda_g,g)$ satisfies
\begin{equation}
\dim_{\mathrm H}\Lambda_g
=\dim_{\mathrm B}\Lambda_g
=s_0(g).
\end{equation}
\end{thmx}
\noindent
\Cref{thm:A} is the set-dimension part of \eqref{eqn:dim formula generic} in \Cref{thm:generic_full}.

In the sequel, we summarize \Cref{thm:A} by saying that the dimension formula holds for \emph{$C^r$-generic surface repellers}.

Our second result is non-perturbative: its hypotheses are two explicit
dynamical conditions on the single map $f$. Since $f|_{\Lambda_f}$ is not
invertible, directions determined by the past of a point are indexed by
its \emph{pre-histories}, sequences $\hat x=(x_j)_{j\leq0}$ in $\Lambda_f$
with $x_0=x$ and $f(x_{j-1})=x_j$ (\cref{subsec:inverse limit}). A
\emph{dominated splitting} is a continuous $Df$-invariant splitting
\begin{equation}\label{eq:intro-splitting}
T_xM=E^{wu}(x)\oplus E^{su}(\hat x),
\qquad x\in\Lambda_f,\ \hat x\text{ a pre-history of }x,
\end{equation}
into a weak and a strong unstable direction, the expansion along $E^{su}$
uniformly dominating that along $E^{wu}$; the weak direction depends only
on $x$, while the strong one depends on the pre-history. The repeller is
\emph{$su$-non-integrable} if each transitive component contains a point
that admits two pre-histories with distinct strong directions
(\Cref{def:su-nonintegrable}). If this fails, the strong unstable
directions form a $Df$-invariant line field on some transitive component,
as for Bedford--McMullen carpets \cite{bedford1984,mcmullen1984}.
Non-integrability plays, for a single map, the role that irreducibility
plays for random matrix products \cite{furstenberg_noncommuting_1963}, and
conditions of this type appear in several settings; see for instance
\cite{benoistMesuresStationnairesFermes2011,brown_rodriguez_hertz_measure_rigidity_2017,katz_measure_rigidity_2023,eskin_potrie_zhang_geometric_2023,dolgopyat_decay_1998,tsujii_zhang_smooth_mixing_2023,gan_shi_rigidity_2020}.
If $(\Lambda_f,f)$ admits a dominated splitting, $su$-non-integrability of
the continuation holds on a $C^r$-open and dense subset of $U_f$
(\Cref{appendix:B}).

\begin{thmx}\label{thm:B}
Let $(\Lambda_f,f)$ be a $C^{1+\alpha}$ surface repeller that admits a dominated splitting and is $su$-non-integrable. Then
\[
\dim_{\mathrm H}\Lambda_f=\dim_{\mathrm B}\Lambda_f
=s_0 ,
\]
where $s_0$ is the unique zero of $P_{\mathrm{sub}}$.
\end{thmx}
\noindent
\Cref{thm:B} is the set-dimension part of \eqref{eqn:dim formula surface repeller} in \Cref{thm:dominated}.

We obtain the lower bound $\dim_{\mathrm H}\Lambda\geq s_0$ from ergodic measures $\mu$ whose Lyapunov dimension $\dim_{\mathrm{LY}}(\mu,f)$ (see \eqref{eq:intro-LY} below) is close to $s_0$; the difficulty is to show that $\dim_{\mathrm H}\mu$ is also close to $s_0$. Indeed, the Lyapunov dimension is an upper bound for $\dim_{\mathrm H}\mu$ \cite[Thm.~1.4]{feng_simon_dimension_part_i_2023}, but in the non-conformal setting, it need not be attained; for example, in the $su$-integrable cases. Establishing the equality $\dim_{\mathrm H}\mu=\dim_{\mathrm{LY}}(\mu,f)$ for a special class of invariant measures is the main step of this paper (\Cref{thm:base}), and we now describe this class.

A repeller is a \emph{horseshoe repeller}
if $(\Lambda_f,f)$ is topologically conjugate, via a homeomorphism
$h\colon\Sigma_m^+\to\Lambda_f$, to the full one-sided shift on $m\geq2$
symbols (\Cref{def:horseshoe repeller}). A measure $\mu$ on $\Lambda_f$ is \emph{$f$-Bernoulli} if
$\mu=h_*\nu_{\mathbf p}$ for a Bernoulli product measure $\nu_{\mathbf p}$
with weights $p_i>0$ (\Cref{def:Bernoulli measure}); such a measure is ergodic and has full support.
For an ergodic $\mu$, let $\lambda_1(\mu,f)\geq\lambda_2(\mu,f)>0$ be its
Lyapunov exponents (\Cref{thm:oseledets-repeller}) and $h_\mu(f)$ its entropy. The \emph{Lyapunov dimension}
$\dim_{\mathrm{LY}}(\mu,f)$, defined in general in \eqref{eqn:def LY}, is, on a surface, the value obtained by charging the entropy
first to the weakly expanding direction:
\begin{equation}\label{eq:intro-LY}
\dim_{\mathrm{LY}}(\mu,f)
=\begin{cases}
h_\mu(f)/\lambda_2(\mu,f), & h_\mu(f)\leq\lambda_2(\mu,f),\\[2pt]
1+\bigl(h_\mu(f)-\lambda_2(\mu,f)\bigr)/\lambda_1(\mu,f), & h_\mu(f)>\lambda_2(\mu,f);
\end{cases}
\end{equation}
by Ruelle's inequality $\dim_{\mathrm{LY}}(\mu,f)\leq2$.

\begin{thmx}\label{thm:base}
Let $(\Lambda_f,f)$ be a $C^{1+\alpha}$ horseshoe surface repeller that admits
a dominated splitting and is $su$-non-integrable. Then every $f$-Bernoulli
measure $\mu$ on $\Lambda_f$ satisfies
\[
\dim_{\mathrm H}\mu=\dim_{\mathrm{LY}}(\mu,f).
\]
\end{thmx}

The idea of the proof of \Cref{thm:B} from \Cref{thm:base} is as follows. The root $s_0$ of the sub-additive pressure equation \eqref{eqn:sub-additive potential} has the variational
characterization
\begin{equation}\label{eq:intro-s0}
s_0=\sup\{\dim_{\mathrm{LY}}(\mu,f):\mu\in\mathcal M_e(\Lambda_f,f)\},
\end{equation}
see \eqref{eqn:s0 variational}; in particular, the lower bound $\dim_{\mathrm H}\Lambda_f\geq s_0$
follows from \Cref{thm:base} once one produces Bernoulli measures $\mu$ on horseshoe sub-repellers satisfying the hypotheses of \Cref{thm:base} with
$\dim_{\mathrm{LY}}(\mu,f)$ close to $s_0$.

\subsection{Applications}

\subsubsection{Dimension of non-linear non-conformal iterated function systems}

The dimension of attractors of iterated function systems (IFSs) is the counterpart of this problem in fractal geometry; see e.g.\ \cite{hutchinson_fractals_1981,bedford1984,mcmullen1984,bedfordurbanski1990,gatzouras1992,hu1998,falconer2003,manning2007,jordan_pollicott_simon_2007,falconer2013,dassimmons2017,barany_triangular_2019,ffl2021,jurga2022}. For self-affine attractors, the expected value is again the zero of a sub-additive pressure \cite{falconer_hausdorff_1988}, and in the plane, this has been proved under algebraic hypotheses on the linear parts: Hochman's inverse theorem for entropy \cite{hochman_self-similar_2014} led to the formula of B\'ar\'any, Hochman and Rapaport \cite{barany_hochman_rapaport_2019}, via the variational principle of Morris and Shmerkin \cite{morris_equality_2016}, and to the advances of Hochman and Rapaport \cite{hochman_hausdorff_2021}; see \cite{lpx,jlpx} for random walks on projective spaces. These methods use the finite-dimensional Lie group containing the maps,
but for $C^r$ non-linear IFSs, $1<r\leq\infty$, there is no such group, and exact formulas have been limited to the special systems and transversal families mentioned above.

A planar $C^r$, $r>1$, contracting IFS is a finite collection $\Phi=\{\phi_i\}_{i=1}^m$ of uniformly contracting $C^r$ diffeomorphisms defined on a bounded open domain $U\subset\mathbb R^2$ such that
$$\overline{\bigcup_{i=1}^m\phi_i(U)}\,\subset\, U.$$ By Hutchinson's theorem \cite{hutchinson_fractals_1981}, such a system admits a unique compact subset $\Lambda_\Phi\subset U$, the \emph{attractor} of the IFS, satisfying
\[
\Lambda_\Phi=\bigcup_{i=1}^m\phi_i(\Lambda_\Phi).
\]
Given a probability vector $\mathbf p$, we denote by $\mu_{\Phi,\mathbf p}$ the stationary measure associated with $(\Phi,\mathbf p)$, namely the unique Borel probability measure on $\Lambda_\Phi$ satisfying
\[
\mu_{\Phi,\mathbf p}=\sum_{i=1}^m p_i\cdot(\phi_i)_*\mu_{\Phi,\mathbf p},
\]
where $(\phi_i)_*$ denotes the push-forward by $\phi_i$.
For the simplicity of exposition, we restrict attention to IFSs satisfying the Strong Separation Condition (SSC), i.e.
\[
\phi_i(\Lambda_\Phi)\cap\phi_j(\Lambda_\Phi)=\emptyset\quad\text{for all } i\neq j .
\]
Under the SSC, the inverse branches $\phi_i^{-1}$ define an expanding map
$f_\Phi$ near $\Lambda_\Phi$, with $\Lambda_\Phi$ as a horseshoe repeller, and
the coding sequence of the IFS plays the role of the pre-history. Under this
dictionary, a dominated splitting for $f_\Phi$ is equivalent to
\emph{uniform non-conformality} of $\Phi$ (the singular values of
$D\Phi^n_{\mathbf i}$ separate exponentially, uniformly in the word
$\mathbf i$), and $su$-non-integrability is
equivalent to \emph{weak irreducibility} (the most contracted direction at
some point depends on the coding sequence); see \cref{subsec:IFS} for the precise definitions. For self-affine systems, these are the non-compactness and irreducibility conditions on the linear
parts familiar from \cite{barany_hochman_rapaport_2019,hochman_hausdorff_2021}.

\begin{thmx}\label{thm:ifs-non-perturbative}
Let $\Phi=\{\phi_i\}_{i=1}^m$ be a $C^{r}$, $r>1$, planar contracting IFS satisfying the SSC. Assume that $\Phi$ is uniformly non-conformal and weakly irreducible. Then
\[
\dim_{\mathrm H}\Lambda_\Phi=\dim_{\mathrm B}\Lambda_\Phi=s_0(\Phi),
\]
where $s_0(\Phi)$ is the unique zero of the sub-additive singular-value pressure $P_{\Phi}$ defined in \eqref{eqn:pressure IFS}. Moreover, for every probability vector $\mathbf p$,
\[
\dim_{\mathrm H}\mu_{\Phi,\mathbf p}=\dim_{\mathrm{LY}}(\mu_{\Phi,\mathbf p},\Phi),
\]
where $\dim_{\mathrm{LY}}(\mu,\Phi)$ denotes the Lyapunov dimension of $\mu$ with respect to $\Phi$ (\cref{subsec:IFS}). 
\end{thmx}

We equip $C^r(U,\mathbb R^2)^m$ with the product of the $C^r$ topologies of uniform convergence on compact subsets of $U$. If $\Phi$ is a $C^r$ contracting IFS on $U$ satisfying the SSC, then every $\Psi\in C^r(U,\mathbb R^2)^m$ sufficiently close to $\Phi$ is a contracting IFS on a neighborhood of $\Lambda_\Phi$ satisfying the SSC, and its attractor $\Lambda_\Psi$ depends continuously on $\Psi$.

\begin{thmx}\label{thm:ifs}
Let $r\in(1,\infty]$. For every $C^r$ planar contracting IFS $\Phi$ satisfying the SSC, there exist a $C^r$-open neighborhood $U_\Phi\subseteq C^r(U,\mathbb R^2)^m$ of $\Phi$ and a subset $V\subseteq U_\Phi$, residual in the $C^r$ topology of $U_\Phi$, such that every $\Psi\in V$ satisfies
\[
\dim_{\mathrm H}\Lambda_\Psi=\dim_{\mathrm B}\Lambda_\Psi=s_0(\Psi).
\]
\end{thmx}

\subsubsection{Dichotomy for Weierstrass-type functions}
A central topic in fractal geometry is the Hausdorff dimension of the graphs of nowhere differentiable functions, among which Weierstrass-type functions are the most extensively studied. For a $1$-periodic $\phi\colon\R\to\R$, $b\in\N_{\geq2}$ and
$\lambda\in(1/b,1)$, let
\[
W^\phi_{\lambda,b}(x):=\sum_{n\geq0}\lambda^n\phi(b^nx), \qquad x\in\R.
\]
For real analytic $\phi$, Ren and Shen \cite{ren_dichotomy_2021} proved
that either $W^\phi_{\lambda,b}$ is real analytic, or the Hausdorff
dimension of its graph equals $2+\log_b\lambda$. This dichotomy contains a result of Shen \cite{Shen18}, which resolved a long-standing conjecture for trigonometric $\phi$. Using \Cref{thm:base}, we obtain the finite-smoothness analog, recovering their
theorem at $r=\omega$:

\begin{thmx}\label{thm:ren-shen-c}
Let $r\in(1,\infty]\cup\{\omega\}$ and let $\phi$ be a $1$-periodic
$C^r$ function. For every integer $b\geq2$ and every
$\lambda\in(1/b,1)$, either the graph of $W^\phi_{\lambda,b}$ has
Hausdorff and box dimensions equal to $2+\log_b\lambda$, or the
function $W^\phi_{\lambda,b}$ is $C^r$. Moreover, the second alternative only happens for finitely many $\lambda$ when $\phi$ is non-constant.
\end{thmx}

\iffalse
For an integer $b\geq2$, a parameter $\lambda\in(1/b,1)$ and a $1$-periodic function $\phi$, consider
\begin{equation}
W^\phi_{\lambda,b}(x)=\sum_{n=0}^\infty\lambda^n\phi(b^nx).
\end{equation}
Ren and Shen \cite{ren_dichotomy_2021} proved that for real-analytic $\phi$, either the graph of $W^\phi_{\lambda,b}$ has the predicted Hausdorff dimension $2+\log_b\lambda>1$, or $W^\phi_{\lambda,b}$ is itself real-analytic. This dichotomy contains the celebrated result of Shen \cite{Shen18}, which resolved a long-standing conjecture for trigonometric $\phi$. Using \Cref{thm:base}, we extend the dichotomy from the real-analytic category to $C^r$ regularity, $r>1$.

\begin{thmx}\label{thm:ren-shen-c}
Let $r\in(1,\infty]\cup\{\omega\}$ and let $\phi$ be a $1$-periodic $C^r$ function. For every integer $b\geq2$ and every $\lambda\in(1/b,1)$, either the graph of $W^\phi_{\lambda,b}$ has Hausdorff and box dimension $2+\log_b\lambda$, or the function $W^\phi_{\lambda,b}$ is $C^r$.
\end{thmx}
\fi

\subsection{Strategy of the proofs}
\iffalse

The lower bound on $\dim_{\mathrm H}\mu$ is obtained from sumset estimates---Hochman's inverse
theorem \cite{hochman_self-similar_2014} and the Balog--Szemer\'edi--Gowers
theorem in the form used by Wu \cite{wuProjectionTheoremsCountably2025} --- rather than from transversality and %
entropy growth arguments
in a finite-dimensional group. Such estimates have been used in fractal
geometry \jialun{\st{and in the dynamics of algebraic and analytic systems}}
\cite{katztao2001,bourgain_2010,bourgaindyatlov2017,shmerkin2023,orponen_shmerkin_jams2026,wang_sticky_2026,benardhezhang2026,khalil2026,lpx,jlpx};
to our knowledge, this is their first use of the dimension of invariant measures for non-linear, non-conformal $C^{1+\alpha}$ systems.
\fi
A key feature of the proof of \Cref{thm:base} is the use of techniques
from additive combinatorics \cite{katztao2001,bourgain_2010,hochman_self-similar_2014,wuProjectionTheoremsCountably2025} in the study of a smooth dynamical system. While additive combinatorics has been extensively used in fractal geometry, see, for example,  \cite{bourgaindyatlov2017,shmerkin2023,orponen_shmerkin_jams2026, wang_sticky_2026,benardhezhang2026,khalil2026}; to the best of our knowledge, this is the first application to smooth dynamics, where no finite-dimensional Lie group structure is available. 

More precisely, establishing the inequality $\dim_{\mathrm H}\mu\geq\dim_{\mathrm{LY}}(\mu,f)$ requires a synthesis of additive combinatorics, smooth dynamics, and the moving-fiber approach developed by Wu \cite{wuProjectionTheoremsCountably2025} for the planar linear setting.
Wu's approach builds upon Hochman's work but handles the interplay between the base and fiber measures differently: in its applications to self-similar and self-affine measures, Hochman's method uses the finite-dimensionality of the acting group or semigroup to run the entropy argument, whereas Wu introduces the spreading property of fiber measures, a non-concentration property obtained from the local structure of the measure, which provides the input for the entropy-increase argument without representing the measure as a convolution over a group of maps. This property offers a more flexible framework, though it
relies on a transversality condition that can be challenging to verify in general. 
The rest of this subsection explains how the approach is carried out for repellers, where the projections are along curved leaves, and every estimate is local.

\iffalse
More precisely, establishing the inequality $\dim_{\mathrm{H}}\mu\geq \dim_{\mathrm{LY}}\mu$ requires an intricate synthesis of additive combinatorics, smooth
dynamics, and the moving-fiber approach developed by Wu \cite{wuProjectionTheoremsCountably2025} for the planar linear setting.
Wu’s approach builds upon Hochman’s works but handles the interplay between the base and
fiber measures differently. While Hochman’s method essentially requires the acting group or
semigroup to be finite-dimensional in order to apply an entropy argument, Wu introduces the
spreading property of fiber measures. This property offers a more flexible framework, though it
relies on a transversality condition that can be challenging to verify in general. 
In this paper, we develop this approach in the setting of repellers by extracting and heavily exploiting local geometric and dynamical information.
\fi

\subsubsection*{A Ledrappier--Young reformulation}
Let $(\Lambda_f,f)$ and $\mu$ be as in \Cref{thm:base}; domination gives
$\lambda_1(\mu,f)>\lambda_2(\mu,f)$. A pre-history of $x\in\Lambda_f$ is
determined by $x$ and a \emph{past} $\bfi\in\Sigma^-$, a left-infinite
sequence of symbols (see \eqref{eqn:pasts}); we write $E^{su}_\bfi(x)$ for the corresponding strong
unstable direction, $W_\bfi(x)$ for the local strong unstable leaf through
$x$ tangent to it, and $\nu_-$ for the Bernoulli measure induced by $\mu$ on
$\Sigma^-$. For a fixed past, the leaves $W_\bfi(x)$
form a local foliation, along which $\mu$ has a Ledrappier-Young
decomposition: there are
$\gamma_1,\gamma_2\in[0,1]$, the \emph{fiber dimension} and the
\emph{transverse dimension} of $\mu$, such that the conditional measures of
$\mu$ on the leaves have dimension $\gamma_1$, the projection of $\mu$ along
the leaves onto a transversal has dimension $\gamma_2$, independently of the
past and of the transversal, and $\mu$ is exact
dimensional, with exact dimension $\dim\mu=\dim_{\mathrm H}\mu$ satisfying
\begin{equation}\label{eq:intro-LYformula}
\dim\mu=\gamma_1+\gamma_2,
\qquad
h_\mu(f)=\lambda_1(\mu,f)\,\gamma_1+\lambda_2(\mu,f)\,\gamma_2
\end{equation}
(\cref{lem:exact dim tran}).
Comparing \eqref{eq:intro-LY} with \eqref{eq:intro-LYformula} gives
\begin{equation}\label{eq:intro-dichotomy}
\dim_{\mathrm H}\mu=\dim_{\mathrm{LY}}(\mu,f)
\quad\Longleftrightarrow\quad
\gamma_1=0\ \text{ or }\ \gamma_2=1 ,
\end{equation}
and \Cref{thm:base} is equivalent to the following statement, which is
the technical heart of the paper.

\begin{thm}\label{thm:main-two}
Let $(\Lambda_f,f)$ be a $C^{1+\alpha}$ horseshoe surface repeller that admits
a dominated splitting and is $su$-non-integrable, and let $\mu$ be an
$f$-Bernoulli measure on $\Lambda_f$. If $\gamma_1>0$, then $\gamma_2=1$.
\end{thm}

The proof transports the
$\gamma_1$-dimensional structure along the leaves of one past onto the
transversal by projecting along the leaves of a second past that makes a
small angle with the first, and it shows that a transverse measure of dimension
less than one cannot absorb it. We describe this first in the linear model
and then in our setting.
\vspace{5mm}

\emph{
In what follows, the linear model, the dictionary, and the two non-linear difficulties (localization and spreading) are meant to be read first. The remaining parts are condensed summaries of \cref{sec: sec5} and
\cref{sec: Bernoulli horseshoe approx},
and are best read as a guide alongside those sections.}

\subsubsection*{The linear model}

The mechanism originates in Hochman's inverse theorem for the entropy of
convolutions \cite{hochman_self-similar_2014} and in the projection theorem of
Wu \cite{wuProjectionTheoremsCountably2025} for CP-distributions on $\mathbb R^2$;
the following is a dimension-counting version of the argument in
\cite{wuProjectionTheoremsCountably2025}. Let $\mu$ be a planar measure
and let $\pi_1,\pi_2$ be two linear projections whose kernels make an angle
$2^{-k}$. Suppose the conditional measures of $\mu$ on the fibers of $\pi_1$
have positive dimension and $\pi_1\mu$ has dimension $\beta<1$. Take a fiber
$\ell$ of $\pi_1$ and the $2^{-k}$-neighborhood $R$ of a unit segment of
$\ell$. Recall that $\cN_r(S)$ is the minimum covering number of a set $S$ in $\mathbb{R}$ by balls of radius $r$. Because $\pi_1\mu$ has dimension $\beta$, the projection
$B=\pi_1(R\cap\supp\mu)$ satisfies $\cN_{2^{-2k}}(B)\approx 2^{\beta k}$. Under $\pi_2$: $\cN_{2^{-2k}}(\pi_2(R\cap\supp\mu))\approx \cN_{2^{-2k}}(\bigcup_{a}(a+B_a))$ of translates of
pieces $B_a\subseteq B$, one for each $2^{-k}$-cell $a$ of the
fiber. The $\pi_2$-projection of $R\cap \supp \mu$ is thus a
``sumset'' of the transverse set $B$ with the fiber.

The positivity of the fiber dimension forces the fiber measure to be \emph{spreading}, {{a quantitative non-concentration property}} defined in \cite[Def.~4.5]{wuProjectionTheoremsCountably2025} (see also \Cref{def-spreading}). The sumset theorem of
\cite[Thm.~4.7]{wuProjectionTheoremsCountably2025}  (see also \Cref{thm:growth})
states that adding a spreading set to a set $B$ that is not filling the
line ($\beta<1$) must increase the covering number by a definite power:
$\cN_{2^{-2k}}\bigl(\bigcup_a(a+B_a)\bigr)\geq \cN_{2^{-2k}}(B)^{1+\delta}$.
Consequently, $\pi_2\mu$ has, at this scale and around this fiber, a larger
dimension than $\pi_1\mu$, which results in a contradiction.

\begin{figure}
  \hspace{-1cm}
  \makebox[\linewidth][c]{%
    \resizebox{1.25\linewidth}{!}{%
\providecommand{\bfi}{\mathbf{i}}
\providecommand{\bfj}{\mathbf{j}}
\providecommand{\supp}{\operatorname{supp}}
\def\panelshift{11.6cm}   %
\begin{tikzpicture}[font=\small,>=Stealth]
\begin{scope}
\node[anchor=base west] at (-1.4,5.5) {\textbf{(a) linear model}};
\draw[fill=gray!8] (-0.45,0) rectangle (0.45,4);
\draw[very thick] (0,-0.05) -- (0,4.4);
\node[anchor=south] at (0,4.42) {$\ell=\pi_1^{-1}(z)$};
\foreach \y in {0.8,1.6,2.4,3.2} \draw[gray!60] (-0.45,\y)--(0.45,\y);
\begin{scope}
\clip (-0.45,0) rectangle (0.45,4);
\fill[blue!12]  (-0.45,0)   rectangle (0.45,0.8);
\fill[red!12]   (-0.45,1.6) rectangle (0.45,2.4);
\fill[green!12] (-0.45,3.2) rectangle (0.45,4);
\end{scope}
\foreach \p in {(-0.32,0.12),(-0.10,0.55),(0.10,0.25),(0.22,0.62),(0.38,0.18)}
  \fill[blue!70!black] \p circle (1.1pt);
\foreach \p in {(-0.38,2.15),(-0.15,1.75),(0.18,1.92),(0.03,2.28),(0.35,2.10)}
  \fill[red!70!black] \p circle (1.1pt);
\foreach \p in {(-0.30,3.85),(-0.12,3.35),(0.28,3.42),(0.10,3.72),(0.40,3.90)}
  \fill[green!50!black] \p circle (1.1pt);
\draw[decorate,decoration={brace,amplitude=4pt}] (-0.78,0)--(-0.78,4)
  node[midway,left=5pt] {$1$};
\draw[decorate,decoration={brace,mirror,amplitude=3pt}] (0.62,1.6)--(0.62,2.4)
  node[midway,right=4pt] {$a$\ ($2^{-k}$-cell)};
\draw[decorate,decoration={brace,mirror,amplitude=3pt}] (-0.45,-0.18)--(0.45,-0.18)
  node[midway,below=3pt] {$2^{-k}$};
\node at (1.15,3.9) {$R$};
\draw[->] (-1.1,-1.35)--(2.7,-1.35) node[right] {$\pi_1$};
\draw[dashed,gray] (-0.32,0.06)--(-0.32,-1.3);
\draw[dashed,gray] (0.35,2.04)--(0.35,-1.3);
\draw[line width=2.2pt,blue!70!black]  (-0.40,-1.62)--(-0.12,-1.62) (0.05,-1.62)--(0.40,-1.62);
\draw[line width=2.2pt,red!70!black]   (-0.35,-1.84)--(-0.05,-1.84) (0.12,-1.84)--(0.38,-1.84);
\draw[line width=2.2pt,green!50!black] (-0.42,-2.06)--(-0.20,-2.06) (0.00,-2.06)--(0.33,-2.06);
\node[anchor=west] at (1.10,-1.80) {$B_a=\pi_1(a\cap R\cap\supp\mu)$};
\draw[decorate,decoration={brace,mirror,amplitude=4pt}] (-0.45,-2.24)--(0.45,-2.24)
  node[midway,below=4pt] {$B=\bigcup_a B_a$:\ the $B_a$ overlap};
\draw[->] (-1.1,-3.35)--(2.7,-3.35) node[right] {$\pi_2$};
\draw[dashed,orange!80!black] (0.10,0.25)--(0.66,-3.30);
\draw[dashed,orange!80!black] (0.18,1.92)--(1.01,-3.30);
\draw[dashed,orange!80!black] (0.28,3.42)--(1.35,-3.30);
\node[anchor=west] at (0.62,3.35) {\textcolor{orange!80!black}{\scriptsize parallel fibers of $\pi_2$, at angle $2^{-k}$ with $\ell$}};
\draw[line width=2.2pt,blue!70!black]  (0.26,-3.62)--(0.54,-3.62) (0.71,-3.62)--(1.06,-3.62);
\draw[line width=2.2pt,red!70!black]   (0.66,-3.84)--(0.96,-3.84) (1.13,-3.84)--(1.39,-3.84);
\draw[line width=2.2pt,green!50!black] (0.93,-4.06)--(1.15,-4.06) (1.35,-4.06)--(1.68,-4.06);
\draw[decorate,decoration={brace,mirror,amplitude=4pt}] (0.24,-4.24)--(1.70,-4.24)
  node[midway,below=4pt] {$\bigcup_a(a+B_a)$:\ each $B_a$ translated by its height $\times\,2^{-k}$};
\end{scope}
\begin{scope}[xshift=\panelshift]
\node[anchor=base west] at (-1.6,5.5) {\begin{tabular}[t]{@{}l@{}}\textbf{(b) non-linear model}\\ \textbf{(straightening coordinates $\sigma_\bfi$)}\end{tabular}};
\draw[gray!60,fill=gray!6] (-0.75,0) rectangle (0.95,4);
\draw[very thick] (0,-0.05) -- (0,4.4);
\node[anchor=south] at (0,4.42) {$L\subset W_\bfi(y)$};
\fill[blue!14]  (-0.4,0.5) rectangle (0.1,1.0);
\fill[red!14]   (0.1,1.5)  rectangle (0.6,2.0);
\fill[green!14] (-0.4,3.5) rectangle (0.1,4.0);
\draw[blue!50]  (-0.4,0.5) rectangle (0.1,1.0);
\draw[red!50]   (0.1,1.5)  rectangle (0.6,2.0);
\draw[green!50!black] (-0.4,3.5) rectangle (0.1,4.0);
\fill[blue!70!black]  (0,0.62) circle (1.6pt);  \node[anchor=east] at (-0.44,0.62) {\textcolor{blue!70!black}{\scriptsize $z_t$}};
\fill[red!70!black]   (0,1.88) circle (1.6pt);  \node[anchor=east] at (-0.44,1.78) {\textcolor{red!70!black}{\scriptsize $z_{t'}$}};
\fill[green!50!black] (0,3.70) circle (1.6pt);  \node[anchor=east] at (-0.44,3.30) {\textcolor{green!50!black}{\scriptsize $z_{t''}$}};
\draw[blue!60,dashed]  (0,0.62) circle (0.55);
\draw[red!60,dashed]   (0,1.88) circle (0.55);
\draw[green!60!black,dashed] (0,3.70) circle (0.55);
\draw[gray!70,line width=0.3pt] (-0.42,3.34)--(-0.04,3.66);
\draw[gray!70,line width=0.3pt] (-0.42,1.80)--(-0.04,1.87);
\foreach \p in {(-0.33,0.60),(-0.20,0.90),(-0.08,0.66),(0.04,0.84)} \fill[blue!70!black] \p circle (1pt);
\foreach \p in {(0.16,1.60),(0.27,1.92),(0.38,1.70),(0.50,1.86)} \fill[red!70!black] \p circle (1pt);
\foreach \p in {(-0.33,3.86),(-0.18,3.58),(-0.05,3.90),(0.05,3.62)} \fill[green!50!black] \p circle (1pt);
\node[anchor=west] at (0.63,2.06) {\textcolor{red!70!black}{\scriptsize $\mathcal B_{t'}$}};
\node[anchor=west] at (1.12,3.35) {\textcolor{orange!80!black}{\scriptsize $\bfj$-leaves, angle $\tau_y\approx2^{-k}$ at $y$}};
\draw[decorate,decoration={brace,amplitude=4pt}] (-1.06,0)--(-1.06,4)
  node[midway,left=5pt] {$2^{-Ck}$};
\draw[decorate,decoration={brace,mirror,amplitude=3pt}] (-0.4,-0.20)--(0.1,-0.20)
  node[midway,below=3pt,xshift=-14pt] {$2^{-(C+1)k}$};
\node at (0.83,2.90) {$R$};
\draw[->] (-1.25,-1.45) .. controls (0.2,-1.28) and (1.5,-1.58) .. (2.8,-1.42) node[right] {$T$};
\node[anchor=east] at (-1.30,-1.45) {\textcolor{gray}{\scriptsize $\pi_\bfi$}};
\fill (0,-1.37) circle (1.4pt); \node[below left=0pt] at (0,-1.37) {\scriptsize $y$};
\draw[dashed,gray] (-0.20,0.86)--(-0.20,-1.34);
\draw[dashed,gray] (0.50,1.82)--(0.50,-1.46);
\draw[line width=2.2pt,blue!70!black]  (-0.35,-1.72)--(-0.20,-1.72) (-0.10,-1.72)--(0.05,-1.72);
\draw[line width=2.2pt,red!70!black]   (0.15,-1.94)--(0.31,-1.94) (0.40,-1.94)--(0.54,-1.94);
\draw[line width=2.2pt,green!50!black] (-0.34,-2.16)--(-0.22,-2.16) (-0.10,-2.16)--(0.06,-2.16);
\node[anchor=west] at (1.00,-1.90) {$B_t=\pi_\bfi(\mathcal B_t)$};
\draw[decorate,decoration={brace,mirror,amplitude=4pt}] (-0.40,-2.34)--(0.60,-2.34)
  node[midway,below=4pt] {$B=\bigcup_t B_t$};
\draw[->] (-1.25,-3.55) .. controls (0.2,-3.38) and (1.5,-3.68) .. (2.8,-3.52) node[right] {$T$};
\node[anchor=east] at (-1.30,-3.55) {\textcolor{gray}{\scriptsize $\pi_\bfj$}};
\draw[dashed,orange!80!black] (0.00,0.56) .. controls (0.30,-1.4) .. (0.13,-3.44);
\draw[dashed,orange!80!black] (0.44,1.74) .. controls (0.80,-0.8) .. (0.85,-3.60);
\draw[dashed,orange!80!black] (0.00,3.56) .. controls (0.70,-0.1) .. (0.69,-3.58);
\draw[line width=2.2pt,blue!70!black]  (-0.07,-3.82)--(0.08,-3.82) (0.18,-3.82)--(0.33,-3.82);
\draw[line width=2.2pt,red!70!black]   (0.65,-4.04)--(0.81,-4.04) (0.90,-4.04)--(1.04,-4.04);
\draw[line width=2.2pt,green!50!black] (0.50,-4.26)--(0.62,-4.26) (0.74,-4.26)--(0.90,-4.26);
\node[anchor=west] at (1.35,-4.13) {\begin{tabular}[t]{@{}l@{}}$\pi_\bfj(\mathcal B_t)\approx B_t+\pi_\bfj(z_t)$\\ $\pi_\bfj(z_t)\approx t$\end{tabular}};
\draw[decorate,decoration={brace,mirror,amplitude=4pt}] (-0.09,-4.44)--(1.06,-4.44)
  node[midway,below=4pt] {$\bigcup_t(t+B_t)$ at resolution $2^{-(C+2)k}$};
\end{scope}
\symmetricbbox
\end{tikzpicture}
    }%
  }
  \caption{The sumset structure of the projection}
  \label{fig:sumset}
\end{figure}
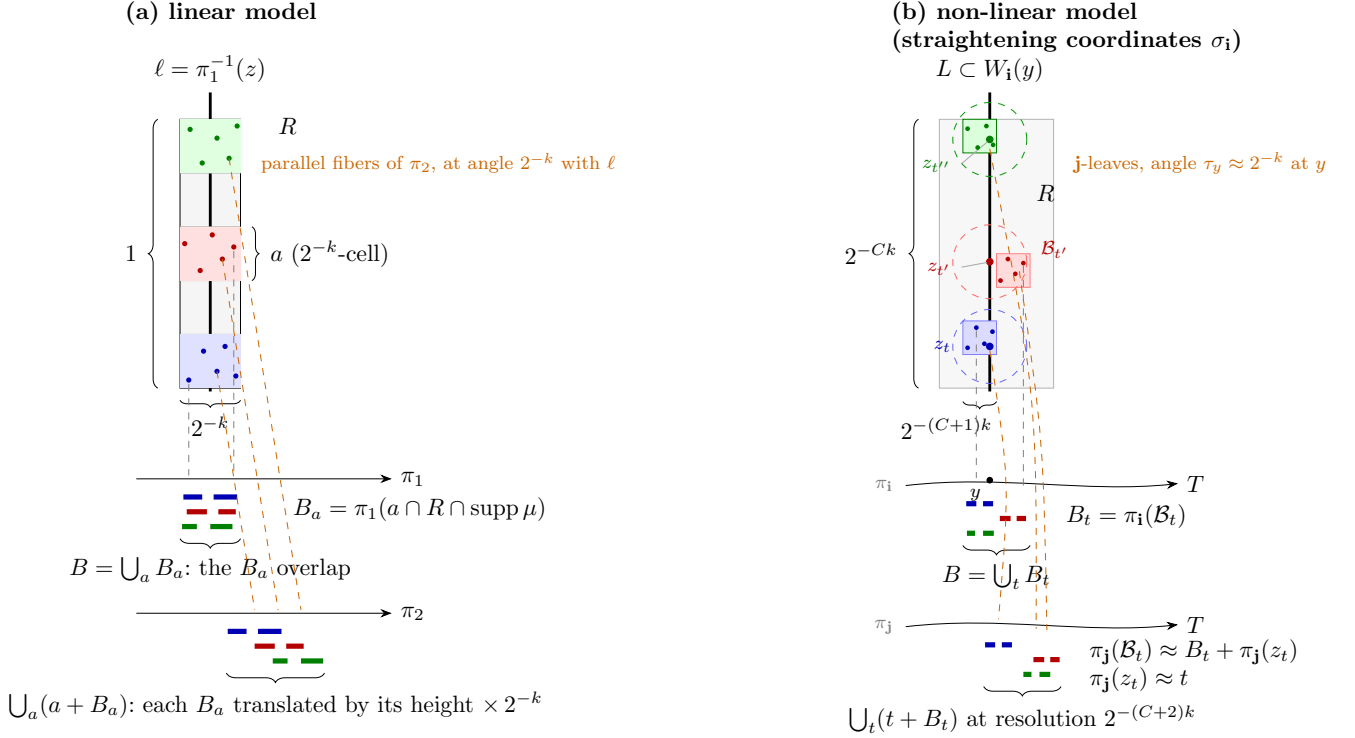

\subsubsection*{The dictionary}
In our setting, the two directions are two pasts $\bfi,\bfj\in\Sigma^-$.
Projecting along the $\bfi$-leaves, respectively $\bfj$-leaves, onto a transversal
$T$ through a base point $y$ defines local projections $\pi_\bfi,\pi_\bfj$,
and by \cref{lem:exact dim tran}, the two projected measures have the
\emph{same} dimension $\gamma_2$, while the fiber measures on the
$\bfi$-leaves have dimension $\gamma_1>0$. The role of the uncountable
exceptional set of \cite{wuProjectionTheoremsCountably2025} is played by
$su$-non-integrability, in the following quantitative form
(\cref{prop: small angle}): there is $\epsilon>0$, depending only on the
system, such that for every point $x\in\Lambda_f$, every $\eta>0$ and every
set $\Sigma'\subseteq\Sigma^-$ of pasts with $\nu_-(\Sigma')>1-\epsilon$,
there are two pasts $\bfi,\bfj\in\Sigma'$ with
\[
0<\angle\bigl(E^{su}_\bfi(x),E^{su}_\bfj(x)\bigr)<\eta .
\]

\subsubsection*{Non-linearity}
\textbf{Localization: } Since there is no global projection, we work in the exponential chart at a base
point and in the \emph{straightening coordinates} $\sigma_\bfi$ of
\cref{prop:straightening map}.
The linear picture is accurate only for boxes that are so small that the H\"older errors are negligible compared to the resolution at which we count. Concretely, for a pair of pasts at angle $2^{-k}$, we replace the $2^{-k}\times1$ box of the linear model with a box of size $2^{-(C+1)k}\times2^{-Ck}$ around a piece of the $\bfi$-leaf of length $2^{-Ck}$, and we count at resolution $2^{-(C+2)k}$. 

\textbf{Spreading on a window of scales: }Due to the localization, the spreading property of the fiber measure requires a positive proportion of  dyadic scales between $2^{-Ck}$ and
$2^{-(C+1)k}$ 
The spreading property itself is established in \cref{sec:nledyadic}. Wu obtains it from the theory of CP-distributions; here, it comes from the dynamics: 
the conditional
measures of a Bernoulli measure along strong unstable leaves are mapped to
one another by the inverse branches of $f$, up to normalization; 
the positivity of the fiber dimension  implies spreading at a single scale; and the maximal ergodic theorem
turns spreading at a single scale into a frequency statement over scales
that holds uniformly on a set of measure close to one, in the dyadic
partitions of the straightening coordinates.

\subsubsection*{Selecting the point}
The counting argument runs at a single base point $y$ and a single pair of
pasts $\bfi,\bfj$, which have to be good in several independent senses at
once (\Cref{sec: step1}): the fiber measure on $W_\bfi(y)$ is spreading on
the relevant window; the projections $\pi_\bfi\mu$ and $\pi_\bfj\mu$ obey
their dimension $\gamma_2$; the good set has density close to one both in
the box and on the leaf through $y$; and the angle at $y$ between
$E^{su}_\bfi$ and $E^{su}_\bfj$ is of order $2^{-k}$. The first two
properties hold for all pairs $(x,\bfi)$ in a set of $\hat\mu$-measure
close to one and for all $k$ beyond a threshold depending on the
property; in order to preserve local properties after the intersection,
we need to apply a density argument. By Fubini, for $\mu$-most base
points $y$, the set of pasts $\bfi$ for which $(y,\bfi)$ is good has
$\nu_-$-measure larger than $1-\epsilon$. The fourth property, which
concerns the pair of pasts, is obtained from \cref{prop: small angle}:
applied at $y$ to this set, with $\eta$ smaller than all the thresholds,
it produces two good pasts $\bfi,\bfj$ at a positive angle $\tau_y<\eta$.
The scale $k$ is then \emph{defined} by $2^{-k-1}<\tau_y\leq2^{-k}$ and
is automatically large.

\subsubsection*{Counting and the contradiction}
Assume $\gamma_1>0$ and $\gamma_2<1$. 
Let $L$ be the piece of the $\bfi$-leaf through $y$ of
length $2^{-Ck}$, and $R$ be the $2^{-(C+1)k}\times2^{-Ck}$ box around it.
Choose one good point $z_t$ of $L$ in each $2^{-(C+1)k}$-cell of $L$ that
carries fiber mass, let $\mathcal B_t$ be the part of the good set of $R$
within $2^{-(C+1)k}$ of $z_t$, and put $B_t=\pi_\bfi(\mathcal B_t)$,
$B=\bigcup_tB_t$. The transverse dimension gives
$\cN_{2^{-(C+2)k}}(B)\approx 2^{\gamma_2k}$.
The commutator lemma (\cref{rem:geom-input-commutator}) shows that $\pi_\bfj(\mathcal B_t)$ is, at resolution
$2^{-(C+2)k}$, a translation of $B_t$ by $\pi_\bfj(z_t)$. The displacement lemma (\cref{rem:geom-input-nonlinear-displacement})
shows that, at a finer resolution, $\pi_\bfj$ acts on $L$ as an affine
contraction by the angle: in the rescaled parameter, $\pi_\bfj(z_t)\approx t$.
Hence, $\pi_\bfj(R\cap\supp\mu)$ contains, at
resolution $2^{-(C+2)k}$, a sumset $\bigcup_t(t+B_t)$ of $B$ with a spreading
set of fiber points. 
The assumption $\gamma_2<1$ guarantees that $B$ is far from filling the
interval $\pi_\bfi(R)$. The sum-set estimate of \cref{sec:additive} then
applies and gives an increase by a definite power: for some $\delta>0$,
\[
\cN_{2^{-(C+2)k}}\bigl(\pi_\bfj(R\cap\supp\mu)\bigr)\gtrsim
\cN_{2^{-(C+2)k}}(B)^{1+\delta}.
\]
Thus, at scale $2^{-(C+2)k}$, the support of $\pi_\bfj\mu$ meets more
cells near $y$ than its dimension $\gamma_2$ permits, which results a contradiction . Therefore $\gamma_2=1$.

\subsubsection*{From measures to sets}

The passage from \Cref{thm:base} to
\Cref{thm:B,thm:A} is based on three approximation steps. Since
\[
\dim_{\mathrm H}\Lambda_f
\leq
\overline{\dim}_{\mathrm B}\Lambda_f
\leq s_0
\]
holds for every $C^1$ repeller \cite{zhang_dynamical_1997,BanCaoHu}, it remains to construct ergodic measures whose Hausdorff dimensions approach $s_0$.

For an $su$-non-integrable horseshoe, an SSE measure
(\Cref{def:SSE measure}) --- whose Ledrappier--Young dimension equals
$s_0$ (\Cref{prop:sse-max-lyap} and \Cref{lem:1.1}) --- is approximated in entropy and Lyapunov exponents by
$f$-invariant averages of Bernoulli measures for suitable iterates.
Applying \Cref{thm:base} at the iterate level and then averaging back to
$f$ yields the horseshoe dimension formula (\Cref{prop: irre-horseshoe}).

For \Cref{thm:B}, the horseshoe approximation of Cao--Pesin--Zhao \cite{cao_dimension_2019} produces horseshoes carrying measures with Lyapunov dimension tending to $s_0$. If such a horseshoe is $su$-integrable, one additional inverse branch from the ambient $su$-non-integrable component enlarges it to an $su$-non-integrable horseshoe without losing the approximation. 
For \Cref{thm:A}, the average-conformal and zero-entropy cases are
treated separately. In the remaining case, given a $C^r$ map and
$\epsilon>0$, the horseshoe produced by the Cao--Pesin--Zhao argument
\cite{cao_dimension_2019}, with Lyapunov dimension larger than
$s_0-\epsilon$, persists under $C^r$ perturbation. The density result
(\Cref{appendix:B}), applied to the related horseshoe repeller, shows that its continuation is $su$-non-integrable for an
open dense set of nearby maps. Pressure continuity and a Baire-category
argument then give the required residual dimension formula.

\subsection{Structure of the paper}
In \cref{sec:2.1} we collect the preliminaries: dimensions of sets and measures, inverse limits, dominated splittings, and $su$-non-integrability, the symbolic coding of repellers, Oseledets' theorem, and the sub-additive pressure and Lyapunov dimension for repellers and for IFSs. 
\Cref{sec:regularity} establishes the Hölder regularity of the strong unstable
directions and develops the quantitative consequences of
non-integrability. \Cref{sec:nledyadic} constructs the local strong unstable
leaves and the associated straightening coordinates and verifies
the spreading properties of the fiber measures.
\Cref{sec:additive} gathers the sumset estimates from additive combinatorics. \Cref{sec: sec5} is the core of the paper: it contains the selection of the base point and of the pair of pasts, the planar-geometry lemmas, and the proofs of \Cref{thm:main-two} and \Cref{thm:base}. \Cref{sec: Bernoulli horseshoe approx} proves \cref{prop: irre-horseshoe} and deduces \Cref{thm:B} (\Cref{thm:dominated}) and \Cref{thm:A} (\Cref{thm:generic_full}). \Cref{sec: IFS and Weierstrass} contains the proofs of the results on planar non-linear IFSs (\Cref{thm:ifs-non-perturbative,thm:ifs}) and on Weierstrass-type functions (\Cref{thm:ren-shen-c}). \Cref{appendix:B} proves that $su$-non-integrability is open and dense in the $C^r$ topology, and \cref{sec:counterexample} contains examples showing that the hypothesis is necessary in \Cref{thm:A} and \Cref{thm:B}.

\subsection{Acknowledgment}
We thank Yongluo Cao, Shaobo Gan, Meng Wu, Meysam Nassiri, and Yun Zhao
for useful comments and discussions. We also thank Jing Zhou, Changguang Dong, Shucheng Yu, and the fourth named author for organizing the conference at Great Bay University in December 2025, where the key ideas were generated. W. Pan is supported by NSERC Discovery Grant RGPIN-2025-06706. D. Xu is supported by National Key R$\&$D Program of China No. 2024YFA1015100.

\subsection{AI disclosure}
The authors used AI tools to assist with literature searches, proof verification and checking, figure preparation, and manuscript editing. 
All main ideas and the mathematical contributions are solely attributed to the authors.

\section{Preliminaries}\label{sec:2.1}
In this section we set up the basic notions we need. We first fix the dimension conventions, then introduce the inverse-limit and symbolic settings, and finally recall the sub-additive pressure and the Lyapunov dimension used in the main statements. Throughout, $M$ is a Riemannian manifold and $f\colon M\to M$ is a $C^{1+\alpha}$ map, $0<\alpha\leq1$, admitting a repeller $\Lambda$; since $C^r\subseteq C^{1+\alpha}$ with $\alpha=\min\{r-1,1\}$, this covers the $C^r$ maps of the introduction.

\subsection{Hausdorff, box, and pointwise dimensions}\label{subsec:dimensions}
We fix the dimension conventions used throughout. All metric notions are taken
with respect to the Riemannian distance $d_M$. For $Z\subset M$ and
$s\geq 0$, define
\[
\mathcal H^s_\rho(Z)
:=
\inf\left\{
\sum_{k=1}^{\infty}(\operatorname{diam} U_k)^s:
Z\subseteq \bigcup_{k=1}^{\infty}U_k,\ 
\operatorname{diam}U_k\leq \rho
\right\},
\qquad
\mathcal H^s(Z):=\lim_{\rho\to 0}\mathcal H^s_\rho(Z).
\]
The \emph{Hausdorff dimension} of $Z$ is
\[
\dim_{\mathrm H}Z
:=
\inf\{s\geq 0:\mathcal H^s(Z)=0\}
=
\sup\{s\geq 0:\mathcal H^s(Z)=\infty\}.
\]
For a totally bounded set $Z\subset M$, let $N(Z,\rho)$ be the smallest
number of $d_M$-balls of radius $\rho$ needed to cover $Z$. Its lower and
upper box dimensions are
\[
\underline{\dim}_{\mathrm B}Z
:=
\liminf_{\rho\to 0}
\frac{\log N(Z,\rho)}{-\log \rho},
\qquad
\overline{\dim}_{\mathrm B}Z
:=
\limsup_{\rho\to 0}
\frac{\log N(Z,\rho)}{-\log \rho}.
\]
If these two quantities coincide, their common value is called the
\emph{box dimension} of $Z$, denoted by $\dim_{\mathrm B}Z$ (see e.g.
\cite{federer_geometric_1969,falconer_fractal_geometry_2014}).

Let $\mu$ be a Borel probability measure on $M$. The Hausdorff dimension of $\mu$ is
defined by
\[
\dim_{\mathrm H}\mu
:=
\inf\left\{
\dim_{\mathrm H}Z:
Z\subseteq M \text{ is Borel and } \mu(Z)=1
\right\}.
\]
For $\delta\in(0,1)$, set
\[
N(\mu,\rho,\delta)
:=
\inf\left\{
N(Z,\rho):
Z\subseteq M \text{ is Borel and } \mu(Z)>1-\delta
\right\}.
\]
The lower and upper box dimensions of $\mu$ are defined by
\[
\underline{\dim}_{\mathrm B}\mu
:=
\lim_{\delta\to 0}
\liminf_{\rho\to 0}
\frac{\log N(\mu,\rho,\delta)}{-\log \rho},
\qquad
\overline{\dim}_{\mathrm B}\mu
:=
\lim_{\delta\to 0}
\limsup_{\rho\to 0}
\frac{\log N(\mu,\rho,\delta)}{-\log \rho}.
\]
If these two values are equal, we denote their common value by
$\dim_{\mathrm B}\mu$. For $x\in\supp\mu$, the lower and upper pointwise dimensions of $\mu$ at
$x$ are defined by
\[
\underline d_\mu(x)
:=
\liminf_{\rho\to 0}
\frac{\log \mu(B(x,\rho))}{\log \rho},
\qquad
\overline d_\mu(x)
:=
\limsup_{\rho\to 0}
\frac{\log \mu(B(x,\rho))}{\log \rho},
\]
where $B(x,\rho)$ denotes the $d_M$-ball centered at $x$ with radius
$\rho$. If these two limits coincide, their common value is denoted by
$d_\mu(x)$ and is called the pointwise dimension of $\mu$ at $x$ (see e.g. \cite{young_dimension_1982,fan_relationships_nodate}).

We say that $\mu$ is \emph{exact dimensional} if there exists a constant
$\dim \mu\geq 0$ such that
\[
\underline d_\mu(x)=\overline d_\mu(x)=\dim \mu
\qquad\text{for $\mu$-a.e. }x.
\]
In this case, by \cite{young_dimension_1982},
\[
\dim_{\mathrm H}\mu
=
\underline{\dim}_{\mathrm B}\mu
=
\overline{\dim}_{\mathrm B}\mu
=
\dim_{\mathrm B}\mu
=
\dim \mu.
\]
For exact dimensionality and dimension formulas for hyperbolic invariant
measures, see
\cite{young_dimension_1982,LEDRAPPIER_YOUNG_B_1985,barreira_pesin_schmeling_1999,liu-xie06,Shu10,Part2}; for invariant measures of iterated function systems and self-affine measures, see e.g. \cite{feng2009dimension,barany_ledrappier-young_2015,barany_ledrappieryoung_2017}.

\subsection{Inverse limit space and dominated splittings}\label{subsec:inverse limit}
Since $f|_{\Lambda}$ is an expanding but non-invertible map, capturing its full dynamics requires tracking the pre-histories of points. The \emph{inverse limit space} $(\hat{\Lambda},\hat{f})$ associated with $(\Lambda,f)$ is defined as follows.
\begin{itemize}
\item $\hat{\Lambda}$ is the closed subset of $\Lambda^{\mathbb{Z}_{\leq 0}}$ given by
\begin{equation}
\label{eqn:hatlambda}
\hat{\Lambda} := \left\{ \hat{x} = \{x_j\}_{j\leq 0} : x_j \in \Lambda \text{ and } f(x_{j-1}) = x_j \text{ for all } j \leq 0 \right\}, 
\end{equation}
equipped with the metric
\begin{equation}
\hat{d}_a(\hat{x},\hat{y}) = \sum_{j=0}^{\infty} a^{j} d_{M}(x_{-j},y_{-j}), \label{eqn:da metric}
\end{equation}
where $a\in (0,1)$. We also write $\hat d=\hat d_{1/2}$.
\item The homeomorphism $\hat{f}\colon\hat{\Lambda}\to \hat{\Lambda}$ is the left shift
$$\hat{f}((\ldots,x_{-1},x_0)) = (\ldots, x_{-1},x_0,f(x_0)).$$
\end{itemize}
The natural projection
\begin{equation}
\label{eqn:projection inverse limit space}
\pi\colon(\hat{\Lambda},\hat{f})\to (\Lambda,f), \quad \pi(\hat{x})=x_0,
\end{equation}
is a semi-conjugacy, $\pi\circ \hat{f} = f\circ \pi$. Since $\Lambda$ is
compact, {{$\hat\Lambda$ is compact as well, and the topology induced by the metric $\hat d_a$, $a\in(0,1)$, is compatible with the product topology}}. For an $f$-invariant Borel
probability measure $\mu$ on $\Lambda$, we denote by $\hat\mu$ the unique
$\hat f$-invariant Borel probability measure on $\hat\Lambda$ such that
$\pi_*\hat\mu=\mu$; the natural extension preserves ergodicity and metric
entropy. For the theory of inverse-limit spaces and natural extensions, see for instance
\cite{Rokhlin_1967,Przytycki1976,KatokHasselblatt1995,QianXiebook}.

\begin{defi}\label{def:dominated splitting}
We say that $(\Lambda,f)$ admits a \emph{dominated splitting} if there is a continuous splitting
\begin{equation}
\label{eqn:TM splitting}
T_{\pi(\hat{x})}M = E_1(\hat x) \oplus E_2(\hat{x}),\qquad \hat x\in\hat\Lambda,
\end{equation}
with the following properties.
\begin{itemize}
\item \textbf{($\hat{f}$-invariance):} $D_{\pi(\hat x)} f\, E_i(\hat{x}) = E_i(\hat{f}\hat{x})$ for $i=1,2$ and every $\hat{x}\in \hat{\Lambda}$.
\item \textbf{(Domination):} there exists $n_0\in \mathbb{N}$ such that for every $\hat{x}\in \hat{\Lambda}$,
$$\frac{m\bigl(D_{\pi(\hat x)} f^{n_0}|_{E_2(\hat{x})}\bigr)}{\bigl\|D_{\pi(\hat x)} f^{n_0}|_{E_1(\hat{x})}\bigr\|} \geq 2,$$
where $m(A):=\|A^{-1}\|^{-1}$ is the conorm of the linear operator $A$.
\end{itemize}
\end{defi}
 
For the general theory of dominated splittings, see, for instance, \cite{bonatti_diaz_viana_2005,KatokHasselblatt1995}.
A dominated splitting is unique, the angle between $E_1(\hat x)$ and $E_2(\hat x)$ is uniformly bounded
away from $0$, and dominated splittings persist under sufficiently small
$C^1$ perturbations. Moreover, the weak bundle is characterized by forward growth. Consequently, $E_1(\hat x)$ depends only on the base point $x_0=\pi(\hat x)$, whereas $E_2(\hat x)$ may depend on the entire pre-history. This is the crucial difference from the invertible setting, and we write accordingly
\[
E^{wu}(\pi(\hat x)):=E_1(\hat x),\qquad E^{su}(\hat x):=E_2(\hat x),
\]
for the \emph{weak unstable} and \emph{strong unstable} bundles.

\subsection{The non-integrability assumption}\label{subsec:non-integrability}
By the spectral decomposition theorem, see for instance \cite{smale_differentiable_1967,Przytycki1976,KatokHasselblatt1995}, any $C^{1+\alpha}$ repeller $(\Lambda_f,f)$ is a finite disjoint union of transitive repellers $(\Lambda_{i,f},f)$ with $f^{-1}(\Lambda_{i,f})\cap\Lambda_f = \Lambda_{i,f}$. 

\begin{defi}\label{def:su-nonintegrable}
A transitive repeller $(\Lambda_f,f)$ is \emph{$su$-non-integrable} if it admits a dominated splitting
$$T_{\pi(\hat x)}M = E^{wu}(\pi(\hat x)) \oplus E^{su}(\hat x)$$
over its inverse limit space, and there exists a base point $x\in\Lambda_f$ such that
$$\#\bigl\{E^{su}(\hat x)\in \mathbb P(T_xM) : \hat x\in\pi^{-1}(x) \bigr\} \geq 2 .$$
A non-transitive repeller is \emph{$su$-non-integrable} if all of its transitive components are $su$-non-integrable, and a repeller is \emph{$su$-integrable} if it admits a dominated splitting and is not $su$-non-integrable.
\end{defi}

The following lemma shows that $su$-non-integrability at one point implies it at every point.

\begin{lem}
\label{lem:pointwise-su-nonintegrability}
Let $(\Lambda_f,f)$ be a transitive $C^{1+\alpha}$ surface repeller admitting a dominated splitting.
Assume that there exist $p\in\Lambda_f$ and two histories
$\hat p^1,\hat p^2\in\hat\Lambda_f$ with
\[
\pi(\hat p^1)=\pi(\hat p^2)=p,
\qquad
\mathbb P(E^{su}(\hat p^1))\neq \mathbb P(E^{su}(\hat p^2)).
\]
Then for every $x\in\Lambda_f$, there exist two histories
$\hat x^1,\hat x^2\in\hat\Lambda_f$ with
\[
\pi(\hat x^1)=\pi(\hat x^2)=x,
\qquad
\mathbb P(E^{su}(\hat x^1))\neq \mathbb P(E^{su}(\hat x^2)).
\]
\end{lem}
\Cref{lem:pointwise-su-nonintegrability} is proved in \cref{sec:3.2}.

\subsection{Markov coding and horseshoe repellers}\label{subsec:coding}
We briefly recall the symbolic coding of repellers; see, for example,
\cite{Przytycki1976,KatokHasselblatt1995,QianXiebook}.
For any $C^{1+\alpha}$ repeller $(\Lambda,f)$,
there exist a finite zero-one matrix
\[
A=(A_{ij})_{1\leq i,j\leq m},
\]
the associated one-sided subshift of finite type
\[
\Sigma_A^+
:=
\left\{
\omega=(\omega_n)_{n\geq0}\in\{1,\ldots,m\}^{\mathbb Z_{\geq0}}:
A_{\omega_n\omega_{n+1}}=1
\text{ for every }n\geq0
\right\},
\]
and a finite to one continuous surjective coding map
\[
\Pi\colon\Sigma_A^+\longrightarrow\Lambda
\]
such that
\[
\Pi\circ\sigma=f\circ\Pi.
\]
A finite word $w=i_0\cdots i_{n-1}$ is called \emph{admissible} if
$A_{i_j i_{j+1}}=1$ for $0\leq j<n-1$. We define the cylinders of repellers as
\[
[w]
:=
\left\{
\omega\in\Sigma_A^+:
\omega_j=i_j,\ 0\leq j<n
\right\},
\qquad
\Lambda[w]:=\Pi([w]).
\]
The coding can be chosen so that there exist constants $D>0$ and
$\vartheta\in(0,1)$ with
\[
\operatorname{diam}\Lambda[w]\leq D\vartheta^n
\]
for every admissible word $w$ of length $n$. In particular, every
nonempty relatively open subset of $\Lambda$ contains $\Lambda[w]$ for
some admissible word $w$.
If $(\Lambda,f)$ is transitive, then $A$ can be chosen to be irreducible, so
any two admissible words can be joined by a finite admissible word. 

\begin{defi}\label{def:horseshoe repeller}
A repeller $(\Lambda_f,f)$ is called a \emph{horseshoe repeller} if, for
some $m\geq2$, there exists a homeomorphism
\begin{equation}
\label{eqn:horseshoe def}
h\colon\Sigma_m^+\longrightarrow\Lambda_f,
\qquad
\Sigma_m^+:=\{1,\ldots,m\}^{\mathbb Z_{\geq0}},
\end{equation}
such that $h\circ\sigma=f\circ h$.
\end{defi}
Then the sets $\Lambda_i:=h([i])$, $1\leq i\leq m$, form a pairwise
disjoint clopen partition of $\Lambda_f$, and each
$f|_{\Lambda_i}\colon\Lambda_i\to\Lambda_f$ is a homeomorphism.
Since $f$ is a local diffeomorphism near $\Lambda_f$, there exists an
open neighborhood $U$ of $\Lambda_f$ such that the inverse branches
$(f|_{\Lambda_i})^{-1}$ extend to $C^{1+\alpha}$ maps
$f_i^{-1}\colon U\to M$ with pairwise disjoint images $U_i:= f_i^{-1}(U)$.

\begin{defi}\label{def:Bernoulli measure}
Let $(\Lambda_f,f)$ be a horseshoe repeller.
A vector $\mathbf p=(p_1,\ldots,p_m)$ is called a \emph{probability vector} if
$p_i>0$ and $\sum_{i=1}^m p_i=1$. The Bernoulli product measure
$\nu_{\mathbf p}$ on $\Sigma_m^+$ is determined by
\[
\nu_{\mathbf p}([i_0\cdots i_{n-1}])
=p_{i_0}\cdots p_{i_{n-1}}
\]
for every cylinder $[i_0\cdots i_{n-1}]$. An $f$-invariant measure $\mu$ on $\Lambda_f$ is called \emph{$f$-Bernoulli}
if $\mu=h_*\nu_{\mathbf p}$ for some Bernoulli product measure
$\nu_{\mathbf p}$ on $\Sigma_m^+$.
\end{defi}
Since $p_i>0$ for all $i$, every $f$-Bernoulli measure is ergodic and fully supported on $\Lambda_f$.

{{For a horseshoe repeller, the inverse limit is coded by the two-sided shift. Let
\begin{equation}\label{eqn:pasts}
\Sigma^-:=\{1,\ldots,m\}^{\mathbb Z_{<0}}
\end{equation}
be the space of \emph{pasts}. Since every $x\in\Lambda_f$ has exactly one preimage in each $\Lambda_i$, the pre-histories of $x$ are in bijection with $\Sigma^-$: for $\bfi=(\ldots,i_{-2},i_{-1})\in\Sigma^-$, we write $\hat x_\bfi$ for the pre-history $(\ldots,x_{-2},x_{-1},x)$ with $x_{-j}\in\Lambda_{i_{-j}}$, and $(\bfi,x)\mapsto\hat x_\bfi$ is a homeomorphism $\Sigma^-\times\Lambda_f\to\hat\Lambda_f$. Under this identification, the natural extension of an $f$-Bernoulli measure $\mu=h_*\nu_{\mathbf p}$ is the product $\hat\mu=\nu_-\times\mu$, where $\nu_-$ is the Bernoulli measure with weights $\mathbf p$ on $\Sigma^-$.}}

\subsection{The $C^r$ topology and continuations of repellers}\label{subsec:continuation}
For $r\in(1,\infty]$ we equip $C^r(M,M)$ with the $C^r$ topology of uniform
convergence of derivatives up to order $r$ on compact subsets of $M$ (of all
derivatives when $r=\infty$). This is a Baire space, so its residual subsets,
those containing a countable intersection of open dense subsets, are dense.
Repellers are structurally stable \cite{shub_endomorphisms_1969,przytycki_omega_1977}: if $(\Lambda_f,f)$ is a $C^r$ repeller with
isolating neighborhood $U$, there is a $C^r$-open neighborhood $U_f$ of $f$
such that for every $g\in U_f$ the set
\[
\Lambda_g:=\{x\in U: g^n(x)\in U\text{ for all }n\geq0\}
\]
is a repeller of $g$ with isolating neighborhood $U$, $(\Lambda_g,g)$ is
topologically conjugate to $(\Lambda_f,f)$ by a homeomorphism close to the
identity, and $\Lambda_g$ depends continuously on $g$ in the Hausdorff
metric. We call $\Lambda_g$ the \emph{continuation} of $\Lambda_f$; the
continuation of a surface repeller with a dominated splitting again has a
dominated splitting.

\subsection{Oseledets' multiplicative ergodic theorem for repellers}\label{subsec:oseledets}
For a linear map $A$ between $d$-dimensional inner product spaces let
$\alpha_1(A)\geq\cdots\geq\alpha_d(A)\geq0$ denote its singular values, and for $x\in\Lambda_f$ and $n\geq1$ write
\[
\alpha_j(x,f^n):=\alpha_j(D_xf^n),\qquad 1\leq j\leq d .
\]
We recall the following form of Oseledets' multiplicative ergodic theorem \cite{oseledets_multiplicative_1968}
adapted to repellers; see \cite{QianXiebook} for instance.

\begin{thm}[Oseledets' multiplicative ergodic theorem]
\label{thm:oseledets-repeller}
Let $M$ be a $d$-dimensional Riemannian manifold. Let
$(\Lambda_f,f)$ be a $C^{1+\alpha}$ repeller and $\mu$ an $f$-ergodic measure on $\Lambda_f$.
Denote by $\hat\mu$ the natural-extension lift of $\mu$ to
$(\hat\Lambda_f,\hat f)$.

Then there exist Lyapunov exponents, counted with multiplicity,
\begin{equation}\label{eq:def of LYexponent}
\lambda_1(\mu,f)\geq\cdots\geq\lambda_d(\mu,f)>0,
\end{equation}
and an $\hat f$-invariant Borel set
$\hat\Lambda_\mu\subseteq\hat\Lambda_f$ with
$\hat\mu(\hat\Lambda_\mu)=1$, such that for every
$\hat x\in\hat\Lambda_\mu$ there is a measurable splitting
\[
T_{\pi(\hat x)}M
=
\bigoplus_{\lambda\in
\{\lambda_1(\mu,f),\ldots,\lambda_d(\mu,f)\}}
E_\lambda(\hat x),
\]
where the direct sum ranges over the distinct values of the Lyapunov
exponents and
\[
\dim E_\lambda(\hat x)
=
\#\{1\leq j\leq d:\lambda_j(\mu,f)=\lambda\}.
\]
Moreover,
\[
D_{\pi(\hat x)}f\bigl(E_\lambda(\hat x)\bigr)
=
E_\lambda(\hat f\hat x),
\]
and, for every $v\in E_\lambda(\hat x)\setminus\{0\}$,
\[
\lim_{n\to\infty}
\frac1n\log\left\|D_{\pi(\hat x)}f^n v\right\|
=
\lambda.
\]
Finally, for $\mu$-almost every $x\in\Lambda_f$,
\begin{equation}\label{eq:oseledets-singular-values}
\lim_{n\to\infty}
\frac1n\log\alpha_j(x,f^n)
=
\lambda_j(\mu,f),
\qquad 1\leq j\leq d.
\end{equation}
\end{thm}
\begin{rem}
Let $\mu$ and $\hat\mu$ be as in \Cref{thm:oseledets-repeller}. If a surface
repeller $(\Lambda_f,f)$ admits a dominated splitting
$E^{wu}\oplus E^{su}$, then the domination condition gives $\lambda_1(\mu,f)-\lambda_2(\mu,f)\geq\log 2/n_0>0$, and the Oseledets splitting coincides $\hat\mu$-a.e.\ with the dominated one, $E_{\lambda_1}=E^{su}$ and $E_{\lambda_2}=E^{wu}$. Consequently
    \begin{align}\label{eq:Lyexpo1}
\lambda_1(\mu,f)
&=
\int_{\hat\Lambda_f}\log\|Df|_{E^{su}(\hat x)}\|\,\mathrm d\hat\mu(\hat x),\\
\label{eq:Lyexpo2}
\lambda_2(\mu,f)
&=
\int_{\Lambda_f}\log\|Df|_{E^{wu}(x)}\|\,\mathrm d\mu(x).
\end{align}
In particular, since $E^{su}$ and $E^{wu}$ are continuous, $\lambda_1(\mu,f)$ and $\lambda_2(\mu,f)$ depend continuously on $\hat\mu$ in the weak-$*$ topology.
\end{rem}

\subsection{Sub-additive topological pressures and Lyapunov dimensions}
\label{subsubsec:sub-additive pressures}
Let $(\Lambda_f,f)$ be a $C^{1+\alpha}$ repeller. For any compact subset $X\subset M$,
denote
\begin{align}
\mathcal M(X,f)&:=\{\mu:\mu \text{ is an $f$-invariant probability measure and } \supp \mu\subseteq X\},\\
\mathcal M_e(X,f)&:=\{\mu:\mu \text{ is an $f$-ergodic probability measure and } \supp \mu\subseteq X\}.
\label{eqn:MeXf}
\end{align}
The sub-additive topological pressure used below belongs to the non-additive thermodynamic formalism of Falconer \cite{falconer_subadditive_1988}, Barreira \cite{barreira_nonadditive_1996} and Cao, Feng and Huang \cite{cao_thermodynamic_2008}; for variational principles, approximation results and dimension-theoretic applications see \cite{feng2009dimension,cao_dimension_2019,morris_equality_2016,morris_variational_2023,bochi_morris_equilibrium_2018}.

For any $s\in[0,\dim M]$, the \emph{sub-additive singular potentials}, following Cao, Pesin and Zhao \cite{cao_dimension_2019}, are defined by
\begin{equation}
\label{eqn:singular potential}
\varphi^s_n(x) = \sum_{i=\dim M-[s]+1}^{\dim M}\log \alpha_i(x,f^n) + (s-[s])\log\alpha_{\dim M-[s]}(x,f^n),
\end{equation}
where $\alpha_1(x,f^n)\geq\cdots\geq\alpha_{\dim M}(x,f^n)>0$ are the singular values of $D_xf^n$ defined above, $[s]$ denotes the integer part of $s$, and
the last term is understood to vanish when $s=\dim M$.
The sequence $(-\varphi^s_n)_{n\geq1}$ is sub-additive. We define the \emph{sub-additive topological pressure} of the singular potentials by
\begin{equation}
\label{eqn:sub-additive potential}
P_{\mathrm{sub}}(s) := \sup_{\mu\in\mathcal{M}_e(\Lambda_f,f)} \left\{ h_\mu(f) - \lim_{n\to\infty}\frac{1}{n}\int\varphi^s_n\,\mathrm d\mu \right\},
\end{equation}
where the limit exists by sub-additivity. By the variational principle of Cao, Feng and Huang \cite{cao_thermodynamic_2008}, see also \cite[Equation 2.4]{cao_dimension_2019}, \eqref{eqn:sub-additive potential} coincides with the topologically defined sub-additive pressure of $(-\varphi^s_n)_n$, and the supremum may equivalently be taken over $\mathcal M(\Lambda_f,f)$; see \cite[Proposition A.1]{feng_lyapunov_2010} for the reduction from invariant to ergodic measures.

Since $f$ is uniformly expanding, the Lyapunov exponents of every $\mu\in\mathcal M_e(\Lambda_f,f)$ satisfy $\log\kappa\leq\lambda_{\dim M}(\mu,f)\leq\lambda_1(\mu,f)\leq\log\sup_{x\in\Lambda_f}\|D_xf\|$. Hence
\eqref{eq:oseledets-singular-values} and the dominated convergence theorem yield
\begin{equation}
P_{\mathrm{sub}}(s) = \sup_{\mu\in\mathcal{M}_e(\Lambda_f,f)} \left\{ h_\mu(f) - \sum_{i=\dim M-[s]+1}^{\dim M}\lambda_i(\mu,f) - (s-[s])\lambda_{\dim M-[s]}(\mu,f) \right\} \label{eq:sub-pressure-lyapunov}.
\end{equation}
The bracket in
\eqref{eq:sub-pressure-lyapunov} is a continuous, piecewise affine function
of $s$ with slopes at most $-\log\kappa$; taking the supremum,
$P_{\mathrm{sub}}$ is Lipschitz continuous and strictly decreasing. Moreover, $P_{\mathrm{sub}}(0)=h_{\mathrm{top}}(f|_{\Lambda_f})\geq0$ by the classical variational principle, and $P_{\mathrm{sub}}(\dim M)\leq0$ by Ruelle's inequality \cite{ruelle_entropy_1978}. Hence $P_{\mathrm{sub}}$ has a unique zero $s_0$.

Given an $f$-ergodic measure $\mu$ on $\Lambda_f$, its \emph{Lyapunov dimension} $\dim_{\mathrm{LY}}(\mu,f)$ is defined as the unique zero $t_\mu\in[0,\dim M]$ of the function
\begin{equation}
\label{eqn:def LY}
F_\mu(t) := h_\mu(f) - \sum_{i=\dim M-[t]+1}^{\dim M}\lambda_i(\mu,f) - (t-[t])\lambda_{\dim M-[t]}(\mu,f).
\end{equation}
It is well defined since $F_\mu$ is
continuous and strictly decreasing in $t$, $F_\mu(0)=h_\mu(f)\geq0$, and $F_\mu(\dim M)\leq0$ by Ruelle's inequality. For a surface, $\dim_{\mathrm{LY}}(\mu,f)$ is given by the explicit formula \eqref{eq:intro-LY}. Since $P_{\mathrm{sub}}(s)=\sup_\mu F_\mu(s)$ and the slopes of the $F_\mu$ are bounded away from $0$, one has
\begin{equation}
\label{eqn:s0 variational}
s_0=\sup_{\mu\in\mathcal M_e(\Lambda_f,f)}\dim_{\mathrm{LY}}(\mu,f),
\end{equation}
which is \eqref{eq:intro-s0}: indeed $F_\mu(s)>0$ for some $\mu$ whenever $s<\sup_\mu t_\mu$, while $F_\mu(s)\leq-\log\kappa\,(s-\sup_\mu t_\mu)<0$ for all $\mu$ whenever $s>\sup_\mu t_\mu$.

\subsection{Uniform non-conformality, weak irreducibility and pressures for IFSs}\label{subsec:IFS}
Let
$\Phi=\{\phi_i\}_{i=1}^m$
be a planar $C^{1+\alpha}$ contracting IFS consisting of diffeomorphisms
onto their images, with attractor $\Lambda_\Phi$, and suppose that $\Phi$
satisfies the Strong Separation Condition. Let $\Sigma^+:=\{1,\ldots,m\}^{\mathbb Z_{\geq0}}$ be the coding space. For a word
$\mathbf i=(i_0,\ldots,i_{n-1})$, write
\[
\Phi_{\mathbf i}^n
:=
\phi_{i_0}\circ\cdots\circ\phi_{i_{n-1}} ;
\]
for $\mathbf i\in\Sigma^+$ we use the same notation for its initial word of length $n$.
Under the SSC, define the associated expanding inverse-branch repeller by
\[
f_\Phi|_{\phi_i(\Lambda_\Phi)}:=\phi_i^{-1},\qquad 1\leq i\leq m ;
\]
$f_\Phi$ extends to a $C^{1+\alpha}$ map of a neighborhood of $\Lambda_\Phi$, and $(\Lambda_\Phi,f_\Phi)$ is a horseshoe repeller. If $y=\Phi^n_{\mathbf i}(x)$, then
\[
(D_yf_\Phi^n)^{-1}=D_x\Phi_{\mathbf i}^n,
\]
so the singular values of $D_x\Phi^n_{\mathbf i}$ are the reciprocals of those of $D_yf^n_\Phi$ in reverse order. A pre-history of $x\in\Lambda_\Phi$ under $f_\Phi$ is the same thing as a sequence $\mathbf i\in\Sigma^+$, namely $(\ldots,\phi_{i_0}(x),x)$; we write $\hat x_{\mathbf i}$ for it. The dynamical conditions on $\Phi$ are the conditions of the introduction for $f_\Phi$, expressed through the maps $\phi_i$.
 
\emph{Uniform non-conformality.} We say that $\Phi$ is \emph{uniformly non-conformal} if there exist constants $D>0$ and $\epsilon>0$ such that for every $\mathbf i\in\Sigma^+$, every $n\geq1$ and every $x\in\Lambda_\Phi$,
\[
\|D_x\Phi^n_{\mathbf i}\|\geq D e^{\epsilon n}\sqrt{|\det D_x\Phi^n_{\mathbf i}|},
\]
i.e.\ $\alpha_1(D_x\Phi^n_{\mathbf i})/\alpha_2(D_x\Phi^n_{\mathbf i})\geq D^2e^{2\epsilon n}$. Geometrically, the images of infinitesimal circles under the IFS are ellipses whose eccentricities grow exponentially (see also \cite{hueter_falconers_1995}). Uniform non-conformality is equivalent to the existence of a dominated splitting for $f_\Phi$ over the inverse limit of $(\Lambda_\Phi,f_\Phi)$, see \cite{bochi_gourmelon_domination_2009}. Under the forward iteration of the contractions the splitting decomposes the tangent space into a weakest and a strongest contracting direction, $E^{ws}(x)\oplus E^{ss}(x,\mathbf i)$, where $E^{ss}(x,\mathbf i):=E^{su}(\hat x_{\mathbf i})$ depends on the coding sequence $\mathbf i$.

\emph{Weak irreducibility.} In the study of planar self-affine IFSs an irreducibility condition prevents the linear parts from preserving a common line, which would otherwise degenerate the fractal into a carpet-like object whose dimension may be strictly smaller than the zero of the sub-additive pressure. The non-linear analogue is $su$-non-integrability of $f_\Phi$: we say that $\Phi$ is \emph{weakly irreducible} if there exist $x\in\Lambda_\Phi$ and $\mathbf i,\mathbf j\in\Sigma^+$ such that
\[
E^{ss}(x,\mathbf i)\neq E^{ss}(x,\mathbf j),
\]
that is, the strongest contracting direction depends on the coding sequence. For self-affine systems this covers both the strongly irreducible case and the non-diagonal upper-triangular case of \cite{barany_hochman_rapaport_2019}. 

\emph{Pressure and Lyapunov dimension.} For $A\in GL_2(\mathbb R)$, let
$\alpha_1(A)\geq\alpha_2(A)>0$
be its singular values. For $s\in[0,2]$, define the singular value function
\[
\varphi^s(A):=\alpha_1(A)^{\min\{s,1\}}\,\alpha_2(A)^{\max\{s-1,0\}},
\]
so that $\log\varphi^s(D_x\Phi^n_{\mathbf i})=-\varphi^s_n(y)$ for $y=\Phi^n_{\mathbf i}(x)$, with $\varphi^s_n$ the singular potential \eqref{eqn:singular potential} of $f_\Phi$.
The sub-additive singular-value pressure of $\Phi$ is
\begin{equation}
\label{eqn:pressure IFS}
P_\Phi(s)
:=
\lim_{n\to\infty}\frac1n
\log
\sum_{\mathbf i\in\{1,\ldots,m\}^n}
\sup_{x\in\Lambda_\Phi}
\varphi^s\bigl(D_x\Phi_{\mathbf i}^n\bigr).
\end{equation}
The limit exists, $P_\Phi$ is continuous and strictly decreasing, and we
denote its unique zero by $s_0(\Phi)$.
For $\nu\in\mathcal M_e(\Lambda_\Phi,f_\Phi)$, write
\[
\lambda_1(\nu,\Phi)
:=
\lambda_1(\nu,f_\Phi)
\geq
\lambda_2(\nu,\Phi)
:=
\lambda_2(\nu,f_\Phi)>0 ,
\]
and define the Lyapunov dimension of $\nu$ with respect to $\Phi$, $\dim_{\mathrm{LY}}(\nu,\Phi):=\dim_{\mathrm{LY}}(\nu,f_\Phi)$. By the relation between $D\Phi^n_{\mathbf i}$ and $Df^n_\Phi$ above, the potentials, pressures, pressure zeros, Lyapunov exponents and Lyapunov dimensions of $\Phi$ agree with the corresponding notions for the repeller $(\Lambda_\Phi,f_\Phi)$; in particular $P_\Phi=P_{\mathrm{sub}}$ and $s_0(\Phi)=s_0$. For these definitions and the variational principle, see
\cite{falconer_hausdorff_1988,falconer_subadditive_1988,
cao_thermodynamic_2008,feng_simon_dimension_part_i_2023}. 
\section{H\"older regularity and non-integrability}\label{sec:regularity}
\subsection{H\"older continuity of dominated splitting}

On a surface $M$, let $\mathbb{P}TM$ denote the projectivized tangent bundle. We fix a Riemannian metric on $\mathbb{P}TM$ and denote the associated distance by $d_{\mathbb{P}TM}$.

For later estimates, it will be convenient to compare this distance with a more explicit local expression. Fix $c>0$ smaller than the injectivity radius of $M$. Whenever $d_M(x,y)<c$, let
\[
\mathcal P_{x\to y}:T_xM\longrightarrow T_yM
\]
denote parallel transport along the unique minimizing geodesic from $x$ to $y$. Then, after possibly decreasing $c$, there exists a constant $C_{\rm{PT}}\geq1$ such that
\[
C_{\rm{PT}}^{-1}
\Big(
d_M(x,y)
+
\angle\big(\mathcal P_{x\to y}E_x,E_y\big)
\Big)
\leq
d_{\mathbb{P}TM}\big((x,E_x),(y,E_y)\big)
\]
and
\[
d_{\mathbb{P}TM}\big((x,E_x),(y,E_y)\big)
\leq
C_{\rm{PT}}
\Big(
d_M(x,y)
+
\angle\big(\mathcal P_{x\to y}E_x,E_y\big)
\Big)
\]
for all $(x,E_x),(y,E_y)\in\mathbb{P}TM$ with $d_M(x,y)<c$, where
\[
\angle(E,F)\in\left[0,\frac{\pi}{2}\right]
\]
denotes the angle between two lines in the same tangent space.

Thus, for the local H\"older estimates below, the distance on $\mathbb{P}TM$ may be measured equivalently by the sum of the distance between the base points and the angle between the corresponding directions after parallel transport.

\begin{prop}\label{prop Holder continuous}
    For any  $C^{1+\alpha}$ surface horseshoe repeller $(\Lambda_f,f)$ with dominated splitting, there exists a constant $C_0>0$ such that   
   for any $\mathbf{i}\in \Sigma^-$, the map
\begin{equation}
x\in \Lambda_f \mapsto (x, E^{su}(\hat{x}_{\mathbf{i}}))\in \ptm 
\end{equation}
    is $\alpha$-H\"older with the H\"older constant $C_0>1$, i.e. 
    $$d_{\mathbb PTM}\left(\left(x,E^{su}(\hat{x}_\bfi)\right),\left(y,E^{su}(\hat{y}_\bfi)\right)\right)\leq C_0\cdot d_{M}(x,y)^{\alpha},\quad\forall x,y\in\Lambda. $$
\end{prop}

To prove this proposition, we first {{recall}} the following version of the classical H\"older section theorem (\Cref{lem: Holder section thm}) in \cite{Pugh_foliation_1997}.

Let $E=X\times Y$, where $X$ is a compact metric space and $Y$ is a closed, bounded subset of a Banach space, and let $\pi_X:E\to X$ and $\pi_Y:E\to Y$ be the projection maps. Let $h:X\to X$ be a homeomorphism on the base, and $F:E\to E$ be a bundle map over $(X,h)$, i.e.,$$\pi_X\circ F=h\circ \pi_X.$$ 
\begin{equation}
\begin{tikzcd}
    E \arrow[r, "F"] \arrow[d, "\pi_X"'] & E \arrow[d, "\pi_X"] \\
    X \arrow[r, "h"'] & X
\end{tikzcd}
\end{equation}
Assume $F$ is a fiber contraction: for all $x\in X$ and all $y,y'\in Y$,
\begin{equation}
d_Y\left(\pi_Y(F(x,y)),\pi_Y(F(x,y'))\right)\leq \kappa_x d_Y(y,y'),
\end{equation}
where 
\begin{equation}
\kappa_{\mathrm{fib}}:=\sup\{\kappa_x:x\in X\}<1.
\end{equation}
We also assume the following conditions.
\begin{enumerate}[(1)]
\item $h$ satisfies that
\begin{equation}
a:=\inf\left\{\frac{d_X(h(x_1),h(x_2))}{d_X(x_1,x_2)}:x_1\neq x_2\in X\right\}>0.
\end{equation}
\item $F$ is $\alpha$-H\"{o}lder continuous: there exists $C>0$ such that for any $x,x'\in X$ and any $y\in Y$,
\begin{equation}
d_Y\left(\pi_Y(F(x,y)),\pi_Y(F(x',y))\right)\leq C d_X(x,x')^{\alpha}.
\end{equation}
\end{enumerate}

\begin{lem}[{\cite[Theorem 3.2]{Pugh_foliation_1997}}]
\label{lem: Holder section thm}
Under the above assumptions, there exists a unique $F$-invariant section $\sigma_F:X\to E$ in the sense that
\begin{equation}
\sigma_F\circ h=F\circ \sigma_F.
\end{equation}
Moreover, suppose $$\kappa_{\mathrm{fib}}\cdot a^{-\alpha}<1.$$ Then $\sigma_F$ is $\alpha$-H\"{o}lder continuous, i.e., if we write $\sigma_F(x)=(x,s(x))\in X\times Y$, then there exists $C>0$ such that
\begin{equation}
d_Y(s(x),s(x'))\leq Cd_X(x,x')^{\alpha} \quad \text{for all}\,\,x,x'\in X.
\end{equation}
\end{lem}

A map $\sigma:\hat{\Lambda}_f\to \mathbb{P}TM$ is called a {\textit{section}} if $\sigma(\hat{x})\in \mathbb{P}T_{x_0}M$ for any $\hat{x}=\{x_j\}_{j\leq 0}\in \hat{\Lambda}_f$. Let $\hat{d}_a$ be the metric on $\hat{\Lambda}_f$ given in \eqref{eqn:da metric}.
 We prove the following lemma on approximation by Lipschitz sections.
\begin{lem}\label{lem: Lip section approximate}
     For any continuous section $\sigma_0:\hat \Lambda_f\to \mathbb PTM$ 
    and any $\epsilon>0$, there exists a Lipschitz section $\sigma_{1}:(\hat \Lambda_f,\hat d_a)\to \mathbb PTM$ such that
    \[
    \ d_{C^0}(\sigma_0,\sigma_1)<\epsilon.
    \]
\begin{proof}
Fix a smooth embedding of the manifold $M$ into Euclidean space
$\mathbb{R}^N$. For any $x\in M$, we have the inclusions
$T_xM\subseteq T_x\mathbb{R}^N$ and
$\mathbb{P}T_xM\subseteq \mathbb{P}T_x\mathbb{R}^N$. Let
$T_M\mathbb{R}^N$ and $\mathbb{P}T_M\mathbb{R}^N$ denote the restrictions
of the respective bundles to $M$. Then
\[
TM\subseteq T_M\mathbb{R}^N,
\qquad
\mathbb{P}TM\subseteq \mathbb{P}T_M\mathbb{R}^N.
\]
Furthermore, the trivialization of the ambient tangent bundle induces the
smooth trivialization
\begin{equation}
\mathbb{P}T_M\mathbb{R}^N
\cong
M\times\mathbb{R}P^{N-1}.
\end{equation}

We first justify the Lipschitz approximation used below. Let $X$ be a
compact metric space. The real-valued Lipschitz functions on $X$ form a
unital subalgebra of $C(X)$ that separates points; hence, they are
uniformly dense in $C(X)$ by the Stone--Weierstrass theorem. Fix a smooth
embedding
\[
\imath:\mathbb{R}P^{N-1}\longrightarrow\mathbb{R}^q
\]
for some $q\in\mathbb N$ and a smooth retraction
\[
r_{\mathbb P}:\mathcal V\longrightarrow
\imath(\mathbb{R}P^{N-1})
\]
from a tubular neighborhood $\mathcal V$ of
$\imath(\mathbb{R}P^{N-1})$. Given a continuous map
$g:X\to\mathbb{R}P^{N-1}$, we may approximate each coordinate of
$\imath\circ g$ uniformly by a Lipschitz function. If the approximation
is sufficiently close, its image is contained in $\mathcal V$, and
composition with $\imath^{-1}\circ r_{\mathbb P}$ gives a Lipschitz map
$X\to\mathbb{R}P^{N-1}$ arbitrarily close to $g$. Thus
\[
\operatorname{Lip}(X,\mathbb{R}P^{N-1})
\quad\text{is dense in}\quad
C(X,\mathbb{R}P^{N-1})
\]
in the uniform topology.

Let
\begin{equation}
\widetilde{\sigma}_0:
\hat{\Lambda}_f
\longrightarrow
\mathbb{P}T_M\mathbb{R}^N
\cong
M\times\mathbb{R}P^{N-1}
\end{equation}
be the continuous map obtained by composing $\sigma_0$ with the inclusion
$\mathbb{P}TM\subseteq\mathbb{P}T_M\mathbb{R}^N$.

Under the above trivialization, write
\[
\widetilde{\sigma}_0(\hat{x})
=
\bigl(\pi(\hat{x}),s_0(\hat{x})\bigr),
\]
where
\[
s_0:\hat{\Lambda}_f\longrightarrow\mathbb{R}P^{N-1}
\]
is continuous. Note that the natural projection
$\pi:(\hat{\Lambda}_f,\hat d_a)\to M$ is $1$-Lipschitz, since
\[
d_M\bigl(\pi(\hat{x}),\pi(\hat{y})\bigr)
=
d_M(x_0,y_0)
\leq
\hat d_a(\hat{x},\hat{y}).
\]

For $x\in M$, let
\[
P_x:\mathbb{R}^N\longrightarrow T_xM
\]
be the Euclidean orthogonal projection. For a line $[v]\in
\mathbb{R}P^{N-1}$ sufficiently close to $\mathbb{P}T_xM$, one has
$P_xv\neq0$, and we define
\[
\pi^*(x,[v])
:=
\bigl(x,[P_xv]\bigr).
\]
Since $x\mapsto P_x$ depends smoothly on $x$, this defines a smooth,
fiber-preserving bundle map
\[
\pi^*:\mathcal U\longrightarrow\mathbb{P}TM
\]
on an open neighborhood
$\mathcal U\subseteq\mathbb{P}T_M\mathbb{R}^N$ of $\mathbb{P}TM$.
Moreover,
\(
\left.\pi^*\right|_{\mathbb{P}TM}
=
\operatorname{Id}_{\mathbb{P}TM}.
\)

The set $\widetilde{\sigma}_0(\hat{\Lambda}_f)$ is compact and is contained
in $\mathbb{P}TM$. Therefore, there exist $\rho>0$ and $L\geq1$ such that
the $\rho$-neighborhood of
$\widetilde{\sigma}_0(\hat{\Lambda}_f)$ is contained in $\mathcal U$ and
$\pi^*$ is $L$-Lipschitz on this neighborhood.
Choose
\[
0<\delta<\min\left\{\rho,\frac{\epsilon}{L}\right\}.
\]
By the Lipschitz approximation established above, there exists a Lipschitz
map
\[
s_1:(\hat{\Lambda}_f,\hat d_a)
\longrightarrow
\mathbb{R}P^{N-1}
\]
such that
\[
\sup_{\hat{x}\in\hat{\Lambda}_f}
d_{\mathbb{R}P^{N-1}}
\bigl(s_0(\hat{x}),s_1(\hat{x})\bigr)
<\delta.
\]
Define
\[
\widetilde{\sigma}_1(\hat{x})
:=
\bigl(\pi(\hat{x}),s_1(\hat{x})\bigr).
\]
Since both $\pi$ and $s_1$ are Lipschitz,
$\widetilde{\sigma}_1$ is Lipschitz. Furthermore,
\[
d_{C^0}\bigl(\widetilde{\sigma}_0,\widetilde{\sigma}_1\bigr)
<\delta
\text{, so that }
\widetilde{\sigma}_1(\hat{\Lambda}_f)\subseteq\mathcal U.
\]

Consequently, define
\[
\sigma_1:
(\hat{\Lambda}_f,\hat d_a)
\longrightarrow
\mathbb{P}TM\qquad
\text{by}\qquad
\sigma_1(\hat{x})
=
\pi^*\bigl(\widetilde{\sigma}_1(\hat{x})\bigr).
\]
Because both $\widetilde{\sigma}_1$ and $\pi^*$ are Lipschitz on the
relevant neighborhood, $\sigma_1$ is Lipschitz. Since $\pi^*$ preserves
the base point,
\[
\pi_M\circ\sigma_1(\hat{x})
=
\pi(\hat{x}),
\]
and hence $\sigma_1$ is a section. Finally, using
$\pi^*\circ\widetilde{\sigma}_0=\sigma_0$ and the choice of $\delta$, we
obtain
\[
\begin{aligned}
d_{C^0}(\sigma_0,\sigma_1)
&=
\sup_{\hat{x}\in\hat{\Lambda}_f}
d_{\mathbb{P}TM}
\left(
\pi^*(\widetilde{\sigma}_0(\hat{x})),
\pi^*(\widetilde{\sigma}_1(\hat{x}))
\right)\\
&\leq
L\,d_{C^0}
\bigl(\widetilde{\sigma}_0,\widetilde{\sigma}_1\bigr)
<
L\delta
<
\epsilon.
\end{aligned}
\]
This completes the proof.
\end{proof}
\end{lem}

Let $\sigma_{su}:\hat{\Lambda}_f\to \mathbb{P}TM$ be the continuous map given by $\sigma_{su}(\hat{x})=E^{su}(\hat{x})$. By the continuity of the splitting, there exists a uniform cone size $\alpha_0>0$ such that the $\alpha_0$-cone field around $E^{su}$ forms an open neighborhood $K_{su}\subset \mathcal{E}$ of the graph $\sigma_{su}$. The domination property implies the following two properties:
consider the fiber map
\[
F(\hat{x},L) = (\hat{f}(\hat{x}), D_{\pi(\hat{x})}f(L)),
\]
then, by passing to an equivalent adapted metric or considering a high iteration without loss of generality, we have 
\begin{enumerate}[label=(\roman*), ref=(\roman*)]
\item (Strict invariance) $K_{su}$ is a trapping area, that is,
\begin{equation}
F(\overline{K_{su}})\subset K_{su}.
\end{equation}

\item (Fiber contraction) The map $F$ strictly contracts fibers inside $K_{su}$. Specifically, there exists a constant $\kappa_{\mathrm{fib}}\in (0,1)$ such that for any $\hat{x}\in \hat{\Lambda}_f$ and any two $L_1,L_2\in K_{su}\cap \mathbb{P}T_{\pi(\hat{x})}M$,
\begin{equation}
\angle \big( F(\hat{x}, L_1), F(\hat{x}, L_2) \big) \leq \kappa_{\mathrm{fib}} \cdot \angle(L_1, L_2).
\end{equation}
\end{enumerate}

Next, we establish the Hölder regularity of the invariant $su$-direction on the inverse limit space, which will later be transferred to fixed-past sections over the base repeller.

\begin{prop}\label{lem: Holder in inverse limit space}
    For any $\kappa_{\mathrm{fib}}<a<1$, the map 
    \begin{equation}
    \hat x\in\hat\Lambda_f\mapsto (\pi (\hat x),E^{su}(\hat x))\in \ptm
    \end{equation}
    is $\alpha$-H\"older continuous with respect to the metric $\hat d_a$ on $\hat\Lambda_f$ and $d_{\ptm}$ on $\ptm$.
   \begin{proof}
Let $\sigma_{su} : \hat{\Lambda}_f \to \mathbb{P}TM$ denote the invariant strong unstable
section, i.e., $\sigma_{su}(\hat{x}) = E^{su}(\hat{x})$ for all
$\hat{x} \in \hat{\Lambda}_f$.
To define suitable fiber coordinates, fix $\epsilon>0$ small and apply
Lemma~\ref{lem: Lip section approximate} to obtain a Lipschitz section
\[
\sigma_1 : \hat{\Lambda}_f \to \mathbb{P}TM
\]
such that
\[
\mathrm{graph}(\sigma_1) \subset K_{su}, \qquad
d_{C^0}(\sigma_{su}, \sigma_1) < \epsilon.
\]
This Lipschitz section $\sigma_1$ defines local fiber coordinates in
$K_{su}\subset \hat{\Lambda}_f\times\mathbb{P}TM$ such that the fiber is an interval in $\mathbb R$ with a metric
that coincides with the angle metric on each $\mathbb{P}T_{\pi(\hat{x})}M$. The coordinate map is Lipschitz by its definition. Thus, this coordinate structure preserves the H\"older regularity of the sections.

Recall the fiber map
\[
F: K_{su} \to K_{su}, \quad F(\hat{x},L) = (\hat{f}(\hat{x}), D_{\pi(\hat{x})}f(L)).
\]
By the fiber contraction property, there exists $0<\kappa_{\mathrm{fib}}<1$ such that for all
$\hat{x}\in \hat{\Lambda}_f$ and $u,v \in K_{su}$ in the same fiber,
\[
d_{\mathbb{P}T_{\pi(\hat f (\hat{x}))}M}(D_{\pi(\hat{x})}f(u), D_{\pi(\hat{x})}f(v))
\leq \kappa_{\mathrm{fib}} \, d_{\mathbb{P}T_{\pi(\hat{x})}M}(u,v).
\]
One can directly check that the conditions of Lemma~\ref{lem: Holder section thm} hold in this setting since $\kappa_{\mathrm{fib}}<a<1$. Apply Lemma~\ref{lem: Holder section thm} to the fiber contraction $F$ in the
coordinates given by $\sigma_1$.  This guarantees the existence of a unique
$F$-invariant section
\[
\tilde{\sigma} : \hat{\Lambda}_f \to K_{su}
\qquad \text{satisfying}
\qquad \tilde{\sigma} \circ \hat{f} = F \circ \tilde{\sigma},
\]
and which is $\alpha$-Hölder continuous with constant $C>0$ with respect to the angle metric
on $\mathbb{P}T_{\pi(\hat{x})}M$.
Since the section $\sigma_{su}$ is invariant, it follows that
$
\tilde{\sigma} = \sigma_{su}.
$
Therefore, for all $\hat{x},\hat{y}\in \hat{\Lambda}_f$,
\[
d_{\mathbb{P}TM}(\sigma_{su}(\hat{x}), \sigma_{su}(\hat{y}))
\leq C \, \hat{d}_a(\hat{x},\hat{y})^\alpha.
\]
Thus $\sigma_{su}$ is $\alpha$-Hölder continuous, with the Hölder
constant  depending on the system $f$. 
\end{proof}
\end{prop}

\begin{proof}[Proof of \Cref{prop Holder continuous}]
Now fix $\kappa_{\mathrm{fib}}<a<1$ as above, \Cref{lem: Holder in inverse limit space} implies that there exists a constant $C_0'>0$ such that
\[
d_{\ptm}\left((\pi\hat x,E^{su}(\hat x)),(\pi\hat y,E^{su}(\hat y))\right)\leq C_0'\cdot\hat d_{a}(\hat x,\hat y)^\alpha.
\]
Since $(\Lambda_f,f)$ is a horseshoe repeller, for any historical branch $\omega\in \Sigma^-$ and $x,y\in\Lambda_f$, one has
\[
d((\hat x_\omega)_i,(\hat y_\omega)_i)\leq d(x,y),\quad \forall\ i\in\mathbb Z^-.
\]
Thus,
\[
\hat d_a(\hat x_\omega,\hat y_\omega)\leq \sum_{i=0}^{\infty}a^{i}\cdot d(x,y)=\frac{1}{1-a}\cdot d(x,y).
\]
Together we obtain 
\[
d_{\ptm}\left((x,E^{su}(\hat x_\omega)),(y,E^{su}(\hat y_\omega))\right)\leq C_0'\cdot(\frac{1}{1-a})^\alpha\cdot d(x,y)^\alpha,
\]
which completes the proof of \Cref{prop Holder continuous} with $C_0= C_0'\cdot(\frac{1}{1-a})^\alpha$.

\end{proof}

\subsection{Small angle and non-integrability}\label{sec:3.2}
We first prove \cref{lem:pointwise-su-nonintegrability}, which yields that for a transitive surface repeller, if one point is $su$-non-integrable, then every point on this repeller is $su$-non-integrable.

\begin{proof}[Proof of \cref{lem:pointwise-su-nonintegrability}]
Let
\[
\theta_0=\angle(E^{su}(\hat p^1),E^{su}(\hat p^2))>0 .
\]
We first observe that the existence of two histories whose ($E^{su}$) directions make an angle uniformly bounded away from zero is an open property with respect to the base point.
Indeed, by the uniform expansion, the two backward orbits 
defining $\hat p^1$ and $\hat p^2$ determine two local inverse branches near $p$ with a uniform size. Thus, after shrinking a relative neighborhood $U\subset\Lambda_f$ of $p$, these inverse branches give two continuous history sections
\[
s_i:U\to\hat\Lambda_f,\qquad \pi\circ s_i=\Id_U,\qquad s_i(p)=\hat p^i,
\quad i=1,2 .
\]
By continuity of the dominated splitting on $\hat\Lambda_f$, after shrinking $U$ if necessary, we have
\[
\angle(E^{su}(s_1(z)),E^{su}(s_2(z)))\geq \frac{\theta_0}{2}
\qquad\text{for all }z\in U .
\]

Now fix $x\in\Lambda_f$. Since $(\Lambda_f,f)$ is a transitive repeller, it admits an irreducible Markov coding
\[
\Pi:\Sigma_A^+\to\Lambda_f,\qquad \Pi\circ\sigma=f\circ\Pi .
\]
Choose a cylinder $\Lambda[w]=\Pi([w])\subset U$. If $\xi\in\Sigma_A^+$ is a coding of $x$, irreducibility gives an admissible connecting word $v$ such that $wv\xi$ is admissible. Hence, for
\[
z:=\Pi(wv\xi)\in U
\]
there exists $n\geq0$ with
\[
f^n(z)=x .
\]

For $i=1,2$, define
\[
\hat x^i:=\hat f^n(s_i(z)).
\]
Then $\pi(\hat x^i)=x$, and by $\hat f$-invariance of the dominated splitting,
\[
E^{su}(\hat x^i)=D_zf^n(E^{su}(s_i(z))).
\]
Since $D_zf^n$ is invertible, its induced map on projective lines is injective. Therefore
\[
\mathbb P(E^{su}(\hat x^1))\neq
\mathbb P(E^{su}(\hat x^2)).
\]
This proves the lemma.
\end{proof}

The next proposition is the key measure-theoretic consequence of \(su\)-non-integrability.  Roughly speaking, it says that the dependence of the \(su\)-direction on the past cannot be removed after discarding a small set of historical branches: every sufficiently large set of pasts still contains two distinct histories whose induced \(su\)-directions at a prescribed base point are arbitrarily close, while not coinciding. 

In this way, \(su\)-non-integrability is strengthened into a measure-quantitative statement on historical branches, which will later allow us to extract small-angle pairs inside large-measure subsets of \(\Sigma^{-}\), where \(\Sigma^{-}\) is the space of left-infinite sequences representing the historical branches.

\begin{prop}\label{prop: small angle}
Let $(\Lambda_f,f)$ be a $C^{1+\alpha}$ $su$-non-integrable surface horseshoe repeller, and let $\nu$ be a fully supported Bernoulli measure on $\Sigma^-$. Then there exists $\epsilon>0$ such that the following holds.

For every $\eta>0$, every $x\in\Lambda_f$, and every measurable set
$\Sigma'\subseteq\Sigma^-$ with
\[
\nu(\Sigma')>1-\epsilon,
\]
there exist two distinct histories $\mathbf i,\mathbf j\in\Sigma'$ such that
\[
0<
\angle\bigl(E^{su}(\hat x_{\mathbf i}),E^{su}(\hat x_{\mathbf j})\bigr)
<\eta.
\]
\end{prop}

We next use the H\"older continuity of the strong-unstable direction on the inverse limit to derive an exponential diameter estimate for the sets of directions associated with history cylinders.

For $x\in\Lambda_f$ and $\mathbf i=(\ldots,i_{-2},i_{-1})\in\Sigma^-$, write
\[
\hat x_{\mathbf i}:=
(\ldots,x_{-2}^{\mathbf i},x_{-1}^{\mathbf i},x),
\]
where
\[
x_{-n}^{\mathbf i}
:=f^{-1}_{i_{-n}}\circ\cdots\circ f^{-1}_{i_{-1}}(x),
~~ n\geq 1.
\]
Thus $f^n(x_{-n}^{\mathbf i})=x$.

For a word
\[
w=(i_{-n},\ldots,i_{-1})\in\{1,\ldots,m\}^n,
\]
let $[w]\subset\Sigma^-$ denote the corresponding $n$-cylinder, and define the associated set of strong-unstable directions by
\[
\mathscr D^{su}_n(x,w)
:=\{E^{su}(\hat x_{\mathbf i}):
\mathbf i\in[w]\} \subset\mathbb P(T_xM).
\]

By \Cref{lem: Holder in inverse limit space} and a direct calculation that is omitted, there exist constants
$C_{\mathrm{inv}}>0$ and $\gamma=-\alpha\log_{2}a>0$ such that
\[
d_{\mathbb PTM}
\left(
(\pi(\hat z),E^{su}(\hat z)),
(\pi(\hat z'),E^{su}(\hat z'))
\right)
\leq
C_{\mathrm{inv}}\hat d(\hat z,\hat z')^\gamma
\]
for all $\hat z,\hat z'\in\hat\Lambda_f$.

Fix $x\in\Lambda_f$ and a word
\[
w=(i_{-n},\ldots,i_{-1})
\]
of length $n$. If $\mathbf i,\mathbf j\in[w]$, then the corresponding histories
$\hat x_{\mathbf i}$ and $\hat x_{\mathbf j}$ have identical coordinates from
time $-n$ to time $0$. By the definition of the inverse-limit metric with $a=\frac{1}{2}$, there
exists a uniform constant $C_{\rm b}>0$ such that
\[
\hat d(\hat x_{\mathbf i},\hat x_{\mathbf j})
\leq C_{\rm{b}} 2^{-n}.
\]
Consequently,
\[
d_{\mathbb PTM}
\left(
(x,E^{su}(\hat x_{\mathbf i})),
(x,E^{su}(\hat x_{\mathbf j}))
\right)
\leq
C_{\mathrm{inv}}C_{\rm b}^\gamma 2^{-\gamma n}.
\]
Since the two directions belong to the same fiber $\mathbb P(T_xM)$, the
restriction of the metric on $\mathbb PTM$ is uniformly equivalent to the
angular metric. Hence, after changing the constant, there exists
$C_{\mathrm H}>0$ such that
\[
\angle\left(
E^{su}(\hat x_{\mathbf i}),
E^{su}(\hat x_{\mathbf j})
\right)
\leq
C_{\mathrm H}2^{-\gamma n}.
\]
Taking the supremum over $\mathbf i,\mathbf j\in[w]$ gives
\[\label{eq:su-direction-cylinder-diameter}
\operatorname{diam}_{\angle}
\mathscr D_n^{su}(x,w)
\leq
C_{\mathrm H}2^{-\gamma n}.
\]
 where $$\operatorname{diam}_{\angle}
\mathscr D^{su}_n(x,w):=\sup_{\mathbf i,\mathbf j\in[w]}
\angle\left(
E^{su}(\hat x_{\mathbf i}),
E^{su}(\hat x_{\mathbf j})
\right).$$

We next obtain a uniform finite-scale separation of strong-unstable
directions.

\begin{lem}\label{lem:uniform-su-separation}
With the notation above, there exist $N_0\in\mathbb N$ and $\delta_0>0$ such that, for every
$x\in\Lambda_f$, there are two distinct words
$u_x,v_x\in\{1,\ldots,m
\}^{N_0}$ satisfying
\begin{equation}\label{eq:uniform-su-cylinder-separation}
\operatorname{dist}_{\angle}
\left(
\mathscr D^{su}_{N_0}(x,u_x),
\mathscr D^{su}_{N_0}(x,v_x)
\right)
\geq\delta_0.
\end{equation}
where $\operatorname{dist}_{\angle}(A,B):=\inf\{\angle (E,F):E\in A, F\in B\}.$
\end{lem}

\begin{proof}
Fix $x\in\Lambda_f$. By
\Cref{lem:pointwise-su-nonintegrability}, there exist two histories
$\mathbf i,\mathbf j\in\Sigma^-$ such that
\[
\theta_x
:=
\angle\bigl(
E^{su}(\hat x_{\mathbf i}),
E^{su}(\hat x_{\mathbf j})
\bigr)
> 0.
\]
 Choose $n_x$ sufficiently large so that
\[
C_{\mathrm H}2^{-\gamma n_x}<\frac{\theta_x}{4}.
\]

Let $u_x$ and $v_x$ be the length-$n_x$ terminal words of $\mathbf i$ and $\mathbf j$, respectively. By
\eqref{eq:su-direction-cylinder-diameter},
\[ \operatorname{dist}_{\angle}\left(\mathscr D^{su}_{n_x}(x,u_x),\mathscr D^{su}_{n_x}(x,v_x)
\right)>\frac{\theta_x}{2}.
\]

The map
\[
(y,\mathbf a)\longmapsto E^{su}(\hat y_{\mathbf a})
\]
is continuous on $\Lambda_f\times\Sigma^-$. Since the cylinders
$[u_x]$ and $[v_x]$ are compact, after shrinking to a neighborhood
$U_x$ of $x$, the same two direction sets remain uniformly separated:
there exists $\delta_x>0$ such that
\[
\operatorname{dist}_{\angle}
\left(
\mathscr D^{su}_{n_x}(y,u_x),
\mathscr D^{su}_{n_x}(y,v_x)
\right)
\geq\delta_x
\]
for every $y\in U_x$.

By compactness of $\Lambda_f$, there exist finitely many such neighborhoods
\[
U_{x_1},\ldots,U_{x_m}
\]
covering $\Lambda_f$. Then we can take a uniform $N_0$ and a uniform $\delta_0$, which prove the lemma.\qedhere
\end{proof}

\begin{proof}[Proof of \Cref{prop: small angle}]
Let
$(p_1,\ldots,p_m)$
be the weights of the fully supported Bernoulli measure $\nu$ on
$\Sigma^-$. Set
\[
p_*:=\min_{1\leq r\leq m}p_r>0
\]
and let $N_0$ be given by
\Cref{lem:uniform-su-separation}. We claim that the proposition holds with
\[
\epsilon:=p_*^{N_0}.
\]

Fix $\eta>0$, $x\in\Lambda_f$, and
$\Sigma'\subseteq\Sigma^-$ with
$\nu(\Sigma')>1-\epsilon.$ 
Choose $N_1$ sufficiently large so that
\begin{equation}\label{eq:choose-N1-small-angle}
C_{\mathrm H}2^{-\gamma N_1}<\eta.
\end{equation}

Consider the partition of $\Sigma^-$ into cylinders $[w]$ of length $N_1$.
For each such word
\[
w=(i_{-N_1},\ldots,i_{-1}),
\]
let
\[
t_w
:=
f^{-1}_{i_{-N_1}}
\circ\cdots\circ
f^{-1}_{i_{-1}}(x)
\]
be the corresponding $N_1$-step preimage of $x$.

Applying \Cref{lem:uniform-su-separation} at $t_w$, choose two words
$u_w,v_w$ of length $N_0$ such that
\[
\mathscr D^{su}_{N_0}(t_w,u_w)
\cap
\mathscr D^{su}_{N_0}(t_w,v_w)=
\emptyset.
\]
Let $u_w\star w$ denote the word obtained by placing the block $u_w$
immediately before the block $w$ in the backward history, and define
$v_w\star w$ analogously. Thus
\[
[u_w\star w],[v_w\star w]\subset[w].
\]

Since $\nu$ is Bernoulli,
\[
\nu([u_w\star w])
=\nu([u_w])\nu([w])
\geq
p_*^{N_0}\nu([w])
=\epsilon\nu([w]),
\]
and similarly
\[
\nu([v_w\star w])
\geq
\epsilon\nu([w]).
\]

We claim that for at least one word $w$ of length $N_1$, the set
$\Sigma'$ meets both cylinders
\[
[u_w\star w]
\quad\text{and}\quad
[v_w\star w].
\]
Suppose otherwise. Then for every $w$, at least one of these two cylinders
is disjoint from $\Sigma'$. Choose one such cylinder and denote it by $Q_w$.
Since the cylinders $[w]$ of length $N_1$ are pairwise disjoint, the sets
$Q_w$ are also pairwise disjoint. Hence
\[
\nu((\Sigma')^c)
\geq
\sum_{|w|=N_1}\nu(Q_w)
\geq
\epsilon\sum_{|w|=N_1}\nu([w])
=\epsilon,
\]
contradicting
\[
\nu((\Sigma')^c)<\epsilon.
\]
Thus such a word $w$ exists.

Choose
\[
\mathbf i\in
\Sigma'\cap[u_w\star w],
\qquad
\mathbf j\in
\Sigma'\cap[v_w\star w].
\]
Since $\mathbf i$ and $\mathbf j$ have the same terminal block $w$ of
length $N_1$, both directions belong to
$\mathscr D^{su}_{N_1}(x,w)$. Therefore, by
\eqref{eq:su-direction-cylinder-diameter} and
\eqref{eq:choose-N1-small-angle},
\[
\angle\bigl(
E^{su}(\hat x_{\mathbf i}),
E^{su}(\hat x_{\mathbf j})
\bigr)
<
\eta.
\]

It remains to show that these two directions are distinct. Remove the common
terminal block $w$ from $\mathbf i$ and $\mathbf j$ and regard the remaining
tails as histories over $t_w$. Denote these tails by
$\widetilde{\mathbf i}$ and $\widetilde{\mathbf j}$. Then
\[
\widetilde{\mathbf i}\in[u_w],
\qquad
\widetilde{\mathbf j}\in[v_w],
\]
and hence
\[
E^{su}(\hat t_{w,\widetilde{\mathbf i}})
\neq
E^{su}(\hat t_{w,\widetilde{\mathbf j}})
\]
by the choice of $u_w$ and $v_w$.

By invariance of the strong-unstable bundle,
\[
E^{su}(\hat x_{\mathbf i}) =Df^{N_1}_{t_w}(
E^{su}(\hat t_{w,\widetilde{\mathbf i}})),~~
E^{su}(\hat x_{\mathbf j})
=Df^{N_1}_{t_w}
(E^{su}(\hat t_{w,\widetilde{\mathbf j}})).
\]
Since $f$ is a local diffeomorphism near $\Lambda_f$,
$Df^{N_1}_{t_w}$ is invertible, and its induced action on
$\mathbb P(T_{t_w}M)$ is injective. Consequently,
\[
E^{su}(\hat x_{\mathbf i})
\neq
E^{su}(\hat x_{\mathbf j}).
\]
Combining this with the preceding upper bound gives
\[
0<
\angle\bigl(
E^{su}(\hat x_{\mathbf i}),
E^{su}(\hat x_{\mathbf j})
\bigr)
<
\eta,
\]
which proves the proposition.
\end{proof}

\section{$(n,\ell,\epsilon)$-Dyadic-Spreading}
     \label{sec:nledyadic}
     Consider a $C^{1+\alpha}$ horseshoe surface repeller $(\Lambda_f, f)$ (\Cref{def:horseshoe repeller}) that admits a dominated splitting and is $su$-non-integrable. Let $h: \Sigma^+ \to \Lambda_f$ be a homeomorphism such that $h \circ \sigma = f \circ h$, where $\Sigma^+ = \{1, \dots, m\}^{\mathbb{Z}_{\geq 0}}$ for some integer $m \geq 2$ and $\sigma$ is the left-shift map on $\Sigma^+$. 
     For each $n \in \mathbb{Z}_{\geq 0}$, the space $\Sigma^+$ contains $m^n$ distinct $n$-cylinders; we let $\Lambda_{n,q}$ ($1 \leq q \leq m^n$) denote their respective images under $h$ in $\Lambda_f$, which we will also refer to as $n$-cylinder sets.
     Let $\nu_+$ be a fully supported Bernoulli measure on $\Sigma^+$, and define the pushforward measure $\mu := h_*\nu_+$ on $\Lambda_f$. Finally, let $f^{-1}_1, \dots, f^{-1}_m$ be the local inverses of $f$, and let $\Sigma^- := \{1, \dots, m\}^{\mathbb{Z}_{< 0}}$ be identified with the space of these inverse branches.

\subsection{Local Euclidean structures on the repeller}

\subsubsection{Local strong unstable manifolds}\label{sec 4.1.1}
Define the bijection
\begin{equation}
\label{eqn:Phi map}
\Phi: \Lambda_f \times \Sigma^- \to \hat{\Lambda}_f 
\end{equation}
that maps $(x, \bfi)$ to $\{x_\ell\}_{\ell \leq 0}$, where $x_0 = x$ and $x_{\ell} = f^{-1}_{i_\ell}(x_{{\ell+1}})$ for all $\ell < 0$. Thus, for a horseshoe repeller, the results established in the preceding sections on the inverse limit space \(\hat{\Lambda}_f\) can be equivalently reformulated on \(\Lambda_f\times\Sigma^-\) via the bijection \(\Phi\).

To extend $\Phi$ to a neighborhood of $\Lambda_f$, we choose a sufficiently small $\delta > 0$ such that the local inverses of $f$ are defined on $B(\Lambda_f, \delta)$ and the $\delta$-neighborhoods of the $1$-cylinders $\Lambda_{1,q}$ (for $1 \leq q \leq m$) are mutually disjoint.
Since the local inverse maps for points close to the repeller are well-defined, the inverse limit space on the neighborhood of the repeller $B(\Lambda_f,\delta)$ is well-defined.
Defining the inverse limit space $\widehat{B(\Lambda_f, \delta)}$ associated with $(B(\Lambda_f, \delta), f)$ as in \eqref{eqn:hatlambda}, we extend the bijection as follows:$$\Phi: B(\Lambda_f, \delta) \times \Sigma^- \to \widehat{B(\Lambda_f, \delta)}$$mapping $(y, \mathbf{i}) \mapsto \{y_\ell\}_{\ell \leq 0}$, where $y_0 = y$ and $y_\ell = f^{-1}_{i_\ell}(y_{\ell+1})$ for all $\ell < 0$.

For any $(x,\bfi)\in\Lambda_f\times\Sigma^-$, the {\emph{global strong unstable manifold}} is defined as
\begin{equation}
W^{su}(x,\bfi):=\left\{y\in {B(\Lambda_f,\delta)}:\, 
         \limsup_{\ell\to -\infty}-\frac{1}{\ell}\log d(y_{\ell},x_{\ell})\leq -\lambda_1(x,\bfi)\right\}
\end{equation}
where $\{x_\ell\}=\Phi(x,\mathbf{i})$, $\{y_\ell\}=\Phi(y,\mathbf{i})$ and
$$\lambda_1(x,\bfi):=\liminf_{\ell\to-\infty}\frac{1}{\ell}\log\|D_x(f^{-1}_{i_\ell}\circ\cdots\circ f^{-1}_{i_{-1}} )|_{E^{su}(x,\bfi)}\|.$$
Here $\lambda_1(x,\bfi)>0$ is uniformly bounded away from zero due to the uniform expanding property of the repeller.

We now  fix a constant $\rho_0 > 0$ strictly less than the injectivity radius of $M$, such that for every $x \in M$, the inverse of the exponential map$$ \exp_x^{-1}: B_{\rho_0}(x) \subset M \to T_xM $$is a diffeomorphism onto its image.

There exists a constant $r\in (0,\rho_0)$ small enough, depending only on $f$, such that for every $(x,\bfi)\in \Lambda_f\times \Sigma^-$, 
the {\textit{local strong unstable manifold}}  $\wloc(x,\bfi)\subseteq M$ is well-defined as
\begin{equation}
    W^{su}_{\mathrm{loc}}(x,\mathbf{i}) := \text{connected component of} \,\,x \,\,\text{in}\,\,B_{r}(x)\cap W^{su}(x,\bfi).
\end{equation}
Furthermore, the following properties are standard consequences of the uniform Hadamard–Perron theorem for hyperbolic sets, applied to the inverse branches determined by each
history; see \cite{HirschPughShub1977} and
\cite{KatokHasselblatt1995}. For the endomorphism setting, see also
\cite{Przytycki1976}.     
\begin{enumerate}[(U1)]
\item \label{U:USR} (Uniform Size and Regularity) There exist constants $r_0 > 0$ and $C_1 > 0$ such that for every $(x, \bfi) \in \Lambda_f \times \Sigma^-$, the preimage under the exponential map, $\exp_x^{-1}(W^{su}_{\mathrm{loc}}(x, \bfi))$, is the graph of a $C^{1+\alpha/2}$ function $\varphi: J\to E^{su}(x, \bfi)^{\perp}$ satisfying the following conditions:
\begin{itemize}
\item the domain $J \subseteq E^{su}(x, \bfi)$ contains the open ball $B_{r_0}^{su}(0)$; 
\item the differential satisfies $\|D\varphi\| < 1/100$ on $J$;
\item the $C^{1+\alpha/2}$ norm is uniformly bounded, with $\|\varphi\|_{C^{1+\alpha/2}} \leq C_1$.
\end{itemize}
\item (Uniform Backward Contraction) There exists a constant $\chi > 0$ such that for every $(x, \bfi) \in \Lambda_f \times \Sigma^-$ and all $y, z \in W^{su}_{\mathrm{loc}}(x, \bfi)$, we have$$d(y_\ell, z_\ell) \leq e^{\chi \ell} d(y, z) \quad \text{for all } \ell \leq 0,$$where $\{y_\ell\}_{\ell\leq 0} = \Phi(y, \bfi)$ and $\{z_\ell\}_{\ell\leq 0} = \Phi(z, \bfi)$.
\end{enumerate}

For all sufficiently large $n_0 \in \mathbb{N}$, each $n_0$-cylinder set $\Lambda_{n_0,q}$ (for $1 \leq q \leq m^{n_0}$), there exists an open neighborhood $U_q \subseteq M$ with a diameter sufficiently small compared to $r_0$, satisfying $U_q \cap \Lambda_f = \Lambda_{n_0,q}$.

The existence of such neighborhoods follows from the fact that the horseshoe
repeller is topologically conjugate to a full shift, whose cylinder sets form
a closed-open basis of the subspace topology.

\subsubsection{Straightening map}
\label{subsubsection:straightening map}
\begin{defi}
For a given $\theta > 0$ and $n_0\in\mathbb N$, a $C^1$ submanifold $T \subseteq U_q=:U$ is said to be \emph{$\theta$-uniform $su$-transversal} for an $n_0$-cylinder $\Lambda_{n_0,q}$ if, for every point $y \in T \cap \Lambda_f$, every $\bfi \in \Sigma^-$, and $z,x_0 \in \Lambda_{n_0,q}$,
the angle between its tangent space $T_y T$ and the strong unstable subspace $E^{su}(z, \bfi)$ satisfies$$ \angle\left(D_y\exp_{x_0}^{-1}(T_y T),\, D_z\exp_{x_0}^{-1} (E^{su}(z, \bfi))\right) \geq \theta.$$
\end{defi}

Recall the following basic properties of the dominated splitting: under the identification of $\hat{\Lambda}_f$ with $\Lambda_f \times \Sigma^-$ via $\Phi$, the splitting defined in \eqref{eqn:TM splitting} is continuous, and the angle between $E^{wu}(x)$ and $E^{su}(x, \mathbf{i})$ is uniformly bounded away from $0$ for all $x \in \Lambda_f$ and $\mathbf{i} \in \Sigma^-$.
From now on, we first fix a $\theta$ much smaller than the uniform angle bound and then fix $n_0$ large enough depending on $\theta$ such that for any $q$, $x\in\Lambda_{n_0,q}$ and 
\[
 T_x:=\exp_x\{v\in E^{wu}(x)\}\cap U,
 \]
$T_x$ is a $\theta$-uniform $su$-transversal for $\Lambda_{n_0,q}$. We moreover require $n_0$ large enough that for every $\bfi\in\Sigma^-$, all $x,x_0\in\Lambda_{n_0,q}$ and every $y\in W^{su}_{\mathrm{loc}}(x,\bfi)$,
\[
\angle\Big(D_y\exp_{x_0}^{-1}(T_yW^{su}_{\mathrm{loc}}(x,\bfi)),\ D_{x_0}\exp_{x_0}^{-1}(E^{su}(x_0,\bfi))\Big)\ \leq\ \arccos(1/1.01).
\]

Given $\mathbf{i} \in \Sigma^-$, define the neighborhood $V_{\mathbf{i}} = \bigcup_{x \in \Lambda_{n_0,q}} W^{su}_{\mathrm{loc}}(x, \mathbf{i})$. The \emph{$su$-holonomy} between any two $su$-transversals $T_1$ and $T_2$ is the map$$h^{su}_{\mathbf{i}}: T_1 \cap V_{\mathbf{i}} \to T_2 \cap V_{\mathbf{i}} \qquad \text{given by} \qquad x \mapsto W^{su}_{\mathrm{loc}}(x, \mathbf{i}) \cap T_2.$$
The following lemma follows from the results in \cite{AaronBrown2022}.
\begin{lem}\label{lem holder of Esu}
For the fixed \(\theta\) and \(n_0\), there exists
\(C_2=C_2(f,\theta,n_0)>1\) such that, for every
\(\mathbf i\in\Sigma^-\) and every two
\(\theta\)-uniform transversals \(T_1,T_2\),
the holonomy
\[
h_{\mathbf i}^{su}:T_1\cap V_{\mathbf i}
\longrightarrow T_2\cap V_{\mathbf i}
\]
is \(C_2\)-bi-Lipschitz onto its image.
 \end{lem}

We follow \cite[Section 8.3]{LEDRAPPIER_YOUNG_B_1985} to construct the straightening maps, a construction that applies identically to every $n_0$-cylinder. For ease of notation, we fix one $q$, drop the subscript $q$ and simply write $\Lambda_{n_0}$ for a fixed $n_0$-cylinder and $U$ for its corresponding isolating neighborhood.
 \begin{equation}
 \label{eqn:x0}
 \text{Fix a point} \,\, x_0 \in \Lambda_{n_0}.
 \end{equation}
 Using the exponential map $\exp_{x_0}$, we introduce a local Euclidean coordinate on $U$, thereby inducing a Euclidean metric via the chart
\[
\exp_{x_0}^{-1}: U \to \mathbb{R}^2,
\]
and choose coordinates such that
\[
D_0\exp_{x_0}(\mathbb{R}\times \{0\}) = E^{wu}(x_0),
\]
so that the horizontal axis in $\mathbb{R}^2$ aligns with the weak-unstable direction at $x_0$.
 We also fixed a $\theta$-uniform $su$-transversal $T_0$ as 
 \begin{equation}
 \label{eqn:tran T0}
 T_0=\exp_{x_0}(\{(x,y)\in\mathbb R^2:y=0\}).
 \end{equation}
Following \cite[Section 8.3]{LEDRAPPIER_YOUNG_B_1985}, we can construct the straightening map as follows:
\begin{prop}
\label{prop:straightening map}
For every $\bi\in \Sigma^-$, we construct a  {\textit{foliation straightening}} map 
\begin{equation}
\label{eqn:sigma i}
 \sigma_\bfi:V_{\bfi}\to \mathbb R^2 
 \end{equation}
 satisfying the following uniform properties:
 \begin{enumerate}[(S1)]
     \item (Straightening): For any $x\in \Lambda_{n_0}$, the image $\sigma_\bfi(W^{su}_{\mathrm{loc}}(x,\bfi))$ is a vertical line segment parallel to the $y$-axis.
     \label{straight:Lip}
     \item (Global bi-Lipschitz):
     The map $\sigma_\bfi$ is a bi-Lipschitz homeomorphism onto its image with  
     a uniform Lipschitz constant $C_3$.
     \label{straight:global Lip}
     \item (Leaf-wise Regularity):
     \label{S:leaf-wise regularity}
    For any local leaf, the restriction
    $$\sigma_\bfi|_{W^{su}_{\mathrm{loc}}(x,\bfi)}:W^{su}_{\mathrm{loc}}(x,\bfi)\to\mathbb R^2$$ is $C^{1+\frac{\alpha}{2}}$ a diffeomorphism onto its image, with a uniform bi-Lipschitz constant $1.01$ and a $C^{1+\frac{\alpha}{2}}$-norm uniformly bounded by $C_4$.
     \item (Inverse Regularity in Tangent Space): 
     \label{S:inverse reg in tang}
    The pullback to the flat tangent space
    $$\exp_{x_0}^{-1}\circ \sigma_\bfi^{-1}|_{\sigma_\bfi(W^{su}_{\mathrm{loc}}(x,\bfi))}: \sigma_\bfi(W^{su}_{\mathrm{loc}}(x,\bfi))\to \mathbb R^2$$ is a $C^{1+\frac{\alpha}{2}}$ curve whose $C^{1+\frac{\alpha}{2}}$-norm is uniformly bounded by $C_5$. Moreover, for any $y\in \sigma_\bfi(W^{su}_{\mathrm{loc}}(x,\bfi))$ one has
    \[
    \|D_y (\exp_{x_0}^{-1}\circ \sigma_\bfi^{-1}|_{\sigma_\bfi(W^{su}_{\mathrm{loc}}(x,\bfi))})\|\in[1,1.01].
    \]
 \end{enumerate}
 We note that $\sigma_\bfi$ is well-defined on $V_{\mathbf i}$, which contains $\Lambda_{n_0}$ and hence has full $\mu_0$-measure, where $\mu_0:=\frac{1}{\mu(\Lambda_{n_0})}\mu|_{\Lambda_{n_0}}$.
 \end{prop}

 \begin{proof}[Construction of the straightening map $\sigma_\bfi$]

For each $x\in\Lambda_{n_0}$ and $y \in W^{su}_{\mathrm{loc}}(x,\bfi)$, we define the straightening coordinates $\sigma_\bfi(y) = (p(y),q(y))$ as follows. The horizontal coordinate $p(y)$ is obtained by projecting $y$ along $W^{su}_{\mathrm{loc}}(x,\bfi)$
to the fixed transversal $T_0$, and the vertical coordinate is the projection of the chart point onto the strong-unstable line at the base point:
\[
q(y):=\big\langle \exp_{x_0}^{-1}y,\ e_{\bfi}\big\rangle,
\]
where $e_{\bfi}$ is a unit vector spanning $D_{x_0}\exp_{x_0}^{-1}(E^{su}(x_0,\bfi)).$

By this construction, each local leaf $W^{su}_{\mathrm{loc}}(x,\bfi)$ is mapped to a vertical line segment parallel to the $y$-axis, while different leaves remain separated according to their projections onto $T_0$. This ensures that the foliation is straightened leafwise and that the $C^{1+\alpha/2}$ regularity of each leaf is preserved.

The bi-Lipschitz property of $\sigma_\bfi$ follows from the composition of the following two facts: (i) each leaf is the graph of $\varphi_x$ with $\|D\varphi_x\|<1/100$, providing uniform Lipschitz control along leaves; (ii) the $su$-holonomy between $\theta$-uniform transversals is bi-Lipschitz with constant $C_2$. Consequently, there exists $C_3>0$, depending only on $f$, $n_0$, and $\theta$, such that
\[
\frac{1}{C_3} d(x,y) \leq \|\sigma_\bfi(x)-\sigma_\bfi(y)\| \leq C_3 d(x,y),\quad \forall x,y\in V_\bfi.
\]
 
Leafwise, let $s\mapsto\gamma(s)$ parametrize $W^{su}_{\mathrm{loc}}(x,\bfi)$ by arc length. Then
\[
\frac{d}{ds}\,q(\gamma(s))=\big\langle \big(D_{\gamma(s)}\exp_{x_0}^{-1}\big)\gamma'(s),\ e_{\bfi}\big\rangle\in [1/1.01,\,1].
\]
 Hence the restriction $\sigma_\bfi|_{W^{su}_{\mathrm{loc}}(x,\bfi)}$ is a $C^{1+\alpha/2}$ diffeomorphism with {bi-Lipschitz constant} $1.01$ and $C^{1+\alpha/2}$-norm bounded by $C_4$, inherited from the uniform $C^{1+\alpha/2}$ bound of $\varphi_x$.

Composing with $\exp_{x_0}^{-1}$, the pullback $\exp_{x_0}^{-1}\circ \sigma_\bfi^{-1}$ of each vertical segment is a $C^{1+\alpha/2}$ curve in the tangent space with uniform norm bounded by $C_5$, depending only on $f$, $n_0$, and $\theta$.
Moreover, writing $\delta(t):=\exp_{x_0}^{-1}\circ\sigma_\bfi^{-1}(p,t)$ for the pullback of a vertical segment, one has $\langle\delta(t),e_\bfi\rangle=t$ identically by the definition of $q$, hence $\langle\delta'(t),e_\bfi\rangle=1$ and
\[
\|\delta'(t)\|=\frac{1}{\cos\angle(\delta'(t),e_\bfi)}\in[1,\,1.01],
\]
since $\delta'(t)$ spans $\big(D\exp_{x_0}^{-1}\big)T W^{su}_{\mathrm{loc}}(x,\bfi)$, which makes an angle at most $\arccos(1/1.01)$ with $e_\bfi$ by the choice of $n_0$ above. Thus,
\[
    \|D_y (\exp_{x_0}^{-1}\circ \sigma_\bfi^{-1}|_{\sigma_\bfi(W^{su}_{\mathrm{loc}}(x,\bfi))})\|\in[1,1.01],\quad \forall y\in \sigma_\bfi(W^{su}_{\mathrm{loc}}(x,\bfi)).
\]

Hence, $\sigma_\bfi: V_\bfi\to \mathbb{R}^2$ straightens the foliation, is globally bi-Lipschitz with constant $C_3$, leafwise $C^{1+\alpha/2}$ with constant $C_4$, and has pullback regularity in tangent space bounded by $C_5$. All constants $C_3,C_4,C_5$ are independent of $x\in \Lambda_{n_0}$ and $\bfi\in \Sigma^-$.
\end{proof}

Finally, we define balls on the strong unstable manifolds with respect to the metric induced by the straightening map. For $\bfi \in \Sigma^-$, $x \in \Lambda_{n_0}$, and sufficiently small $r > 0$, these are given by
\begin{equation}
\label{eqn:ball straightening}
B^{su}_{\bfi,s}(x,r) := \sigma_\bfi^{-1} \left( \left\{ (\pi_1\circ\sigma_\bfi(x),y) \in \mathbb{R}^2 \;\middle|\; |y-\pi_2\circ \sigma_\bfi(x)| \leq r \right\} \right).
\end{equation}

\begin{rem}[Standing constants]\label{rem:constants}
Let $\alpha>0$ be the H\"older exponent. Let $C_0$ be the constant of \Cref{prop Holder continuous}, enlarged if necessary, so that it also dominates the Lipschitz constant of $(z,E)\mapsto(D_z\exp^{-1}_{x_0})E$ on $\ptm|_{\Lambda_{n_0}}$. Let $C_1$ be the constant of \ref{U:USR}, $C_2$ that of \Cref{lem holder of Esu}, and $C_3,\ldots,C_5$ those of \Cref{prop:straightening map}. Each of these results remains valid when its constant is enlarged, so we may assume $1<C_0<C_1<\cdots<C_5$. The map $\exp^{-1}_{x_0}$ is bi-Lipschitz on $\Lambda_{n_0}$; enlarging $C_5$ if necessary, both of its Lipschitz constants are at most $C_5$. In particular, for every $\bfi\in\Sigma^-$, $\sigma_\bfi$ is bi-Lipschitz with constant $C_3\leq C_5$ and $\sigma_\bfi\circ\exp_{x_0}$ is bi-Lipschitz with constant $C_5^2$. These constants depend only on the dynamical system and $n_0$; all constants introduced later are expressed in terms of them.
\end{rem}

\subsection{Spreading property of fiber measures}

 The starting point of this section is the $(n,\ell,\epsilon)$-dyadic-spreading condition introduced in \cite{wuProjectionTheoremsCountably2025}. Let $\mathcal{D}^h_k$ be the standard dyadic partition of $\mathbb{R}$ shifted by $h \in \mathbb{R}$ at scale $2^{-k}$ for $k \in \mathbb{N}$. The elements in $\mathcal{D}^h_k$ are of the form $[h+\frac{N}{2^k}, h+\frac{N+1}{2^k})$ for some $N \in \mathbb{Z}$. When $h=0$, we simply write $\mathcal{D}_k$. We first introduce a spreading property for points.
	 \begin{defi}
		Given $n,\ell\in\N$, $\epsilon>0$, $h\in \R$ and a probability measure $\eta$ on $\mathbb{R}$, we say a point $x\in \R$ satisfies \emph{$(n,\ell,\epsilon,h)$-dyadic-spreading} with respect to $\eta$ if  
			\[ \frac{1}{n}\#\{1\leq k\leq n,\ \eta(\mathcal{D}^h_k(x))>2 \eta(\mathcal{D}^h_{k+\ell}(x)) \}>1-\epsilon.\] 
	\end{defi}
	\begin{defi}\label{def-spreading}
        Given $n, \ell \in \mathbb{N}$, $\epsilon > 0$, and $h \in \mathbb{R}$, a probability measure $\eta$ on $\mathbb{R}$ is defined to be \emph{$(n, \ell, \epsilon, h)$-dyadic-spreading} if
        $$\eta \left(\left\{ x:\ x \text{ is }(n,\ell,\epsilon,h)\text{-dyadic-spreading} \right\}\right)  > 1-\epsilon.$$
        In the special case where $h=0$, the measure $\eta$ is referred to as \emph{$(n, \ell, \epsilon)$-dyadic-spreading}.
	\end{defi}

The core of Wu's argument \cite{wuProjectionTheoremsCountably2025} is to demonstrate that the fiber measure satisfies this property, allowing him to use modified sum-set estimates to yield dimension growth, provided that the fiber has positive dimension and the projection is not full-dimensional. 
Meanwhile, the proof by Bárány-Hochman-Rapaport \cite{barany_hochman_rapaport_2019} relies on projection measures—an approach that requires establishing properties such as porosity and is more demanding to handle than fiber measures. The current approach bypasses these difficulties by requiring only the projection dimension. The key point in verifying the dyadic-spreading property of the fiber measure is the positive entropy and the use of fiber dynamics.

\subsubsection{Fiber measures}
Define the map
\begin{align}
F: \Lambda_f \times \Sigma^- &\to \Lambda_f \times \Sigma^-\\
(x, \mathbf{i}) &\mapsto (f^{-1}_{i_{-1}}(x), \sigma(\mathbf{i})),
\end{align}
where, by a slight abuse of notation, $\sigma: \Sigma^- \to \Sigma^-$ denotes the right-shift map. We can naturally view $F$ as the inverse of $\hat{f}$ acting on the inverse limit space.

 Recall that $\mu$ is the Bernoulli measure with full support on $\Lambda_f$. Thus, the lift of $\mu$ on the inverse limit space is the product measure $\hat{\mu} := \mu \times \nu_-$ on $\Lambda_f \times \Sigma^-$, where $\nu_-$ is the fully supported Bernoulli measure that assigns the same weight to each $1$-cylinder set as $\nu_+$ (see the introduction to \Cref{sec:nledyadic}).
    Then $\hat{\mu}$ is $F^{-1}$-ergodic.
    
    We introduce the definition of subordinate partitions.
  \begin{defi}[Subordinate partition in the non-invertible setting]
A measurable partition $\mathcal{A}$ of $\Lambda_f \times \Sigma^-$ is said to be \emph{subordinate to the local strong unstable manifolds} if it satisfies:
\begin{enumerate}[(1)]
    \item For $\hat\mu$-almost every $\xx=(x,\bfi)$, the atom $\mathcal{A}(\xx)$ is contained in the local leaf:
    \begin{equation}
        \mathcal{A}(\xx) \subseteq W^{su}_{\mathrm{loc}}(x,\bfi),
    \end{equation}
    and contains a neighborhood of $\xx$ in $\Lambda_f\cap W^{su}_{\mathrm{loc}}(x,\bfi)$.
    \item The partition is increasing under the inverse map $F$, i.e., 
    for $\hat\mu$-almost every $\xx$,
    \begin{equation}
       [F\calA](\xx):=F(\mathcal{A}(F^{-1}(\xx))) \subseteq \mathcal{A}(\xx).
    \end{equation}
    \item The atoms are generated by negative iterates of $F$ and shrink to points:
    \begin{equation}
        \bigcap_{n\geq 0} [F^n \mathcal{A}](\xx) = \{\xx\}.
    \end{equation}
\end{enumerate}
\end{defi}
     
We define the partition $\mathcal{A}^{su}$ such that the element containing $(x,\mathbf{i}) \in \Lambda_f\times \Sigma^-$ is$$\mathcal{A}^{su}_{\mathbf{i}}(x) := \Lambda_{n_0,q}(x) \cap W^{su}_{\mathrm{loc}}(x,\mathbf{i}).$$This partition is indeed subordinate to the local strong unstable manifolds. Property $(1)$ is immediate from its construction, and the increasing property follows naturally from the structure of horseshoes. Finally, Property $(3)$ holds due to the uniform contraction along the strong unstable leaves.

Following Rokhlin \cite{Rokhlin_1967}, for $\nu_-$-a.e. $\bi\in \Sigma^-$, $\mu$ admits conditional measures with respect to the partition $\calA^{su}_{\bi}$ for $\mu$-a.e. $x\in \Lambda_f$, which we denote by $\mu_{x,\calA^{su}_\bfi}$ or  $\mu_{\xx,\calA^{su}}$.

    The \emph{partial entropy} is defined as follows, using the conditional entropy:
    \begin{equation}
	 H_{\hat{\mu}}(F\calA^{su}|\calA^{su})=\int-\log{\mu}_{\xx,\calA^{su}}([F\calA^{su}](\xx))d \hat{\mu}(\xx). 
    \end{equation}
    We have the following Ledrappier-Young formula for repellers.
	\begin{lem}[\cite{QianXie}]
    \label{lem:fiber measure}
		For $\hat{\mu}$-a.e. $\xx\in\Lambda_f\times\Sigma^-$, the fiber measure $\mu_{x,\calA^{su}_\bfi}$ is exact dimensional with 
        \begin{equation}
		\dim\mu_{x,\calA_\bfi^{su}}=\lim_{r\to 0} \frac{\log \mu_{x,\calA_\bfi^{su}}(B^{su}_{\bi,s}(x,r))}{\log r}
        =\frac{H_{\hat\mu}(F\calA^{su}|\calA^{su})}{\lambda_1{}(\mu,f)}, 
        \end{equation}
        where $\lambda_1(\mu,f)$ is defined as in \eqref{eq:def of LYexponent} and $B^{su}_{\bi,s}(x,r)$ is defined as in \eqref{eqn:ball straightening}.
	\end{lem}
    \begin{rem}
        The results in \cite{QianXie} are originally stated for $C^2$ repellers. However, using the main theorem in \cite{AaronBrown2022}, these results can be extended to $C^{1+\alpha}$ repellers in a standard way.
    \end{rem}

In \cite[Proposition 2.4]{QianXie}, they proved the exact dimension result for $\mu_{x,\calA^{su}_{\bfi}}$ with respect to the induced Riemannian metric on $W^{su}_{\mathrm{loc}}(x,\bfi)$. Since $\sigma_{\bfi}$ is bi-Lipschitz, we recover the same exact dimension when using balls in $W^{su}_{\mathrm{loc}}$ equipped with the metric induced by the straightening map.

    For brevity in what follows, we denote this dimension by
    \begin{equation}
    \label{eqn:fiber dim}
    \gamma_1:=\frac{H_{\hat\mu}(F\calA^{su}|\calA^{su})}{\lambda_1{}(\mu,f)},
    \end{equation}
    and call it the \emph{fiber dimension of $\mu$}.

	\begin{lem}[Invariance of fiber measure]\label{lem:invariance}
		For any measurable set $B\subset \Lambda_f\times \Sigma^-$ and for $\hat{\mu}$-a.e. $\xx\in \Lambda_f\times \Sigma^-$, 
        \begin{equation}
		 \frac{\mu_{x,\calA^{su}_\bfi}(F(B\cap \cA^{su}(F^{-1}\xx)))}{\mu_{x,\calA^{su}_\bfi}([F\cA^{su}](\xx))}= 
        \mu_{F^{-1}\xx,\cA^{su}}(B). 
        \end{equation}
    \begin{proof}
For any $i \in \{1,\dots,m\}$, denote by $\mu_i$ the conditional measure of $\mu$ on the $i$-th $1$-cylinder. Since $\mu$ is a Bernoulli measure, we have
\[
f_* \mu_i = \mu.
\]
Also note that $F$ is the inverse of $\hat{f}$. 
Therefore, for any measurable set $B \subset \Lambda_f \times \Sigma^-$ and $\hat{\mu}$-a.e. $\xx$, by the disintegration theorem,
\[
\mu_{\xx,\calA^{su}}(F(B \cap \calA^{su}(F^{-1}(\xx)))) = \mu_{\xx,\calA^{su}}([F\calA^{su}](\xx)) \cdot {\mu_{F^{-1}(\xx),\calA^{su}}(B)},
\]
which completes the proof.
\end{proof}
\end{lem}

   \subsubsection{Ball-spreading} 
   We introduce another variant of the spreading property.
	 \begin{defi}
		Given $n,\ell\in\N$, $\epsilon>0$, and a probability measure $\eta$ on $\mathbb{R}$, we say a point $x\in \R$ satisfies \emph{$(n,\ell,\epsilon)$-ball-spreading} with respect to $\eta$ if  
			\[ \frac{1}{n}\#\{1\leq k\leq n,\ \eta(B(x,2^{-k}))>2 \eta(B(x,2^{-k-\ell})) \}>1-\epsilon.\] 
        We say that $x$ satisfies \emph{uniform $(\ell,\epsilon)$-ball-spreading} with respect to $\eta$ if, for every $n\in \N$,
        it is $(n,\ell,\epsilon)$-ball-spreading with respect to $\eta$.
	\end{defi}

  In what follows, we restrict ourselves to a fixed  $n_0$-cylinder $\Lambda_{n_0,q}$ and let $\mu_0:=\frac{1}{\mu(\Lambda_{n_0,q})}\mu|_{\Lambda_{n_0,q}}$ be the conditional measure of $\mu$ on this cylinder. For $\nu_-$-a.e. $\bfi \in \Sigma^-$, the conditional measures of $\mu$ and $\mu_0$ with respect to $\calA^{su}_{\bfi}$ are identical; both equal $\mu_{x,\calA^{su}_{\bfi}}$ for almost every $x \in \Lambda_{n_0}$ with respect to both $\mu$ and $\mu_0$.
  In the following proposition, we use the positive entropy of the fiber measures to obtain the ball-spreading property.

	\begin{prop}\label{prop:ball}
		Suppose $\gamma_1>0$. For any $\epsilon>0$, there exist $\ell\in \N$ and $X_1\subset \Lambda_{n_0}\times\Sigma^-$ with $\mu_0\times\nu_-(X_1)>1-\epsilon$, such that for any $(x,\bfi)\in X_1$, $\sigma_\bfi(x)$ is uniform $(\ell,\epsilon)$-ball-spreading with respect to the push-forward of the fiber measure $({\sigma_\bfi})_*\mu_{x,\calA^{su}_\bfi}$ on  $\sigma_\bfi(W^{su}_{\mathrm{loc}}(x,\bfi))\subset\mathbb R^2$ with respect to the Euclidean metric. Here $\sigma_\bfi$ is the straightening map defined in \eqref{eqn:sigma i}.	
	\end{prop}
    \begin{rem}
        Note that if a point is $(n,\ell,\epsilon)$-ball-spreading with respect to a measure, then for any $\ell'>\ell$, this point is $(n,\ell',\epsilon)$-ball-spreading by definition. This observation will be used in Dyadic Property 1 in \cref{sec: step1}.
    \end{rem}

    In the rest of the section, we assume $\gamma_1>0$.
    To prove \cref{prop:ball}, we first use the positive partial entropy assumption to establish a spreading property for the fiber measures with respect to the subordinate partition. We then pass from this partition-level spreading property to the required ball-spreading property.
	We start with two lemmas.
	
	\begin{lem}\label{lem:one-scale}
		For any $\epsilon>0$, there exists $\ell\in \N$ such that
		\[ \hat\mu\left(\Big\{\xx\in\Lambda_{f}\times\Sigma^-:\ -\log\mu_{\xx,\calA^{su}}([F^\ell\calA^{su}](\xx))>\log 2\Big\}\right)>1-\epsilon. \]
	\end{lem}
	\begin{proof}
    For every $k\in \N$, we define a measurable function $\phi_k$ on $\Lambda_f\times \Sigma^-$ by 
		$$\phi_k(\xx)=-\log\mu_{\xx,\calA^{su}}([F^k\calA^{su}] (\xx)).$$ Then, due to the invariance of the fiber measure (\cref{lem:invariance}),  
		\begin{align}\label{equ:birkhoff-sum}
			\phi_k(\xx)&=\sum_{1\leq j\leq k}-\log\frac{ \mu_{\xx,\calA^{su}}([F^j\calA^{su}](\xx))}{ \mu_{\xx,\calA^{su} }([F^{j-1}\calA^{su}](\xx))}\\&=\sum_{1\leq j\leq k}-\log\frac{ \mu_{F^{-j+1}\xx,\calA^{su}}([F\calA^{su}](F^{-j+1}\xx))}{ \mu_{F^{-j+1}\xx,\calA^{su}}( \calA^{su}(F^{-j+1}\xx))}=\sum_{0\leq j\leq k-1} \phi_1(F^{-j}\xx) 
		\end{align}
		Since $\hat\mu$ is $F^{-1}$-ergodic, by the Birkhoff ergodic theorem, we have for $\hat\mu$-a.e. $\xx$, as $k\rightarrow\infty$,
		\[\frac{1}{k}\phi_k(\xx) \longrightarrow \int \phi_1(\yy)d\hat\mu(\yy)=H_{\hat\mu}(F\calA^{su}|\calA^{su  })>0\]
		Therefore, for $\hat\mu$-a.e. $\xx$, the function $\phi_k(\xx)\rightarrow\infty$ as $k\rightarrow\infty$, which implies the lemma from the convergence in measure.
	\end{proof}
	
	Recall the maximal ergodic theorem.
    Let $(X, \mathcal{B}, \eta)$ be a probability space, and let $T: X \to X$ be a measure-preserving transformation. Let $\phi \in L^1(X, \eta)$, and denote
	\[ M\phi(x):=\sup_{n} \frac{1}{n}\sum_{1\leq k\leq n} \phi(T^kx). \] 
	Then, for any $R>0$,
	\[ \eta(\{x:\ M\phi(x)>R\})\leq \frac{\|\phi\|_1}{R}. \]
	Using the maximal ergodic theorem, we strengthen the previous lemma to the following:
	\begin{lem}\label{lem:multi-scale}
		For any $\epsilon>0$, there exists $\ell\in \N$ such that 
        \begin{equation}
		\hat\mu \Big(\Big\{\xx:\inf_{n\in\N}\frac{1}{n}\#\left\{1\leq k\leq n, \mu_{\xx,\calA^{su}}([F^{k}\calA^{su}](\xx))>2 \mu_{\xx,\calA^{su}}([F^{k+\ell}\calA^{su}](\xx))\right \}>1-\epsilon\Big\} \Big)>1-\epsilon. 
        \end{equation}
	\end{lem} 
	\begin{proof}
    Fix $\epsilon>0$.
    Given any $\epsilon_0>0$, let $\ell\in \N$ be given as in \Cref{lem:one-scale}. Set
        \begin{equation}
        X_{\epsilon_0}:=\Big\{\xx\in\Lambda_{f}\times\Sigma^-:\ -\log\mu_{\xx,\calA^{su}}([F^\ell\calA^{su}](\xx))>\log 2\Big\}.
        \end{equation}
        Set $\phi=1_{X_{\epsilon_0}^c}$, and  we have $\|\phi\|_1<\epsilon_0$. 
        
        We apply the maximal ergodic theorem to the dynamical system  $(\Lambda_f\times \Sigma^-,\hat{\mu},F^{-1})$ and the function $\phi$. More precisely, 
       notice that by the invariance of the fiber measures (\Cref{lem:invariance}), we have for every $k\in \N$,
		\[ \frac{\mu_{\xx,\calA^{su}}([F^{k}\calA^{su}](\xx)) }{ \mu_{\xx,\calA^{su}}([F^{k+\ell}\calA^{su}](\xx))}=\frac{\mu_{F^{-k}\xx,\calA^{su}}(\calA^{su}(F^{-k}\xx)) }{ \mu_{F^{-k}\xx,\calA^{su}}([F^{\ell}\calA^{su}](F^{-k}\xx))} \]
		So $x$ satisfies $\mu_{\xx,\calA^{su}}([F^{k}\calA^{su}](\xx))>2 \mu_{\xx,\calA^{su} }([F^{k+\ell}\calA^{su}](\xx))$ if and only if $F^{-k}\xx\in X_{\epsilon_0}$. 
		  
		Therefore, 
		\[ M\phi(\xx)=\sup_{n} \frac{1}{n}\#\{ 1\leq k\leq n,\ \mu_{\xx,\calA^{su}}([F^{k}\calA^{su}](\xx))\leq 2 \mu_{\xx,\calA^{su}}([F^{k+\ell}\calA^{su}](\xx) \}, \]
		and by the maximal ergodic theorem
		\[\hat\mu\left(\left\{\xx:M\phi(\xx)>{\epsilon_0}^{1/2}\right\}\right)\leq \epsilon_0^{1/2} . \]
        By choosing $\epsilon_0<\epsilon^2$,
        we obtain the lemma. 
	\end{proof}

	\begin{proof}[Proof of \cref{prop:ball}]
    The proof is to compare atoms of $\calA^{su}$ with balls in $\R^2$. For every $\ell\in \N$ and for $\hat{\mu}$-a.e. $\xx$, we define the subset $N_{\ell}(\xx)\subset \N$ by
        \[
        N_{\ell}(\xx)=\{k\in \N:\mu_{x,\calA^{su}_\bfi}([F^{k}\calA^{su}_\bfi](x))>2 \mu_{x,\calA^{su}_\bfi}([F^{k+\ell}\calA^{su}_\bfi](x))\}.
        \]
        Given  $\epsilon>0$, it follows from \cref{lem:multi-scale} that there exists $\ell'\in \N$ such that 
        \begin{equation}
		 \hat{\mu}\left(\left\{\xx: \inf_{n\in \N}\frac{1}{n}\#(N_{\ell'}(\xx)\cap [1,n])>1-\epsilon\right\}\right)\geq 1-\epsilon .
        \end{equation}
		Restricting to $\Lambda_{n_0}$, we obtain
        \begin{equation}\label{equ: size of Nl}
		 \mu_0\otimes \nu_-\left(\left\{\xx: \inf_{n\in \N}\frac{1}{n}\#(N_{\ell'}(\xx)\cap [1,n])>1-\epsilon\right\}\right)\geq 1-\epsilon/\mu(\Lambda_{n_0}).
        \end{equation}
       For a word $\bfi\in\Sigma^-$ and $k\in \N$,  set $[\bfi]_k:=(i_{-1},\cdots,i_{-k})$, and
        \[ 
        \alpha(\bfi,k):=\|Df^{-1}_{i_{-k}}\circ\cdots\circ Df^{-1}_{i_{-1}}(x_0)|_{E^{su}(x_0,\bfi)}\|,
        \]
        where $x_0$ is the point in $\Lambda_{n_0}$ fixed in \eqref{eqn:x0}.

        We claim that there exists a uniform constant $C_1'>1$ such that for any $\bi\in \Sigma^-$, $k\in \mathbb{N}$, $z\in \Lambda_{n_0}$, and $y\in W^{su}_{\mathrm{loc}}(z,\bi)$,
        \[
        \frac{1}{C'_1}\leq \frac{\|Df^{-1}_{i_{-k}}\circ\cdots\circ Df^{-1}_{i_{-1}}(y)|_{T_yW^{su}_{\mathrm{loc}}(z,\bfi)}\|}{\alpha(\bfi,k)}\leq C_1'.
        \]
The proof of the claim is standard, but this property is crucial for establishing the spreading property in the non-linear case with respect to the self-affine case.
\begin{proof}[Proof of the claim]
We use the uniform H\"older continuity and the bounded distortion calculation. More precisely, fix $\bi\in \Sigma^-$, $k\in \N$, $z\in \Lambda_{n_0}$ such that
$
y\in W^{su}_{\mathrm{loc}}(z,\bfi).
$
Let
\[
F_j:=f^{-1}_{i_{-j}}\circ\cdots\circ f^{-1}_{i_{-1}},
\qquad 0\leq j\leq k,
\]
with $F_0=\Id$, and set
\[
x_j:=F_j(x_0),\qquad y_j:=F_j(y),\qquad z_j:=F_j(z).
\]

By construction, we have $y_j\in W^{su}_{\mathrm{loc}}(z_j,\sigma^j\bfi)$ for every
$0\leq j\leq k$, and the local strong unstable manifold is uniformly contracting under the inverse iterations of $f$. Moreover, since each inverse branch is uniformly contracting on the repeller, there exist constants $C_2'>0$ and
$\rho\in(0,1)$ such that
\begin{equation}\label{eq:orbit-estimates}
d(x_j,z_j)+d(y_j,z_j)
\leq
C_2'\cdot\rho^{j},
\qquad
0\leq j\leq k .
\end{equation}

By \cref{prop Holder continuous} and the uniform $\frac{\alpha}{2}$-H\"older continuity of the local strong unstable manifolds \ref{U:USR}, for any $\bfj\in \Sigma^-$, the function
\[
(u,\bfj)\longmapsto
\log\left\|Df^{-1}_{j_{-1}}(u)\big|_{E^{su}(u,\bfj)}\right\|
\]
is uniformly $\frac{\alpha}{2}$-H\"older with respect to $u$ on $\Lambda_f$ or along each local strong unstable leaf. Hence, there exists $C_3'>0$ such that for every
$0\leq j\leq k-1$,
\[
\begin{aligned}
&
\left|
\log\left\|Df^{-1}_{i_{-(j+1)}}(y_j)\big|_{E^{su}(y_j,\sigma^j\bfi)}\right\|
-
\log\left\|Df^{-1}_{i_{-(j+1)}}(x_j)\big|_{E^{su}(x_j,\sigma^j\bfi)}\right\|
\right|
\\
&\leq
C_3'\cdot d(y_j,z_j)^\frac{\alpha}{2}
+
C_3'\cdot d(z_j,x_j)^\frac{\alpha}{2}.
\end{aligned}
\]

Using the chain rule and \eqref{eq:orbit-estimates}, we obtain
\[
\begin{aligned}
\left|
\log
\frac{
\left\|
D\bigl(f^{-1}_{i_{-k}}\circ\cdots\circ f^{-1}_{i_{-1}}\bigr)(y)
\big|_{T_yW^{su}_{\mathrm{loc}}(z,\bfi)}
\right\|
}{
\alpha(\bfi,k)
}
\right|
&\leq
C_3'\cdot
\sum_{j=0}^{k-1}
\bigl(
d(y_j,z_j)^\frac{\alpha}{2}
+
d(z_j,x_j)^\frac{\alpha}{2}
\bigr)
\\
&\leq
2C_3' (C_2')^\frac{\alpha}{2}
\sum_{m=0}^{\infty}\rho^{\frac{\alpha m}{2}}
=:\log C_1' .
\end{aligned}
\]
Therefore,
\[
\frac{1}{C_1'}
\leq
\frac{
\left\|
D\bigl(f^{-1}_{i_{-k}}\circ\cdots\circ f^{-1}_{i_{-1}}\bigr)(y)
\big|_{T_yW^{su}_{\mathrm{loc}}(z,\bfi)}
\right\|
}{
\alpha(\bfi,k)
}
\leq
{C_1'},
\]
which proves the claim.
\end{proof}

        Moreover, note that by the construction of the partition $\calA^{su}$, there exists a constant $c_1\in (0,1)$ such that for any $x\in\Lambda_f$ and $\bfi\in\Sigma^-$, 
        \[
         B^{su}_\bfi(x,c_1)\cap\Lambda_f \subseteq \calA^{su}_\bfi(x)=\cA^{su}(x,\bfi)\subseteq B^{su}_\bfi(x,c_1^{-1})\cap\Lambda_f,
        \]
        where $B^{su}_\bfi(x,r)$ denotes the ball of radius $r$ with center $x$ on $W^{su}_{\mathrm{loc}}(x,\bfi)$ under the Riemannian metric. Combining this with the claim and the fact that the straightening map $\sigma_\bfi$ is uniformly bi-Lipschitz,
        there exists a uniform constant $C_4'>0$ such that the following holds: for any $x\in\Lambda_{n_0}$, $\bfi\in\Sigma^-$, and $k\in \N$,

		\begin{equation} \label{equ:change-scale} B^{su}_{\bfi,s}(x, \alpha (\bfi,k)/C_4')\cap\Lambda_f \subset [F^k\calA^{su}](\xx)\subset  B^{su}_{\bfi,s}(x, C_4'\cdot\alpha(\bfi,k)),
		\end{equation}
        where $B^{su}_{\bfi,s}(x,\cdot)$ are balls on the strong unstable manifolds with respect to the metric induced by the straightening map \eqref{eqn:ball straightening}.
        
		    For every $k\in\mathbb N$, $x\in \Lambda_{n_0}$, and $\bfi\in \Sigma^-$, let $n_1(k,\xx)$ be the maximal $n\in \N$ such that
		\[   B^{su}_{\bfi,s}(x,2^{-k})\cap\Lambda_f\subset[ F^n\calA^{su}](\xx), \]
		and $n_2(k,\xx)$ be the minimal $m\in \N$ such that 
		\[  [F^m\calA^{su}](\xx)\subset B^{su}_{\bfi,s}(x,2^{-k}).
        \]
		Therefore, for $\mu_0\times\nu_-$-a.e. $\xx=(x,\bfi)\in\Lambda_{n_0}\times\Sigma^-$ and $k,\ell\in \N$,
		\[ \frac{\mu_{\xx,\calA^{su}}(B^{su}_{\bfi,s}(x,2^{-k}))}{\mu_{\xx,\calA^{su}}(B^{su}_{\bfi,s}(x,2^{-k-\ell}))}\geq \frac{\mu_{\xx,\calA^{su}}([F^{n_2(k,\xx)}\calA^{su}](\xx))}{\mu_{\xx,\calA^{su}}([F^{n_1(k+\ell,\xx)}\calA^{su}](\xx))}. \]
		By \eqref{equ:change-scale}, there exists $\ell>0$ large enough that depends only on $\ell'$ such that $$n_1(k+\ell,\xx)-n_2(k,\xx)\geq \ell'.$$ Hence, if $n_2(k,\xx)\in N_{\ell'}(\xx)$, we have 
        \[
        {\mu_{\xx,\calA^{su}}(B^{su}_{\bfi,s}(x,2^{-k}))}\geq 2 \mu_{\xx,\calA^{su}}(B^{su}_{\bfi,s}(x,2^{-k-\ell})).
        \]
It remains to obtain an upper estimate for
\[
    \#\{1\leq k\leq n,\ n_2(k,\xx)\in N_{\ell'}(\xx)^c \}.
\]
By \eqref{equ:change-scale}, there exists a constant $C>0$ such that, for every
$\xx$, the map $k\mapsto n_2(k,\xx)$ is at most $C$-to-one and satisfies
$n_2(k,\xx)\leq Ck$ for all $k\in\mathbb N$. 
Together with \eqref{equ: size of Nl}, for all $\xx$ in a set whose $\mu_0\times\nu_-$ measure is greater than {$1-\epsilon/\mu(\Lambda_{n_0})$}, we have
\[
    \#\{1\leq k\leq n,\ n_2(k,\xx)\in N_{\ell'}(\xx)^c \}
    \leq
    C \#\{1\leq n_2\leq Cn,\ n_2\in N_{\ell'}(\xx)^c \}
    \leq
    C^2\epsilon n.
\]
Choosing $\epsilon$ in the preceding estimate sufficiently small completes the proof.
	\end{proof}

\subsubsection{Ball-spreading and dyadic-spreading}

In the previous subsection, we showed that the positive entropy of the fiber measure implies the ball-spreading property. For the additive-combinatorial argument used later, however, the more classical and convenient formulation is the dyadic-spreading property, and the following lemma provides the passage from the former to the latter.
 
  \begin{lem}[Ball-spreading to dyadic-spreading]\label{lem:ball-dyadic-spreading}
  Let $\eta$ be a probability measure on $\R$, $\ell\in \N$, and $\epsilon>0$ with $\epsilon>2^{-\ell+1}$. Suppose
 	 a point $x\in \R$ satisfies uniform $(\ell,\epsilon)$-ball-spreading with respect to $\eta$. Then there exist 
     $N=N(x,\ell,\epsilon)\in \N$ and
     $I=I(x,\ell,\epsilon)\subset [0,1] $ with $\mathrm{Leb}(I)>1-\epsilon$, such that for any $h\in I$ and $n>N$, the point $x$ is also $(n,2\ell,10\epsilon, h)$-dyadic-spreading with respect to $\eta$. The $N$ can be chosen to be measurable with respect to $x$.
	\end{lem}

   \begin{proof}
   For every $n\in \N$ and $h\in [0,1]$, define the index sets:
\begin{align}
&A_n(x,h) := \{ 1\leq k\leq n,\, d(2^k(x-h),\mathbb{Z})>2^{-\ell} \},\\
&B_n(x) := \{ 1\leq k\leq n,\, \eta(B(x,2^{-k-\ell}))>2\eta (B(x,2^{-k-2\ell}))\}.
\end{align}
   By Borel's normal number theorem, for $\mathrm{Leb}$-a.e.  $y\in[0,1]$, the sequence $\{2^k y\}_{k\in\mathbb N}$ is equidistributed in $\R/\Z$ with respect to $\mathrm{Leb}$ (see, for instance, \cite{KuipersNiederreiter}).  
This implies that, 
for 
$\mathrm{Leb}$-a.e. $h\in[0,1]$,
the sequence 
$\{2^k(x-h)\}_{k\in\mathbb N}$ is equidistributed  in $\R/\Z$ 
with respect to $\mathrm{Leb}$. In particular, we have for $\mathrm{Leb}$-a.e. $h\in[0,1]$,
		\begin{equation}
			\lim_{n\to\infty} \frac{1}{n}\# A_n(x,h)=1-2^{-\ell+1}. 
		\end{equation}
		By Egorov's theorem, for any $\epsilon\in (2^{-\ell+1},1)$, there exist a set $I\subset [0,1]$ with $\mathrm{Leb}(I)>1-\epsilon$ and $\ N\in \N$ such that for each $h\in I$ and for all $n>N$, 
		\begin{equation}\label{equ:separate-z}
			\frac{1}{n}\#A_n(x,h)>1-\epsilon. 
		\end{equation}  
		For every $k\in A_n(x,h)$, we obtain
		\[ \cD_k^h(x)\supset B(x,2^{-k-\ell})\supset B(x,2^{-k-2\ell})\supset \cD_{k+2\ell}^h(x),\]
		which yields 
		\[ \frac{\eta(\cD_k^h(x))}{\eta(\cD_{k+2\ell}^h(x))}\geq\frac{\eta(B(x,2^{-k-\ell}))}{\eta(B(x,2^{-k-2\ell}))} \]
        whenever $\eta(\cD_{k+2\ell}^h(x))>0$. From this, we deduce the set inclusion:
        \begin{equation}
        A_n(x,h) \cap B_n(x) \subset \{ 1\leq k\leq n : \eta(\mathcal{D}^h_k(x))>2\eta(\mathcal{D}^h_{k+2\ell}(x))\}.
        \end{equation}
        Thus, the lemma holds provided we take $N\gg \ell/\epsilon$.
	\end{proof}

\section{Sum set estimate}\label{sec:additive}

The version of the sum set estimate needed here is Theorem 4.7 from \cite{wuProjectionTheoremsCountably2025}, whose proof is based on Hochman's inverse theorem \cite{hochman_self-similar_2014} and the Balog-Szemer\'edi-Gowers theorem (Corollary 2.36 of \cite{TV}).

\begin{thm}\label{thm:growth}
For any $\epsilon\in (0,1)$ and $\ell\in \N$, there exists $\delta>0$ such that the following holds for $n\geq n(\epsilon,\ell)$.

Let $\eta$ be a probability measure on $[0,1]\cap 2^{-n}\Z$ and $A\subset [0,1]\cap 2^{-n}\Z$ with $\eta(A)\geq 1/2$ and each element in $A$ satisfying $(n,\ell,\epsilon)$-dyadic-spreading.
Let $B$ be a subset of $[0,1]\cap 2^{-n}\Z$ with $|B|\leq 2^{(1-2\epsilon)n}$. Suppose for each $a\in A$, there is a subset $B_a\subset B$ with $|B_a|\geq |B|^{1-\delta}$. Then we have 
\[
\left|\bigcup_{a\in A}(a+B_a)\right|\geq |B|^{1+\delta}.
\]
\end{thm}

Wu's theorem is formulated on the dyadic grid and gives a cardinality estimate for linear sumsets. In the next section, we will instead need a covering-number estimate at a rescaled local scale, where the translating parameters are also allowed to undergo a small non-linear perturbation. The following corollary records precisely the form needed for that application, and will be deduced from \Cref{thm:growth} through the reductions developed in this section.

The use of the covering number is motivated by the need to analyze sets $B \subset \R$ that are not $2^{-n}$-discrete; by discretizing at scale $2^{-n}$, we can naturally frame our conclusions in terms of covering numbers. Specifically, for a subset $B \subset \R$ and $r > 0$, let $\cN_r(B)$ denote the minimal number of open balls of radius $r$ required to cover $B$.

\begin{cor}[Nonlinear, rescaled form in covering]\label{cor:final-growth-addcom}
For every \(\epsilon_*\in(0,\frac12)\), \(\ell_*\in\mathbb N\), $C,C'\in\mathbb N$, and $s>\frac{1}{100}$, there exists
\(0<\delta=\delta(\epsilon_*,\ell_*)\leq1\) such that the following holds for all
\(n\geq N(\epsilon_*,\ell_*,C',C)\). Let \(I\) be an interval of length
\(s 2^{-(C+1)n}\), and let \(\eta\) be a probability measure on \(I\) whose
affine normalization to \([0,1]\) is \((n,\ell_*,\epsilon_*/6)\)-dyadic-spreading with
respect to the standard dyadic partition. Let \(A\subset\supp\eta\) be a discrete set such that, for the affine
normalization \(T_I:I\to[0,1]\),
\begin{equation}
\label{eqn:lower bound measure}
\eta\left(\bigcup_{a\in A}T_I^{-1}(\cD_n(T_I(a)))\right)>\frac23.
\end{equation}
Let
\[
B\subset [-3C'2^{-(C+1)n},3C' 2^{-(C+1)n}]
\]
satisfy
\[
\cN_{s 2^{-(C+2)n}}(B)
\leq 2^{(1-\epsilon_*)n }.
\]
Assume that for every \(a\in A\) there is a set \(B_a\subset B\) with
\begin{equation}
\label{eqn:lower bound Ba}
\cN_{s 2^{-(C+2)n}}(B_a)
\geq
\cN_{s 2^{-(C+2)n}}(B)^{1-\delta}.
\end{equation}
Finally, let \(\phi:A\to\mathbb R\) satisfy
\[
|\phi(a)-a|\leq s 2^{-(C+3)n},
\qquad a\in A.
\]
Then
\[
\cN_{s 2^{-(C+2)n}}
\left(
\bigcup_{a\in A}(\phi(a)+B_a)
\right)
\geq
\frac1{200}
\cN_{s 2^{-(C+2)n}}(B)^{1+\delta}.
\]
\end{cor}

We postpone the proof of \Cref{cor:final-growth-addcom} until the end of the section. We first record the elementary covering-number estimates needed in this section. Let $d_{\mathrm{H}}(\cdot, \cdot)$ denote the Hausdorff distance between subsets of $\R$.

\begin{lem}\label{lem:discritization}
Let $A$ be a subset of $\R$. Then the following holds.
\begin{enumerate}[(1)]

\item For any $C>1$ and $n\in \N$, we have
\begin{equation}
\cN_{C2^{-n}}(A)\leq \cN_{2^{-n}}(A) \leq (2C) \cN_{C2^{-n}}(A).
\end{equation}

\item Given another subset $A'\subset \R$ with $d_{\mathrm{H}}(A,A')\leq 2^{-n}$ for some $n\in \N$, we have
\begin{equation}
\frac{1}{3} \cN_{2^{-n}}(A')\leq \cN_{2^{-n}}(A) \leq 3 \cN_{2^{-n}}(A').
\end{equation}

\item For $r>0$, let $A^{(r)}$ be the $r$-neighborhood of $A$, that is $A^{(r)}:=A+B(0,r)$. Then, for any $n\in \N$,
\begin{equation}
3\cN_{2^{-n}}(A)\geq \cN_{2^{-n}}(A^{(2^{-n})})\geq \frac{1}{2} |A^{(2^{-n})}\cap 2^{-n}\Z|\geq \frac{1}{2}\cN_{2^{-n}}(A). 
\end{equation}

\end{enumerate}
\end{lem}

We now combine the enlargement of the support, the discretization argument, and the non-linear perturbation of \cref{thm:growth} in the following corollary.

\begin{cor}[Non-linear]\label{cor:growth-nonlinear}
For any $\epsilon\in (0,\frac{1}{12})$ and $\ell\in \N$, there exists $0<\delta\leq1$ such that the following holds for any $C>1$ and all $n\geq n(\epsilon,\ell,C)$.

Let $\eta$ be a probability measure on $[0,1]$ satisfying $(n,\ell,\epsilon)$-dyadic-spreading. Let $A\subset\supp \eta$ be a discrete set with
\[
\eta\left(\bigcup_{a\in A}\cD_n(a)\right)>\frac23.
\]
Let $B$ be a subset of $[-C,C]$ with
\[
\cN_{2^{-n}}(B)\leq 2^{(1-6\epsilon)n}.
\]
Suppose for each $a\in A$, there is a subset $B_a\subset B$ with
\[
\cN_{2^{-n}}(B_a)\geq \left(6\cN_{2^{-n}}(B)\right)^{1-\delta}.
\]
Let $\phi:[0,1]\rightarrow [-1,2]$ be a map satisfying
\[
|\phi(x)-x|\leq 2^{-2n}
\qquad\text{for every }x\in A.
\]
Then
\[
\cN_{2^{-n}}
\left(
\bigcup_{a\in A}(\phi(a)+B_a)
\right)
\geq
\frac{1}{200}\cN_{2^{-n}}(B)^{1+\delta}.
\]
\end{cor}

\begin{proof}
Fix $\epsilon\in (0,\frac{1}{12})$, $\ell\in\N$, and $C>1$. Let $\delta_0>0$ be the constant given by \Cref{thm:growth} for the parameters $(3\epsilon,\ell)$, and set
\(
\delta:=\min\{\delta_0,1\}.
\)
We take $n$ sufficiently large depending on $\epsilon,\ell$, and $C$, and write
\(
r:=2^{-n}.
\)

\medskip\noindent\textbf{Step 1: Discretization.}
For every $\mathcal D\in\mathcal D_n$, let $x_{\mathcal D}$ be the left endpoint of $\mathcal D$, and define
\[
\eta':=\sum_{\mathcal D\in\mathcal D_n}\eta(\mathcal D)\delta_{x_{\mathcal D}}.
\]
Let
\[
A':=\{x_{\mathcal D_n(a)}:a\in A\}.
\]
Then
\[
\eta'(A')
=
\eta\left(\bigcup_{a\in A}\mathcal D_n(a)\right)
>
\frac23.
\]
For each $a'\in A'$, set
\[
\widetilde B_{a'}
:=
\bigcup_{a\in A\cap\mathcal D_n(a')}B_a.
\]
Thus $\widetilde B_{a'}\subset B$ and
\[
\cN_r(\widetilde B_{a'})
\geq
\left(6\cN_r(B)\right)^{1-\delta}.
\]

We discretize $B$ and the sets $\widetilde B_{a'}$ at scale $r$ by setting
\[
D:=B^{(r)}\cap r\Z,
\qquad
D_{a'}:=\widetilde B_{a'}^{(r)}\cap r\Z.
\]
By \Cref{lem:discritization},
\[
\cN_r(B)\leq |D|\leq 6\cN_r(B)
\]
and, for every $a'\in A'$,
\[
|D_{a'}|
\geq
\cN_r(\widetilde B_{a'})
\geq
\left(6\cN_r(B)\right)^{1-\delta}
\geq
|D|^{1-\delta}
\geq
|D|^{1-\delta_0}.
\]
Moreover,
\[
|D|
\leq
6\cdot 2^{(1-6\epsilon)n}.
\]

We next verify the spreading property after the discretization. Let $G$ be the set of points which satisfy $(n,\ell,\epsilon)$-dyadic-spreading with respect to $\eta$. By assumption,
\[
\eta(G)>1-\epsilon.
\]
If $\mathcal D\in\mathcal D_n$ intersects $G$, then for every $x\in G\cap\mathcal D$ the points $x$ and $x_{\mathcal D}$ belong to the same dyadic atom at every level $k\leq n$. Hence, for $1\leq k\leq n-\ell$, the corresponding dyadic masses for $\eta$ and $\eta'$ agree. Consequently, if $n$ is sufficiently large so that $\ell<\epsilon n$, then $x_{\mathcal D}$ satisfies $(n,\ell,2\epsilon)$-dyadic-spreading with respect to $\eta'$. Therefore there exists
\[
G'\subset\supp\eta'
\]
with
\[
\eta'(G')>1-\epsilon
\]
such that every point of $G'$ satisfies $(n,\ell,2\epsilon)$-dyadic-spreading with respect to $\eta'$.

\medskip\noindent\textbf{Step 2: Reduction to Wu's theorem.}
Choose an integer $m=m(\epsilon,C)$ sufficiently large so that, with
\[
S_m(x):=2^{-m}x,
\]
we have
\[
S_m(D)+\frac12\subset[0,1]
\]
and
\[
6\leq 2^{(1-6\epsilon)m}.
\]
Set
\[
\eta^\#=(S_m)_*\eta',
\qquad
A^\#=S_m(A'),
\qquad
D^\#=S_m(D)+\frac12,
\]
and, for $a^\#=S_m(a')$,
\[
D_{a^\#}^\#:=S_m(D_{a'})+\frac12.
\]
All these discrete objects lie on the grid $2^{-(n+m)}\Z$.

For a point of $G'$, the more than $n(1-2\epsilon)$ good scales for $\eta'$ become good scales for $\eta^\#$ after shifting the scale index by $m$. Since $m$ is fixed, for all sufficiently large $n$,
\[
n(1-2\epsilon)>(n+m)(1-3\epsilon).
\]
It follows that every point of $S_m(G')$ satisfies $(n+m,\ell,3\epsilon)$-dyadic-spreading with respect to $\eta^\#$. Moreover,
\[
\eta^\#(S_m(G'))>1-\epsilon.
\]
Hence
\[
\eta^\#\left(A^\#\cap S_m(G')\right)
>
\frac23-\epsilon
>
\frac12.
\]

The cardinality bound on $D^\#$ is
\[
|D^\#|
=
|D|
\leq
6\cdot 2^{(1-6\epsilon)n}
\leq
2^{(1-6\epsilon)(n+m)}.
\]
Also, for every $a^\#\in A^\#$,
\[
|D_{a^\#}^\#|
=
|D_{a'}|
\geq
|D|^{1-\delta_0}
=
|D^\#|^{1-\delta_0}.
\]
We may therefore apply \Cref{thm:growth}, with parameters $(3\epsilon,\ell)$ and scale $n+m$, to
\[
\eta^\#,\qquad
A^\#\cap S_m(G'),\qquad
D^\#,\qquad
D_{a^\#}^\#.
\]
This gives
\[
\left|
\bigcup_{a^\#\in A^\#\cap S_m(G')}
\left(a^\#+D_{a^\#}^\#\right)
\right|
\geq
|D^\#|^{1+\delta_0}.
\]
Scaling back and enlarging the union to all of $A'$ yields
\begin{equation}\label{eq:merged-discrete-growth}
\left|
\bigcup_{a'\in A'}(a'+D_{a'})
\right|
\geq
|D|^{1+\delta}.
\end{equation}

\medskip\noindent\textbf{Step 3: Passage to covering numbers.}
Set
\[
U:=\bigcup_{a\in A}(a+B_a),
\qquad
U':=\bigcup_{a'\in A'}(a'+\widetilde B_{a'}).
\]
Since
\[
|a-x_{\mathcal D_n(a)}|\leq r
\]
for every $a\in A$, we have
\[
d_{\mathrm H}(U,U')\leq r.
\]
Therefore, by \Cref{lem:discritization},
\begin{equation}\label{eq:merged-U-Uprime}
\cN_r(U)\geq\frac13\cN_r(U').
\end{equation}
On the other hand,
\[
\bigcup_{a'\in A'}(a'+D_{a'})
\subset
(U')^{(r)}\cap r\Z.
\]
Using \Cref{lem:discritization} once more,
\[
\cN_r(U')
\geq
\frac16
\left|
\bigcup_{a'\in A'}(a'+D_{a'})
\right|.
\]
Together with \eqref{eq:merged-discrete-growth}, \eqref{eq:merged-U-Uprime}, and
\[
|D|\geq\cN_r(B),
\]
we obtain
\begin{equation}\label{eq:merged-linear-covering}
\cN_r(U)
\geq
\frac1{18}\cN_r(B)^{1+\delta}.
\end{equation}

\medskip\noindent\textbf{Step 4: Non-linear perturbation.}
Set
\[
U_\phi:=\bigcup_{a\in A}(\phi(a)+B_a).
\]
For every $a\in A$,
\[
d_{\mathrm H}(a+B_a,\phi(a)+B_a)
\leq
|\phi(a)-a|
\leq
r^2,
\]
and hence
\[
d_{\mathrm H}(U,U_\phi)\leq r^2<r.
\]
By \Cref{lem:discritization} and \eqref{eq:merged-linear-covering},
\[
\cN_r(U_\phi)
\geq
\frac13\cN_r(U)
\geq
\frac1{54}\cN_r(B)^{1+\delta}
\geq
\frac1{200}\cN_r(B)^{1+\delta}.
\]
This proves the corollary.
\end{proof}

It remains to obtain the rescaled form stated in \Cref{cor:final-growth-addcom}. This follows by applying the preceding non-linear covering-number estimate after an affine normalization of the interval $I$.

\begin{proof}[Proof of \cref{cor:final-growth-addcom}]
Let
\(
\epsilon:=\frac{\epsilon_*}{6}.
\)
Then $\epsilon\in(0,\frac1{12})$. Let
\(
0<\delta_0\leq1
\)
be the constant given by \Cref{cor:growth-nonlinear} for the parameters $(\epsilon,\ell_*)$, and set
\(
\delta:=\frac{\delta_0}{4}.
\)

Let \(T_I:I\to[0,1]\) be the affine normalization and extend it affinely to
\(\mathbb R\). Its slope is \(s^{-1}2^{(C+1)n}\). For
\(a^\#=T_I(a)\), define
\[
A^\#:=T_I(A),
\qquad
B^\#:=s^{-1}2^{(C+1)n}B,
\qquad
B_{a^\#}^\#:=s^{-1}2^{(C+1)n}B_a,
\]
and
\[
\phi^\#(a^\#):=T_I(\phi(a)).
\]
The scale \(s 2^{-(C+2)n}\) becomes \(2^{-n}\), and covering numbers are
preserved under this affine normalization:
\[
\cN_{2^{-n}}(B^\#)
=
\cN_{s 2^{-(C+2)n}}(B),
\qquad
\cN_{2^{-n}}(B_{a^\#}^\#)
=
\cN_{s 2^{-(C+2)n}}(B_a).
\]
Moreover, since \(s>\frac1{100}\),
\[
B^\#
\subset
[-300C',300C']
\subset
[-1000C',1000C'].
\]
The mass condition becomes
\[
(T_I)_*\eta
\left(
\bigcup_{a^\#\in A^\#}\cD_n(a^\#)
\right)
>
\frac23,
\]
and \((T_I)_*\eta\) is \((n,\ell_*,\epsilon)\)-dyadic-spreading by assumption. Also,
\[
\cN_{2^{-n}}(B^\#)
\leq
2^{(1-\epsilon_*)n}
=
2^{(1-6\epsilon)n},
\]
and
\[
|\phi^\#(a^\#)-a^\#|
\leq
s^{-1}2^{(C+1)n}\cdot
s2^{-(C+3)n}
=
2^{-2n}.
\]
For all sufficiently large $n$, the values $\phi^\#(a^\#)$ lie in $[-1,2]$; extend $\phi^\#$ arbitrarily to a map
\(
\phi^\#:\,[0,1]\to[-1,2].
\)

Write
\(
N_B:=\cN_{2^{-n}}(B^\#).
\)
We distinguish two cases.

If
\(
N_B^{2\delta}<200,
\)
then, since $A^\#$ is nonempty and every $B_{a^\#}^\#$ satisfies the assumed lower bound,
\[
\begin{split}
\cN_{2^{-n}}
\left(
\bigcup_{a^\#\in A^\#}
(\phi^\#(a^\#)+B_{a^\#}^\#)
\right)
&\geq
\cN_{2^{-n}}(B_{a^\#}^\#)\geq
N_B^{1-\delta}\geq
\frac1{200}N_B^{1+\delta}.
\end{split}
\]
Thus the desired conclusion holds.

Suppose now that
\(
N_B^{2\delta}\geq200.
\)
Since $\delta=\delta_0/4$, we have
\[
N_B^{\delta_0-\delta}
=
N_B^{3\delta_0/4}
=
\left(N_B^{\delta_0/2}\right)^{3/2}
\geq
200^{3/2}
>
6^{1-\delta_0}.
\]
Therefore, for every $a^\#\in A^\#$,
\[
\cN_{2^{-n}}(B_{a^\#}^\#)
\geq
N_B^{1-\delta}
\geq
(6N_B)^{1-\delta_0}.
\]
All the hypotheses of \Cref{cor:growth-nonlinear} are now satisfied with
\[
\epsilon=\frac{\epsilon_*}{6},
\qquad
\ell=\ell_*,
\qquad
C_1=1000C'.
\]
Hence,
\[
\cN_{2^{-n}}
\left(
\bigcup_{a^\#\in A^\#}
\left(\phi^\#(a^\#)+B_{a^\#}^\#\right)
\right)
\geq
\frac1{200}
N_B^{1+\delta_0}
\geq
\frac1{200}
N_B^{1+\delta}.
\]
Scaling back by \(T_I^{-1}\) proves \Cref{cor:final-growth-addcom}.
\end{proof}

\section{Projection and growth of dimension}\label{sec: sec5}	
\subsection{Statement and preparation}
\label{subsec:statement and preparation}
Following the setup and notations of \Cref{subsubsection:straightening map},
throughout \cref{sec: sec5}, we work with a fixed cylinder $\Lambda_{n_0}$ and consider the conditional measure $\mu_0$ of $\mu$ on $\Lambda_{n_0}$, given by
\begin{equation}
\mu_0=\frac{1}{\mu(\Lambda_{n_0})}\mu|_{\Lambda_{n_0}}.
\end{equation}

For every $z\in\Lambda_{n_0}$ and $\bfi\in\Sigma^-$, let
$T_{z,\bfi}$ be the local transversal through $z$ such that
$\exp^{-1}_{x_0}(T_{z,\bfi})$ is a segment of the line through
$\exp^{-1}_{x_0}(z)$ orthogonal to
\[
\big(D_z\exp^{-1}_{x_0}\big)\big(E^{su}_{\bfi}(z)\big).
\]
As in \Cref{subsubsection:straightening map}, after shrinking $\Lambda_{n_0}$ if necessary, these transversals are $\theta$-uniform, and every
relevant local $su$-manifold intersects each transversal considered as above in
exactly one point. Throughout this section, all $\theta$-uniform transversals
are chosen with this unique-intersection
property.

For $\bfi\in\Sigma^-$ and any such $\theta$-uniform $su$-transversal $T$, define
\begin{equation}
\pi^{su}_{\bfi,T}:\Lambda_{n_0}\to T
\end{equation}
by letting $\pi^{su}_{\bfi,T}(x)$ be the unique point of
$W^{su}_{\mathrm{loc}}(x,\bfi)\cap T$. We denote the corresponding transverse measure by
\begin{equation}
\label{eqn:muTbfi}
\mu^T_\bfi:=(\pi^{su}_{\bfi,T})_*\mu_0 .
\end{equation}
We define balls with respect to the transverse metric. For $\bfi\in \Sigma^-$, $x\in \Lambda_{n_0}$, and sufficiently small $r>0$, these are given by
\begin{equation}
B^T_\bfi(x,r):=\{u\in T: d_T(\pi^{su}_{\bfi,T}(x),u)\leq r\},
\end{equation}
where $d_T$ is the induced Riemannian metric on $T$.

We have the following result on the exact dimensionality of transverse measures.
\begin{lem}[\cite{QianXie}]
\label{lem:exact dim tran}
There exists $\gamma_2\in[0,1]$
such that for $\mu_0\times\nu_-$-a.e. $(x,\bfi)$ and any $\theta$-uniform transversal $T$,
\begin{equation}
\lim_{r\to0}
\frac{\log \mu^T_\bfi\bigl(B^T_\bfi(x,r)\bigr)}
{\log r}
=
\gamma_2 .
\end{equation}
Moreover, the measure $\mu$ itself is exact dimensional, and one has
\[
\dim \mu=\gamma_1+\gamma_2,
\]
where $\gamma_1$ is the fiber dimension of $\mu$ given in \eqref{eqn:fiber dim}.
{We also have the Ledrappier-Young type entropy formula:
\begin{equation}\label{equ:entropy-formula}
    h_\mu(f)=\gamma_2\lambda_2(\mu,f)+\gamma_1\lambda_1(\mu,f) 
\end{equation}}
\end{lem}

The value of \(\gamma_2\) is independent of the particular choice of the
transversal \(T\), since the \(su\)-holonomies between \(\theta\)-uniform
transversals are uniformly bi-Lipschitz (see \Cref{lem holder of Esu}). We call
this common value the \emph{transverse dimension of \(\mu\)}.

\Cref{thm:main-two} is our main theorem in \cref{sec: sec5}. 
Before the proof, we first explain how \cref{thm:base} follows from it.
\begin{proof}[Proof of \cref{thm:base}]
By definition of Lyapunov dimension \eqref{eqn:def LY} and the entropy formula \eqref{equ:entropy-formula},
one can see that \cref{thm:base} follows directly from \cref{thm:main-two}.
\end{proof}

We first exclude the zero-dimensional transverse case.
\begin{lem}\label{lem:gamma2-positive}
If $\gamma_1>0$, then $\gamma_2>0$.
\end{lem}

\begin{proof}
Suppose, by contradiction, that $\gamma_2=0$.
Hence, by the zero transverse dimension criterion as in \cite{Ledrappier_Xie}, the following local support property holds on a full $\mu_0\times\nu_-$ measure set: for $(x,\bfi)$ in this set, after restricting to a sufficiently small neighborhood of $x$, the measure $\mu_0$ is supported on the local strong unstable manifold $W^{su}_{\mathrm{loc}}(x,\bfi)$.

Let $\mathcal G\subset\Lambda_{n_0}\times\Sigma^-$ be this full measure set. By Fubini's theorem, there is a full $\mu_0$-measure subset $X_0\subset\Lambda_{n_0}$ such that, for every $x\in X_0$, the section
\[
    \Sigma_x:=\{\bfi\in\Sigma^-:(x,\bfi)\in\mathcal G\}
\]
has full $\nu$-measure. Choose $x\in X_0\cap\supp\mu_0$. Since $\mu_0$ is the normalized restriction of a fully supported Bernoulli measure to the cylinder $\Lambda_{n_0}$, we have $\supp\mu_0=\Lambda_{n_0}$ and $\mu_0$ has no atoms.

Applying {\cref{prop: small angle}} to the full measure set $\Sigma_x$, we can choose two histories $\bfi,\bfj\in\Sigma_x$ such that
\[
    \angle\bigl(E^{su}(x,\bfi),E^{su}(x,\bfj)\bigr)>0.
\]
Since $(x,\bfi),(x,\bfj)\in\mathcal G$, in a sufficiently small neighborhood $U_x$ of $x$ the measure $\mu_0$ is supported both on $W^{su}_{\mathrm{loc}}(x,\bfi)$ and on $W^{su}_{\mathrm{loc}}(x,\bfj)$. These two $C^1$ plaques have distinct tangent directions at $x$; hence, after shrinking $U_x$ if necessary,
\[
    W^{su}_{\mathrm{loc}}(x,\bfi)\cap W^{su}_{\mathrm{loc}}(x,\bfj)\cap U_x=\{x\}.
\]
Therefore, $\mu_0|_{U_x}$ is supported on the single point $x$. Since $x\in\supp\mu_0$, this forces $\mu_0(\{x\})>0$, contradicting the non-atomicity of $\mu_0$. Thus $\gamma_2>0$.
\end{proof}

    Proving \Cref{thm:main-two} requires bounding the measure of $2^{-C\ell}\times 2^{-C'\ell}$ rectangles. The common strategy of applying the Lebesgue density theorem (e.g., \cite[Proposition 3.8]{barany_ledrappier-young_2015}) cannot be used here, as it fails for rectangles \cite[2.8.20]{federer_geometric_1969}. As an alternative, we utilize a variant of the Lebesgue density theorem formulated for partitions (see \cite[2.9.8]{federer_geometric_1969}).

	\begin{defi}
For $\ell,\ell'\in\N$ and $h\in[0,1]$, let $\cD^{h}_{\ell\times\ell'}$ be the partition of $\R^2$ into the half-open rectangles
\[
\big[i2^{-\ell},(i+1)2^{-\ell}\big)\times\big[h+j2^{-\ell'},h+(j+1)2^{-\ell'}\big),\qquad i,j\in\Z,
\]
that is, the dyadic partition into $2^{-\ell}\times2^{-\ell'}$ rectangles with sides parallel to $e_1,e_2$ and grid origin $(0,h)$.
For $\bfi\in\Sigma^-$, let $\cD^{\bfi,h}_{\ell\times\ell'}$ be the pullback partition of $\Lambda_{n_0}\cap V_\bfi$ whose atoms are the nonempty sets $\sigma_\bfi^{-1}(R)\cap\Lambda_{n_0}$, $R\in\cD^{h}_{\ell\times\ell'}$, where $\sigma_\bfi\colon V_\bfi\to\R^2$ is the straightening map of \eqref{eqn:sigma i}; for $x\in\Lambda_{n_0}\cap V_\bfi$, we write $\cD^{\bfi,h}_{\ell\times\ell'}(x)$ for the atom containing $x$. When $\ell=\ell'$ we abbreviate $\cD^{\bfi,h}_{\ell}:=\cD^{\bfi,h}_{\ell\times\ell}$.
\end{defi}
It follows from \cref{prop:straightening map} that every atom $Q\in\cD^{\bfi,h}_{\ell}$ satisfies
\begin{equation}
\label{eqn:atom diam}
\diam (Q)\leq C_3\sqrt{2}\,2^{-\ell},
\end{equation}
where $C_3>1$ is the constant of \ref{straight:global Lip}.

	\begin{lem}\label{lem:2ll-lower}
		For any $C,C'\in\mathbb N$, $h\in[0,1]$, and $\mu_0\times\nu_-$-a.e. $(x,\bfi)\in\Lambda_{n_0}\times\Sigma^-$, we have
		\[ \liminf_{\ell\rightarrow\infty}\frac{\log \mu_0(\cD^{\bfi,h}_{C\ell\times C'\ell}(x)) }{-\ell\log2}\geq C'\gamma_1+C\gamma_2. \]
	\end{lem}
	\begin{proof}
        Since the fiber measures and the transverse measures are exact dimensional (\Cref{lem:fiber measure} and \Cref{lem:exact dim tran}), Egorov's theorem implies that for any $\epsilon>0$, there exist $\ell_0\in \N$ and 
        a subset $J_1\subset\Lambda_{n_0}\times\Sigma^-$ with $ \mu_0\times\nu_-(J_1)>1-\epsilon$ such that for all $(x,\bfi)\in J_1$ and $r<2^{-\ell_0}$, we have 
		\begin{align}
			\label{equ:fiber-upp}\mu_{x,\calA^{su}_\bfi}(B_{\bfi,s}^{su}(x,r))&\leq r^{\gamma_1-\epsilon}, \\
			\label{equ:proj-upp}\mu_0(B^{T_0}_\bfi(x,r))&\leq r^{\gamma_2-\epsilon},
		\end{align}
        where $T_0$ is the $\theta$-uniform $su$-transversal given in \eqref{eqn:tran T0}.
 		
		Applying the Lebesgue density theorem for partitions (\cite[2.9.8, 2.8.19]{federer_geometric_1969})
        and Egorov's theorem to every partition $\cD^{\bfj,h}_{C\ell\times C'\ell}$ for $\bfj\in \Sigma^-$, we obtain a subset $J_2\subset J_1$ with ${\mu_0}\times\nu_-(J_2)>1-2\epsilon$ such that for all $(x,\bfi)\in J_2$ and $\ell>\ell_1(\bfi)$, we have
		\[ \mu_0(J_1\cap \cD^{\bfi,h}_{C\ell\times C'\ell}(x))\geq \frac{1}{2}\mu_0(\cD^{\bfi,h}_{C\ell\times C'\ell}(x)). \]
        Consequently, for any $(x,\bfi)\in J_2$, disintegrating the measure over the transversal yields
		\[   \mu_0(\cD^{\bfi,h}_{C\ell\times C'\ell}(x))\leq 2\mu_0(J_1\cap \cD^{\bfi,h}_{C\ell\times C'\ell}(x))\leq 2\int_{B^{T_0}_{\bfi}(x,C''2^{-C\ell})} \mu_{y,\calA^{su}_\bfi}(\cD^{\bfi,h}_{C\ell\times C'\ell}(x)\cap J_1) d{\mu_0}(y),\]
        where $C''>0$ is a constant arising from the bi-Lipschitz property of $\sigma_{\bfi}$.
        To bound the integrand, suppose
$\calA^{su}_{\bfi}(y)\cap J_1\cap \cD^{\bfi,h}_{C\ell\times C'\ell}(x)\neq \emptyset$. Choosing a point
$z$ in the intersection and 
applying \eqref{equ:fiber-upp}, we obtain
		\[\mu_{y,\calA^{su}_\bfi}(\cD^{\bfi,h}_{C\ell\times C'\ell}(x)\cap J_1)\leq \mu_{z,\calA^{su}_\bfi}(B^{su}_{\bfi,s}(z,2^{-C'\ell}))\leq 2^{-C'\ell(\gamma_1-\epsilon)}.  \]
        Substituting this bound into the integral and applying \eqref{equ:proj-upp} to the transverse measure yields
		\[   \mu_0(\cD^{\bfi,h}_{C\ell\times C'\ell}(x))\leq 2^{-C'\ell(\gamma_1-\epsilon)+1} \mu_0(B_\bfi^{T_0}(x,C''2^{-C\ell}) )\leq (C'')^{\gamma_2-\epsilon}\cdot 2^{-C'\ell\gamma_1-C\ell\gamma_2+(C+C')\ell\epsilon+1}.\]
        Letting $\epsilon\to0$ completes the proof.
	\end{proof}
	
	The proof of \cref{thm:main-two} also requires the following two auxiliary lemmas. The first lemma provides a doubling scale, which is essential for transferring density from balls to boxes in the later proof.
	
\begin{lem}[Existence of a doubling scale]\label{lem:doubling-scale}
Suppose $x\in\Lambda_{n_0}$ satisfies the following property: there exist $\ell_0\in\mathbb{N}$ and $\epsilon\in (0,\frac{1}{3})$ such that for any $r\leq 2^{-\ell_0}$, we have
\begin{equation}
\label{eqn:measure r 2r}
r^{\gamma_1+\gamma_2+\epsilon} \leq \mu_0(B(x,C_5^{-10}r)) \leq \mu_0(B(x,C_5^{10}r)) \leq r^{\gamma_1+\gamma_2-\epsilon},
\end{equation}
where $C_5>1$ is the constant given in \ref{S:inverse reg in tang}. Then, for any $\ell > \ell_0+10\log_2C_5$, there exists a scale $r=2^{-k}$ with $\ell < k \leq 2\ell$ such that
$$ \frac{\mu_0(B(x,2r))}{\mu_0( B(x,r))} < 8. $$
\end{lem}
\begin{proof}
Let $\gamma := \gamma_1+\gamma_2$. Fix $\ell > \ell_0+10\log_2C_5$. Suppose for contradiction that for all $k\in\mathbb{N}$ with $\ell < k \leq 2\ell$, we have
$$ \frac{\mu_0(B(x,2^{-k+1}))}{\mu_0(B(x,2^{-k}))} \geq 8. $$
Iterating this inequality over the $\ell$ steps from $k=2\ell$ down to $k=\ell+1$, we obtain
$$ \mu_0(B(x,2^{-\ell})) \geq 8^{\ell} \cdot \mu_0(B(x,2^{-2\ell})). $$
To bound these terms using \eqref{eqn:measure r 2r}, we set $r_1 = C_5^{-10}2^{-\ell}$ and $r_2 = C_5^{10}2^{-2\ell}$. Since $C_5>1$ and $\ell > \ell_0+10\log_2C_5$, it follows that $r_1 \leq 2^{-\ell_0}$ and $r_2\leq 2^{-\ell_0}$ . Applying the upper bound of \eqref{eqn:measure r 2r} with $r_1$ yields $\mu_0(B(x,2^{-\ell})) \leq r_1^{\gamma-\epsilon}$. Applying the lower bound with $r_2$ yields $\mu_0(B(x,2^{-2\ell})) \geq r_2^{\gamma+\epsilon}$. Substituting these into our iterated estimate gives$$ (C_5^{-10}2^{-\ell})^{\gamma-\epsilon} \geq 8^{\ell} \cdot (C_5^{10}2^{-2\ell})^{\gamma+\epsilon}. $$
Since $C_5 > 1$, this forces $2^{\ell(\gamma+3\epsilon)} > 2^{3\ell}$, which implies $3 \leq \gamma+3\epsilon$, providing the desired contradiction.\end{proof}
    
	The second lemma utilizes the existence of a local doubling scale to transfer density bounds from balls to their constituent subsets.
    \begin{lem}[Density transfer]\label{lem:density-of-box}
Let $x\in\Lambda_{n_0}$ and $r>0$ be such that 
\begin{equation}
\label{eqn:2rr8}
\frac{\mu_0(B(x,2r))}{\mu_0(B(x,r))}<8,
\end{equation}
and let $\cF$ be a pairwise disjoint countable collection of Borel subsets satisfying 
\begin{equation}
\label{eqn:xrF}
B(x,r)\subseteq \bigcup_{F\in\cF}F\subseteq B(x,2r).
\end{equation}
Then for any Borel set $G$ and $c>0$ satisfying
\begin{equation}
\label{eqn:xrG}
\frac{\mu_0(B(x,2r)\cap G)}{\mu_0(B(x,2r))}>1-c, 
\end{equation}
there exists a set $F$ in $\cF$ such that 
\[\frac{\mu_0(F\cap G)}{\mu_0(F)}>1-8c. \]
\end{lem}

\begin{proof}
Suppose for contradiction that for all $F\in\cF$, we have
\[ \frac{\mu_0(F\cap G)}{\mu_0(F)}\leq 1-8c. \]
This implies $\mu_0(F \setminus G)\geq 8c\cdot\mu_0(F)$ for all $F\in\cF$.
Summing over all sets in $\cF$, and using the fact that they are pairwise disjoint, we obtain
\[ \mu_0\left(\bigcup_{F\in\cF}(F \setminus G)\right) = \sum_{F\in\cF}\mu_0(F \setminus G) \geq 8c\sum_{F\in\cF}\mu_0(F). \]
It follows from the inclusion relation \eqref{eqn:xrF} that $\bigcup_{F\in\cF}F \supseteq B(x,r)$ and $\bigcup_{F\in\cF}(F \setminus G) \subseteq B(x,2r) \setminus G$, which yields
\begin{equation}
\mu_0(B(x,2r) \setminus G)\geq 8c\cdot \mu_0(B(x,r)).
\end{equation}
From the density assumption on $G$ \eqref{eqn:xrG}, we know that $\mu_0(B(x,2r) \setminus G) < c \mu_0(B(x,2r))$. Combining these gives
\begin{equation}
c \mu_0(B(x,2r)) > 8c \mu_0(B(x,r)).
\end{equation}
Since $c>0$, we obtain $\mu_0(B(x,2r)) > 8 \mu_0(B(x,r))$, which contradicts the doubling assumption \eqref{eqn:2rr8}.
\end{proof}
	
        Now we start to prove \cref{thm:main-two}.
			 We prove by contradiction. Suppose $\gamma_1>0$, but $\gamma_2<1$. \Cref{lem:gamma2-positive} established that  $\gamma_2=0$ cannot occur. Hence, in the sequel, we assume
\begin{equation}
\label{eqn:assumption gamma 2}
    0<\gamma_2<1.
\end{equation}

 \subsection{Selection of the points}\label{sec: step1}
In this subsection, we construct a good set of large measure that satisfies several properties and will be used in the subsequent steps.
		We define the space $X$ and the measure $m$:
		\[
		X:=\Lambda_{n_0}\times\Sigma^-\times[0,1], \quad m:=\mu_0\times \nu_-\times \mathrm{Leb}.
		\]
		We introduce several definitions and properties.
		Throughout, sets denoted $\widetilde{G}$ are subsets of $\Lambda_{n_0}\times\Sigma^-$, while sets denoted $G$ are subsets of $X$.
\begin{defi}
Given a set $S\subset X$, we define the fiber sets
\begin{align}
&S(x,\cdot,h):=\{\bfi\in\Sigma^-:\, (x,\bfi,h)\in S \} \quad \text{for}\,\,(x,h)\in \Lambda_{n_0}\times [0,1],\\
&S(\cdot,\bfi,h):=\{x\in\Lambda_{n_0}:\, (x,\bfi,h)\in S \}\quad \text{for}\,\,(\bfi,h)\in \Sigma^-\times [0,1].
\end{align}
\end{defi}

Throughout, 
\[
1<C_0<C_1<\cdots<C_5.
\]
These are the constants of \Cref{rem:constants}. Let $\epsilon>0$ be the constant given in \Cref{prop: small angle}. 
Fix $C\in 3\N$ such that
\begin{equation}\label{eqn:choice of C}
C\alpha>10 \quad\text{and}\quad C>3(1+\log_2C_3),
\end{equation}
where $\alpha$ is the exponent in the standing assumption $f\in C^{1+\alpha}$. The constant $C$ is taken to be divisible by $3$ so that $Ck/3$ is an integer in \cref{prop:output-step1}, and $C\alpha>10$ so that the angle is stable on $B(x,2^{-Ck/3})$ there. Choose $\epsilon_1>0$ such that
\begin{equation}\label{eq:choice-eps1}
{
\epsilon_*:=60(C+1)\epsilon_1<\frac{1-\gamma_2}{2}
}
\quad\text{and}\quad
\epsilon_1<\frac{\epsilon}{6},
\end{equation}
{where $\epsilon_*/6=10(C+1)\epsilon_1$ is the dyadic-spreading error produced in \Cref{claim:spreading-normalized};} such $\epsilon_1$ exists since $\gamma_2\in(0,1)$ by \eqref{eqn:assumption gamma 2}.

\medskip\noindent\textbf{Dyadic spreading (Property~1).}
		For the given $\epsilon_1>0$, by \cref{prop:ball}, there exist $\ell_1=\ell_1(\epsilon_1)\in \N$ and $\widetilde{G}_1\subseteq \Lambda_{n_0}\times \Sigma^-$ with $\mu_0\times \nu_-(\widetilde{G}_1)>1-\epsilon_1$ such that for any $(x,\bfi)\in \widetilde{G}_1$, the point $\sigma_\bfi(x)$ is $(\ell_1,\epsilon_1)$-ball-spreading on $\sigma_\bfi(\calA^{su}_\bfi(x))$ with respect to $(\sigma_\bfi)_*\mu_{x,\calA^{su}_\bfi}$. Enlarging $\ell_1$ if necessary, we may assume $2^{-\ell_1+1}<\epsilon_1$.
		By \cref{lem:ball-dyadic-spreading}, for any $(x,\bfi)\in \widetilde{G}_1$ the set of $h\in[0,1]$ satisfying the following dyadic-spreading property:
		\begin{enumerate}[(P1)]
			\item The point $\sigma_{\bfi}(x)$ is $(n,2\ell_1,10\epsilon_1,h)$-dyadic-spreading on $\sigma_{\bfi}(\calA^{su}_\bfi(x))$ with respect to $(\sigma_{\bfi})_*\mu_{x,\calA^{su}_\bfi}$ for every $n\geq n_1(x,\bfi,\ell_1,\epsilon_1)$, where $n_1(x,\bfi,\ell_1,\epsilon_1)$ denotes the least such integer
            \label{P1:spreading property}
		\end{enumerate}
		has Lebesgue measure larger than $1-\epsilon_1$. Setting
		\[
		G_1:=\{(x,\bfi,h)\in \widetilde{G}_1\times[0,1]:\ (x,\bfi,h) \text{ satisfies (P1)}\}\subset X,
		\]
		Fubini's theorem gives $m(G_1)>1-2\epsilon_1$.

We use the following remark repeatedly. If $S\subseteq X$ is measurable and $\ell\colon S\to\N\cup\{\infty\}$ is measurable and finite $m$-a.e. on $S$, then the sets $S\cap\{\ell\leq L\}$ increase to $S$ up to an $m$-null set. So, by continuity from below, for every $\rho>0$, there is $L^*\in\N$ with
		\begin{equation}\label{eq:threshold}
			m\big(S\cap\{\ell\leq L^*\}\big)>m(S)-\rho .
		\end{equation}
        Since $n_1$ is measurable in $(x,\bfi)$, \eqref{eq:threshold} applied with $S=G_1$, $\ell=n_1$, and $\epsilon=\epsilon_1$ yields $n^*_1=n^*_1(\ell_1,\epsilon_1)$; replacing $G_1$ with $G_1\cap\{n_1\leq n^*_1\}$, property (P1) holds at every point of $G_1$ with the common threshold $n^*_1$ and $m(G_1)>1-3\epsilon_1$.

		Let $\delta=\delta(\epsilon_*,2\ell_1)\in (0,1)$ be the constant given in \Cref{cor:final-growth-addcom}. Since $\delta$ depends on $\ell_1$, the parameter $\epsilon_2$ can only be fixed at this point, after (P1). Choose $\epsilon_2>0$ such that
		\begin{equation}\label{eq:choice-eps2}
		\gamma_2+(2C+3)\epsilon_2<1-\epsilon_*,\quad \frac{\gamma_2-(2C+4)\epsilon_2}{\gamma_2+(2C+3)\epsilon_2}>1-\delta,\quad  \text{and}\quad {\epsilon_2<\min\{\frac{\delta \gamma_2}{100 C},\,\frac{\epsilon}{6},\,\frac{1}{48},\,\frac{\epsilon_*}{96}\}}.
		\end{equation}
		the first condition is compatible with the others because \eqref{eq:choice-eps1} gives $\gamma_2+\epsilon_*<1$. Note that \eqref{eq:choice-eps1} and \eqref{eq:choice-eps2} together give $3\epsilon_1+3\epsilon_2<\epsilon$, which is the hypothesis of \Cref{prop: small angle}, and that the final condition in \eqref{eq:choice-eps2} forces $\epsilon_2\leq1/48$, so that $16\epsilon_2<\tfrac13$, which will be used in  \eqref{equ:A-large-mass}. {Moreover, $\epsilon_2<\epsilon_*/96$ gives $16\epsilon_2<\epsilon_*/6$, which will be used in \Cref{claim:spreading-normalized}.}

\medskip\noindent\textbf{Dimension estimates (Properties~2--4).}
		Throughout the argument, we only use the dyadic scales $r=2^{-n}$, $n\in\N$, and we fix a single $\theta$-uniform $su$-transversal $T$: any two such transversals are bi-Lipschitz equivalent with a universal constant, which affects $\log\mu^{T}_\bfi\big(B^{T}_\bfi(x,r)\big)/\log r$ only to a lower order. Thus, by the exact rate of \Cref{lem:exact dim tran}, the estimate (P3) for $T$ yields (P3) for every $\theta$-uniform $su$-transversal at all sufficiently small $r$, the loss being absorbed into $\ell_2$.
		By \Cref{lem:exact dim tran} and \Cref{lem:2ll-lower}, for $m$-a.e. $(x,\bfi,h)\in X$ there is a least integer $\ell_2(x,\bfi,h)\in\N$ such that, for every $n\geq \ell_2(x,\bfi,h)$ and $r=2^{-n}$:
       
        \begin{enumerate}[label=(P\arabic*),start=2]
			\item Measure dimension estimate:
			\[r^{\gamma_1+\gamma_2+\epsilon_2}\leq\mu_0(B(x,C_5^{-10}r))\leq\mu_0(B(x,C_5^{10}r))\leq r^{\gamma_1+\gamma_2-\epsilon_2},\]
			where the constant $C_5>1$ is given in \ref{S:inverse reg in tang};
            \item Transverse dimension estimate:
			\[r^{\gamma_2+\epsilon_2}\leq \mu_{\bfi}^T\big(B^T_\bfi(x,C^{-10}_7r)\big)\leq \mu_{\bfi}^T\big(B^T_\bfi(x,C^{10}_7r)\big)\leq r^{\gamma_2-\epsilon_2},\]
			where $B^T_\bfi(x,\rho)$ denotes the ball in $T$ of radius $\rho$ centred at $\pi^{su}_{\bfi,T}(x)$, and the constant $C_7>1$, depending only on the dynamical system, is specified in \eqref{eq:C7} below; without loss of generality, we assume $C_7\geq \max\{C_5, 4\}$, which is used in \eqref{equ:npij};
            \label{P3:transverse dimension}
			\item Rectangle upper bound:
			\[\mu_0\big(\cD^{\bfi,h}_{(C+2)n\times (C+1)n}(x)\big)\leq r^{(C+1)\gamma_1+(C+2)\gamma_2-\epsilon_2}.\]\label{P4:retangle bound}
		\end{enumerate}
        Since (P2)--(P4) are imposed only at the countably many scales $r=2^{-n}$, each condition defining a Borel subset of $X$, the function $\ell_2$ is measurable; the same applies to $\ell_3$ and $\ell_4$ below. By \eqref{eq:threshold} with $S=X$, $\ell=\ell_2$, $\epsilon=\epsilon_2$, fix $\ell^*_2=\ell^*_2(\epsilon_1,\epsilon_2)\geq\ell_1$ and set
		\[
		G_2:=G_1\cap\{(x,\bfi,h)\in X:\ \ell_2(x,\bfi,h)\leq \ell^*_2\},
		\]
		so that $m(G_2)>1-3\epsilon_1-\epsilon_2$ and every point of $G_2$ enjoys (P1)--(P4) at the common thresholds $n^*_1$ and $\ell^*_2$.

\medskip\noindent\textbf{Ball density (Properties~5--6).}
		Since $G_2$ is Borel, for $\nu_-\times\mathrm{Leb}$-a.e. $(\bfi,h)$ the fiber $G_2(\cdot,\bfi,h)$ is $\mu_0$-measurable; applying the density theorem for Radon measures \cite[Cor.~2.14]{Mattila_1995} to this fiber and integrating in $(\bfi,h)$, we obtain for $m$-a.e. $(x,\bfi,h)\in G_2$ a least $\ell_3(x,\bfi,h)\in\N$ such that:
		\begin{enumerate}[(P5)]
			\item Ball density of $G_2$:
			\[\mu_0(B(x,r)\cap G_2(\cdot,\bfi,h))\geq (1-\epsilon_2)\mu_0(B(x,r))\qquad \text{for all } r=2^{-n},\ n\geq\ell_3(x,\bfi,h).\]
		\end{enumerate}
        By \eqref{eq:threshold} with $S=G_2$, $\ell=\ell_3$, $\epsilon=\epsilon_2$, fix $\ell^*_3\geq\ell^*_2$ and set $G_3:=G_2\cap\{\ell_3\leq\ell^*_3\}$, so that $m(G_3)>1-3\epsilon_1-2\epsilon_2$.
        
		Repeating the argument with $G_3$ in place of $G_2$, we obtain for $m$-a.e. $(x,\bfi,h)\in G_3$ a least $\ell_4(x,\bfi,h)\in\N$ such that:
		\begin{enumerate}[(P6)]
			\item Ball density of $G_3$:
			\[\mu_0(B(x,r)\cap G_3(\cdot,\bfi,h))\geq (1-\epsilon_2)\mu_0(B(x,r))\qquad \text{for all } r=2^{-n},\ n\geq\ell_4(x,\bfi,h).\]
		\end{enumerate}
		together with $\ell^*_4=\ell^*_4(\epsilon_1,\epsilon_2)\geq\ell^*_3$ such that
		\[
		G_4:=G_3\cap\{(x,\bfi,h)\in X:\ \ell_4(x,\bfi,h)\leq\ell^*_4\}
		\]
		satisfies $m(G_4)>1-3\epsilon_1-3\epsilon_2$.
		The two applications play different roles: (P6) is applied at the points of $G_4$ at which the argument is centered, whereas (P5) is available at \emph{every} point of $G_3$, in particular at the nearby points supplied by (P6). This is how they are used in \cref{prop:output-step1} and \cref{claim:lowerRa}.

		For later reference, the constants of this subsection are fixed in the order
		\[
		\epsilon\ \rightsquigarrow\ C\ \rightsquigarrow\ \epsilon_1\ \rightsquigarrow\ \ell_1,\,n^*_1\ \rightsquigarrow\ \delta\ \rightsquigarrow\ \epsilon_2\ \rightsquigarrow\ \ell^*_2,\,\ell^*_3,\,\ell^*_4 ,
		\]
		each depending only on those preceding it; the constant $K$ is fixed in \eqref{eqn:choice of K}.

\medskip\noindent\textbf{Choice of the point.}
\begin{defi}\label{def:angle}
For $z\in\Lambda_{n_0}$ and $\bfi,\bfj\in\Sigma^-$, we set
\[
\angle_z(\bfi,\bfj):=\angle\Big(\big(D_z\exp_{x_0}^{-1}\big)\big(E^{su}_{\bfi}(z)\big),\ \big(D_z\exp_{x_0}^{-1}\big)\big(E^{su}_{\bfj}(z)\big)\Big).
\]
\end{defi}
This is the angle used in the planar geometry arguments below. Since the Riemannian metric is smooth and $\Lambda_{n_0}$ lies in a small neighborhood of $x_0$, the maps $D_z\exp^{-1}_{x_0}$ and their inverses are uniformly bounded on $\Lambda_{n_0}$, and $z\mapsto D_z\exp^{-1}_{x_0}$ is Lipschitz there, in the sense of the identification of nearby tangent spaces by parallel transport used in \Cref{prop Holder continuous}. Consequently, $\tan\angle_z(\bfi,\bfj)$ is comparable to $\tan\angle\big(E^{su}_{\bfi}(z),E^{su}_{\bfj}(z)\big)$ up to a universal constant, and the H\"older estimate of \Cref{prop Holder continuous} transfers to $\angle_z$. Enlarging $C_0$ if necessary, we assume that the constant $C_0>1$ of \Cref{prop Holder continuous} bounds this comparison constant and this Lipschitz constant as well.

        Since $m(G_4)>1-3\epsilon_1-3\epsilon_2$, Fubini's theorem provides a pair $(x,h)\in\Lambda_{n_0}\times[0,1]$ with
		\[
		\nu_-\big(G_4(x,\cdot,h)\big)>1-3\epsilon_1-3\epsilon_2 ,
		\]
		which we fix for the rest of the proof.
		By \eqref{eq:choice-eps1} and \eqref{eq:choice-eps2}, we have $3\epsilon_1+3\epsilon_2<\epsilon$, so \Cref{prop: small angle} applies to $G_4(x,\cdot,h)$ for \emph{every} value of its parameter.

Let $k_0$ be the threshold in \Cref{rem:geom-input-nonlinear-displacement}, and set
\(
C'_0:=\lceil C_5^3\rceil.
\)
We now fix $K$ by
\begin{align}\label{eqn:choice of K}
K:=\max\Bigg\{\lceil\log_2 C_0\rceil,\ &
\bigg\lceil\frac{\log_2(10C_5^2)}{\epsilon_2}\bigg\rceil,
\Bigg\lceil\frac{4\log_2(200C_8)}{\delta\gamma_2}\Bigg\rceil,
 k_0,
 N(\epsilon_*,2\ell_1,C'_0,C),\\
&\Bigg\lceil\frac{2}{C\epsilon_2}
\Big(\big(10+\lceil\log_2 C_5\rceil\big)
\big(\gamma_1+\gamma_2+\epsilon_2\big)+2\Big)\Bigg\rceil
\Bigg\}+4 .
\end{align}
Applying \Cref{prop: small angle} to $G_4(x,\cdot,h)$ with parameter $\eta=2^{-(\ell_4^*+n_1^*+K)}/4C_0$, and using the comparison above, we obtain $\bfi,\bfj\in G_4(x,\cdot,h)$ such that
		\[
		\tau:=\tan\angle_x(\bfi,\bfj)
		\]
		satisfies $0<\tau\leq 2C_0\eta$. Let $k\in\N$ be the unique integer with $2^{-k-1}<\tau\leq2^{-k}$. Then
		\begin{equation}\label{eqn:k large}
		k\geq \ell^*_4+n^*_1+K.
		\end{equation}

With $k':=(C+1)k$, we obtain
\begin{align}
&4C_0\,2^{-\alpha Ck/3}<2^{-k-2},\label{eqn:K-angle}\\
&\tfrac{k'\epsilon_2}{2}\geq\big(10+\lceil\log_2 C_5\rceil\big)\big(\gamma_1+\gamma_2+\epsilon_2\big)+2,\label{eqn:K-count}\\
&\tfrac{k\delta\gamma_2}{4}\geq\log_2(200C_8).
\label{eqn:K-C8}
\end{align}
Moreover,
\[
k\geq k_0
\qquad\text{and}\qquad
k\geq N(\epsilon_*,2\ell_1,C'_0,C).
\]
These bounds follow from \eqref{eqn:k large} and \eqref{eqn:choice of K}; they are used, respectively, in \eqref{eq:angle-stable}, \Cref{claim:lowerRa}, \eqref{eqn:k final}, the application of \Cref{rem:geom-input-nonlinear-displacement}, and the application of \Cref{cor:final-growth-addcom}.

		Since $\bfi,\bfj\in G_4(x,\cdot,h)$, we have $x\in G_4(\cdot,\bfi,h)\cap G_4(\cdot,\bfj,h)$.
		By \Cref{prop Holder continuous} and the choice \eqref{eqn:choice of C} of $C$,
		\begin{equation}\label{eq:angle-stable}
		\tau_y:=\tan\angle_y(\bfi,\bfj)\in[2^{-k-2},2^{-k+2}]\qquad\text{for every }y\in B\big(x,2^{-Ck/3}\big);
		\end{equation}
		indeed, $C\alpha>10$ gives $4C_0\,2^{-\alpha Ck/3}<4C_0\,2^{-3k}<2^{-k-2}$, where the last inequality follows from \eqref{eqn:k large}.
		Throughout we write
		\[
		G:=G_3(\cdot,\bfi,h)\cap G_3(\cdot,\bfj,h)\subseteq\Lambda_{n_0}.
		\]
		We now establish the main output of this step.

        \begin{prop}\label{prop:output-step1}
There exists $y\in G\cap B\big(x,2^{-Ck/3}\big)$ such that
            \[
            \tau_y\in[2^{-k-2},2^{-k+2}]
            \qquad\text{and}\qquad
            \frac{\mu_{y,\calA^{su}_\bfi}\big(\cD^{\bfi,h}_{Ck}(y)\cap G\big)}{\mu_{y,\calA^{su}_\bfi}\big(\cD^{\bfi,h}_{Ck}(y)\big)}\geq1-16\epsilon_2 .
            \]
        \end{prop}
        \begin{proof}
            By \eqref{eq:angle-stable}, the angle condition holds at every point of $B(x,2^{-Ck/3})$, so it suffices to produce a point of $G$ in that ball satisfying the density estimate.

           Since $x\in G_4(\cdot,\bfi,h)\subseteq G_2(\cdot,\bfi,h)$, it satisfies \eqref{eqn:measure r 2r}, so \cref{lem:doubling-scale}, applied with parameter $k_{\mathrm{doub}}=Ck/3$, produces a scale $r=2^{-\ell}$ with $Ck/3<\ell\leq2Ck/3$ such that
            \begin{equation}\label{eq:doubling}
                \frac{\mu_0\big(B(x,2r)\big)}{\mu_0\big(B(x,r)\big)}<8 .
            \end{equation}
            Since $\ell\geq Ck/3+1$, we have $2r\leq2^{-Ck/3}$, hence
            \begin{equation}\label{eq:balls-nested}
                B(x,2r)\subseteq B\big(x,2^{-Ck/3}\big),
            \end{equation}
            so every point of $B(x,2r)$ satisfies the angle condition.

            As $C\geq3$, we have $\ell-1\geq Ck/3\geq k\geq\ell^*_4$, so (P6) applies at $x$ at the scale $2r=2^{-(\ell-1)}$, once for $\bfi$ and once for $\bfj$. Each application gives density at least $1-\epsilon_2$, so
            \begin{equation}\label{eq:G-density}
                \mu_0\big(B(x,2r)\setminus G\big)\leq2\epsilon_2\,\mu_0\big(B(x,2r)\big).
            \end{equation}

           Let $\cF\subseteq\cD^{\bfi,h}_{Ck}$ be the collection of atoms of positive $\mu_0$-measure meeting $B(x,r)$. By \eqref{eqn:atom diam} and $\ell\leq2Ck/3$, every $Q\in\cF$ satisfies
            \[
            \diam(Q)\leq\sqrt2\,C_3\,2^{-Ck}\leq\sqrt2\,C_3\,2^{-Ck/3}\,r\leq r
            \]
            where we used $C>3(1+\log_2C_3)$ \eqref{eqn:choice of C}, and hence $Q\subseteq B(x,2r)$. The atoms of $\cF$ are therefore pairwise disjoint, cover $B(x,r)$, and are contained in $B(x,2r)$, so \cref{lem:density-of-box} applies with the doubling estimate \eqref{eq:doubling} and the density estimate \eqref{eq:G-density}: there is $Q\in\cF$ with
            \begin{equation}\label{eq:box-density}
                \mu_0(Q\setminus G)\leq16\epsilon_2\,\mu_0(Q).
            \end{equation}
            Write $\mu_{y,\bfi}:=\mu_{y,\calA^{su}_\bfi}$, so that $\mu_0=\int\mu_{y,\bfi}\,d\mu_0(y)$. As $\mu_0(Q)>0$, averaging \eqref{eq:box-density} yields a point $y$ with $\mu_{y,\bfi}(Q)>0$ and
            \begin{equation}\label{eq:fiber-density}
                \mu_{y,\bfi}(Q\setminus G)\leq16\epsilon_2\,\mu_{y,\bfi}(Q).
            \end{equation}
            Since $16\epsilon_2<1$, this gives $\mu_{y,\bfi}(Q\cap G)>0$; as $\mu_{y,\bfi}$ is carried by $\calA^{su}_\bfi(y)$, we may pick $y'\in Q\cap G\cap\calA^{su}_\bfi(y)$. The conditional measures being constant on atoms, $\mu_{y',\bfi}=\mu_{y,\bfi}$, so \eqref{eq:fiber-density} holds at $y'$. Finally $Q=\cD^{\bfi,h}_{Ck}(y')$ and $y'\in Q\subseteq B(x,2r)\subseteq B\big(x,2^{-Ck/3}\big)$, so $y'$ satisfies both conclusions.
        \end{proof}

	\subsection{Local projection for sets}\label{sec:local-projection-for-sets}
We now use the output of \Cref{sec: step1} to compare the projections induced by
two nearby but distinct strong unstable directions. The local strong unstable
manifolds are used only as plaques through points of $\Lambda_{n_0}$; we do not
use an ambient foliation in a neighborhood of $\Lambda_{n_0}$. All H\"older
estimates below are used either for points in $\Lambda_{n_0}$ or for points lying
on each local plaque.

Throughout this subsection, $\exp^{-1}_{x_0}$ is taken on the fixed normal-coordinate
neighborhood $U$, while the straightening map has its full domain
\[
\exp^{-1}_{x_0}:U\to T_{x_0}M\simeq\R^2,
\qquad
\sigma_\bfi:V_\bfi\to\R^2.
\]
We use their restrictions to $\Lambda_{n_0}$ when discussing dynamical sets,
but we retain the full domain $V_\bfi$ when working with ambient plaque
segments. Put
\[
\Omega:=\exp^{-1}_{x_0}(\Lambda_{n_0})\subseteq\R^2,
\qquad
\widetilde V_\bfi:=\exp^{-1}_{x_0}(V_\bfi).
\]
The planar geometry arguments on the repeller are carried out on $\Omega$,
whereas the composition
\[
S:=\sigma_\bfi\circ\exp_{x_0}:\widetilde V_\bfi\to\R^2
\]
is also available on the ambient plaque segments used below. Sets carrying
dynamical meaning --- $G$, $G'$, $R_0$, $R_1$, $R$, and later $L$,
$L_{\mathrm{good}}$, $\cR_t$, $\cB_t$ --- are subsets of $\Lambda_{n_0}$;
we write $\exp^{-1}_{x_0}(A)$ or $\sigma_\bfi(A)$ when a statement about such
a set $A$ is made in one of the two coordinate pictures, and abbreviate
$\widetilde A:=\exp^{-1}_{x_0}(A)$. Objects produced by the planar arguments
themselves --- the interval $I_k$, the finite set $F$, the measure $\zeta_k$,
and the projections $B_t$ --- live in the plane or on the line
$\exp^{-1}_{x_0}(T)$.

Let $y$ be the point obtained in \Cref{sec: step1}. As there, we write
\[
G:=G_3(\cdot,\bfi,h)\cap G_3(\cdot,\bfj,h),
\qquad
G':=G_2(\cdot,\bfi,h)\cap G_2(\cdot,\bfj,h),
\]
so that $y\in G\subseteq G'\subseteq\Lambda_{n_0}$.
We choose a local transversal $T$ through $y$ such that $\exp^{-1}_{x_0}(T)$ is the line through $\exp^{-1}_{x_0}(y)$ orthogonal to $\big(D_y\exp^{-1}_{x_0}\big)\big(E^{su}_{\bfi}(y)\big)$. As explained at the beginning of \Cref{subsec:statement and preparation}, $T$ is a $\theta$-uniform $su$-transversal, and every relevant
$\bfi$- or $\bfj$-plaque through a point of $\Lambda_{n_0}$ intersects $T$
in exactly one point.

Set
\[
R_0:=\cD^{\bfi,h}_{(C+1)k\times Ck}(y),
\]
and let $R_1$ be the union of all atoms of $\cD^{\bfi,h}_{(C+1)k}$ whose closure meets $R_0$. In the straightened picture, $\sigma_\bfi(R_0)$ is a rectangle of size $2^{-(C+1)k}\times2^{-Ck}$; $\sigma_\bfi(R_1)$ is contained in a rectangle of size
$3\cdot 2^{-(C+1)k}
\ \times\
\big(2^{-Ck}+2\cdot 2^{-(C+1)k}\big)$.

Finally set
\[
R:=R_1\cap G' .
\]
In particular $R\subseteq\Lambda_{n_0}$, and every point of $R$ enjoys (P1)--(P4).

We record the local regularity estimates used below. Let
\begin{align}
L'
    :=
    \sigma_{\bfi}^{-1}\bigl(\cD^{h}_{Ck}(\sigma_{\bfi}(y))
    \cap
    \sigma_{\bfi}(\mathcal{W}^{su}_{\mathrm{loc}}(y,\bfi))\bigr) \qquad \text{and} \qquad
    L
    :=
    \cD^{\bfi,h}_{Ck}(y)
    \cap
    \calA^{su}_\bfi(y).
\end{align}

By construction, $L'\subset V_\bfi$ and $L\subset\Lambda_{n_0}\cap V_\bfi$.
Hence $\sigma_\bfi^{-1}$ and
$\exp^{-1}_{x_0}\circ\sigma_\bfi^{-1}$ are well-defined on all plaque
segments used below.

Let
\[
    c_0:=\left\|D_{\sigma_\bfi(y)}\left(\exp_{x_0}^{-1}\circ\sigma_\bfi^{-1}\right)\right\| ,
\]
the derivative being taken in the one-dimensional leafwise variable; it measures the distortion of leaf-length by the straightening map at $y$. By \Cref{prop:straightening map}\ref{S:inverse reg in tang}, $c_0\in[1,1.01]$.

By \Cref{prop:straightening map}, \Cref{prop Holder continuous}, and \ref{U:USR}, the plaques, the straightening map, and its leafwise derivative satisfy the uniform regularity assumptions \ref{R1:leaf graph}--\ref{R4:leafwise derivative} used in the appendix. Hence the three planar geometry inputs cited below apply to the present data with uniform constants.

For $\bfk\in\{\bfi,\bfj\}$, define
\[
\pi_\bfk:\Omega\longrightarrow\exp^{-1}_{x_0}(T),
\qquad
\pi_\bfk(\widetilde z)
:=
\exp^{-1}_{x_0}\Big(
W^{su}_{\mathrm{loc}}(z,\bfk)\cap T
\Big),
\qquad
z:=\exp_{x_0}(\widetilde z).
\]
The intersection in this definition uniquely exists. Thus $\pi_\bfk$ is well-defined on all of
$\Omega$.

By \Cref{prop:output-step1},
\[
    2^{-k-2}
    \leq
    \tau_{y}=\tan\angle_y(\bfi,\bfj)
    \leq
    2^{-k+2},
\]
so that, since $c_0\in[1,1.01]$,
\begin{equation}
\label{eqn:sa bound}
    s_{\mathrm{a}}:=c_0\,\tau_y\,2^{k}\in\left[\tfrac14,\ 4.04\right].
\end{equation}

Put
\[
\widetilde y:=\exp^{-1}_{x_0}(y),
\qquad
\widetilde R:=\exp^{-1}_{x_0}(R).
\]
Recall that
\[
S=\sigma_\bfi\circ\exp_{x_0}:
\widetilde V_\bfi\to\R^2
\]
was defined above on the full straightening domain.
Since $y\in R_0$, the definition of $R_1$ gives
\[
S(\widetilde R)=\sigma_\bfi(R)
\subset
\big[S_1(\widetilde y)-3\cdot2^{-(C+1)k},\,
S_1(\widetilde y)+3\cdot2^{-(C+1)k}\big]
\times
\big[S_2(\widetilde y)-3\cdot2^{-Ck},\,
S_2(\widetilde y)+3\cdot2^{-Ck}\big].
\]
Moreover, $C\alpha>10>6$ and
$\tau_y\in[2^{-k-2},2^{-k+2}]$. Thus
\Cref{lem:geom-input-proj-curved-rectangle}, applied with $A=3$, yields a
uniform constant $C_7>1$ such that
\begin{equation}\label{eq:C7}
    \pi_\bfj(\widetilde R)
    \subset
    B^T\left(
        \pi_\bfj(\widetilde y),
        C_7 2^{-(C+1)k}
    \right).
\end{equation}
By a slight abuse of notation, for $\bfk\in\{\bfi,\bfj\}$ we write
\[
    \pi_\bfk\mu_0:=(\pi_\bfk\circ\exp_{x_0}^{-1})_*\mu_0
    =\bigl(\exp_{x_0}^{-1}\bigr)_*\mu^T_\bfk ,
\]
that is, the transverse measure $\mu^T_\bfk$ defined in \eqref{eqn:muTbfi} read in the chart $\exp^{-1}_{x_0}$; since $\exp^{-1}_{x_0}$ is bi-Lipschitz on $\Lambda_{n_0}$, the estimates (P3) apply to it, the distortion being absorbed by the exponents $C_7^{\pm10}$ there.
Since $R\subseteq G'$, the estimate (P3) is available for $\bfk=\bfj$ at every point of $\pi_\bfj\bigl(\exp^{-1}_{x_0}(G')\bigr)$. Since $s_{\mathrm{a}}\geq\tfrac14\geq C_7^{-10}$, \ref{P3:transverse dimension} applies at both radii, and comparing the measures of the large and the small balls gives
\begin{equation}\label{equ:npij}
\begin{split}
\cN_{s_{\mathrm{a}}\cdot2^{-(C+2)k}}
\left(
\pi_\bfj\left(\exp_{x_0}^{-1}(R)\right)
\right)
&\leq
\frac{
\pi_\bfj\mu_0
\left(
B^T\left(
\pi_\bfj(\exp_{x_0}^{-1}(y)),
C_7 2^{-(C+1)k}
\right)
\right)
}{
\inf_{z\in\pi_\bfj(\exp^{-1}_{x_0}(G'))}
\pi_\bfj\mu_0
\left(
B^T\left(z,s_{\mathrm{a}} 2^{-(C+2)k}\right)
\right)
}
\\
&\leq
2^{(C+2)k(\gamma_2+\epsilon_2)-(C+1)k(\gamma_2-\epsilon_2)}
\\
&=
2^{k(\gamma_2+(2C+3)\epsilon_2)} .
\end{split}
\end{equation}

\subsection{Lower estimates on covering number}\label{sec: 6.4}
Now we establish the matching lower bound for the $\bfj$-projection.
\begin{lem}\label{lem:proj-tube}
With $\delta$ as in \Cref{cor:final-growth-addcom},
\[
\cN_{s_{\mathrm{a}}\cdot2^{-(C+2)k}}
\bigl(\pi_\bfj(\exp_{x_0}^{-1}(R))\bigr)
\geq
2^{k\gamma_2(1+\delta/2)}.\]
\end{lem}
Assuming \Cref{lem:proj-tube} and comparing it with the upper bound \eqref{equ:npij}, we get 
\begin{equation}
\label{eqn:contradiction}
    \gamma_2(1+\delta/2)\leq\gamma_2+(2C+3)\epsilon_2 ,
\end{equation}
which contradicts \eqref{eq:choice-eps2}. This rules out the standing assumption $\gamma_2<1$; hence $\gamma_2=1$, and the proof of \cref{thm:main-two} is complete. 

It therefore remains only to prove \Cref{lem:proj-tube}; the proof occupies
the rest of this subsection and \Cref{subsec:dimension growth}. In the rest of
this subsection, we use the fiber measure of $y$ to construct subsets
$\widetilde{\cB}_t\subseteq\exp^{-1}_{x_0}(R)$, $t\in F$, whose
$\bfi$-projections $B_t$ are almost as large as their union $B$
(\eqref{equ:upperbound-NB}--\eqref{equ:Ba}); in \Cref{subsec:dimension growth},
we apply the non-linear sum-set estimate \Cref{cor:final-growth-addcom} to show
that the $\bfj$-projection of $\bigcup_{t\in F}\widetilde{\cB}_t$ is larger
than $B$ by a definite power, which is the claimed bound.

\medskip
\textbf{Proof of \Cref{lem:proj-tube}.}
Let $\widetilde y:=\exp^{-1}_{x_0}(y)$, and let $(u_1,u_2)$ be the orthonormal frame at $\widetilde y$ with $u_1$ spanning $\exp^{-1}_{x_0}(T)$ and $u_2$ spanning $\bigl(D_y\exp^{-1}_{x_0}\bigr)\bigl(E^{su}_{\bfi}(y)\bigr)$; these are orthogonal by the choice of $T$. All components below are taken in this frame. We identify $\exp^{-1}_{x_0}(T)$ with $\R$ via the $u_1$-coordinate, taking $\widetilde y$ as the origin. 

Let $\mathfrak m_0$ be the signed slope of $T_{\widetilde y}\bigl(\exp^{-1}_{x_0}(W_\bfj(y))\bigr)$ relative to $u_2$, that is, the ratio of its $u_1$-component to its $u_2$-component. Since $u_2$ spans $\bigl(D_y\exp^{-1}_{x_0}\bigr)\bigl(E^{su}_{\bfi}(y)\bigr)$, the definition of $\angle_y(\bfi,\bfj)$ gives
\[
    |\mathfrak m_0|=\tan\angle_y(\bfi,\bfj)=\tau_y .
\]
Since $\sigma_\bfi$ carries the $\bfi$-leaves to the lines $\{p=\mathrm{const}\}$, the set $\sigma_\bfi(L)$ lies in one such line; let $v$ be its unit tangent at $\sigma_\bfi(y)$, oriented so that the $u_2$-component of
\[
    D_{\sigma_\bfi(y)}\bigl(\exp_{x_0}^{-1}\circ\sigma_\bfi^{-1}\bigr)v
\]
is positive. That image vector spans $\bigl(D_y\exp^{-1}_{x_0}\bigr)\bigl(E^{su}_{\bfi}(y)\bigr)=\R u_2$; hence its $u_1$-component vanishes and its $u_2$-component equals its norm, namely $c_0$.

We write the points of the line $\sigma_\bfi(y)+\R v$ in the form $\sigma_\bfi(y)+\lambda v$, $\lambda\in\R$, and define the affine map $\varphi_1^{-1}$ on this line by
\[
    \varphi^{-1}_1\big(\sigma_{\bfi}(y)+\lambda v\big):=-\lambda\,c_0\,\mathfrak m_0 ,
\]
which is injective since $\mathfrak m_0\neq0$ by $\tau_y>0$; we write $\varphi_1$ for its inverse. Set
\[
    I_k:=\varphi_1^{-1}\Big(\cD^{h}_{Ck}\big(\sigma_\bfi(y)\big)\cap\big(\sigma_\bfi(y)+\R v\big)\Big),
\]
where $\cD^{h}_{Ck}(z)$ denotes the atom of $\cD^{h}_{Ck}$ in $\mathbb{R}^2$ containing $z\in\R^2$. Since the line is parallel to the vertical sides of the atoms, the intersection is a dyadic segment of length $2^{-Ck}$, and it contains $\sigma_\bfi(L)$, as $\sigma_\bfi$ maps $\cD^{\bfi,h}_{Ck}(y)$ into the atom and $\calA^{su}_\bfi(y)$ into the line. Hence $I_k$ is an interval containing the origin, of length
\begin{equation}
\label{eqn:length of Ik}
    c_0\tau_y\,2^{-Ck}=s_{\mathrm{a}}2^{-(C+1)k}.
\end{equation}
For a Borel set $A$, we define the normalized restriction by $\mu^A(\cdot)=\mu(\cdot \cap A)/\mu(A)$. Set
\[
    \zeta_k:=\big(\varphi_1^{-1}\circ\sigma_\bfi\big)_*\big(\mu_{y,\calA^{su}_\bfi}\big)^{\cD^{\bfi,h}_{Ck}(y)},
\]
a probability measure supported in $\overline{I_k}$.

For $0\leq i\leq k$ and $t\in I_k$, set
\[
    s_{\mathrm{a}}\!\cdot\!\cD_{(C+1)k+i}(t)
    :=\varphi_1^{-1}\Big(\cD^{h}_{Ck+i}\big(\varphi_1(t)\big)\cap\big(\sigma_\bfi(y)+\R v\big)\Big),
\]
the cell containing $t$ of the $\varphi_1^{-1}$-image of the level-$(Ck+i)$ dyadic partition of the line; since $\big|(\varphi_1^{-1})'\big|=c_0\tau_y=s_{\mathrm{a}}2^{-k}$, this cell is an interval of length $s_{\mathrm{a}}2^{-(C+1)k-i}$, so that the subscript records the length. By dyadic nesting, these cells lie in $I_k$ and, for each fixed $i$, form a partition of $I_k$, which we denote $s_{\mathrm{a}}\!\cdot\!\cD_{(C+1)k+i}$; it is with respect to this family of partitions that $\zeta_k$ is dyadic-spreading.

Let $L_{\mathrm{good}}:=L\cap G$. For each cell $Q\in s_{\mathrm{a}}\!\cdot\!\cD_{(C+2)k}$ with
\[
    \zeta_k\Big(Q\cap\big(\varphi_1^{-1}\circ\sigma_\bfi\big)(L_{\mathrm{good}})\Big)>0 ,
\]
choose one point of $Q\cap\big(\varphi_1^{-1}\circ\sigma_\bfi\big)(L_{\mathrm{good}})\cap\supp\zeta_k$, and let $F\subset I_k$ be the resulting finite set. 

For $t\in F$, let $z_t\in L_{\mathrm{good}}$ be the unique point with $\big(\varphi_1^{-1}\circ\sigma_\bfi\big)(z_t)=t$. We claim that
\begin{equation}\label{equ:A-large-mass}
    \zeta_k\left(\bigcup_{t\in F}s_{\mathrm{a}}\cdot\cD_{(C+2)k}(t)\right)>\frac23.
\end{equation}
We verify \eqref{equ:A-large-mass}. The measure $\zeta_k$ is carried by $(\varphi_1^{-1}\circ\sigma_\bfi)(L)$, and every cell containing no point of $F$ meets $(\varphi_1^{-1}\circ\sigma_\bfi)(L_{\mathrm{good}})$ in zero $\zeta_k$-measure; hence the mass outside the union in \eqref{equ:A-large-mass} is at most
\[
    \zeta_k\Big(I_k\setminus(\varphi_1^{-1}\circ\sigma_\bfi)(L_{\mathrm{good}})\Big)
    =\big(\mu_{y,\calA^{su}_\bfi}\big)^{\cD^{\bfi,h}_{Ck}(y)}\Big(\cD^{\bfi,h}_{Ck}(y)\setminus G\Big)
    \leq16\epsilon_2 ,
\]
the equality because $\varphi_1^{-1}\circ\sigma_\bfi$ is injective and $\zeta_k(I_k)=1$, and the inequality by \Cref{prop:output-step1}. Since $16\epsilon_2<\tfrac13$ by \eqref{eq:choice-eps2}, the bound \eqref{equ:A-large-mass} follows.

\begin{defi}\label{def:neighbor-family}
For $z\in\Lambda_{n_0}$ and $m\in\N$, the \emph{neighbor family} of $z$ at scale $2^{-m}$ is defined by
\begin{equation}
N(z,m):=\Big\{\sigma_{\bfi}^{-1}(Q)\cap\Lambda_{n_0}:\ Q\in\cD^{h}_{m},\ Q\cap B\big(\sigma_{\bfi}(z),2^{-m-10}\big)\neq\emptyset\Big\},
\end{equation}
a finite collection of atoms of $\cD^{\bfi,h}_{m}$; the dependence on $\bfi$ and $h$ is suppressed from the notation.
\end{defi}

Two properties of the neighbor family will be used. First,
\begin{equation}\label{eqn:neighbor-card}
\#N(z,m)\leq4 .
\end{equation}
Second, since $C_5\geq C_3\geq\mathrm{Lip}(\sigma_\bfi)$ by \Cref{prop:straightening map}, every $w\in\Lambda_{n_0}\cap B\big(z,C_5^{-1}2^{-m-10}\big)$ satisfies $\sigma_\bfi(w)\in B\big(\sigma_\bfi(z),2^{-m-10}\big)$, and therefore
\begin{equation}\label{eqn:neighbor-cover}
\bigcup_{\cR\in N(z,m)}\cR\ \supseteq\ B\big(z,C_5^{-1}2^{-m-10}\big)\cap\Lambda_{n_0}.
\end{equation}

\begin{claim}\label{claim:lowerRa}
For every $t\in F$, there exists $\cR_t\in N\big(z_t,(C+1)k\big)$ with $\cR_t\subseteq R_1$ and
\begin{equation}\label{equ:muatom}
\mu_0\big(\cR_t\cap G'\big)\geq2^{-(C+1)k(\gamma_1+\gamma_2+2\epsilon_2)}.
\end{equation}
\end{claim}

\begin{proof}[Proof of the claim]
Write $k':=(C+1)k$, $c_5:=\lceil\log_2C_5\rceil$, and $n:=k'+10+c_5$, so that $2^{-n}\leq C_5^{-1}2^{-k'-10}$.
Since $z_t\in L_{\mathrm{good}}\subseteq G\subseteq G'$, the point $z_t$ satisfies the density property (P5), for both $\bfi$ and $\bfj$, and the measure estimate (P2), at every dyadic scale $2^{-m}$ with $m\geq\ell^*_4$; note that $n\geq k\geq\ell^*_4$.
By \eqref{eqn:neighbor-cover} at scale $m=k'$ and $2^{-n}\leq C_5^{-1}2^{-k'-10}$,
\begin{align}
\mu_0\Big(\bigcup_{\cR\in N(z_t,k')}\big(\cR\cap G'\big)\Big)
&\geq\mu_0\big(B(z_t,2^{-n})\cap G'\big)%
\geq(1-2\epsilon_2)\,\mu_0\big(B(z_t,2^{-n})\big)\\
&\geq(1-2\epsilon_2)\,2^{-n(\gamma_1+\gamma_2+\epsilon_2)}%
\geq2^{-k'(\gamma_1+\gamma_2+\frac32\epsilon_2)},
\end{align}
where the second inequality is (P5), applied once for $\bfi$ and once for $\bfj$; the third is the lower bound of (P2), together with $C_5^{-10}\leq1$; and the last holds since \eqref{eqn:K-count} gives
\[
\tfrac{k'\epsilon_2}{2}\geq(10+c_5)(\gamma_1+\gamma_2+\epsilon_2)+2
\geq(n-k')(\gamma_1+\gamma_2+\epsilon_2)+\log_2\tfrac{1}{1-2\epsilon_2}.
\]
By \eqref{eqn:neighbor-card}, some atom $\cR_t\in N(z_t,k')$ satisfies
\[
\mu_0\big(\cR_t\cap G'\big)\geq\tfrac14\,2^{-k'(\gamma_1+\gamma_2+\frac32\epsilon_2)}\geq2^{-k'(\gamma_1+\gamma_2+2\epsilon_2)},
\]
the last inequality again by \eqref{eqn:K-count}, whose right-hand side is at least $2$.
It remains to check $\cR_t\subseteq R_1$. Since the coordinate $p$ is constant on $\calA^{su}_\bfi(y)$ and $z_t\in\cD^{\bfi,h}_{Ck}(y)$, we have $z_t\in R_0$, so $\sigma_\bfi(z_t)$ lies in the cell $\cD^{h}_{(C+1)k\times Ck}\big(\sigma_\bfi(y)\big)$. The cell of $\cR_t$ lies within distance $2^{-k'-10}<2^{-k'}$ of $\sigma_\bfi(z_t)$, and two cells of the level-$k'$ grid with disjoint closures are at distance at least $2^{-k'}$; hence the closure of the cell of $\cR_t$ meets $\overline{\cD^{h}_{(C+1)k\times Ck}\big(\sigma_\bfi(y)\big)}$, and $\cR_t\subseteq R_1$ by the definition of $R_1$.
\end{proof}

For each $t\in F$, define
\[
    \widetilde{\cB}_t:=\exp_{x_0}^{-1}\left(\cR_t\cap G'\right),
    \qquad
    B_t:=\pi_\bfi\big(\widetilde{\cB}_t\big)\subseteq\exp^{-1}_{x_0}(T),
\]
and set
\[
    \widetilde{\cB}:=\bigcup_{t\in F}\widetilde{\cB}_t,
    \qquad
    B:=\bigcup_{t\in F}B_t=\pi_\bfi\big(\widetilde{\cB}\big).
\]
Since $\cR_t\subseteq R_1$ for each $t\in F$ and $R=R_1\cap G'$, we have 
\begin{equation}
\label{eqn:inclusion B}
B\subseteq\pi_\bfi\big(\exp^{-1}_{x_0}(R)\big). 
\end{equation}
The argument leading to \eqref{eq:C7} and \eqref{equ:npij} applies verbatim with $\pi_\bfi$ in place of $\pi_\bfj$ and gives
\begin{equation}\label{equ:upperbound-NB}
    \cN_{s_{\mathrm{a}}\cdot2^{-(C+2)k}}(B)
    \leq
    2^{k(\gamma_2+(2C+3)\epsilon_2)} .
\end{equation}
Next, we give a lower estimate for each $B_t$. By \eqref{equ:muatom} and \ref{P4:retangle bound},
\begin{equation}\label{equ:Ba}
\begin{split}
\cN_{s_{\mathrm{a}}\cdot2^{-(C+2)k}}(B_t)
&\geq
\frac{
\mu_0(\cR_t\cap G')
}{
10C_5^2\cdot
\sup_{x\in G'}
\mu_0\left(\cD^{\bfi,h}_{(C+2)k\times(C+1)k}(x)\right)
}
\\
&\geq
\frac{
2^{-(C+1)k(\gamma_1+\gamma_2+2\epsilon_2)}
}{
10C_5^2\cdot
2^{-k((C+1)\gamma_1+(C+2)\gamma_2-\epsilon_2)}
}
\geq
2^{k(\gamma_2-(2C+4)\epsilon_2)}
\\&
\geq
\cN_{s_{\mathrm{a}}\cdot2^{-(C+2)k}}(B)^{1-\delta}.
\end{split}
\end{equation}
For the first inequality, fix $p\in\exp^{-1}_{x_0}(T)$ and consider the fiber $\pi_\bfi^{-1}\big(B(p,s_{\mathrm{a}}2^{-(C+2)k})\big)$. By the bi-Lipschitz property of the holonomy between $T$ and $T_0$, which defines the $p$-coordinate of $\sigma_\bfi$ (\Cref{lem holder of Esu}), and the bi-Lipschitz bound on $\exp^{-1}_{x_0}|_{\Lambda_{n_0}}$, the $p$-coordinate of this fiber ranges over an interval of length at most $2C_5^2\,s_{\mathrm{a}}2^{-(C+2)k}$; by \eqref{eqn:sa bound}, this interval meets at most $10C_5^2$ grid intervals of width $2^{-(C+2)k}$. On the other hand, $\sigma_\bfi\circ\exp_{x_0}\big(\widetilde{\cB}_t\big)=\sigma_\bfi\big(\cR_t\cap G'\big)$ is contained in the cell of $\cR_t$, which has height $2^{-(C+1)k}$. Hence
\[
\sigma_{\bfi}\circ\exp_{x_0}\Big(\pi_\bfi^{-1}\big(B(p,s_{\mathrm{a}}2^{-(C+2)k})\big)\cap\widetilde{\cB}_t\Big)
\]
meets at most $10C_5^2$ atoms of $\cD^{h}_{(C+2)k\times(C+1)k}$, each of which contains a point of $\sigma_\bfi(G')$; summing over a covering of $B_t$ by $\cN_{s_{\mathrm{a}}2^{-(C+2)k}}(B_t)$ balls gives the first inequality.
The third inequality amounts to $2^{k\epsilon_2}\geq10C_5^2$, which holds since $k\geq K$ by \eqref{eqn:k large} and $K\geq\log_2(10C_5^2)/\epsilon_2$ by \eqref{eqn:choice of K}. The last inequality follows from \eqref{equ:upperbound-NB} and the second condition in \eqref{eq:choice-eps2}.

\subsection{Dimension growth}
\label{subsec:dimension growth}
We now complete the proof of \Cref{lem:proj-tube} by applying the non-linear sum-set estimate \Cref{cor:final-growth-addcom}. 

By the definition of $N\big(z_t,(C+1)k\big)$ (\Cref{def:neighbor-family}), the cell of $\cR_t$ has diameter at most $\sqrt2\,2^{-(C+1)k}$ and lies within $2^{-(C+1)k-10}$ of $\sigma_\bfi(z_t)$; pulling back by $\sigma_\bfi^{-1}$, with bi-Lipschitz constant $C_3$, and applying $\exp^{-1}_{x_0}$, with Lipschitz constant at most $C_5$, we obtain
\[
    \widetilde{\cB}_t\subseteq B\big(\widetilde z_t,\,10C_5^2\,2^{-(C+1)k}\big).
\]
Note also that $\pi_\bfi(\widetilde z_t)=0$: the point $z_t$ lies on $\calA^{su}_\bfi(y)$, whose local leaf meets $T$ at $y$, and $\widetilde y$ is the origin of our identification of $\exp^{-1}_{x_0}(T)$ with $\R$.
Since $z_t$ and $y$ lie on the same $\bfi$-plaque and in the same level-$Ck$ dyadic atom,
\[
\big|\sigma_\bfi(z_t)-\sigma_\bfi(y)\big|\leq2^{-Ck}.
\]
The bi-Lipschitz bound for $S^{-1}$ therefore gives
\[
d(\widetilde z_t,\widetilde y)\leq C_5^2\,2^{-Ck}.
\]
Together with
\[
\widetilde{\cB}_t
\subseteq
B\bigl(\widetilde z_t,10C_5^2\,2^{-(C+1)k}\bigr),
\]
this verifies the hypotheses of \Cref{rem:geom-input-commutator} with
$A_0=10C_5^2$, $\widetilde z=\widetilde z_t$, and
$B_{\widetilde z}=\widetilde{\cB}_t$. Hence \Cref{rem:geom-input-commutator} gives a constant $C_8>1$ such that
\begin{equation}\label{equ:dhba-}
    d_{\mathrm{H}}\Big(\pi_\bfj(\widetilde{\cB}_t)-\pi_\bfj(\widetilde z_t),\ \pi_\bfi(\widetilde{\cB}_t)\Big)\leq \frac{C_8}{10}\,2^{-(C+2)k}.
\end{equation}
Consequently,
\begin{equation}\label{equ:dhba}
    d_{\mathrm{H}}\left(
    \pi_\bfj(\widetilde{\cB}),
    \bigcup_{t\in F}\left(B_t+\pi_\bfj(\exp_{x_0}^{-1}(z_t))\right)
    \right)
    \leq \frac{C_8}{10} 2^{-(C+2)k}.
\end{equation}
Since \(s_{\mathrm{a}}\) is bounded below \eqref{eqn:sa bound}, one-dimensional covering stability under
Hausdorff perturbations gives, 
\begin{equation}\label{equ:covering-perturbation}
\cN_{s_{\mathrm{a}} 2^{-(C+2)k}}(\pi_\bfj(\widetilde{\cB}))
\geq
\frac{1}{C_8}
\cN_{s_{\mathrm{a}} 2^{-(C+2)k}}
\left(
\bigcup_{t\in F}\left(B_t+\pi_\bfj(\exp_{x_0}^{-1}(z_t))\right)
\right).
\end{equation}

We verify the hypotheses of
\Cref{cor:final-growth-addcom} with the help of \cref{rem:geom-input-nonlinear-displacement}. 

Consider the affine normalization $\varphi_2\colon\cD^{h}_{Ck}(\sigma_\bfi(y))\cap(\sigma_\bfi(y)+\R v)\to[0,1)$; then $\varphi_2\circ\varphi_1\colon I_k\to[0,1)$ is an affine normalization of $I_k$, and by the definition of $\zeta_k$,
\begin{equation}\label{eqn:normalized zeta}
    (\varphi_2\circ\varphi_1)_*\zeta_k=(\varphi_2\circ\sigma_\bfi)_*\mu_y^{\square},
    \qquad
    \mu_y^{\square}:=\big(\mu_{y,\calA^{su}_\bfi}\big)^{\cD^{\bfi,h}_{Ck}(y)},
\end{equation}
where $\big(\mu_{y,\calA^{su}_\bfi}\big)^{\cD^{\bfi,h}_{Ck}(y)}$ is the conditional measure on $\cD^{\bfi,h}_{Ck}(y)\cap \supp (\mu_{y,\calA^{su}_\bfi})$.

\begin{claim}\label{claim:spreading-normalized}
Every point of $(\varphi_2\circ\sigma_\bfi)\big(G\cap\supp\mu_y^{\square}\big)$ is {$(k,2\ell_1,\epsilon_*/6)$-dyadic-spreading} with respect to $(\varphi_2\circ\sigma_\bfi)_*\mu_y^{\square}$, and this set of points has $(\varphi_2\circ\sigma_\bfi)_*\mu_y^{\square}$-measure at least {$1-\epsilon_*/6$}.
\end{claim}
\begin{proof}
Let $z\in G\cap\supp\mu_y^{\square}$. Since $\mu_y^{\square}$ is carried by $\calA^{su}_\bfi(y)\cap\cD^{\bfi,h}_{Ck}(y)$, we have $z\in\calA^{su}_\bfi(y)$, hence $\mu_{z,\calA^{su}_\bfi}=\mu_{y,\calA^{su}_\bfi}$. As $z\in G\subseteq G_1(\cdot,\bfi,h)$ and $(C+1)k\geq n^*_1$, \ref{P1:spreading property} applies at $z$ with $n=(C+1)k$: the point $\sigma_\bfi(z)$ is $\big((C+1)k,2\ell_1,10\epsilon_1,h\big)$-dyadic-spreading on $\sigma_\bfi\big(\calA^{su}_\bfi(y)\big)$ with respect to $(\sigma_\bfi)_*\mu_{y,\calA^{su}_\bfi}$.
For $m\in\N$, say that $\sigma_\bfi(z)$ is \emph{$(m,2\ell_1)$-good} if
\[
    (\sigma_\bfi)_*\mu_{y,\calA^{su}_\bfi}\big(\cD^h_{m}(\sigma_\bfi(z))\big)
    >2\,(\sigma_\bfi)_*\mu_{y,\calA^{su}_\bfi}\big(\cD^h_{m+2\ell_1}(\sigma_\bfi(z))\big).
\]
For $m\geq Ck$, both cells above are contained in $\cD^h_{Ck}(\sigma_\bfi(y))$, on which the measures $(\sigma_\bfi)_*\mu_{y,\calA^{su}_\bfi}$ and $(\sigma_\bfi)_*\mu_y^{\square}$ are proportional; hence for such $m$, goodness may equivalently be read with respect to $(\sigma_\bfi)_*\mu_y^{\square}$. The spreading property gives
\begin{align}
    &\#\big\{Ck+1\leq m\leq(C+1)k:\ \sigma_\bfi(z)\ \text{is}\ (m,2\ell_1)\text{-good}\big\}\\
    \geq&\#\big\{1\leq m\leq(C+1)k:\ \sigma_\bfi(z)\ \text{is}\ (m,2\ell_1)\text{-good}\big\}-Ck\\
    \geq&(C+1)k(1-10\epsilon_1)-Ck=k\big(1-10(C+1)\epsilon_1\big){ =k(1-\epsilon_*/6)}.
\end{align}
Since $\varphi_2$ carries the segments $\cD^h_{Ck+i}\cap\big(\sigma_\bfi(y)+\R v\big)$, $i\geq0$, onto the standard level-$i$ dyadic intervals of $[0,1)$, and $(\sigma_\bfi)_*\mu_y^{\square}$ is carried by the segment $\cD^h_{Ck}(\sigma_\bfi(y))\cap(\sigma_\bfi(y)+\R v)$, it follows that $\varphi_2(\sigma_\bfi(z))$ is {$(k,2\ell_1,\epsilon_*/6)$-dyadic-spreading} with respect to $(\varphi_2\circ\sigma_\bfi)_*\mu_y^{\square}$.
Finally, \Cref{prop:output-step1} gives $\mu_y^{\square}(G)\geq1-16\epsilon_2>1/2$, and therefore
\[
    (\varphi_2\circ\sigma_\bfi)_*\mu_y^{\square}\Big((\varphi_2\circ\sigma_\bfi)\big(G\cap\supp\mu_y^{\square}\big)\Big)\geq\mu_y^{\square}\big(G\cap\supp\mu_y^{\square}\big)=\mu_y^{\square}(G){\geq1-16\epsilon_2>1-\epsilon_*/6}. 
\]
\end{proof}
\Cref{claim:spreading-normalized} therefore shows that the affine normalization of $\zeta_k$ is $(k,2\ell_1,\epsilon_*/6)$-dyadic-spreading, and hence verifies the spreading hypothesis of \Cref{cor:final-growth-addcom}.

\begin{claim}\label{lem:apply-final-growth}
The set \(F\subset I_k\), the set
\(B\), the subsets $\{B_t\subset B:t\in F\}$, and the map
\[
    \phi:F\to\mathbb R,
    \qquad
    \phi(t):=\pi_\bfj(\exp_{x_0}^{-1}(z_t)),
\]
satisfy the hypotheses of \Cref{cor:final-growth-addcom} with \(\ell_*=2\ell_1\).
\end{claim}

\begin{proof}

\begin{enumerate}[(1)]
\item Recall that $I_k$ is an interval of length $s_{\mathrm{a}}2^{-(C+1)k}$ \eqref{eqn:length of Ik}. Moreover, the choice of $K$ gives
\[
k\geq N(\epsilon_*,2\ell_1,C'_0,C),
\]
which is the large-scale hypothesis in \Cref{cor:final-growth-addcom}.

\item By construction, we have $F\subset \supp \zeta_k$, and \eqref{equ:A-large-mass} verifies \eqref{eqn:lower bound measure}.

\item For $B$, recall that \eqref{eqn:inclusion B} gives
\[B\subseteq \pi_{\bfi}(\exp_{x_0}^{-1}R_1),\]
with $\sigma_{\bfi}(R_1)$ contained in a rectangle of size $(3\cdot 2^{-(C+1)k})\times (2^{-Ck}+2\cdot 2^{-(C+1)k})$. Hence, using the bi-Lipschitz property of holonomy between $T$ and $T_0$ (\Cref{lem holder of Esu}), which defines $\sigma_{\bfi}$, and  the Lipschitz bound for $\exp_{x_0}^{-1}$ and $\sigma_{\bfi}^{-1}$, we obtain 
\[B\subset 2^{-(C+1)k}[-3C^3_5,3C^3_5].\]
Moreover,
\eqref{equ:upperbound-NB} and \eqref{eq:choice-eps2} give
\[
    \cN_{s_{\mathrm{a}} 2^{-(C+2)k}}(B)
    \leq
    2^{k(\gamma_2+(2C+3)\epsilon_2)}
    \leq
    2^{(1-\epsilon_*)k}.
\]

\item For each $t\in F$, the lower bound for $\cN_{s_{\mathrm{a}} 2^{-(C+2)k}}(B_t)$ given in \eqref{equ:Ba} verifies \eqref{eqn:lower bound Ba}.

\item Finally, since $k\geq k_0$ by \eqref{eqn:k large} and \eqref{eqn:choice of K}, \Cref{rem:geom-input-nonlinear-displacement} gives for every $t\in F$,
\[
    |\phi(t)-t|=|\pi_\bfj(\exp_{x_0}^{-1}z_t)-t|\leq s_{\mathrm{a}} 2^{-(C+3)k}.
\]
\end{enumerate}
Thus, all hypotheses of \Cref{cor:final-growth-addcom} are satisfied.
\end{proof}

Combining \Cref{cor:final-growth-addcom} with \Cref{claim:spreading-normalized} and  \Cref{lem:apply-final-growth}, we
get
\[
\cN_{s_{\mathrm{a}} 2^{-(C+2)k}}
\left(
\bigcup_{t\in F}(B_t+\pi_\bfj(\exp_{x_0}^{-1}(z_t)))
\right)
>
\frac1{200}
\cN_{s_{\mathrm{a}} 2^{-(C+2)k}}(B)^{1+\delta}.
\]
Together with \eqref{equ:covering-perturbation} and the inclusion
\(\widetilde{\cB}\subset\exp_{x_0}^{-1}R\), this gives
\begin{equation}\label{equ:lower-proj-final}
\cN_{s_{\mathrm{a}} 2^{-(C+2)k}}
\left(
\pi_\bfj(\exp_{x_0}^{-1}R)
\right)
\geq
\frac1{200C_8}
\cN_{s_{\mathrm{a}} 2^{-(C+2)k}}(B)^{1+\delta}.
\end{equation}
Since \(B_t\subset B\) and \(F\neq\emptyset\) by \eqref{equ:A-large-mass}, 
\eqref{equ:Ba} gives
\[
 \cN_{s_{\mathrm{a}} 2^{-(C+2)k}}(B)\geq \cN_{s_{\mathrm{a}} 2^{-(C+2)k}}(B_t)
    \geq
    2^{k(\gamma_2-(2C+4)\epsilon_2)}.
\]
Substituting this into \eqref{equ:lower-proj-final} gives
\[
\cN_{s_{\mathrm{a}}2^{-(C+2)k}}\big(\pi_\bfj(\exp^{-1}_{x_0}(R))\big)
\geq\frac{1}{200C_8}\,2^{k(1+\delta)(\gamma_2-(2C+4)\epsilon_2)}
\geq2^{k\gamma_2(1+\delta/2)} ,
\]
where the last inequality holds because
\begin{equation}
\label{eqn:k final}
k\Big[(1+\delta)\big(\gamma_2-(2C+4)\epsilon_2\big)-\gamma_2\big(1+\tfrac{\delta}{2}\big)\Big]
=k\Big[\tfrac{\delta\gamma_2}{2}-(1+\delta)(2C+4)\epsilon_2\Big]
\geq\frac{k\delta\gamma_2}{4}
\geq\log_2(200C_8).
\end{equation}
The middle inequality uses $\epsilon_2<\delta\gamma_2/100C$, which gives $(1+\delta)(2C+4)\epsilon_2\leq\tfrac{2(2C+4)}{100C}\delta\gamma_2\leq\tfrac{\delta\gamma_2}{4}$, and the final one is \eqref{eqn:choice of K}.
This proves \Cref{lem:proj-tube}. \qed

As explained in \eqref{eqn:contradiction}, \Cref{lem:proj-tube}
contradicts \eqref{equ:npij} under the assumption \(\gamma_2<1\). Therefore
\(\gamma_2=1\), and the proof of \Cref{thm:main-two} is complete.

\section{Approximation by Bernoulli horseshoes}\label{sec: Bernoulli horseshoe approx}

This section introduces \textit{sub-additive singular equilibrium} (SSE) measures and studies their basic properties. We show that, on $su$-non-integrable horseshoes, SSE measures can be approximated by Bernoulli measures supported on horseshoes. Using the Lyapunov dimension formula for Bernoulli measures, we prove that the Hausdorff dimension of an $su$-non-integrable surface horseshoe is the zero of the sub-additive pressure function $P_{\mathrm{sub}}(s)$ (\Cref{prop: irre-horseshoe}), which was previously obtained as an upper bound in \cite{cao_dimension_2019}.
Then, using the horseshoe approximation method in \cite{cao_dimension_2019}, we extend this result to $su$-non-integrable surface repellers (\Cref{thm:dominated}) and to $C^r$-generic surface repellers (\Cref{thm:generic_full}). Moreover, for these repellers, the following mass variational principle holds:
\[
\dim_{\mathrm H}\Lambda_f
=
\sup_{\mu\in\mathcal{M}_e(\Lambda_f,f)}\{\dim_{\mathrm H}\mu\},
\]
where the notation $\mathcal M_e(\Lambda_f,f)$ is defined in \eqref{eqn:MeXf}.
We also present a polynomial counterexample showing that the generic or $su$-non-integrability assumption cannot be removed. Finally, we give a short proof of a general lower bound for repellers in arbitrary dimensions using the Ledrappier-Young entropy formula.

Then, we complete the proofs of our main theorems in the following subsections.

\subsection{SSE measures and approximation by Bernoulli horseshoes}

Let $(\Lambda_f,f)$ be a $C^{1+\alpha}$ surface repeller. We use the sub-additive singular potentials, the sub-additive topological pressure $P_{\mathrm{sub}}$, the function $F_\mu$, and the Lyapunov dimension $\dim_{\mathrm{LY}}(\mu,f)$ introduced in \Cref{sec:2.1}. In particular, by \eqref{eq:sub-pressure-lyapunov} and \eqref{eqn:def LY},
\begin{equation}\label{eq:hsub-pressure-lyapunov}
P_{\mathrm{sub}}(s)
=
\sup_{\mu\in\mathcal M_e(\Lambda_f,f)}F_\mu(s)
\qquad \text{for any}\quad 0\leq s\leq 2.
\end{equation}

The main result of this subsection is as follows.
\begin{prop}
\label{prop: irre-horseshoe}
Let $(\Lambda_f,f)$ be a $C^{1+\alpha}$ surface horseshoe repeller. Assume that it admits a dominated splitting and is $su$-non-integrable. Then
\[
\dim_{\mathrm H}\Lambda_f
=
\dim_{\mathrm B}\Lambda_f
=
\sup_{\nu\in\mathcal{M}_e(\Lambda_f,f)}
\{\dim_{\mathrm H}\nu\}
=
s_0,
\]
where $s_0$ is the unique zero of $P_{\mathrm{sub}}(s)$.
\end{prop}

\begin{prop}\label{prop zhang upperbound}
Let $\Lambda_f$ be a repeller for a $C^1$ expanding map $f$. Then
\[
\dim_{\mathrm H}\Lambda_f
\leq
\underline{\dim}_{\mathrm B}\Lambda_f
\leq
\overline{\dim}_{\mathrm B}\Lambda_f
\leq s_0,
\]
where $s_0$ is the unique zero of $P_{\mathrm{sub}}(s)$.
\end{prop}
\begin{proof}
The inequality $\dim_{\mathrm H}\Lambda_f\leq s_0$ was established in \cite{zhang_dynamical_1997}, and the one $\overline{\dim}_{\mathrm B}\Lambda_f\leq s_0$ was proved in  \cite[Definition~6.1 and Theorem~1.3]{feng_simon_dimension_part_i_2023}. The middle inequalities are the standard relations between Hausdorff and lower and upper box dimensions.\qedhere
\end{proof}

We approximate the lower bound for $\dim_{\mathrm H}\Lambda_f$ by approximating a sub-additive singular equilibrium measure by Bernoulli measures.

\begin{defi}
\label{def:SSE measure}
An $f$-invariant ergodic measure $\mu$ on $\Lambda_f$ is called a \emph{sub-additive singular equilibrium measure}, or an \emph{SSE measure}, if 
\begin{equation}
F_{\mu}(s_0)=0,
\end{equation}
where $F_{\mu}$ is defined as in \eqref{eqn:def LY} and $s_0$ is the unique zero of  $P_{\mathrm{sub}}$.
\end{defi}

SSE measures always exist in {{our setting}}. Indeed, let $\{-\varphi_n^{s_0}\}_{n\geq1}$ be defined as in \eqref{eqn:singular potential}. Since it is
sub-additive, the associated Lyapunov functional satisfies
\[
\eta\longmapsto
\lim_{n\to\infty}\frac1n\int-\varphi_n^{s_0}\,\mathrm d\eta
=
\inf_{n\geq1}\frac1n\int-\varphi_n^{s_0}\,\mathrm d\eta
\]
and is upper semi-continuous on the compact space $\mathcal M(\Lambda_f,f)$. The entropy map is also upper semi-continuous for repellers. Hence, {{the functional}} $F_\mu(s_0)$ attains its maximum, and the ergodic-decomposition formulas for entropy and sub-additive Lyapunov functionals yield an ergodic maximizing component; see \cite{cao_thermodynamic_2008,cao_dimension_2019}. Since $P_{\mathrm{sub}}(s_0)=0$, this component is an SSE measure.

\begin{prop}\label{prop:sse-max-lyap}
An $f$-invariant ergodic measure $\mu$ on $\Lambda_f$ is an SSE measure if and only if
\begin{equation}
\label{eqn:SSE LY}
\dim_{\mathrm{LY}}(\mu,f)
=
\sup_{\nu\in\mathcal{M}_e(\Lambda_f,f)}
\{\dim_{\mathrm{LY}}(\nu,f)\}.
\end{equation}
\end{prop}

We start the proof with the following lemma.
\begin{lem}\label{lem:1.1}
Let $s_0$ be the unique zero of $P_{\mathrm{sub}}(s)$. Then
\[
s_0
=
\sup_{\mu\in\mathcal{M}_e(\Lambda_f,f)}\{\dim_{\mathrm{LY}}(\mu,f)\} .
\]
\end{lem}
\begin{proof}
Set
\[
D:=\sup_{\mu\in\mathcal{M}_e(\Lambda_f,f)}\{\dim_{\mathrm{LY}}(\mu,f)\} .
\]
We first show that $P_{\mathrm{sub}}(s)>0$ whenever $s<D$. Indeed, choose $\mu\in\mathcal M_e(\Lambda_f,f)$ such that $\dim_{\mathrm{LY}}(\mu,f)>s$. Since $F_\mu$ is strictly decreasing and $F_\mu(\dim_{\mathrm{LY}}(\mu,f))=0$, we have $F_\mu(s)>0$. By \eqref{eq:hsub-pressure-lyapunov},
\[
P_{\mathrm{sub}}(s)\geq F_\mu(s)>0 .
\]
Hence $s\leq s_0$ for every $s<D$, and therefore $D\leq s_0$.

Conversely, if $s>D$, then $\dim_{\mathrm{LY}}(\mu,f)<s$ for every $\mu\in\mathcal M_e(\Lambda_f,f)$. Again using the monotonicity of $F_\mu$, we get
\[
F_\mu(s)<0,
\qquad \forall \mu\in\mathcal M_e(\Lambda_f,f).
\]
Thus $P_{\mathrm{sub}}(s)\leq0$. Since $P_{\mathrm{sub}}$ is decreasing and its zero is unique, this implies $s\geq s_0$. Letting $s\downarrow D$ gives $s_0\leq D$. Therefore $s_0=D$.
\end{proof}

\begin{proof}[Proof of \Cref{prop:sse-max-lyap}]
By \Cref{lem:1.1}, the right-hand side of \eqref{eqn:SSE LY} is equal to $s_0$. Given any $f$-invariant ergodic measure $\mu$ on $\Lambda_f$, since $F_\mu(t)$ is strictly decreasing and has the unique zero $\dim_{\mathrm{LY}}(\mu,f)$, 
\[
F_\mu(s_0)=0\quad \text{if and only if} \quad \dim_{\mathrm{LY}}(\mu,f)=s_0.
\]
This is exactly the definition of an SSE measure.
\end{proof}

We now prove \Cref{prop: irre-horseshoe}. Let $(\Lambda_f,f)$ be a $C^{1+\alpha}$ surface horseshoe repeller with $(\Lambda_f,f)\cong (\Sigma^+,\sigma)$ \eqref{eqn:horseshoe def}. Assume that it admits a dominated splitting and is $su$-non-integrable. Let $(\hat\Lambda_f,\hat f)$ be the corresponding inverse limit space and $\pi:(\hat \Lambda, \hat f)\to (\Lambda,f)$ be the projection map \eqref{eqn:projection inverse limit space}. 
For $\mu\in\mathcal M(\Lambda_f,f)$, let $\hat\mu$ be the unique $\hat{f}$-invariant probability measure on $\hat \Lambda_f$ such that 
$\pi_*\hat\mu=\mu$.

Recall that for $C^{1+\alpha}$ repellers, the dominated splitting is H\"older
continuous (\cref{lem: Holder in inverse limit space}); consequently, the integrands in
\eqref{eq:Lyexpo1} and \eqref{eq:Lyexpo2} are H\"older continuous on
$\hat\Lambda_f$, and so is the potential
$$\phi(\hat x):=-\left(\max\{s_0-1,0\}\log\|Df|_{E^{su}(\hat x)}\|+\min\{s_0,1\}\log\|Df|_{E^{wu}(\pi(\hat x))}\| \right)  $$
in
\eqref{eq:sub-pressure-lyapunov}, being a linear combination of them. By definition, one can prove that the SSE measures of $(\Lambda_f,f)$ are the $\pi$ projection of
the equilibrium states of $\phi$ on $(\hat\Lambda_f,\hat f)$.
Since $f|_{\Lambda_f}$ is transitive (a horseshoe
repeller is coded by a full shift of finite type), and the potential is H\"older, the classical thermodynamic formalism
for expanding maps \cite[Thm.~1.22]{bowen_equilibrium_1975}, \cite{ruelle_thermodynamic_1978} gives a unique
equilibrium state of $\phi$, which is a Gibbs measure and therefore has full
support in $\hat \Lambda_f$. Hence, $(\Lambda_f,f)$ has a unique SSE measure
$\mu_0$, and $\operatorname{supp}\mu_0=\Lambda_f$. We approximate $\mu_0$ by
Bernoulli measures with respect to the iterates of $f$.

For each $n\geq1$, let $\mathcal P_n$ be the set of $n$-cylinders of the full shift $\Sigma^+$. Then the system $(\Lambda_f,f^n)$ is homeomorphic to the one-sided full shift coded by the alphabet set $\mathcal{P}_n$. Define the Bernoulli measure $\overline\mu_{B,n}$ for $(\Lambda_f,f^n)$ by assigning to each $n$-cylinder $C\in\mathcal P_n$ the weight
\[
\overline\mu_{B,n}(C)=\mu_0(C),
\]
and then taking the product measure on the full shift over the alphabet $\mathcal P_n$. Since $\mu_0$ has full support, $\overline\mu_{B,n}$ has full support for $(\Lambda_f,f^n)$. Moreover,
\[
\overline\mu_{B,n}\xrightarrow[n\to\infty]{*}\mu_0.
\]
Indeed, given any $k\geq 1$ and any $C\in\mathcal P_k$, for any $n\geq k$, $C$ is a disjoint union of $n$-cylinders; by the definition of $\overline\mu_{B,n}$,
\[
\overline\mu_{B,n}(C)=\mu_0(C).
\]

The measures $\overline\mu_{B,n}$ are $f^n$-invariant but not necessarily $f$-invariant. Hence, we define
\[
\mu_{B,n}
:=
\frac1n\sum_{i=0}^{n-1}f_*^i\overline\mu_{B,n}.
\]
Then $\mu_{B,n}\in\mathcal M_e(\Lambda_f,f)$.  Indeed, if $A$ is an $f$-invariant Borel set, then $A$ is $f^n$-invariant and
\[
f_*^i\overline\mu_{B,n}(A)
=
\overline\mu_{B,n}(f^{-i}A)
=
\overline\mu_{B,n}(A).
\]
Therefore
\[
\mu_{B,n}(A)=\overline\mu_{B,n}(A).
\]
Since $\overline\mu_{B,n}$ is Bernoulli and hence $f^n$-ergodic, $\overline\mu_{B,n}(A)\in\{0,1\}$. Thus $\mu_{B,n}$ is $f$-ergodic. The convergence
\begin{equation}
\label{eqn:Bernoulli weak con}
\mu_{B,n}\xrightarrow[n\to\infty]{*}\mu_0
\end{equation}
follows from the weak-$*$ convergence of $\overline\mu_{B,n}$ to $\mu_0$ and the fact that for any $(n-i)-$cylinder $C$, one has $f^{i}_*\overline\mu_{B,n}(C)=\mu_0(C)$.

\begin{lem}\label{lem:bernoulli-dim-equality}
For every $n\geq1$,
\[
\dim_{\mathrm{LY}}(\mu_{B,n},f)=\dim_{\mathrm{H}}\mu_{B,n}.
\]
\begin{proof}
By assumption,
$(\Lambda_f,f)$ is $su$-non-integrable. Then for any $n\in\mathbb N$, $(\Lambda_f,f^n)$ is also $su$-non-integrable. Hence, we can apply \Cref{thm:base} to the $su$-non-integrable horseshoe $(\Lambda_f,f^n)$ 
and the Bernoulli measure $\overline\mu_{B,n}$, and obtain that
\[
\dim_{\mathrm{LY}}(\overline\mu_{B,n},f^n)
=
\dim_{\mathrm H}\overline\mu_{B,n}.
\]
It remains to compare the Hausdorff dimension, entropy, and Lyapunov exponents of $\overline\mu_{B,n}$ with those of its $f$-invariant average $\mu_{B,n}$.

First, we use the definition of the Hausdorff dimension of a measure from \Cref{sec:2.1}. Since $f$ is a local diffeomorphism near the compact repeller, for every $i\geq0$ there is a finite Borel partition of $\Lambda_f$ such that the restriction of $f^i$ to each element is bi-Lipschitz. Using the finite stability of Hausdorff dimension, and the analogous finite family of local inverse branches, we obtain
\[
\dim_{\mathrm H} f^i(A)=\dim_{\mathrm H}A,
\qquad
\dim_{\mathrm H} f^{-i}(A)=\dim_{\mathrm H}A
\]
for every Borel set $A\subseteq\Lambda_f$. Consequently,
\[
\dim_{\mathrm H}(f_*^i\eta)=\dim_{\mathrm H}\eta
\]
for every Borel probability measure $\eta$ on $\Lambda_f$. Since $\mu_{B,n}$ is a finite convex combination with positive weights of the measures $f_*^i\overline\mu_{B,n}$, a Borel set has full $\mu_{B,n}$-measure if and only if it has full measure for every component. Therefore
\[
\dim_{\mathrm H}\mu_{B,n}
=
\max_{0\leq i<n}\dim_{\mathrm H}(f_*^i\overline\mu_{B,n})
=
\dim_{\mathrm H}\overline\mu_{B,n}.
\]

Next, we compare the Lyapunov exponents. For each $i=0,\ldots,n-1$, the measure $f_*^i\overline\mu_{B,n}$ is $f^n$-invariant and has the same Lyapunov exponents for $f^n$ as $\overline\mu_{B,n}$. Hence
\[
\lambda_j(\mu_{B,n},f^n)
=
\lambda_j(\overline\mu_{B,n},f^n),
\qquad j=1,2.
\]
Since $\mu_{B,n}$ is $f$-ergodic,
\[
\lambda_j(\mu_{B,n},f^n)=n\lambda_j(\mu_{B,n},f),
\]
and therefore
\[
\lambda_j(\overline\mu_{B,n},f^n)=n\lambda_j(\mu_{B,n},f),
\qquad j=1,2.
\]

Finally, we compare entropy. By the property of the entropy map, it follows that
\[
h_{\mu_{B,n}}(f^n)
=
\frac1n\sum_{i=0}^{n-1}h_{f_*^i\overline\mu_{B,n}}(f^n).
\]
For each $i$, the systems $(f^i_*\overline\mu_{B,n},f^n)$ and $(\overline\mu_{B,n},f^n)$ have the same entropy. More explicitly, for any finite partition $\mathcal P$,
\begin{align}
 h_{f_*^i\overline\mu_{B,n}}(f^n,\mathcal P)
&=
 h_{\overline\mu_{B,n}}(f^n,f^{-i}\mathcal P)
 \leq h_{\overline\mu_{B,n}}(f^n),
\end{align}
and applying the same argument to $f^{n-i}_*(f^i_*\overline\mu_{B,n})=\overline\mu_{B,n}$ gives the reverse inequality. Hence
\[
h_{f_*^i\overline\mu_{B,n}}(f^n)=h_{\overline\mu_{B,n}}(f^n).
\]
It follows that
\[
h_{\mu_{B,n}}(f^n)=h_{\overline\mu_{B,n}}(f^n).
\]
Since $\mu_{B,n}$ is $f$-invariant,
\[
h_{\mu_{B,n}}(f^n)=n h_{\mu_{B,n}}(f).
\]
Therefore
\begin{equation}
\label{eqn:entropy fn}
h_{\overline\mu_{B,n}}(f^n)=n h_{\mu_{B,n}}(f).
\end{equation}
The Lyapunov dimension is determined only by the entropy and the Lyapunov exponents, with the same normalization under passing from $f$ to $f^n$. Hence
\[
\dim_{\mathrm{LY}}(\overline\mu_{B,n},f^n)
=
\dim_{\mathrm{LY}}(\mu_{B,n},f).
\]
Combining this with the Hausdorff dimension equality proves the lemma.
\end{proof}
\end{lem}

\begin{lem}\label{lem:bernoulli-lyap-convergence}
We have
\[
\dim_{\mathrm{LY}}(\mu_{B,n},f)
\xrightarrow[n\to\infty]{}
\dim_{\mathrm{LY}}(\mu_0,f)
\]
\end{lem}

\begin{proof}
We have already shown that $\mu_{B,n}\to\mu_0$ in the weak-$*$ topology \eqref{eqn:Bernoulli weak con}. Passing to the inverse limit gives
\[
\hat\mu_{B,n}\xrightarrow[n\to\infty]{*}\hat\mu_0 .
\]
By \Cref{lem: Holder in inverse limit space}, the functions
\[
\hat x\mapsto \log\|Df|_{E^{su}(\hat x)}\|,
\qquad
\hat x\mapsto \log\|Df|_{E^{wu}(\pi(\hat x))}\|
\]
are continuous. Hence, it follows from \eqref{eq:Lyexpo1} and \eqref{eq:Lyexpo2} that for $j=1,2$,
\[
\lambda_j(\mu_{B,n},f)
\xrightarrow[n\to\infty]{}
\lambda_j(\mu_0,f).
\]

It remains to prove the convergence of entropies. Recall that $\mathcal P_k$ is the set of $k$-cylinders in the full-shift $\Sigma^+$. Hence, $\mathcal P_1$ is a generating partition for $(\Lambda_f,f)$ and 
\[
h_{\mu_0}(f)=h_{\mu_0}(f,\mathcal P_1)
=
\lim_{k\to\infty}\frac1k H_{\mu_0}(\mathcal P_k).
\]
By the construction of $\overline\mu_{B,n}$, for each $n$,
\[
H_{\overline\mu_{B,n}}(\mathcal P_n)=H_{\mu_0}(\mathcal P_n).
\]
Moreover, $\overline\mu_{B,n}$ is Bernoulli for the map $f^n$ with alphabet $\mathcal P_n$, so
\[
h_{\overline\mu_{B,n}}(f^n)=H_{\overline\mu_{B,n}}(\mathcal P_n)=H_{\mu_0}(\mathcal P_n).
\]
Combining with \eqref{eqn:entropy fn}, we obtain
\[
h_{\mu_{B,n}}(f)=\frac1n h_{\overline\mu_{B,n}}(f^n)
=
\frac1n H_{\mu_0}(\mathcal P_n).
\]
Hence
\[
h_{\mu_{B,n}}(f)
\xrightarrow[n\to\infty]{}
h_{\mu_0}(f).
\]
Since the Lyapunov dimension is a continuous function of the entropy and Lyapunov exponents as long as the exponents remain positive, we conclude that
\[
\dim_{\mathrm{LY}}(\mu_{B,n},f)
\xrightarrow[n\to\infty]{}
\dim_{\mathrm{LY}}(\mu_0,f) .
\]
\iffalse
The last equality follows from \Cref{prop:sse-max-lyap}, because $\mu$ is the SSE measure.
\fi
\end{proof}

\begin{proof}[Proof of \Cref{prop: irre-horseshoe}]
By \Cref{lem:1.1} and \Cref{prop:sse-max-lyap},
\[
s_0=\dim_{\mathrm{LY}}(\mu_0,f)
=
\sup_{\nu\in\mathcal M_e(\Lambda_f,f)}
\{\dim_{\mathrm{LY}}(\nu,f)\}.
\]
By \Cref{lem:bernoulli-lyap-convergence} and \Cref{lem:bernoulli-dim-equality},
\[
\sup_{n\geq1}\{\dim_{\mathrm H}\mu_{B,n}\}=s_0.
\]
Since each $\mu_{B,n}$ is an ergodic $f$-invariant measure on $\Lambda_f$,
\[
s_0
\leq
\sup_{\nu\in\mathcal M_e(\Lambda_f,f)}
\{\dim_{\mathrm H}\nu\}
\leq
\dim_{\mathrm H}\Lambda_f.
\]
We complete the proof by using the inequalities established in \Cref{prop zhang upperbound}.
\end{proof}

\subsection{Dimensions of $su$-non-integrable surface repellers}

In this subsection, we establish the following non-perturbative dimension formula.  
\begin{thm}
\label{thm:dominated}
Let $(\Lambda_f,f)$ be a $C^{1+\alpha}$ surface repeller. Assume that $(\Lambda_f,f)$ admits a dominated splitting and is $su$-non-integrable. Then
\begin{equation}
\label{eqn:dim formula surface repeller}
\dim_{\mathrm{H}}\Lambda_f = \dim_{\mathrm{B}}\Lambda_f = \sup_{\mu\in\mathcal{M}_e(\Lambda_f,f)} \{\dim_{\mathrm{H}} \mu\} = s_0,\end{equation}where $s_0$ is the unique zero of $P_{\mathrm{sub}}(s)$.
\end{thm}
\noindent
\Cref{thm:B} is \eqref{eqn:dim formula surface repeller} without the supremum.

By the spectral decomposition, $\Lambda_f$ is a finite union of $f$-invariant transitive components. The sub-additive pressure $P_{\mathrm{sub}}$ of the union is the maximum of the component pressures, and once the theorem is proved for each transitive component, the Hausdorff and box dimensions, as well as the variational supremum over ergodic measures, are the corresponding maxima. It therefore suffices to work on one transitive component, and we assume below that
\[(\Lambda_f,f)\quad \text{is transitive}.\]

We shall use the following horseshoe approximation result, which is implicitly derived from the construction in \cite[Section 5]{cao_dimension_2019}. See also \cite[Section 3.4]{CaoZhao2017}.
\begin{prop}\label{prop: cao approxi}
Let $(\Lambda_f,f)$ be a $C^{1+\alpha}$ surface repeller, and let $\mu\in\mathcal M_e(\Lambda_f,f)$ satisfy
\[
h_\mu(f)>0,
\qquad
\lambda_1(\mu,f)>\lambda_2(\mu,f)>0 .
\]
Then for every sufficiently small $\epsilon>0$, there exist $N\in\mathbb N$ and a horseshoe repeller $(Q_\epsilon,f^N)$ with $Q_\epsilon\subseteq\Lambda_f$ such that:
\begin{itemize}
    \item %
    $h_{\mathrm{top}}(f^N,Q_\epsilon) 
    \geq N(h_\mu(f)-\epsilon).$

    \item $(Q_\epsilon,f^N)$ admits a dominated splitting $E^{su}\oplus E^{wu}$, and for every $\nu\in\mathcal M(Q_\epsilon,f^N)$ and every $i\in\{1,2\}$,
    \[
    \lambda_i(\nu,f^N)
    \in
    [N(\lambda_i(\mu,f)-\epsilon),N(\lambda_i(\mu,f)+\epsilon)].
    \]
\end{itemize}
\end{prop}

\begin{rem}\label{rem 2.2}
In particular, if $\nu_{\mathrm{MME}}$ denotes the measure of maximal entropy of $(Q_\epsilon,f^N)$, then the two estimates in \Cref{prop: cao approxi} imply
\[
\dim_{\mathrm{LY}}(\nu_{\mathrm{MME}},f^N)
\geq
\dim_{\mathrm{LY}}(\mu,f)-\epsilon'(\epsilon),
\]
where $\epsilon'(\epsilon)\to0$ as $\epsilon\to0$.
\end{rem}

\begin{rem}\label{rem:iterate-approximation}
If $(Q_\epsilon,f^N)$ satisfies the conclusions of \Cref{prop: cao approxi}, then $(Q_\epsilon,f^{mN})$ satisfies the analogous conclusions for every $m\in\mathbb N$, after replacing $N$ by $mN$. Indeed, entropy and Lyapunov exponents scale linearly under iteration.
\end{rem}

Now, we prove \cref{thm:dominated}.
\begin{proof}[Proof of \cref{thm:dominated}]
By reducing to transitive components, assume that $(\Lambda_f,f)$ is transitive. Let $\mu$ be an ergodic SSE measure of $(\Lambda_f,f)$. Since the repeller admits a dominated splitting, it follows that $\lambda_1(\mu,f)>\lambda_2(\mu,f)>0$.

If $h_\mu(f)=0$, then $\dim_{\mathrm{LY}}(\mu,f)=0$. Since $\mu$ is an SSE measure, \Cref{prop:sse-max-lyap} gives $s_0=0$. By \Cref{prop zhang upperbound},
\[
0\leq
\dim_{\mathrm H}\Lambda_f
\leq
\underline{\dim}_{\mathrm B}\Lambda_f
\leq
\overline{\dim}_{\mathrm B}\Lambda_f
\leq s_0=0.
\]
Hence the conclusion holds. We therefore assume below that $h_\mu(f)>0$.

Fix $\epsilon>0$. Applying \Cref{prop: cao approxi} to $\mu$, we obtain a horseshoe $(Q_\epsilon,f^N)\subseteq(\Lambda_f,f^N)$. The dominated splitting of $(Q_\epsilon,f^N)$ is the restriction of the dominated splitting of $(\Lambda_f,f^N)$.

We first treat the case where there exists a sequence of $\epsilon$ goes to zero such that $(Q_\epsilon,f^N)$ is $su$-non-integrable for each $\epsilon$. By \Cref{prop: irre-horseshoe,lem:1.1},
\[
\dim_{\mathrm H}Q_\epsilon
=
\dim_{\mathrm B}Q_\epsilon
=
\sup_{\nu\in\mathcal{M}_e(Q_\epsilon,f^N)}
\{\dim_{\mathrm H}\nu\}
=
\sup_{\nu\in\mathcal{M}_e(Q_\epsilon,f^N)}
\{\dim_{\mathrm{LY}}(\nu,f^N)\}.
\]

Let $\nu_{\mathrm{MME}}$ be the measure of maximal entropy of $(Q_\epsilon,f^N)$. By \Cref{rem 2.2},
\[
\dim_{\mathrm{LY}}(\nu_{\mathrm{MME}},f^N)
\geq
\dim_{\mathrm{LY}}(\mu,f)-\epsilon'(\epsilon)
=s_0-\epsilon'(\epsilon),
\]
where $\epsilon'(\epsilon)\to0$ as $\epsilon\to0$. Hence
\[
\dim_{\mathrm H}\Lambda_f
\geq
\dim_{\mathrm H}Q_\epsilon
\geq
s_0-\epsilon'(\epsilon).
\]
Letting $\epsilon\to0$ and using \Cref{prop zhang upperbound}, we obtain
\[
\dim_{\mathrm H}\Lambda_f=s_0.
\]

It remains to obtain the mass variational principle. For every $a>0$, choose $\epsilon>0$ sufficiently small so that
\[
\dim_{\mathrm H}Q_\epsilon\geq \dim_{\mathrm H}\Lambda_f-a.
\]
By \Cref{prop: irre-horseshoe}, there exists an ergodic measure
\[
\overline\nu\in\mathcal M_e(Q_\epsilon,f^N)
\]
such that
\[
\dim_{\mathrm H}\overline\nu
\geq
\dim_{\mathrm H}Q_\epsilon-a.
\]
Define the $f$-invariant averaged measure
\[
\nu:=\frac1N\sum_{i=0}^{N-1}f_*^i\overline\nu .
\]
As in the proof of \Cref{lem:bernoulli-dim-equality}, $\nu$ is $f$-ergodic and
\[
\dim_{\mathrm H}\nu
=
\dim_{\mathrm H}\overline\nu.
\]
Therefore
\begin{equation}
\label{eqn:measure geq fractal}
\dim_{\mathrm H}\nu
\geq
\dim_{\mathrm H}\Lambda_f-2a.
\end{equation}
Since $a>0$ is arbitrary,
\[
\dim_{\mathrm H}\Lambda_f
=
\sup_{\eta\in\mathcal M_e(\Lambda_f,f)}
\{\dim_{\mathrm H}\eta\}
=
s_0.
\]

It remains to explain how to reduce the case where $(Q_\epsilon,f^N)$ is $su$-integrable to the previous case. Let $(\Lambda_{f,*},f^N)$ be the transitive component of $(\Lambda_f,f^N)$ that contains $(Q_\epsilon,f^N)$.
We use the $su$-non-integrability of $(\Lambda_{f,},f^N)$ to add one extra inverse branch to the horseshoe construction.

In the proof of \Cref{prop: cao approxi}, the horseshoe is obtained as follows. One first chooses a ball $B(x_0,\rho)$ and points $x_1,\ldots,x_h\in B(x_0,\rho)$ such that
\[
U_i:=f^{-N}_{x_i}B(x_0,\rho)\subseteq B(x_0,\rho),
\qquad 1\leq i\leq h,
\]
are pairwise disjoint. 
Here, \(f^{-N}_{x_i}\) is the
inverse branch of \(f^N\) on \(B(x_0,\rho)\) characterized by
\[
f^N\circ f^{-N}_{x_i}
=
\operatorname{Id}_{B(x_0,\rho)},
\qquad
f^{-N}_{x_i}\bigl(f^N(x_i)\bigr)=x_i.
\]
In particular, \(f^{-N}_{x_i}\) maps \(B(x_0,\rho)\)
diffeomorphically onto \(U_i\). Let
\[
U:=\bigcup_{i=1}^h U_i.
\]
then the maximal invariant set in $U$ is the horseshoe $Q_\epsilon$. In this case, we call the open neighborhoods $U_i$ the \emph{Markov domains} of the horseshoe $(Q_\epsilon,f^N)$.

{
Put
\[
F:=f^N,
\]
and let $\widehat{Q_\epsilon}$ and
$\widehat{\Lambda_{f,*}}$ be the inverse limit spaces of $(Q_{\epsilon},F)$ and $(\Lambda_{f,*},F)$, respectively; we
identify an $F$-history with the corresponding $f$-history obtained by
inserting the intermediate iterates. By abuse of notation, we denote both projection maps $\widehat{Q_\epsilon}\to Q_{\epsilon}$ and $\widehat{\Lambda_{f,*}}\to \Lambda_{f,*}$ by $\pi$.

Fix $p\in Q_\epsilon$, and let $E^{su}_Q(p)$ be the common strong unstable
direction of the histories in $\widehat{Q_\epsilon}$ over $p$. Since
$(\Lambda_{f,*},F)$ is $su$-non-integrable, there exists
$\hat y\in\widehat{\Lambda_{f,*}}$ with $\pi(\hat{y})=p$ such that
\[
E^{su}(\hat y)\neq E^{su}_Q(p).
\]
By the continuity of $E^{su}$, we may choose $k_1$ sufficiently large and set
\[
q:=\pi(\hat F^{-k_1}\hat y),
\qquad
F^{k_1}q=p,
\]
so that every $\hat x_2\in\widehat{\Lambda_{f,*}}$ satisfying
\[
\pi(\hat x_2)=p,
\qquad
\pi(\hat F^{-k_1}\hat x_2)=q
\]
also satisfies
\begin{equation}
\label{eqn:Esu different}
E^{su}(\hat x_2)\neq E^{su}_Q(p).
\end{equation}

By compactness and uniform expansion, there exists $r_{\mathrm{inv}}>0$
such that, for every $z\in\Lambda_{f,*}$, the inverse branch of
$F=f^N$ sending $F(z)$ to $z$ is defined on
$B(F(z),r_{\mathrm{inv}})$ and maps it into $B(z,r_{\mathrm{inv}})$.
Shrinking $\rho$ at the first choice if necessary, we can assume that $\rho<10^{-3}\cdot r_{\mathrm{inv}}$
and that $B(x_0,\rho)$ lies in an isolating neighborhood of
$\Lambda_{f,*}$ without loss of generality.
Then, the inverse branch
\[
f^{-k_1N}_q=F^{-k_1}_q
\]
is defined on $B(x_0,\rho)$, and we set
\[
V:=F^{-k_1}_q B(x_0,\rho).
\]

By transitivity of $(\Lambda_{f,*},F)$, there exist
\[
w\in B\left(x_0,\frac{\rho}{100}\right)\cap\Lambda_{f,*}
\qquad\text{and}\qquad
k_2\in\mathbb N
\]
such that $F^{k_2}w\in V$. Since $V$ can be made arbitrarily small by
increasing $k_1$, the inverse branch $F^{-k_2}_w$ is defined on $V$ and
maps it into $B(x_0,\rho)$.

Set
\[
k:=k_1+k_2,
\qquad
g_+:=F^{-k_2}_w\circ F^{-k_1}_q,
\qquad
U'_+:=g_+\bigl(B(x_0,\rho)\bigr).
\]
If $g_i:=f^{-N}_{x_i}$, then, for every word
\[
\mathbf a=(a_1,\ldots,a_k)\in\{1,\ldots,h\}^k,
\]
define
\[
g_{\mathbf a}:=
g_{a_k}\circ\cdots\circ g_{a_1},
\qquad
U'_{\mathbf a}:=
g_{\mathbf a}\bigl(B(x_0,\rho)\bigr).
\]
The sets $U'_{\mathbf a}$ are precisely the pairwise disjoint Markov
domains of $(Q_\epsilon,F^k)$.

We claim that
\[
U'_+\cap U'_{\mathbf a}=\emptyset
\qquad
\text{for every }\mathbf a\in\{1,\ldots,h\}^k.
\]
Indeed, if the intersection were nonempty, the two inverse branches
$g_+$ and $g_{\mathbf a}$ of $F^k$ on the connected ball
$B(x_0,\rho)$ would coincide, and hence
\[
q_0:=g_+(p)=g_{\mathbf a}(p)\in Q_\epsilon.
\]
It would follow that
\[
q=F^{k_2}q_0\in Q_\epsilon,
\]
so the orbit segment from $q$ to $p$ would extend to a history
$\hat x\in\widehat{Q_\epsilon}$ with
\[
\pi(\hat F^{-k_1}\hat x)=q \quad \text{and} \quad E^{su}(\hat{x})=E^{su}_Q(p).
\]
By taking $\hat x_2=\hat x$, this contradicts \eqref{eqn:Esu different}.

Let $Q_\epsilon^+$ be the maximal $F^k$-invariant set in
\[
\left(\bigcup_{\mathbf a\in\{1,\ldots,h\}^k}U'_{\mathbf a}\right)
\cup U'_+.
\]
All orbit segments defining these branches remain in the fixed isolating
neighborhood; therefore,
\[
Q_\epsilon\subseteq Q_\epsilon^+\subseteq\Lambda_{f,*}\subseteq\Lambda_f.
\]
An $F^k$-history of $p$ using the new branch $g_+$ expands to an
$F$-history $\hat x_2\in\widehat{\Lambda_{f,*}}$ satisfying
\[
\pi(\hat F^{-k_1}\hat x_2)=q.
\]
For every $\hat x_1\in\widehat{Q_\epsilon}$ over $p$, the choice of $q$
gives
\[
E^{su}(\hat x_1)\neq E^{su}(\hat x_2).
\]
Since the dominated splitting and its strong unstable direction are
unchanged under iteration, $(Q_\epsilon^+,F^k)
=(Q_\epsilon^+,f^{kN})$ is $su$-non-integrable.
}

Since $Q_\epsilon\subseteq Q_\epsilon^+$, the measure of maximal entropy  $\nu_0$ of $(Q_\epsilon,f^{kN})$ is an $f^{kN}$-ergodic invariant measure on $Q_\epsilon^+$.
 Applying \Cref{prop: irre-horseshoe,lem:1.1} to $(Q_\epsilon^+,f^{kN})$ gives
\[
\dim_{\mathrm H}Q_\epsilon^+
=
\dim_{\mathrm B}Q_\epsilon^+=
\sup_{\nu\in\mathcal{M}_e(Q_\epsilon^+,f^{kN})}
\{\dim_{\mathrm{H}} \nu\}
\geq \dim_{\mathrm{LY}}(\nu_0,f^{kN}).
\]
Applying \Cref{rem 2.2} and \Cref{rem:iterate-approximation} to $(Q_{\epsilon},f^{kN})$ and $\nu_0$, we conclude that
\[
\dim_{\mathrm H}\Lambda_f
\geq
\dim_{\mathrm H}Q_\epsilon^+
\geq
s_0-\epsilon'(\epsilon).
\]
Meanwhile, arguing as in \eqref{eqn:measure geq fractal} for $f^{kN}$-ergodic invariant measures on $Q_{\epsilon}^+$, we obtain that
\[
\dim_{\mathrm H}\Lambda_f
\geq
\sup_{\nu\in\mathcal M_e(\Lambda_f,f)}
\{\dim_{\mathrm H}\nu\}
\geq \dim_{\mathrm{H}} Q_{\epsilon}^+-2\epsilon'(\epsilon).
\]
By letting $\epsilon\to 0$ and using \Cref{prop zhang upperbound}, we obtain \eqref{eqn:dim formula surface repeller}, completing the proof of \Cref{thm:dominated}.
\end{proof}

\begin{rem}
Let $(\Lambda_f,f)$ be a $C^{1+\alpha}$ transitive $su$-non-integrable surface repeller. If an ergodic measure of full dimension exists, then it must be the unique SSE measure and have full support. Equivalently, this happens if and only if
{
\(
\dim_{\mathrm{LY}}(\mu_{\mathrm{SSE}},f)
=
\dim_{\mathrm H}\mu_{\mathrm{SSE}} .
\)
}
\end{rem}

\subsection{Dimensions of $C^r$-generic surface repellers}

In this subsection, we prove the following generic dimension formula..

\begin{thm}
\label{thm:generic_full}
Let $r\in(1,\infty]$. For every $C^r$ surface repeller
$(\Lambda_f,f)$, there exist a $C^r$-open neighborhood
$U_f\subseteq C^r(M,M)$ of $f$ and a subset $V\subseteq U_f$ that is
residual in the $C^r$ topology of $U_f$ such that, for every $g\in V$,
the continuation
$(\Lambda_g,g)$ satisfies
\begin{equation}
\label{eqn:dim formula generic}
\dim_{\mathrm{H}}\Lambda_g
=
\dim_{\mathrm{B}}\Lambda_g
=
\sup_{\mu\in\mathcal{M}_e(\Lambda_g,g)}
\{\dim_{\mathrm{H}}\mu\}
=
s_0(g),
\end{equation}
where $s_0(g)$ is the unique zero of
$s\mapsto P_{\mathrm{sub}}(s,g)$.
\end{thm}
\noindent
\Cref{thm:A} is \eqref{eqn:dim formula generic} without the supremum.

Recall that repellers are $C^1$ structurally stable
\cite{shub_endomorphisms_1969,przytycki_omega_1977}. More precisely, if
$U$ is an isolating neighborhood of a repeller $(\Lambda_f,f)$, then for
every map $g$ in a sufficiently small $C^1$ neighborhood of $f$,
\[
\Lambda_g:=\{x\in U:g^n(x)\in U\text{ for all }n\geq0\}
\]
is a repeller and $g|_{\Lambda_g}$ is topologically conjugate to
$f|_{\Lambda_f}$. We call $(\Lambda_g,g)$ the \emph{continuation} of
$(\Lambda_f,f)$. For each fixed $r\in(1,\infty]$, the intersection of
such a $C^1$ neighborhood with $C^r(M,M)$ is open in the $C^r$ topology.

We first isolate the average-conformal case. An ergodic measure
$\mu\in\mathcal M_e(\Lambda_f,f)$ is called \emph{average conformal} if
\[
\lambda_1(\mu,f)=\lambda_2(\mu,f).
\]
By the Ledrappier--Young dimension formula \cite[Theorem 2.7]{QianXie} and the exactness of the dimension of invariant measures \cite[Theorem 2.8]{QianXie}, we have the following.

\begin{lem}\label{lem:average-conformal}
Let $(\Lambda_f,f)$ be a $C^{1+\alpha}$ surface repeller and
$\mu\in\mathcal{M}_e(\Lambda_f,f)$. If $\mu$ is average conformal, then
\[
\dim_{\mathrm H}\mu=\dim_{\mathrm{LY}}(\mu,f).
\]
\end{lem}

\begin{lem}
\label{lem:average-conformal-sse}
Let $(\Lambda_f,f)$ be a $C^{1+\alpha}$ surface repeller. Suppose that
it admits an SSE measure $\mu$ which is average conformal. Then
\[
\dim_{\mathrm H}\Lambda_f
=
\dim_{\mathrm B}\Lambda_f
=
\sup_{\nu\in\mathcal M_e(\Lambda_f,f)}\{\dim_{\mathrm H}\nu\}
=
s_0(f).
\]
\end{lem}

\begin{proof}
By \Cref{prop zhang upperbound},
\begin{align}
\sup_{\nu\in\mathcal{M}_e(\Lambda_f,f)}
\{\dim_{\mathrm H}\nu\}
&\leq \dim_{\mathrm H}\Lambda_f
\leq \underline{\dim}_{\mathrm B}\Lambda_f
\leq \overline{\dim}_{\mathrm B}\Lambda_f
\leq s_0(f).
\end{align}
On the other hand, \Cref{lem:1.1}, \Cref{prop:sse-max-lyap}, and \Cref{lem:average-conformal}
give
\begin{align}
s_0(f)
&=
\sup_{\nu\in\mathcal{M}_e(\Lambda_f,f)}
\{\dim_{\mathrm{LY}}(\nu,f)\}
=
\dim_{\mathrm{LY}}(\mu,f)
=
\dim_{\mathrm H}\mu
\leq
\sup_{\nu\in\mathcal{M}_e(\Lambda_f,f)}
\{\dim_{\mathrm H}\nu\}.
\end{align}
Thus all the inequalities above are equalities.
\end{proof}

\begin{proof}[Proof of \Cref{thm:generic_full}]
Fix $r\in(1,\infty]$ and choose
\[
0<\alpha_r<\min\{r-1,1\}.
\]
Every map considered below
is therefore of class $C^{1+\alpha_r}$, so all the results from the
previous subsections apply with the exponent $\alpha_r$.

Let $U_f\subseteq C^r(M,M)$ be a sufficiently small $C^r$-open
neighborhood of $f$ contained in a $C^1$ structural-stability
neighborhood. Thus every $g\in U_f$ admits the continuation
$(\Lambda_g,g)$, and this continuation admits an SSE measure. Define
\begin{align}
U_{f,1}
&:=
\{g\in U_f:\text{ some SSE measure of }(\Lambda_g,g)
\text{ is average conformal or has zero metric entropy}\},
\\
U_{f,2}&:=U_f\setminus U_{f,1}.
\end{align}

For every $g\in U_{f,1}$, the desired conclusion follows from
\Cref{lem:average-conformal-sse} if the corresponding SSE measure is
average conformal. If an SSE measure $\mu$ has zero entropy, then
$\dim_{\mathrm{LY}}(\mu,g)=0$; hence
\Cref{lem:1.1,prop:sse-max-lyap} give $s_0(g)=0$, and
\Cref{prop zhang upperbound} forces every quantity in
\eqref{eqn:dim formula generic} to be zero.

We now prove a local generic approximation statement on $U_{f,2}$.

\begin{claim}
\label{claim:generic-local}
For every $\epsilon>0$ and every $g\in U_{f,2}$, there exist a
$C^r$-open neighborhood $U_g(\epsilon)\subseteq U_f$ of $g$ and a
$C^r$-open dense subset
$V_g(\epsilon)\subseteq U_g(\epsilon)$ such that, for every
$h\in V_g(\epsilon)$,
\begin{equation}
\label{eqn:dim formula  3 epsilon}
\sup_{\nu\in\mathcal{M}_e(\Lambda_h,h)}
\{\dim_{\mathrm{LY}}(\nu,h)\}-3\epsilon
\leq
\dim_{\mathrm H}\Lambda_h
\leq
\sup_{\nu\in\mathcal{M}_e(\Lambda_h,h)}
\{\dim_{\mathrm H}\nu\}+3\epsilon.
\end{equation}
\end{claim}

\begin{proof}[Proof of the claim]
Fix $g\in U_{f,2}$ and choose an SSE measure $\mu_g$ of
$(\Lambda_g,g)$. By the definition of $U_{f,2}$,
\[
h_{\mu_g}(g)>0,
\qquad
\lambda_1(\mu_g,g)>\lambda_2(\mu_g,g)>0.
\]
Choose the approximation parameter in \Cref{prop: cao approxi} so that
the error in \Cref{rem 2.2} is at most $\epsilon$ and the resulting
horseshoe has positive topological entropy. We obtain a horseshoe
$(Q_\epsilon,g^N)$ whose measure of maximal entropy $\nu_g$ satisfies
\[
\dim_{\mathrm{LY}}(\nu_g,g^N)
\geq
\dim_{\mathrm{LY}}(\mu_g,g)-\epsilon
=
s_0(g)-\epsilon,
\]
where the last equality follows from
\Cref{lem:1.1,prop:sse-max-lyap}.

Shrink a $C^r$-open neighborhood $U_g(\epsilon)$ of $g$ so that
$(Q_\epsilon,g^N)$ has a continuation
$(Q_\epsilon(h),h^N)$ for every $h\in U_g(\epsilon)$ and its dominated
splitting persists. Let $\nu_h$ be the measure of maximal entropy of
$(Q_\epsilon(h),h^N)$. The conjugacy between the horseshoes sends
$\nu_g$ to $\nu_h$, while the dominated bundles and the derivative
cocycles vary continuously in the $C^1$ topology. After shrinking
$U_g(\epsilon)$ if necessary, we therefore have
\[
\dim_{\mathrm{LY}}(\nu_h,h^N)
\geq
\dim_{\mathrm{LY}}(\nu_g,g^N)-\epsilon,
\qquad
h\in U_g(\epsilon).
\]

Let
\[
\widetilde Q_\epsilon
:=
\bigcup_{j=0}^{N-1}g^j(Q_\epsilon)
\]
be the cyclic \(g\)-repeller generated by the \(g^N\)-horseshoe
\(Q_\epsilon\). Since \((Q_\epsilon,g^N)\) is a full-shift horseshoe,
\((\widetilde Q_\epsilon,g)\) is transitive and non-invertible.
Applying \Cref{appendix:B} to this cyclic repeller, we obtain a
\(C^r\)-open dense set $V_g(\epsilon)$ in $U_g(\epsilon)$ for which
\((\widetilde Q_\epsilon(h),h)\) is \(su\)-non-integrable. Sampling two
histories with distinct strong unstable directions at the times
\(0,-N,-2N,\ldots\) shows that the \(h^N\)-component
\((Q_\epsilon(h),h^N)\) is also \(su\)-non-integrable.
Hence
\Cref{prop: irre-horseshoe,lem:1.1} give
\[
\dim_{\mathrm H}Q_\epsilon(h)
=
\dim_{\mathrm B}Q_\epsilon(h)
=
\sup_{\nu\in\mathcal{M}_e(Q_\epsilon(h),h^N)}
\{\dim_{\mathrm H}\nu\}
=
\sup_{\nu\in\mathcal{M}_e(Q_\epsilon(h),h^N)}
\{\dim_{\mathrm{LY}}(\nu,h^N)\}.
\]
Since $Q_\epsilon(h)\subseteq\Lambda_h$, it follows that
\begin{align}
\dim_{\mathrm H}\Lambda_h
&\geq
\dim_{\mathrm H}Q_\epsilon(h)
\geq
\dim_{\mathrm{LY}}(\nu_h,h^N)
\geq
s_0(g)-2\epsilon.
\end{align}

We next compare $s_0(g)$ and $s_0(h)$. By
\cite[Theorem~3.5]{cao_dimension_2019}, for every fixed $s\in[0,2]$ the
sub-additive pressure is continuous within the
class of $C^{1+\alpha_r}$ expanding maps. Applying this result to the
finitely many continued transitive components and taking their maximum,
we conclude that
\[
h\longmapsto P_{\mathrm{sub}}(s,h)
\]
is continuous on $U_g(\epsilon)$ in the $C^r$ topology. Since
$s\mapsto P_{\mathrm{sub}}(s,h)$ is continuous and strictly decreasing
with a unique zero, the zero $h\mapsto s_0(h)$ is continuous as well.
Shrinking $U_g(\epsilon)$ once more, we may assume that
\[
|s_0(h)-s_0(g)|<\epsilon,
\qquad
h\in U_g(\epsilon).
\]
Consequently, for every $h\in V_g(\epsilon)$,
\[
\dim_{\mathrm H}\Lambda_h
\geq
s_0(h)-3\epsilon
=
\sup_{\nu\in\mathcal{M}_e(\Lambda_h,h)}
\{\dim_{\mathrm{LY}}(\nu,h)\}-3\epsilon,
\]
where the equality follows from \Cref{lem:1.1}. This is the left-hand
inequality in \eqref{eqn:dim formula  3 epsilon}.

For the opposite inequality, \Cref{prop zhang upperbound} gives
\[
\dim_{\mathrm H}\Lambda_h\leq s_0(h).
\]
Moreover,
\begin{align}
s_0(h)
&\leq s_0(g)+\epsilon
\leq \dim_{\mathrm{LY}}(\nu_g,g^N)+2\epsilon %
\leq \dim_{\mathrm{LY}}(\nu_h,h^N)+3\epsilon\\
&\leq
\sup_{\nu\in\mathcal M_e(Q_\epsilon(h),h^N)}
\{\dim_{\mathrm H}\nu\}+3\epsilon.
\end{align}
For every
$\overline\nu\in\mathcal M_e(Q_\epsilon(h),h^N)$, the measure
\[
\nu:=\frac1N\sum_{j=0}^{N-1}h_*^j\overline\nu
\]
is $h$-ergodic, is supported on $\Lambda_h$, and satisfies
$\dim_{\mathrm H}\nu=\dim_{\mathrm H}\overline\nu$, as in
\Cref{lem:bernoulli-dim-equality}. Therefore,
\[
\sup_{\nu\in\mathcal M_e(Q_\epsilon(h),h^N)}
\{\dim_{\mathrm H}\nu\}
\leq
\sup_{\nu\in\mathcal M_e(\Lambda_h,h)}
\{\dim_{\mathrm H}\nu\}.
\]
Combining the last three displays gives the right-hand inequality in
\eqref{eqn:dim formula  3 epsilon}, and proves the claim.
\end{proof}

For $\epsilon>0$, let
\[
W(\epsilon)
:=
\left\{
h\in U_f:
\eqref{eqn:dim formula  3 epsilon}
\text{ holds for }(\Lambda_h,h)
\right\},
\]
and let $O(\epsilon)$ be the interior of $W(\epsilon)$ in the $C^r$
topology of $U_f$. We claim that $O(\epsilon)$ is dense in $U_f$.
Indeed, let $B\subseteq U_f$ be a nonempty $C^r$-open set. If
$B\cap U_{f,2}\neq\emptyset$, choose $g\in B\cap U_{f,2}$. The set
$B\cap U_g(\epsilon)$ is nonempty and open, so the density of
$V_g(\epsilon)$ in $U_g(\epsilon)$ gives
\[
\emptyset\neq B\cap V_g(\epsilon)\subseteq O(\epsilon).
\]
If $B\cap U_{f,2}=\emptyset$, then $B\subseteq U_{f,1}\subseteq
W(\epsilon)$, and hence $B\subseteq O(\epsilon)$. Thus
$O(\epsilon)$ is $C^r$-open and dense in $U_f$.

Set
\[
V:=\bigcap_{n=1}^{\infty}O(1/n).
\]
Then $V$ is $C^r$-residual in $U_f$. For every $g\in V$, letting
$n\to\infty$ in \eqref{eqn:dim formula  3 epsilon} and using
\Cref{lem:1.1} gives
\[
s_0(g)
=
\sup_{\nu\in\mathcal M_e(\Lambda_g,g)}
\{\dim_{\mathrm{LY}}(\nu,g)\}
\leq
\dim_{\mathrm H}\Lambda_g
\leq
\sup_{\nu\in\mathcal M_e(\Lambda_g,g)}
\{\dim_{\mathrm H}\nu\}
\leq
\dim_{\mathrm H}\Lambda_g.
\]
Hence
\[
\dim_{\mathrm H}\Lambda_g
=
\sup_{\nu\in\mathcal M_e(\Lambda_g,g)}
\{\dim_{\mathrm H}\nu\}
=
s_0(g).
\]
Finally, \Cref{prop zhang upperbound} yields
\[
s_0(g)
=
\dim_{\mathrm H}\Lambda_g
\leq
\underline{\dim}_{\mathrm B}\Lambda_g
\leq
\overline{\dim}_{\mathrm B}\Lambda_g
\leq
s_0(g).
\]
Thus the box dimension exists and equals $s_0(g)$, completing the proof.
\end{proof}

\subsection{Polynomial counterexample}\label{sec:counterexample}

We present a polynomial example showing that the generic condition or the $su$-non-integrability assumption cannot be removed. Consider the square $[-1,1]^2\subseteq\mathbb R^2$ and the polynomial map
\[
f(x,y)=(1.1x,-200y^2+100).
\]
Let $(\Lambda_f,f)$ be the repeller determined by the isolating neighborhood $[-1,1]^2$. In this example, the dynamics splits as a product of a weak expanding direction and a one-dimensional expanding Cantor dynamics. The strong unstable direction is integrable, and one checks that the Hausdorff dimension of $\Lambda_f$ attains the lower bound in \Cref{prop: general dim estimates} rather than the upper bound $s_0$. Hence this gives a counterexample to the conclusions of \Cref{thm:A} and \Cref{thm:B} without the generic or $su$-non-integrability assumption.

\subsection{Dimension lower bound for arbitrary repellers}

In this section, we give Hausdorff dimension estimates for arbitrary $C^{1+\alpha}$ repellers. Let $f$ be a $C^{1+\alpha}$ map of a Riemannian manifold $M$ of dimension $d$, and let $\Lambda_f$ be a repeller of $f$.

We first introduce the super-additive singular pressure $P_{\mathrm{sup}}(s)$. For $s\in[0,d]$, define
\[
\psi^s_n(x)
=
\sum_{i=1}^{[s]}\log\alpha_i(x,f^n)
+
(s-[s])
\log\alpha_{[s]+1}(x,f^n),
\]
where the singular values $\alpha_i(x,f^n)$ are those fixed in \Cref{sec:2.1}. The sequence $\{-\psi_n^s\}_{n\geq1}$ is super-additive, and its variational pressure is
\[
P_{\mathrm{sup}}(s)
=
\sup_{\mu\in\mathcal{M}_e(\Lambda_f,f)}
\left\{
h_\mu(f)
-
\lim_{n\to\infty}\frac1n
\int\psi^s_n\,\mathrm d\mu
\right\};
\]
see \cite{cao_dimension_2019}. Equivalently, if
$\lambda_1(\mu,f)\geq\cdots\geq\lambda_d(\mu,f)>0$
are the Lyapunov exponents of an ergodic measure $\mu$, then
\[
P_{\mathrm{sup}}(s)
=
\sup_{\mu\in\mathcal{M}_e(\Lambda_f,f)}
\left\{
h_\mu(f)
-
\sum_{i=1}^{[s]}\lambda_i(\mu,f)
-
(s-[s])
\lambda_{[s]+1}(\mu,f)
\right\}.
\]
Let $s_0'$ be the unique zero of $P_{\mathrm{sup}}(s)$.

{
\begin{prop}\label{prop: general dim estimates}
For any $C^{1+\alpha}$ repeller $(\Lambda_f,f)$,
\[
s_0'
\leq
\sup_{\nu\in\mathcal M_e(\Lambda_f,f)}
\{\dim_{\mathrm H}\nu\}
\leq
\dim_{\mathrm H}\Lambda_f
\leq
\underline{\dim}_{\mathrm B}\Lambda_f
\leq
\overline{\dim}_{\mathrm B}\Lambda_f
\leq
s_0,
\]
where $s_0$ is the unique zero of $P_{\mathrm{sub}}(s)$ and $s_0'$ is the unique zero of $P_{\mathrm{sup}}(s)$.
\begin{proof}
The second inequality follows from the definition of the Hausdorff dimension of a measure recalled in \Cref{sec:2.1}. The remaining upper bounds are given by \Cref{prop zhang upperbound}. It therefore remains to prove
\[
s_0'
\leq
\sup_{\nu\in\mathcal M_e(\Lambda_f,f)}
\{\dim_{\mathrm H}\nu\}.
\]

For an ergodic measure $\nu$, define its lower Lyapunov dimension
$\dim_{\mathrm{LLY}}(\nu,f)$ as the unique zero of
\[
G_\nu(t)
:=
h_\nu(f)
-
\sum_{i=1}^{\lfloor t\rfloor}\lambda_i(\nu,f)
-
(t-\lfloor t\rfloor)
\lambda_{\lfloor t\rfloor+1}(\nu,f).
\]
By the Ledrappier--Young entropy formula for repellers,
\[
\dim_{\mathrm{LLY}}(\nu,f)
\leq
\dim_{\mathrm H}\nu;
\]
see \cite{QianXie}. Consequently,
\[
\sup_{\nu\in\mathcal M_e(\Lambda_f,f)}
\{\dim_{\mathrm{LLY}}(\nu,f)\}
\leq
\sup_{\nu\in\mathcal M_e(\Lambda_f,f)}
\{\dim_{\mathrm H}\nu\}.
\]

It remains to identify $s_0'$ with the supremum of the lower Lyapunov dimensions. The proof is the same monotonicity argument as in \Cref{lem:1.1}: each $G_\nu$ is continuous and strictly decreasing, and
\[
P_{\mathrm{sup}}(s)
=
\sup_{\nu\in\mathcal M_e(\Lambda_f,f)}
G_\nu(s).
\]
Hence
\[
s_0'
=
\sup_{\nu\in\mathcal M_e(\Lambda_f,f)}
\{\dim_{\mathrm{LLY}}(\nu,f)\}.
\]
Combining the preceding two displays proves the result.
\end{proof}
\end{prop}
}

\begin{rem}
\Cref{prop: general dim estimates} recovers the lower estimate in \cite[Theorem~3.1]{cao_dimension_2019}, namely $s_0'\leq \dim_{\mathrm H}\Lambda_f$. In light of this proposition, \Cref{thm:dominated} and \Cref{thm:generic_full} show that, for $su$-non-integrable surface repellers and for generic surface repellers, the upper estimate given by the sub-additive topological pressure coincides with the lower estimate given by the mass distribution principle.
\end{rem}

\section{Proofs of results on graphs of Weierstrass-type functions and non-linear IFS}\label{sec: IFS and Weierstrass}
\subsection{Dimensions of graphs of Weierstrass-type functions}
\label{sec:weierstrass-graph-dimension}

In this subsection, we prove \Cref{thm:ren-shen-c}. We follow the notation of
Ren and Shen \cite{ren_dichotomy_2021}. Let \(b\geq 2\) be an integer,
\(\lambda\in(1/b,1)\), and let \(\phi\in C^r(\mathbb T)\), where
\(r\in(1,\infty]\cup\{\omega\}\) and \(\mathbb T=\mathbb R/\mathbb Z\).
Set
\[
    W(x)=W_{\lambda,b}^{\phi}(x)
    :=
    \sum_{n=0}^{\infty}\lambda^n\phi(b^n x).
\]
Then \(W\) satisfies the functional equation
\begin{equation}\label{eq:weierstrass-cohomological-equation}
    W(x)=\phi(x)+\lambda W(bx).
\end{equation}

Consider the skew-product map
\[
    F:\mathbb T\times\mathbb R\to \mathbb T\times\mathbb R,
    \qquad
    F(x,y)=\left(bx,\frac{y-\phi(x)}{\lambda}\right).
\]
By \eqref{eq:weierstrass-cohomological-equation}, the graph
\[
    \Gamma_W:=\{(x,W(x)):x\in\mathbb T\}
\]
is \(F\)-invariant. Moreover, the projection
\[
    \pi_x:\Gamma_W\to\mathbb T,
    \qquad
    \pi_x(x,W(x))=x,
\]
conjugates \(F|_{\Gamma_W}\) to the expanding map \(T(x)=bx\):
\[
    \pi_x\circ F|_{\Gamma_W}=T\circ\pi_x.
\]
Therefore
\begin{equation}\label{eq:entropy-weierstrass}
    h_{\mathrm{top}}(F|_{\Gamma_W})=h_{\mathrm{top}}(T)=\log b .
\end{equation}

\paragraph{We first verify that \((\Gamma_W,F)\) is a repeller.} Choose \(C>0\) such that
\[
    |W(x)|<C
    \qquad \text{for every }x\in\mathbb T,
\]
and set
\[
    U:=\mathbb T\times(-C,C).
\]
For any \((x,y)\in\mathbb T\times\mathbb R\), using
\eqref{eq:weierstrass-cohomological-equation}, we have
\[
    F(x,y)-F(x,W(x))
    =
    \left(0,\frac{y-W(x)}{\lambda}\right).
\]
Thus the vertical distance to the graph is multiplied by \(\lambda^{-1}>1\)
under one forward iterate. Hence if \((x,y)\notin\Gamma_W\), then
\[
    \left|\pi_y(F^n(x,y))-W(b^n x)\right|
    =
    \lambda^{-n}|y-W(x)|
    \to\infty.
\]
Since \(W\) is bounded, every point of \(U\setminus\Gamma_W\) eventually
leaves \(U\). Consequently,
\[
    \Gamma_W
    =
    \{z\in U:F^n(z)\in U\text{ for every }n\geq 0\}.
\]
The non-wandering condition follows from the conjugacy with \(T\).

It remains to verify uniform expansion. Along \(\Gamma_W\), we have
\[
    D_{(x,y)}F
    =
    \begin{pmatrix}
        b & 0\\
        -\lambda^{-1}\phi'(x) & \lambda^{-1}
    \end{pmatrix}.
\]
Since
\[
    \min\{b,\lambda^{-1}\}>1,
\]
choose \(\kappa\in(1,\min\{b,\lambda^{-1}\})\). For \(A>0\), consider the
equivalent Riemannian norm
\[
    \|(u,v)\|_A:=\sqrt{A^2u^2+v^2}.
\]
In the coordinates \((U,V)=(Au,v)\), the matrix of \(DF\) becomes
\[
    \begin{pmatrix}
        b & 0\\
        -A^{-1}\lambda^{-1}\phi'(x) & \lambda^{-1}
    \end{pmatrix}.
\]
As \(A\to\infty\), these matrices converge uniformly to the diagonal matrix
\[
    \begin{pmatrix}
        b & 0\\
        0 & \lambda^{-1}
    \end{pmatrix},
\]
whose smallest singular value is \(\min\{b,\lambda^{-1}\}>\kappa\). Hence,
for \(A\) sufficiently large, the smallest singular value of \(DF\) with
respect to \(\|\cdot\|_A\) is uniformly larger than \(\kappa\). Therefore
\[
    \|D_zF(v)\|_A\geq \kappa\|v\|_A
    \qquad
    \text{for all }z\in\Gamma_W,\ v\in T_z(\mathbb T\times\mathbb R).
\]
Thus \(F\) is uniformly expanding on \(\Gamma_W\), and consequently
\((\Gamma_W,F)\) is a surface repeller generated by the \(C^r\) map \(F\).

\paragraph{We now identify the dominated splitting with the slope functions %
in \cite{ren_dichotomy_2021}.} Let
\[
    \hat x=(\ldots,x_{-2},x_{-1},x_0)
\]
be a prehistory for the base map \(T(x)=bx\). The weak unstable direction is
the vertical direction
\[
    E^{wu}(x_0,W(x_0))=\mathbb R(0,1),
\]
with expansion rate \(\lambda^{-1}\). The strong unstable direction is the
direction expanded by \(b\). Indeed,
\[
    D_{(x,y)}F
    =
    \begin{pmatrix}
        b & 0\\
        -\lambda^{-1}\phi'(x) & \lambda^{-1}
    \end{pmatrix}.
\]
If the strong direction has slope \(s\) at \((x,W(x))\), then its image has
slope
\[
    s_+
    =
    \frac{-\lambda^{-1}\phi'(x)+\lambda^{-1}s}{b}
    =
    -\gamma\phi'(x)+\gamma s,
    \qquad
    \gamma:=\frac{1}{b\lambda}.
\]
Iterating this relation along the backward orbit gives
\[
    E^{su}(\hat x)=\mathbb R(1,Y(\hat x)),
\]
where
\begin{equation}\label{eq:ren-shen-Y}
    Y(\hat x)
    =
    -\sum_{n=1}^{\infty}\gamma^n\phi'(x_{-n}).
\end{equation}
This is precisely the slope function used by Ren and Shen. Since
\[
    b>\lambda^{-1}>1,
\]
the expansion along \(E^{su}\) dominates the expansion along \(E^{wu}\).
Therefore
\[
    T_{\pi(\hat z)}(\mathbb T\times\mathbb R)
    =
    E^{wu}(\pi(\hat z))\oplus E^{su}(\hat z)
\]
is a dominated splitting over the inverse limit of \((\Gamma_W,F)\).

We split the proof according to whether the slopes \eqref{eq:ren-shen-Y}
depend on the prehistory.

\paragraph{First suppose that the slopes depend on the prehistory.} Then
there exist \(x\in\mathbb T\) and two prehistories \(\hat x,\hat x'\) over
\(x\) such that
\[
    Y(\hat x)\neq Y(\hat x').
\]
Equivalently, for the corresponding prehistories over the same point of
\(\Gamma_W\),
\[
    E^{su}(\hat x)\neq E^{su}(\hat x').
\]
Thus the repeller \((\Gamma_W,F)\) is \(su\)-non-integrable in the sense of
\Cref{sec:2.1}. By \Cref{thm:dominated},
\[
    \dim_{\mathrm H}\Gamma_W=s_0,
\]
where \(s_0\) is the unique zero of the sub-additive pressure
\(P_{\mathrm{sub}}(s)\).

We compute this zero explicitly. For every ergodic
\(F|_{\Gamma_W}\)-invariant measure \(\mu\), the Lyapunov exponents are
constant:
\[
    \lambda_1(\mu,F)=\log b,
    \qquad
    \lambda_2(\mu,F)=-\log\lambda .
\]
Indeed, \(DF^n\) is lower triangular with diagonal entries \(b^n\) and
\(\lambda^{-n}\). Since \(b>\lambda^{-1}\), the off-diagonal term grows at
most at rate \(b^n\), and the determinant is \((b/\lambda)^n\). Hence the two
singular-value exponents are \(\log b\) and \(-\log\lambda\). Moreover, by
\eqref{eq:entropy-weierstrass},
\[
    \sup_{\mu\in\mathcal M_e(\Gamma_W,F)}h_\mu(F)=\log b.
\]
Since \(\lambda\in(1/b,1)\), the desired zero lies in \((1,2)\). For
\(s\in[1,2]\), the variational formula for the sub-additive pressure gives
\[
\begin{aligned}
    P_{\mathrm{sub}}(s)
    &=
    \sup_{\mu\in\mathcal M_e(\Gamma_W,F)}
    \left\{
        h_\mu(F)-\lambda_2(\mu,F)
        -(s-1)\lambda_1(\mu,F)
    \right\}  \\
    &=
    \log b+\log\lambda-(s-1)\log b .
\end{aligned}
\]
Therefore the unique zero is
\[
    s_0
    =
    2+\frac{\log\lambda}{\log b}
    =
    2+\log_b\lambda .
\]
Hence, in the non-integrable case, we can apply \cref{thm:dominated} to obtain
\[
    \dim_{\mathrm H}\Gamma_W=2+\log_b\lambda .
\]

\paragraph{It remains to treat the integrable case. }In the terminology of
Ren and Shen \cite{ren_dichotomy_2021}, this is precisely Condition \(H^*\): the slope function
\(Y\) is independent of the prehistory. Therefore there exists a function
\(u:\mathbb T\to\mathbb R\) such that
\[
    Y(\hat x)=u(x_0)
\]
for every prehistory \(\hat x\) with a terminal point \(x_0\).

Assume first that \(r<\infty\). Write
\[
    r-1=k+\alpha,
\]
where \(k\in\mathbb N\cup\{0\}\) and \(\alpha\in[0,1)\), with the usual
convention that \(\alpha=0\) corresponds to an integer regularity class.
Since \(\phi\in C^r\), we have \(\phi'\in C^{r-1}\). For each inverse branch
\(\tau_{\mathbf i}^n\) of \(T^n\), the \(n\)-th term in the series defining
\(Y\) has the form
\[
    \gamma^n\phi'\circ \tau_{\mathbf i}^n .
\]
For every \(0\leq \ell\leq k\),
\[
    \left\|D^\ell(\phi'\circ \tau_{\mathbf i}^n)\right\|_{C^0}
    \leq
    b^{-n\ell}\|D^\ell\phi'\|_{C^0},
\]
and, if \(\alpha>0\),
\[
    [D^k(\phi'\circ \tau_{\mathbf i}^n)]_{\alpha}
    \leq
    b^{-n(k+\alpha)}[D^k\phi']_{\alpha}.
\]
Multiplying by \(\gamma^n\) and summing over \(n\), we obtain uniform
convergence in the \(C^{r-1}\) norm, uniformly over all prehistories.
Therefore the prehistory-independent function \(u\) belongs to
\(C^{r-1}(\mathbb T)\). For \(r=\infty\), the same argument applies in every
finite \(C^k\) norm.

The invariance of the slope gives
\[
    u(bx)
    =
    -\gamma\phi'(x)+\gamma u(x),
    \qquad
    \gamma=\frac{1}{b\lambda}.
\]
Equivalently,
\begin{equation}\label{eq:u-cohomological}
    u(x)=\phi'(x)+\lambda b\cdot u(bx).
\end{equation}
Integrating \eqref{eq:u-cohomological} over \(\mathbb T\), we get (using the fact the integral of $\phi'$ vanishes)
\[
    \int_{\mathbb T}u(x)\,dx
    =
    \lambda b\int_{\mathbb T}u(bx)\,dx
    =
    \lambda b\int_{\mathbb T}u(x)\,dx .
\]
Since \(\lambda b\neq 1\), it follows that
\[
    \int_{\mathbb T}u(x)\,dx=0.
\]
Thus \(u\) admits a periodic primitive \(G_0\), namely \[ G_0'(x)=u(x). \] By \eqref{eq:u-cohomological}, the
function
\[
    G_0(x)-\lambda G_0(bx)-\phi(x)
\]
has zero derivative, and hence is constant. After adding a suitable constant
to \(G_0\), we obtain a function \(G\in C^r(\mathbb T)\) satisfying
\[
    G(x)=\phi(x)+\lambda G(bx).
\]
Comparing this equation with \eqref{eq:weierstrass-cohomological-equation},
we see that \(H:=W-G\) satisfies
\[
    H(x)=\lambda H(bx).
\]
Iterating gives \(H(x)=\lambda^n H(b^n x)\). Since \(H\) is bounded on
\(\mathbb T\) and \(0<\lambda<1\), letting \(n\to\infty\) yields
\(H\equiv0\). Therefore \(W=G\in C^r(\mathbb T)\).

For the analytic case \(r=\omega\), Condition \(H^*\) implies that
\(W\) is real analytic by \cite{ren_dichotomy_2021}. Hence the
same regularity conclusion also holds in the analytic category.

Combining the two cases, either
\[
    \dim_{\mathrm H}\operatorname{graph}(W)=2+\log_b\lambda,
\qquad\text{or}\qquad
    W\in C^r(\mathbb T).
\]
Finally, assume that $\phi$ is non-constant and fix $b\geq2$. Following the final argument in
\cite[p.~1061]{ren_dichotomy_2021}, it can be proved that
the exceptional set
\[
\mathcal E_{\phi,b}^{r}
:=
\left\{
\lambda\in(1/b,1):
W_{\lambda,b}^{\phi}\in C^r(\mathbb T)
\right\}
\]
is finite.
This proves \Cref{thm:ren-shen-c}.

\subsection{Dimensions of non-linear iterated function systems}
\label{sec:proof-ifs-non-perturbative}

In this subsection we prove \Cref{thm:ifs-non-perturbative} and \Cref{thm:ifs}. The proof is based
on three correspondences: an IFS satisfying SSC gives a horseshoe repeller;
uniform non-conformality gives a dominated splitting; and weak irreducibility of
the IFS is exactly the \(su\)-non-integrability condition for the associated
repeller.

Let
\[
    \Phi=\{\phi_i\}_{i=1}^m
\]
be a \(C^{1+\alpha}\) planar contracting IFS on a bounded open domain
\(U\subset\mathbb R^2\), and let \(\Lambda_\Phi\) be its attractor. Assume
that \(\Phi\) satisfies the strong separation condition
\[
    \phi_i(\Lambda_\Phi)\cap \phi_j(\Lambda_\Phi)=\emptyset,
    \qquad i\neq j.
\]

We first realize \(\Lambda_\Phi\) as a horseshoe repeller. After shrinking
neighborhoods if necessary, choose pairwise disjoint open sets \(U_i\) such
that
\[
    \phi_i(\Lambda_\Phi)\subset U_i
\]
and such that each inverse map \(\phi_i^{-1}\) is well-defined and expanding
on \(U_i\). Put
\[
    V:=\bigcup_{i=1}^m U_i
\]
and define
\[
    f:V\to U,
    \qquad
    f|_{U_i}:=\phi_i^{-1}.
\]
Then \(f(\Lambda_\Phi)=\Lambda_\Phi\). Moreover, one has
\[
    \Lambda_\Phi
    =
    \{x\in V:f^n(x)\in V\text{ for every }n\geq0\}.
\]
Furthermore, by SSC, the coding map
\[
    \pi:\Sigma_m^+\to\Lambda_\Phi,
    \qquad
    \pi(i_0i_1i_2\ldots)
    =
    \lim_{n\to\infty}
    \phi_{i_0}\circ\phi_{i_1}\circ\cdots\circ\phi_{i_n}(x_*)
\]
is one-to-one, where \(x_*\in U\) is arbitrary, and it satisfies
\[
    f\circ\pi=\pi\circ\sigma.
\]
Therefore \((\Lambda_\Phi,f)\) is a horseshoe repeller.

We next prove that uniform non-conformality gives a dominated splitting for
this repeller. Let \(\hat\Lambda_\Phi\) be the inverse limit space of
\((\Lambda_\Phi,f)\), and let \(\hat f\) be the natural extension. Consider
the normalized derivative cocycle over the homeomorphism
\(\hat f:\hat\Lambda_\Phi\to\hat\Lambda_\Phi\):
\[
    A(\hat x)
    :=
    \frac{Df_{\pi(\hat x)}}{\sqrt{|\det Df_{\pi(\hat x)}|}} .
\]
For \(n\geq1\),
\[
    A^{(n)}(\hat x)
    :=
    A(\hat f^{n-1}\hat x)\cdots A(\hat f\hat x)A(\hat x)
    =
    \frac{Df^n_{\pi(\hat x)}}{\sqrt{|\det Df^n_{\pi(\hat x)}|}} .
\]
The normalization is scalar, so it does not change projective directions.

The uniform non-conformality of the IFS is equivalent to uniform
hyperbolicity of this normalized \(Df\)-cocycle. Indeed, if the forward orbit
of \(x=\pi(\hat x)\) follows the word \(\mathbf i=(i_0,\ldots,i_{n-1})\), then
the inverse branch of \(f^n\) sending \(f^n(x)\) to \(x\) is
\[
    \Phi_{\mathbf i}^n
    =
    \phi_{i_0}\circ\cdots\circ\phi_{i_{n-1}},
\]
and
\[
    Df^n_x
    =
    \left(D\Phi_{\mathbf i}^n(f^n x)\right)^{-1}.
\]
After normalization by \(\sqrt{|\det|}\), taking the inverse does not change
the operator norm in dimension two. Hence the uniform non-conformality
condition
\[
    \frac{\|D\Phi_{\mathbf i}^n(y)\|}
         {\sqrt{|\det D\Phi_{\mathbf i}^n(y)|}}
    \geq C e^{\epsilon n}
\]
is equivalent to
\[
    \|A^{(n)}(\hat x)\|\geq C e^{\epsilon n}
    \qquad
    \text{for every }\hat x\in\hat\Lambda_\Phi,\ n\geq1.
\]
Thus the normalized \(Df\)-cocycle is uniformly hyperbolic in the sense of
\cite[Section 2.1]{avila_bochi_trieste_2008}. 

By \cite[Theorem 2.3]{avila_bochi_trieste_2008}, the uniformly hyperbolic
normalized cocycle has a unique continuous invariant stable direction. Since
the base \((\hat\Lambda_\Phi,\hat f)\) is a homeomorphism,
\cite[Corollary 2.5]{avila_bochi_trieste_2008} gives a continuous invariant
splitting
\[
    T_{\pi(\hat x)}\mathbb R^2
    =
    E^u_A(\hat x)\oplus E^s_A(\hat x)
\]
for the normalized \(Df\)-cocycle.

We now identify this splitting with the dominated splitting required in
\Cref{def:dominated splitting}. The stable direction \(E^s_A\) is the
direction with the smaller expansion rate for \(Df\). Since
\(A^{(n)}(\hat x)\) depends only on the forward orbit of \(x=\pi(\hat x)\),
the stable direction obtained in \cite[Theorem 2.3]{avila_bochi_trieste_2008}
is also determined only by this forward orbit. For the repeller, the forward
orbit is uniquely determined by \(x\). Hence \(E^s_A(\hat x)\) depends only on
\(\pi(\hat x)\), and we define
\[
    E^{wu}(\pi(\hat x)):=E^s_A(\hat x).
\]
The unstable direction \(E^u_A(\hat x)\), on the other hand, is determined by
the prehistory. We define
\[
    E^{su}(\hat x):=E^u_A(\hat x).
\]
Thus
\[
    T_{\pi(\hat x)}\mathbb R^2
    =
    E^{wu}(\pi(\hat x))\oplus E^{su}(\hat x).
\]
The exponential separation of the uniformly hyperbolic normalized cocycle
gives domination for \(Df\): after increasing the iterate if necessary, there
exists \(n_0\geq1\) such that
\[
    \frac{\|Df^{n_0}_{\pi(\hat x)}|_{E^{su}(\hat x)}\|}
         {\|Df^{n_0}_{\pi(\hat x)}|_{E^{wu}(\pi(\hat x))}\|}
    \geq 2
\]
for every \(\hat x\in\hat\Lambda_\Phi\). Therefore
\((\Lambda_\Phi,f)\) admits a dominated splitting in the sense of
\Cref{def:dominated splitting}.

We now compare weak irreducibility with \(su\)-non-integrability. By the definition
of weak irreducibility for uniformly non-conformal IFSs, the IFS is not weakly irreducible
exactly when the strong direction, equivalently the most expanded direction
for the associated repeller, descends to a line field on the
attractor. Under the above identification, this direction is precisely
\(E^{su}(\hat x)\). Therefore non-weak irreducibility is equivalent to the condition that
for every \(x\in\Lambda_\Phi\),
\[
    \#\left\{
    \mathbb P(E^{su}(\hat x)):\hat x\in\hat\Lambda_\Phi,\,
    \pi(\hat x)=x
    \right\}=1.
\]
This is exactly \(su\)-integrability in \Cref{sec:2.1}. Consequently,
\[
    \Phi\text{ is weakly irreducible}
    \quad\Longleftrightarrow\quad
    (\Lambda_\Phi,f)\text{ is }su\text{-non-integrable}.
\]

We can now apply the repeller results. The system \((\Lambda_\Phi,f)\) is a
\(C^{1+\alpha}\) horseshoe repeller, admits a dominated splitting, and is
\(su\)-non-integrable. Hence \Cref{thm:dominated} gives
\[
    \dim_{\mathrm H}\Lambda_\Phi=s_0(\Phi).
\]

Finally, let \(\nu\) be a fully supported stationary measure on
\(\Lambda_\Phi\) associated with a Bernoulli probability vector
\((p_1,\ldots,p_m)\), with \(p_i>0\) for every \(i\). Since the coding map is
one-to-one under SSC, \(\nu\) is the push-forward of the fully supported
Bernoulli measure on \(\Sigma_m^+\). Viewed through the conjugacy
\(\pi:\Sigma_m^+\to\Lambda_\Phi\), this is precisely a fully supported
\(f\)-Bernoulli measure on the horseshoe repeller. Applying \Cref{thm:base}
yields
\[
    \dim_{\mathrm H}\nu=\dim_{\mathrm{LY}}(\nu,\Phi).
\]
This proves \Cref{thm:ifs-non-perturbative}.

We finally explain why \Cref{thm:ifs} follows from \Cref{thm:generic_full}. By the
first part of the proof, every \(C^r\) planar IFS satisfying SSC can be
identified, through its inverse branches, with a horseshoe repeller, and this
identification is local in the \(C^r\) topology. Thus the \(C^r\)-
generic subset of repellers given by \Cref{thm:generic_full} pulls back to a
\(C^r\)-generic subset in the space of \(C^r\) IFSs satisfying SSC.
Hence \Cref{thm:generic_full} gives
\[
    \dim_{\mathrm H}\Lambda_\Phi=s_0(\Phi)
\]
for a \(C^r\)-generic \(C^r\) planar IFS satisfying SSC, which is
\Cref{thm:ifs}.

\appendix

\section{Linearization of the holonomy projections}
Throughout this section, we work in the following setting. Fix $y\in\Lambda_{n_0}$ and $\bfi,\bfj\in\Sigma^-$; write $\widetilde y:=\exp^{-1}_{x_0}(y)$ and $\Omega:=\exp^{-1}_{x_0}(\Lambda_{n_0})$. For $z\in\Lambda_{n_0}$ and $\bfk\in\Sigma^-$, put $W_\bfk(z):= W^{su}_{\mathrm{loc}}(z,\bfk)$, and write
\[
\widetilde z:=\exp^{-1}_{x_0}(z),\qquad
\widetilde W_\bfk(\widetilde z):=\exp^{-1}_{x_0}\bigl(W_\bfk(z)\bigr)
\]
for their images under the chart; every point of $\Omega$ is of the form $\widetilde z$ with $z\in\Lambda_{n_0}$. For a local transversal $T'$, we write $d_{T'}$ for the restricted Riemannian distance on $T'$.

For $z\in\Lambda_{n_0}$, let $T_z$ be a local transversal through $z$ such that $\exp^{-1}_{x_0}(T_z)$ is a segment of the line through $\widetilde z$ orthogonal to $(D_z\exp^{-1}_{x_0})(E^{su}_\bfi(z))$, and put $T:=T_y$. As explained at the beginning of \Cref{subsec:statement and preparation}, each $T_z$ is a $\theta$-uniform $su$-transversal. Moreover, each $W_\bfk(z')$, $z'\in\Lambda_{n_0}$, $\bfk\in\Sigma^-$ meets $T_z$ in exactly one point. 
We use orthonormal coordinates $w=(w_T,w_N)$ on $\R^2$ centered at $\widetilde y$, with $\exp^{-1}_{x_0}(T)$ on the horizontal axis and $(D_y\exp^{-1}_{x_0})(E^{su}_\bfi(y))$ on the vertical axis; since $T_yW_\bfi(y)=E^{su}_\bfi(y)$ for $y\in\Lambda_{n_0}$,
\[
\exp^{-1}_{x_0}(T)\subseteq\R\times\{0\},\qquad
T_{\widetilde y}\widetilde W_\bfi(\widetilde y)=\{0\}\times\R .
\]
We identify $\exp^{-1}_{x_0}(T)$ with a subset of $\R$ via $w\mapsto w_T$; for $u\in\exp^{-1}_{x_0}(T)$ and $r>0$, $B^T(u,r)$ denotes the ball of radius $r$ about $u$ in $\exp^{-1}_{x_0}(T)$. For $\bfk\in\Sigma^-$ and $\widetilde z\in\Omega$, let $\pi_\bfk(\widetilde z)$ be the unique point of $\widetilde W_\bfk(\widetilde z)\cap\exp^{-1}_{x_0}(T)$; this defines the holonomy projection $\pi_\bfk\colon\Omega\to\exp^{-1}_{x_0}(T)$ along the $\bfk$-leaves onto $T$, read in the chart. Since $y\in T$, we have $\pi_\bfk(\widetilde y)=\widetilde y=0$ for every $\bfk$.

Let
\[
\widetilde V_\bfi:=\exp_{x_0}^{-1}(V_\bfi),
\qquad
S:=\sigma_\bfi\circ\exp_{x_0}:
\widetilde V_\bfi\longrightarrow\mathbb R^2.
\]
where $\sigma_\bfi$ is the straightening map of \Cref{prop:straightening map}, constructed with the fixed transversal $T_0$; thus, $S_1$ is the $T_0$-holonomy coordinate $p$ of the $\bfi$-plaques.

As in \Cref{sec: sec5}, let
\[
L':=\sigma_{\bfi}^{-1}\bigl(\cD^{h}_{Ck}(\sigma_{\bfi}(y))\cap\sigma_{\bfi}(W_\bfi(y))\bigr)
\qquad\text{and}\qquad
L:=\cD^{\bfi,h}_{Ck}(y)\cap\calA^{su}_\bfi(y)=L'\cap\Lambda_{n_0};
\]
thus $L'\subset W_\bfi(y)$ contains $y$, and $\sigma_\bfi(L')$ is a vertical segment of length $2^{-Ck}$ containing $\sigma_\bfi(y)$. Let $v$ be a unit vector spanning the vertical direction of the $\sigma_\bfi$-coordinates, so that $\sigma_\bfi(L')\subset\sigma_\bfi(y)+\R v$, and set
\[
\gamma(\lambda):=\bigl(\exp^{-1}_{x_0}\circ\sigma^{-1}_\bfi\bigr)\bigl(\sigma_\bfi(y)+\lambda v\bigr),
\qquad \sigma_\bfi(y)+\lambda v\in\sigma_\bfi(L');
\]
thus $\gamma$ parametrizes $\exp^{-1}_{x_0}(L')\subset\widetilde W_\bfi(\widetilde y)$, $\gamma(0)=\widetilde y=0$, and $\lambda$ ranges over an interval of length $2^{-Ck}$ containing $0$. Let
\[
    c_0:=\|\gamma'(0)\|=\Bigl\|D_{\sigma_\bfi(y)}\bigl(\exp_{x_0}^{-1}\circ\sigma_\bfi^{-1}\bigr)\,v\Bigr\| ,
\]
the leafwise derivative of $\exp^{-1}_{x_0}\circ\sigma^{-1}_\bfi$ at $\sigma_\bfi(y)$; by \Cref{prop:straightening map}\ref{S:inverse reg in tang}, $c_0\in[1,1.01]$. Since $\gamma'(0)$ is tangent to $\widetilde W_\bfi(\widetilde y)$ at $\widetilde y$, it is vertical; we orient $v$ so that $\gamma'(0)=(0,c_0)$.

We use the constants
\[
1<C_0<C_1<\cdots<C_5
\]
of \Cref{rem:constants}. The following local regularity facts hold in the chart, for all $z\in\Lambda_{n_0}$ and $\bfk\in\Sigma^-$:
\begin{enumerate}[label=(R\arabic*),ref=(R\arabic*)]
\item\label{R1:leaf graph} By \ref{U:USR}, $\widetilde W_\bfk(\widetilde z)$ is the graph over its tangent line at $\widetilde z$ of a $C^{1+\alpha/2}$ function $\varphi$ with $\varphi(0)=\varphi'(0)=0$ and $|\varphi'(s_1)-\varphi'(s_2)|\leq C_1|s_1-s_2|^{\alpha/2}$; in particular $|\varphi(s)|\leq C_1|s|^{1+\alpha/2}$, and the tangent lines of $\widetilde W_\bfk(\widetilde z)$ at any two of its points $w_1,w_2$ make an angle at most $C_1\,d(w_1,w_2)^{\alpha/2}$.
\item\label{R2:holder Esu} By \Cref{prop Holder continuous} and \Cref{rem:constants}, applied to lifts with the same past $\bfk$: for $z_1,z_2\in\Lambda_{n_0}$, the tangent lines $T_{\widetilde z_1}\widetilde W_\bfk(\widetilde z_1)$ and $T_{\widetilde z_2}\widetilde W_\bfk(\widetilde z_2)$ make an angle at most $C_0\,d(\widetilde z_1,\widetilde z_2)^{\alpha}$.
\item\label{R3:transversality} By the $\theta$-uniformity of $T$, together with \ref{R1:leaf graph} and the smallness of $\Lambda_{n_0}$: every tangent line of $\widetilde W_\bfk(\widetilde z)$ makes an angle at least $\theta/2$ with the horizontal axis $\R\times\{0\}$; consequently, $\widetilde W_\bfk(\widetilde z)$ is a graph over the vertical axis. For a line $\ell$ making an angle at least $\theta/2$ with the horizontal axis, let $s(\ell)\in\R$ be its signed slope with respect to the vertical axis, that is, $\ell$ is spanned by $(s(\ell),1)$. Then $|s(\ell)|\leq\cot(\theta/2)\leq2/\theta$, and for two such lines $\ell_1,\ell_2$,
\begin{equation}\label{eq:slope-vs-angle}
|s(\ell_1)-s(\ell_2)|\leq\Bigl(\frac{\pi}{\theta}\Bigr)^2\angle(\ell_1,\ell_2).
\end{equation}
Indeed, let $\vartheta_1,\vartheta_2\in[-\pi/2+\theta/2,\,\pi/2-\theta/2]$ be the signed angles of $\ell_1,\ell_2$ from the vertical axis, so that $s(\ell_i)=\tan\vartheta_i$ and $\cos\vartheta_i\geq\sin(\theta/2)\geq\theta/\pi$. Since lines are unoriented, $\angle(\ell_1,\ell_2)=\min\{|\vartheta_1-\vartheta_2|,\pi-|\vartheta_1-\vartheta_2|\}$, so $|\sin(\vartheta_1-\vartheta_2)|=\sin\angle(\ell_1,\ell_2)\leq\angle(\ell_1,\ell_2)$. Hence
\[
|s(\ell_1)-s(\ell_2)|
=\frac{|\sin(\vartheta_1-\vartheta_2)|}{\cos\vartheta_1\cos\vartheta_2}
\leq\frac{\angle(\ell_1,\ell_2)}{\sin^2(\theta/2)}
\leq\Bigl(\frac{\pi}{\theta}\Bigr)^2\angle(\ell_1,\ell_2).
\]
\item\label{R4:leafwise derivative} By \Cref{prop:straightening map}\ref{S:inverse reg in tang}, there is a constant $C_6>1$, depending only on the dynamical system, such that the leafwise derivative of $\exp^{-1}_{x_0}\circ\sigma^{-1}_\bfi$ along $\sigma_\bfi(L')$ satisfies
\[
c_0\bigl(1-C_62^{-\alpha Ck/2}\bigr)\leq\|\gamma'(\lambda)\|\leq c_0\bigl(1+C_62^{-\alpha Ck/2}\bigr)
\qquad\text{for all }\lambda\text{ with }\sigma_\bfi(y)+\lambda v\in\sigma_\bfi(L');
\]
in particular, as $c_0\leq1.01$,
\begin{equation}\label{eq:leafwise-derivative}
\bigl|\,\|\gamma'(\lambda)\|-c_0\bigr|\leq 1.01\,C_6\,2^{-\alpha Ck/2}.
\end{equation}
\end{enumerate}

Finally, fix $C,k\in\N$ with $C\alpha>6$, and assume
\[
\tau_y:=\tan\angle_y(\bfi,\bfj)
=\tan\angle\bigl(T_{\widetilde y}\widetilde W_\bfi(\widetilde y),\,T_{\widetilde y}\widetilde W_\bfj(\widetilde y)\bigr)
\in [2^{-k-2}, 2^{-k+2}].
\]
For $\bfk\in\Sigma^-$ let $\mathfrak m_\bfk:=s\bigl(T_{\widetilde y}\widetilde W_\bfk(\widetilde y)\bigr)$ be the signed slope of the $\bfk$-plaque through $\widetilde y$; thus $\mathfrak m_\bfi=0$ and $|\mathfrak m_\bfj|=\tau_y$.

\subsection*{Plaques as graphs and the slope estimate}
For $\bfk\in\Sigma^-$ and $\widetilde q=(\widetilde q_T,\widetilde q_N)\in\Omega$, by \ref{R3:transversality} the arc of $\widetilde W_\bfk(\widetilde q)$ between $\widetilde q$ and $\pi_\bfk(\widetilde q)$ is a graph over the vertical axis,
\[
\Gamma^\bfk_{\widetilde q}=\{(\psi^\bfk_{\widetilde q}(t),t):\ t\text{ lies between }\widetilde q_N\text{ and }0\},
\qquad
\psi^\bfk_{\widetilde q}(\widetilde q_N)=\widetilde q_T,\quad \psi^\bfk_{\widetilde q}(0)=\pi_\bfk(\widetilde q),
\]
and $(\psi^\bfk_{\widetilde q})'(t)=s\bigl(T_w\Gamma^\bfk_{\widetilde q}\bigr)$ for $w=(\psi^\bfk_{\widetilde q}(t),t)$; in particular $|(\psi^\bfk_{\widetilde q})'|\leq2/\theta$. Consequently,
\begin{equation}\label{eq:projection-formula}
\pi_\bfk(\widetilde q)=\widetilde q_T+\int_{\widetilde q_N}^{0}(\psi^\bfk_{\widetilde q})'(t)\,dt .
\end{equation}
All three lemmas below rest on \eqref{eq:projection-formula} and the following estimate for the integrand.

\begin{lem}[Slope estimate]\label{lem:slope-estimate}
Let $K_*:=60\,C_1\theta^{-3}$. For every $\bfk\in\Sigma^-$, $\widetilde q\in\Omega$, and $t$ between $\widetilde q_N$ and $0$,
\[
\bigl|(\psi^\bfk_{\widetilde q})'(t)-\mathfrak m_\bfk\bigr|
\leq K_*\bigl(d(\widetilde q,\widetilde y)^{\alpha}+|\widetilde q_N|^{\alpha/2}\bigr).
\]
\end{lem}

\begin{proof}
Put $w:=(\psi^\bfk_{\widetilde q}(t),t)\in\Gamma^\bfk_{\widetilde q}$. Comparing $T_w\Gamma^\bfk_{\widetilde q}$ with $T_{\widetilde q}\widetilde W_\bfk(\widetilde q)$ (by \ref{R1:leaf graph}, both being tangent lines of $\widetilde W_\bfk(\widetilde q)$) and then with $T_{\widetilde y}\widetilde W_\bfk(\widetilde y)$ (by \ref{R2:holder Esu}), we get
\[
\angle\bigl(T_w\Gamma^\bfk_{\widetilde q},\,T_{\widetilde y}\widetilde W_\bfk(\widetilde y)\bigr)
\leq C_1\,d(w,\widetilde q)^{\alpha/2}+C_0\,d(\widetilde q,\widetilde y)^{\alpha}.
\]
Since $|(\psi^\bfk_{\widetilde q})'|\leq2/\theta$, the mean value theorem gives $|\psi^\bfk_{\widetilde q}(t)-\widetilde q_T|\leq(2/\theta)|t-\widetilde q_N|$, hence, as $\theta\leq\pi/2<2$ and $|t-\widetilde q_N|\leq|\widetilde q_N|$,
\[
d(w,\widetilde q)\leq(1+2/\theta)|t-\widetilde q_N|\leq(4/\theta)|\widetilde q_N| .
\]
Both lines make an angle at least $\theta/2$ with the horizontal axis, and $(\psi^\bfk_{\widetilde q})'(t)=s(T_w\Gamma^\bfk_{\widetilde q})$, $\mathfrak m_\bfk=s(T_{\widetilde y}\widetilde W_\bfk(\widetilde y))$. Therefore, by \eqref{eq:slope-vs-angle},
\[
\bigl|(\psi^\bfk_{\widetilde q})'(t)-\mathfrak m_\bfk\bigr|
\leq\Bigl(\frac{\pi}{\theta}\Bigr)^2\Bigl(C_1\bigl(\tfrac{4}{\theta}\bigr)^{\alpha/2}|\widetilde q_N|^{\alpha/2}+C_0\,d(\widetilde q,\widetilde y)^{\alpha}\Bigr)
\leq K_*\bigl(d(\widetilde q,\widetilde y)^{\alpha}+|\widetilde q_N|^{\alpha/2}\bigr),
\]
where we used $(4/\theta)^{\alpha/2}\leq4/\theta$, $C_0\leq C_1$, $1+4/\theta\leq6/\theta$, and $6\pi^2\leq60$.
\end{proof}

\subsection*{The three lemmas}

\begin{lem}[Planar geometry input I]
\label{lem:geom-input-proj-curved-rectangle}
Given $A>0$, there exists $A'>0$, depending only on $A$, $\alpha$, $\theta$, $C_1$ and $C_5$, such that the following holds. If $\widetilde{R}\subset \Omega$ satisfies
\begin{equation}
\label{eqn:rectangle inclusion}
    S(\widetilde R)
    \subset
    [S_1(\widetilde{y})-A2^{-(C+1)k},S_1(\widetilde{y})+A 2^{-(C+1)k}]
    \times
    [S_2(\widetilde{y})-A2^{-Ck},S_2(\widetilde{y})+A2^{-Ck}],
\end{equation}
then
\[
    \pi_\bfj(\widetilde R)
    \subset
    B^T\bigl(\pi_\bfj(\widetilde y),A'2^{-(C+1)k}\bigr).
\]
\end{lem}

\begin{rem}\label{rem:scales-I}
The rectangle has a height of order $2^{-Ck}$, much larger than its width $2^{-(C+1)k}$, yet its $\bfj$-projection has a length of order $2^{-(C+1)k}$ only: the $\bfj$-plaques are almost vertical, with a tilt $\tau_y\leq2^{-k+2}$, so following them over the height $2^{-Ck}$ displaces the projection by at most $\tau_y\cdot O(2^{-Ck})=O(2^{-(C+1)k})$, the order of the width, while the curvature of the plaques contributes only $O(2^{-(1+\alpha/2)Ck})$. The lemma is sharp at the scale $2^{-(C+1)k}$.
 
The proof uses only the upper bound $\tau_y\leq2^{-k+2}$. In particular, the conclusion also holds for $\bfj=\bfi$, that is, for the projection $\pi_\bfi$ along the $\bfi$-leaves, with the same constant $A'$; this case is used in \eqref{equ:upperbound-NB}.
\end{rem}

\begin{proof}
Set $a:=A2^{-(C+1)k}$ and $b:=A2^{-Ck}$, so that $a\leq b$. Let $\widetilde{z}=(\widetilde{z}_T,\widetilde{z}_N)\in\widetilde R$.

\smallskip\noindent
\emph{Step 1: distance to $\widetilde y$.}
By \eqref{eqn:rectangle inclusion}, $S(\widetilde z)$ lies in the rectangle of side lengths $2a\times2b$ centered at $S(\widetilde y)$, so $d(S(\widetilde z),S(\widetilde y))\leq a+b\leq 2b$. Since $S$ is bi-Lipschitz with constant $C_5^2$,
\[
    d(\widetilde{z},\widetilde{y})
    \leq C_5^2\, d(S(\widetilde{z}),S(\widetilde{y}))
    \leq K_1b ,
    \qquad K_1:=2C_5^2 .
\]
Hence
\begin{equation}\label{eq:v-bound-from-bilip}
    |\widetilde{z}_N|\leq d(\widetilde{z},\widetilde{y})\leq K_1b .
\end{equation}

\smallskip\noindent
\emph{Step 2: the horizontal coordinate $\widetilde z_T$.}
Since $S_1$ is the $T_0$-holonomy coordinate, \eqref{eqn:rectangle inclusion} gives
\[
    d_{T_0}\bigl(W_{\bfi}(z)\cap T_0,\,W_{\bfi}(y)\cap T_0\bigr)\leq C_5a .
\]
Let $q$ be the unique point of $W_\bfi(y)\cap T_z$. Since $W_\bfi(z)\cap T_z=\{z\}$, the bi-Lipschitz property of the holonomy along the $\bfi$-leaves between the transversals $T_0$ and $T_z$ (\Cref{lem holder of Esu}) gives
\[
    d_{T_z}(z,q)\leq C_5^2 a .
\]
Hence $\widetilde q:=\exp^{-1}_{x_0}(q)=(\widetilde q_T,\widetilde q_N)\in\widetilde W_\bfi(\widetilde y)$ satisfies
\[
    d(\widetilde{z},\widetilde{q})\leq C_5\,d(z,q)\leq C_5\,d_{T_z}(z,q)\leq C_5^3a .
\]
Together with \eqref{eq:v-bound-from-bilip}, this gives
\[
    |\widetilde{q}_N|\leq |\widetilde{z}_N|+d(\widetilde{z},\widetilde{q})\leq K_1b+C_5^3a\leq K_2b,
    \qquad K_2:=K_1+C_5^3 .
\]
By \ref{R1:leaf graph} applied to $\widetilde W_\bfi(\widetilde y)$, whose tangent line at $\widetilde y$ is the vertical axis, we have $\widetilde q=(\varphi(\widetilde q_N),\widetilde q_N)$ with $|\varphi(s)|\leq C_1|s|^{1+\alpha/2}$; hence
\[
    |\widetilde{q}_T|\leq C_1|\widetilde{q}_N|^{1+\alpha/2}\leq K_3 b^{1+\alpha/2},
    \qquad K_3:=C_1K_2^{1+\alpha/2}.
\]
Therefore
\begin{equation}\label{eq:u-bound-from-H1H2}
    |\widetilde{z}_T|
    \leq |\widetilde{z}_T-\widetilde{q}_T|+|\widetilde{q}_T|
    \leq d(\widetilde{z},\widetilde{q})+|\widetilde{q}_T|
    \leq K_4\bigl(a+b^{1+\alpha/2}\bigr),
    \qquad K_4:=C_5^3+K_3 .
\end{equation}

\smallskip\noindent
\emph{Step 3: the $\bfj$-projection.}
By \Cref{lem:slope-estimate} with $\bfk=\bfj$ and $\widetilde q=\widetilde z$, together with $|\mathfrak m_\bfj|=\tau_y$ and \eqref{eq:v-bound-from-bilip}, for all $t$ between $\widetilde z_N$ and $0$,
\[
    |(\psi^\bfj_{\widetilde{z}})'(t)|
    \leq \tau_y+K_*\bigl(K_1^{\alpha}b^{\alpha}+K_1^{\alpha/2}b^{\alpha/2}\bigr)
    \leq K_5\bigl(\tau_y+b^\alpha+b^{\alpha/2}\bigr),
    \qquad K_5:=K_*K_1^{\alpha},
\]
where we used $K_1\geq1$ and $K_*\geq1$. Since $\pi_\bfj(\widetilde y)=0$, \eqref{eq:projection-formula}, \eqref{eq:u-bound-from-H1H2} and \eqref{eq:v-bound-from-bilip} give
\[
\begin{aligned}
    |\pi_\bfj(\widetilde z)-\pi_\bfj(\widetilde y)|
    &\leq |\widetilde z_T|+|\widetilde z_N|\sup_t|(\psi^\bfj_{\widetilde z})'(t)|
    \leq
    K_4\bigl(a+b^{1+\alpha/2}\bigr)
    +
    K_1b\cdot K_5\bigl(\tau_y+b^\alpha+b^{\alpha/2}\bigr)\\
    &\leq
    K_6\bigl(a+\tau_y b+b^{1+\alpha}+b^{1+\alpha/2}\bigr),
    \qquad K_6:=2 (K_4+K_1K_5).
\end{aligned}
\]

\smallskip\noindent
\emph{Step 4: conclusion.}
We have $a=A2^{-(C+1)k}$ and $\tau_y b\leq 2^{-k+2}A2^{-Ck}=4A2^{-(C+1)k}$. Moreover, since $C\alpha>6$, the exponents $C\alpha-1$ and $C\alpha/2-1$ are positive, so
\begin{align}
    b^{1+\alpha}
    &=
    A^{1+\alpha}2^{-(C+1)k}\,2^{-(C\alpha-1)k}\leq A^{1+\alpha}2^{-(C+1)k},\\
    b^{1+\alpha/2}
    &=
    A^{1+\alpha/2}2^{-(C+1)k}\,2^{-(C\alpha/2-1)k}\leq A^{1+\alpha/2}2^{-(C+1)k}.
\end{align}
Hence
\[
    |\pi_\bfj(\widetilde z)-\pi_\bfj(\widetilde y)|
    \leq A'2^{-(C+1)k},
    \qquad A':=K_6\bigl(5A+A^{1+\alpha}+A^{1+\alpha/2}\bigr),
\]
and $A'$ depends only on $A$, $\alpha$, $\theta$, $C_1$ and $C_5$. Since $\widetilde z\in\widetilde R$ was arbitrary, the proof is complete.
\end{proof}

\begin{figure}[H]
\centering
\begin{tikzpicture}[scale=1.9,>=Latex,font=\small]
  \draw[->] (-2.1,0) -- (2.3,0) node[right] {$\exp_{x_0}^{-1}(T)$};
  \draw[->] (0,-2.3) -- (0,2.85) node[above left] {$T_{\widetilde y}\widetilde W_{\bfi}(\widetilde y)$};
  \fill[blue!12]
    plot[domain=-1.8:1.8,samples=40] ({0.02*\x*\x*\x-0.30},{\x})
    -- plot[domain=1.8:-1.8,samples=40] ({0.02*\x*\x*\x+0.30},{\x}) -- cycle;
  \draw[blue!60] plot[domain=-1.8:1.8,samples=40] ({0.02*\x*\x*\x-0.30},{\x});
  \draw[blue!60] plot[domain=-1.8:1.8,samples=40] ({0.02*\x*\x*\x+0.30},{\x});
  \draw[blue!60] (-0.4166,-1.8) -- (0.1834,-1.8);
  \draw[blue!60] (-0.1834,1.8) -- (0.4166,1.8);
  \node[blue!70!black] at (-0.75,1.2) {$\widetilde R$};
  \draw[<->,gray] (-1.3,-1.8) -- (-1.3,1.8) node[pos=0.62,left,black] {$\approx 2b$};
  \draw[<->,gray] (-0.1834,2.05) -- (0.4166,2.05) node[midway,above,black] {$\approx 2a$};
  \draw[thick,blue] plot[domain=-2.25:2.6,samples=60] ({0.02*\x*\x*\x},{\x})
    node[right] {$\widetilde W_{\bfi}(\widetilde y)$};
  \draw[thick,red] plot[domain=-2.25:2.35,samples=60] ({0.30*\x+0.01*\x*\x*\x},{\x})
    node[above right] {$\widetilde W_{\bfj}(\widetilde y)$};
  \coordinate (O) at (0,0);
  \coordinate (V) at (0,1.5);
  \coordinate (J) at (0.45,1.5);
  \pic[draw=red,angle radius=0.55cm] {angle=J--O--V};
  \draw[->,red,shorten >=1pt] (0.72,0.62) node[right,font=\scriptsize,black] {$\angle_y(\bfi,\bfj)$, $\tan\angle_y(\bfi,\bfj)=\tau_y$} -- (0.07,0.31);
  \coordinate (Z) at (0.261,1.45);
  \coordinate (Q) at (0.061,1.45);
  \draw[densely dotted,thick] (Q) -- (Z) node[midway,below=1pt,xshift=7pt,font=\scriptsize] {$\le C_5^3a$};
  \fill (Q) circle (0.9pt);
  \draw[gray,thin] (Q) -- (-0.32,1.72) node[left,black,font=\scriptsize] {$\widetilde q\in\widetilde W_{\bfi}(\widetilde y)$};
  \fill (Z) circle (1.1pt) node[right] {$\widetilde z=(\widetilde z_T,\widetilde z_N)$};
  \draw[thick,red] plot[domain=0:1.45,samples=40] ({0.261+0.30*(\x-1.45)+0.01*(\x*\x*\x-3.049)},{\x});
  \draw[thick,red,dashed] plot[domain=1.45:2.3,samples=30] ({0.261+0.30*(\x-1.45)+0.01*(\x*\x*\x-3.049)},{\x});
  \node[red,font=\scriptsize] (G) at (-0.24,0.95) {$\Gamma_{\widetilde z}$};
  \draw[red,thin,->,shorten >=1pt] (G) -- (0.05,0.85);
  \coordinate (PJ) at (-0.204,0);
  \fill[red] (PJ) circle (1.1pt);
  \node[red,font=\scriptsize,below left] at (PJ) {$\pi_{\bfj}(\widetilde z)$};
  \fill (O) circle (1.1pt);
  \node[font=\scriptsize,below right,yshift=-3pt] at (0.03,-0.03) {$\widetilde y=\pi_{\bfj}(\widetilde y)$};
  \draw[line width=2.4pt,red!70!black,opacity=0.5] (-0.62,0) -- (0.62,0);
  \draw[<-,red!70!black] (0.64,-0.06) -- (1.15,-0.85)
    node[right,align=left,black] {$\pi_{\bfj}(\widetilde R)\subseteq B^{T}\bigl(\pi_{\bfj}(\widetilde y),\,A'2^{-(C+1)k}\bigr)$};
\end{tikzpicture}
 \caption{The 
$\bfj$-projection of an 
$\bfi$-rectangle (\Cref{lem:geom-input-proj-curved-rectangle}).}
\label{fig:j projection of an i rectangle}
\end{figure}

\begin{lem}[Planar geometry input II]\label{rem:geom-input-commutator}
Given $A_0>0$, there exists $A'>0$, depending only on $A_0$, $\alpha$, $\theta$ and $C_1$, such that the following holds. Let $\widetilde z\in\Omega$ with $d(\widetilde z,\widetilde y)\leq A_02^{-Ck}$, and let $B_{\widetilde z}\subset\Omega\cap B\bigl(\widetilde z,A_02^{-(C+1)k}\bigr)$. Then for every $\widetilde w\in B_{\widetilde z}$,
\begin{equation}\label{eq:commutator-pointwise}
\Bigl|\bigl(\pi_\bfj(\widetilde w)-\pi_\bfj(\widetilde z)\bigr)-\bigl(\pi_\bfi(\widetilde w)-\pi_\bfi(\widetilde z)\bigr)\Bigr|\leq A'2^{-(C+2)k}.
\end{equation}
In particular, if $\widetilde z\in\widetilde W_\bfi(\widetilde y)$, then $\pi_\bfi(\widetilde z)=0$ and
\[
    d_{\mathrm{H}}\bigl(\pi_\bfj(B_{\widetilde{z}})-\pi_\bfj(\widetilde{z}),\,\pi_\bfi(B_{\widetilde{z}})\bigr)
    \leq A'2^{-(C+2)k}.
\]
\end{lem}

\begin{rem}\label{rem:scales-II}
The point of the lemma is the scale of the error: the ball has radius $A_02^{-(C+1)k}$, while the two displacements agree up to $A'2^{-(C+2)k}$, one scale smaller. The gain of the factor $2^{-k}$ comes from the precise control of the angle. 
\end{rem}

\begin{proof}
Let $\widetilde w\in B_{\widetilde z}$ and put $A_1:=2A_0$. Then $d(\widetilde w,\widetilde z)\leq A_02^{-(C+1)k}$ and $d(\widetilde w,\widetilde y)\leq d(\widetilde w,\widetilde z)+d(\widetilde z,\widetilde y)\leq A_12^{-Ck}$; hence, for $\widetilde q\in\{\widetilde z,\widetilde w\}$,
\[
    d(\widetilde q,\widetilde y)\leq A_12^{-Ck},\qquad
    |\widetilde q_N|\leq d(\widetilde q,\widetilde y)\leq A_12^{-Ck},\qquad
    |\widetilde w_N-\widetilde z_N|\leq A_02^{-(C+1)k}.
\]
By \Cref{lem:slope-estimate}, for $\bfk\in\{\bfi,\bfj\}$, $\widetilde q\in\{\widetilde z,\widetilde w\}$, and $t$ between $\widetilde q_N$ and $0$,
\[
    \bigl|(\psi^\bfk_{\widetilde q})'(t)-\mathfrak m_\bfk\bigr|\leq\epsilon_k,
    \qquad
    \epsilon_k:=K_*\bigl(A_1^{\alpha}2^{-\alpha Ck}+A_1^{\alpha/2}2^{-\alpha Ck/2}\bigr).
\]
Applying \eqref{eq:projection-formula} to $\widetilde w$ and to $\widetilde z$ and writing $(\psi^\bfk_{\widetilde q})'=\mathfrak m_\bfk+\bigl((\psi^\bfk_{\widetilde q})'-\mathfrak m_\bfk\bigr)$, we obtain, for $\bfk\in\{\bfi,\bfj\}$,
\[
    \pi_\bfk(\widetilde w)-\pi_\bfk(\widetilde z)
    =(\widetilde w_T-\widetilde z_T)-\mathfrak m_\bfk(\widetilde w_N-\widetilde z_N)+E_\bfk,
    \qquad
    |E_\bfk|\leq\bigl(|\widetilde w_N|+|\widetilde z_N|\bigr)\epsilon_k\leq 2A_12^{-Ck}\epsilon_k .
\]
Subtracting the two identities and using $\mathfrak m_\bfi=0$, $|\mathfrak m_\bfj|=\tau_y\leq2^{-k+2}$, we get
\[
\begin{aligned}
\Bigl|\bigl(\pi_\bfj(\widetilde w)-\pi_\bfj(\widetilde z)\bigr)-\bigl(\pi_\bfi(\widetilde w)-\pi_\bfi(\widetilde z)\bigr)\Bigr|
&\leq \tau_y|\widetilde w_N-\widetilde z_N|+|E_\bfj|+|E_\bfi|\\
&\leq 4A_02^{-(C+2)k}+4A_1K_*\bigl(A_1^{\alpha}2^{-(1+\alpha)Ck}+A_1^{\alpha/2}2^{-(1+\alpha/2)Ck}\bigr).
\end{aligned}
\]
Since $C\alpha>6$, we have $(1+\alpha/2)C\geq C+3$, hence $2^{-(1+\alpha)Ck}\leq2^{-(1+\alpha/2)Ck}\leq2^{-(C+2)k}$, and \eqref{eq:commutator-pointwise} follows with
\[
    A':=4A_0+4A_1K_*\bigl(A_1^{\alpha}+A_1^{\alpha/2}\bigr),
\]
which depends only on $A_0$, $\alpha$, $\theta$ and $C_1$.

If $\widetilde z\in\widetilde W_\bfi(\widetilde y)$, then $\widetilde W_\bfi(\widetilde z)=\widetilde W_\bfi(\widetilde y)$ meets $\exp^{-1}_{x_0}(T)$ at $\widetilde y$, so $\pi_\bfi(\widetilde z)=0$, and \eqref{eq:commutator-pointwise} says that the maps $\widetilde w\mapsto\pi_\bfj(\widetilde w)-\pi_\bfj(\widetilde z)$ and $\widetilde w\mapsto\pi_\bfi(\widetilde w)$ differ by at most $A'2^{-(C+2)k}$ at every point of $B_{\widetilde z}$; hence each of the two sets $\pi_\bfj(B_{\widetilde z})-\pi_\bfj(\widetilde z)$ and $\pi_\bfi(B_{\widetilde z})$ lies in the $A'2^{-(C+2)k}$-neighborhood of the other, which is the Hausdorff bound.
\end{proof}

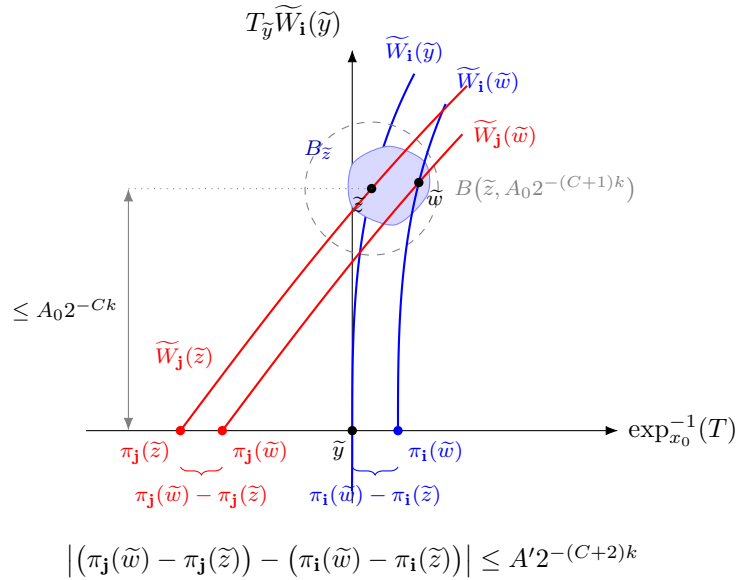
\begin{figure}[H]
\centering
\begin{tikzpicture}[scale=1.6,>=Latex,font=\small]
  \draw[->] (-2.2,0) -- (2.2,0) node[right] {$\exp_{x_0}^{-1}(T)$};
  \draw[->] (0,-0.6) -- (0,3.15) node[above left] {$T_{\widetilde y}\widetilde W_{\bfi}(\widetilde y)$};
  \draw[thick,blue] plot[domain=-0.5:2.95,samples=60] ({0.02*\x*\x*\x},{\x})
    node[above,font=\scriptsize] {$\widetilde W_{\bfi}(\widetilde y)$};
  \coordinate (Z) at (0.16,2.0);
  \coordinate (W) at (0.55,2.05);
  \draw[dashed,gray] (Z) circle (0.55);
  \node[gray,font=\scriptsize,right] at (0.76,2.0) {$B\bigl(\widetilde z,A_02^{-(C+1)k}\bigr)$};
  \fill[blue!15,draw=blue!50] plot[smooth cycle] coordinates {(0.0,1.85) (0.35,1.7) (0.62,1.95) (0.6,2.22) (0.3,2.35) (0.0,2.2)};
  \node[blue!70!black,font=\scriptsize] at (-0.28,2.32) {$B_{\widetilde z}$};
  \draw[thick,blue] plot[domain=0:2.7,samples=50] ({0.02*\x*\x*\x+0.378},{\x})
    node[above right,font=\scriptsize] {$\widetilde W_{\bfi}(\widetilde w)$};
  \draw[thick,red] plot[domain=0:2.85,samples=50] ({0.16+0.75*(\x-2)+0.01*(\x*\x*\x-8)},{\x});
  \draw[thick,red] plot[domain=0:2.45,samples=50] ({0.55+0.75*(\x-2.05)+0.01*(\x*\x*\x-8.615)},{\x})
    node[right,font=\scriptsize] {$\widetilde W_{\bfj}(\widetilde w)$};
  \node[red,font=\scriptsize,left] at (-1.06,0.62) {$\widetilde W_{\bfj}(\widetilde z)$};
  \fill (Z) circle (1.1pt) node[below left,font=\scriptsize,xshift=1pt,yshift=1pt] {$\widetilde z$};
  \fill (W) circle (1.1pt) node[below right,font=\scriptsize,xshift=-1pt,yshift=1pt] {$\widetilde w$};
  \coordinate (PIZ) at (0,0);
  \coordinate (PIW) at (0.378,0);
  \coordinate (PJZ) at (-1.42,0);
  \coordinate (PJW) at (-1.074,0);
  \fill (PIZ) circle (1.1pt);
  \fill[blue] (PIW) circle (1.1pt);
  \fill[red] (PJZ) circle (1.1pt);
  \fill[red] (PJW) circle (1.1pt);
  \node[font=\scriptsize,below left,red] at (PJZ) {$\pi_{\bfj}(\widetilde z)$};
  \node[font=\scriptsize,below right,red] at (PJW) {$\pi_{\bfj}(\widetilde w)$};
  \node[font=\scriptsize,below left] at (0.02,0) {$\widetilde y$};
  \node[font=\scriptsize,below right,blue] at (PIW) {$\pi_{\bfi}(\widetilde w)$};
  \draw[decorate,decoration={brace,mirror,amplitude=3pt},red] (-1.42,-0.3) -- (-1.074,-0.3)
    node[midway,below=3pt,font=\scriptsize] {$\pi_{\bfj}(\widetilde w)-\pi_{\bfj}(\widetilde z)$};
  \draw[decorate,decoration={brace,mirror,amplitude=3pt},blue] (0,-0.3) -- (0.378,-0.3)
    node[midway,below=3pt,font=\scriptsize] {$\pi_{\bfi}(\widetilde w)-\pi_{\bfi}(\widetilde z)$};
  \node[align=center,font=\small,anchor=north] at (0.0,-0.85)
    {$\bigl|\bigl(\pi_{\bfj}(\widetilde w)-\pi_{\bfj}(\widetilde z)\bigr)-\bigl(\pi_{\bfi}(\widetilde w)-\pi_{\bfi}(\widetilde z)\bigr)\bigr|\le A'2^{-(C+2)k}$};
  \draw[<->,gray] (-1.85,0) -- (-1.85,2.0) node[midway,left,black,font=\scriptsize] {$\le A_02^{-Ck}$};
  \draw[dotted,gray] (-1.85,2.0) -- (Z);
\end{tikzpicture}
 \caption{The commutator estimate: agreement of the two projections at scale $2^{-(C+2)k}$ (\Cref{rem:geom-input-commutator}).}
\label{fig:the commutator}
\end{figure}

\subsection*{The parametrization $\varphi_1$ and the non-linear displacement}
Define the affine bijection $\varphi_1^{-1}\colon\sigma_\bfi(y)+\R v\to\R$ by
\[
    \varphi^{-1}_1\bigl(\sigma_{\bfi}(y)+\lambda v\bigr):=-\lambda\,c_0\,\mathfrak m_{\bfj},
\]
which is injective since $\mathfrak m_{\bfj}\neq0$ by $\tau_y>0$, and write $\varphi_1$ for its inverse. Recall $s_{\mathrm a}:=c_0\tau_y2^{k}$; since $c_0\geq1$ and $\tau_y\geq2^{-k-2}$,
\begin{equation}\label{eq:sa-lower}
s_{\mathrm a}\geq\tfrac14 .
\end{equation}

\begin{lem}[Planar geometry input III]\label{rem:geom-input-nonlinear-displacement}
Let $K_7:=11\,K_*C_5^4+8C_6$ and
\[
k_0:=\Bigl\lceil\frac{2\log_2(4K_7)}{C\alpha-6}\Bigr\rceil ,
\]
which depends only on $C$, $\alpha$, $\theta$, $C_1$, $C_5$ and $C_6$. If $k\geq k_0$, then for every $t\in\varphi_1^{-1}(\sigma_\bfi(L))$, the unique point $z_t\in L$ with $t=\varphi_1^{-1}(\sigma_\bfi(z_t))$ satisfies
\[
    \bigl|\pi_\bfj(\widetilde z_t)-t\bigr|\leq s_{\mathrm a}\,2^{-(C+3)k}.
\]
\end{lem}

\begin{rem}
In contrast to \Cref{lem:geom-input-proj-curved-rectangle} and \Cref{rem:geom-input-commutator}, the conclusion here has the fixed constant $s_{\mathrm a}$ rather than a free constant $A'$; this is why a threshold $k_0$ appears, and why the lower bound $\tau_y\geq2^{-k-2}$ is used, through \eqref{eq:sa-lower}.
\end{rem}

\begin{proof}
Fix $t\in\varphi_1^{-1}(\sigma_\bfi(L))$ and put $\lambda:=-t/(c_0\mathfrak m_\bfj)$, so that $\sigma_\bfi(z_t)=\sigma_\bfi(y)+\lambda v$, $\widetilde z_t=\gamma(\lambda)$, and
\begin{equation}\label{eq:t-lambda}
t=-c_0\mathfrak m_\bfj\lambda,\qquad |\lambda|\leq2^{-Ck}\leq1 .
\end{equation}
Write $\widetilde z_t=(\widetilde z_{t,T},\widetilde z_{t,N})$. Note that $\widetilde z_t\in\Omega$, since $z_t\in L\subset\Lambda_{n_0}$.

\smallskip\noindent
\emph{Step 1: position of $\widetilde z_t$.}
For $|u|\leq|\lambda|$ we have $S(\gamma(u))-S(\widetilde y)=uv$, so, $S$ being bi-Lipschitz with constant $C_5^2$,
\begin{equation}\label{eq:gamma-dist}
d(\gamma(u),\widetilde y)\leq C_5^2|u|;\qquad\text{in particular}\qquad
d(\widetilde z_t,\widetilde y)\leq C_5^2|\lambda|,\quad |\widetilde z_{t,N}|\leq C_5^2|\lambda| .
\end{equation}
Since $\widetilde z_t\in\widetilde W_\bfi(\widetilde y)$, \ref{R1:leaf graph} applied to $\widetilde W_\bfi(\widetilde y)$ (a graph over the vertical axis, its tangent line at $\widetilde y$) gives
\begin{equation}\label{eq:zT-bound}
|\widetilde z_{t,T}|\leq C_1|\widetilde z_{t,N}|^{1+\alpha/2}\leq C_1C_5^{3}|\lambda|^{1+\alpha/2},
\end{equation}
using $C_5^{2+\alpha}\leq C_5^3$. For the vertical coordinate, let $\vartheta(u)$ be the signed angle of $\gamma'(u)$ from the vertical axis. It is continuous in $u$, $\vartheta(0)=0$, and $|\vartheta(u)|<\pi/2$ by \ref{R3:transversality}; hence $\gamma'(u)=\|\gamma'(u)\|\bigl(\sin\vartheta(u),\cos\vartheta(u)\bigr)$ with $\cos\vartheta(u)>0$, and \ref{R1:leaf graph}, applied to the points $\widetilde y$ and $\gamma(u)$ of $\widetilde W_\bfi(\widetilde y)$, together with \eqref{eq:gamma-dist} gives $|\vartheta(u)|\leq C_1d(\gamma(u),\widetilde y)^{\alpha/2}\leq C_1C_5|u|^{\alpha/2}$. Combining this with \eqref{eq:leafwise-derivative}, $c_0\leq1.01$, $1-\cos\vartheta\leq|\vartheta|$ and $|\lambda|\leq2^{-Ck}$, we obtain
\begin{equation}\label{eq:zN-bound}
\begin{aligned}
\bigl|\widetilde z_{t,N}-c_0\lambda\bigr|
&=\Bigl|\int_0^{\lambda}\bigl(\|\gamma'(u)\|\cos\vartheta(u)-c_0\bigr)\,du\Bigr|\\
&\leq\int_0^{|\lambda|}\bigl(1.01\,C_6\,2^{-\alpha Ck/2}+1.01\,C_1C_5\,u^{\alpha/2}\bigr)\,du\\
&\leq1.01\,C_6\,2^{-\alpha Ck/2}|\lambda|+1.01\,C_1C_5|\lambda|^{1+\alpha/2}\\
&\leq2\,(C_6+C_1C_5)\,2^{-(1+\alpha/2)Ck},
\end{aligned}
\end{equation}
where we used $\bigl|\|\gamma'\|\cos\vartheta-c_0\bigr|\leq\bigl|\|\gamma'\|-c_0\bigr|+c_0(1-\cos\vartheta)$.

\smallskip\noindent
\emph{Step 2: the $\bfj$-projection.}
By \eqref{eq:projection-formula} and \Cref{lem:slope-estimate} with $\bfk=\bfj$ and $\widetilde q=\widetilde z_t$,
\[
    \pi_\bfj(\widetilde z_t)=\widetilde z_{t,T}-\mathfrak m_\bfj\,\widetilde z_{t,N}+E,
    \qquad
    |E|\leq|\widetilde z_{t,N}|\,K_*\bigl(d(\widetilde z_t,\widetilde y)^{\alpha}+|\widetilde z_{t,N}|^{\alpha/2}\bigr),
\]
so that, by \eqref{eq:gamma-dist}, $|\lambda|\leq1$ and $C_5^{2+2\alpha},C_5^{2+\alpha}\leq C_5^4$,
\[
    |E|\leq 2K_*C_5^4|\lambda|^{1+\alpha/2}.
\]
Subtracting \eqref{eq:t-lambda},
\[
    \pi_\bfj(\widetilde z_t)-t=\widetilde z_{t,T}-\mathfrak m_\bfj\bigl(\widetilde z_{t,N}-c_0\lambda\bigr)+E .
\]
Hence, by \eqref{eq:zT-bound}, \eqref{eq:zN-bound}, $|\lambda|\leq2^{-Ck}$, $|\mathfrak m_\bfj|=\tau_y\leq4$, and $C_1C_5^3+8C_1C_5+2K_*C_5^4\leq11K_*C_5^4$ (recall $K_*\geq C_1$),
\[
    \bigl|\pi_\bfj(\widetilde z_t)-t\bigr|\leq\bigl(11K_*C_5^4+8C_6\bigr)2^{-(1+\alpha/2)Ck}=K_7\,2^{-(1+\alpha/2)Ck}.
\]

\smallskip\noindent
\emph{Step 3: conclusion.}
By \eqref{eq:sa-lower} it suffices to have $K_72^{-(1+\alpha/2)Ck}\leq\tfrac14 2^{-(C+3)k}$, that is, $2^{(C\alpha/2-3)k}\geq4K_7$. Since $C\alpha>6$, this holds precisely when $k\geq2\log_2(4K_7)/(C\alpha-6)$, in particular for $k\geq k_0$.
\end{proof}

\begin{figure}[H]
\centering
\begin{tikzpicture}[scale=1.9,>=Latex,font=\small]
  \node[font=\small\itshape] at (-4.0,2.3) {$\sigma_\bfi$-coordinates};
  \draw[dashed,gray] (-4.7,-0.5) rectangle (-3.3,1.7);
  \node[gray,font=\scriptsize,anchor=south west] at (-4.7,1.72) {$\cD^{h}_{Ck}(\sigma_\bfi(y))$};
  \draw[blue!50] (-4.0,-0.95) -- (-4.0,2.05);
  \node[blue,font=\scriptsize,right] at (-3.98,-0.85) {$\sigma_\bfi(W_\bfi(y))$};
  \draw[very thick,blue] (-4.0,-0.5) -- (-4.0,1.7);
  \node[blue!70!black,font=\scriptsize,right] at (-3.98,-0.35) {$\sigma_\bfi(L')$};
  \foreach \yy in {-0.46,-0.4,-0.14,0.12,0.18,0.55,0.66,1.2,1.4,1.48,1.62}
    \fill[black] (-4.0,\yy) circle (0.6pt);
  \fill (-4.0,0) circle (1.1pt) node[left,font=\scriptsize] {$\sigma_\bfi(y)$};
  \draw[->,thick] (-4.0,0) -- (-4.0,0.42) node[pos=0.9,right,font=\scriptsize] {$v$};
  \coordinate (SZ) at (-4.0,1.2);
  \fill (SZ) circle (1.1pt) node[right,font=\scriptsize] {$\sigma_\bfi(z_t)=\sigma_\bfi(y)+\lambda v$};
  \draw[decorate,decoration={brace,amplitude=3pt}] (-4.15,0) -- (-4.15,1.2) node[midway,left=3pt,font=\scriptsize] {$\lambda$};
  \draw[<->,gray] (-4.85,-0.5) -- (-4.85,1.7) node[midway,left,font=\scriptsize,black] {$2^{-Ck}$};
  \node[font=\small\itshape] at (1.3,2.3) {chart $\exp^{-1}_{x_0}$};
  \draw[->] (-2.0,0) -- (2.0,0) node[right] {$\exp_{x_0}^{-1}(T)$};
  \draw[->] (0,-0.95) -- (0,2.15) node[above left] {$T_{\widetilde y}\widetilde W_{\bfi}(\widetilde y)$};
  \draw[blue!50] plot[domain=-0.9:2.0,samples=50] ({0.02*\x*\x*\x},{\x}) node[right,font=\scriptsize] {$\widetilde W_\bfi(\widetilde y)$};
  \draw[very thick,blue] plot[domain=-0.5:1.7,samples=40] ({0.02*\x*\x*\x},{\x});
  \node[blue,font=\scriptsize,anchor=west] at (0.14,1.62) {$\gamma=\exp^{-1}_{x_0}(L')$};
  \draw[red!60,thin] plot[domain=-0.7:1.9,samples=40] ({0.75*\x+0.03*\x*\x},{\x})
    node[right,font=\scriptsize] {$\widetilde W_\bfj(\widetilde y)$};
  \coordinate (Z) at (0.035,1.2);
  \draw[thick,red] plot[domain=0:1.2,samples=40] ({0.035+0.75*(\x-1.2)+0.15*(\x*\x-1.44)},{\x});
  \draw[thick,red,dashed] plot[domain=1.2:1.45,samples=10] ({0.035+0.75*(\x-1.2)+0.15*(\x*\x-1.44)},{\x});
  \fill (Z) circle (1.1pt) node[left,font=\scriptsize,xshift=-1pt] {$\widetilde z_t=\gamma(\lambda)$};
  \node[red,font=\scriptsize,left] at (-0.86,0.72) {$\widetilde W_\bfj(\widetilde z_t)$};
  \node[red,font=\scriptsize,right] at (-0.55,0.22) {$\pi_\bfj$};
  \draw[gray,densely dotted] (Z) -- (-0.865,0);
  \coordinate (PJ) at (-1.081,0);
  \coordinate (Tt) at (-0.9,0);
  \fill[red] (PJ) circle (1.1pt);
  \fill (Tt) circle (1.1pt);
  \node[red,font=\scriptsize,above left,yshift=1pt] at (PJ) {$\pi_\bfj(\widetilde z_t)$};
  \node[font=\scriptsize,below right,xshift=-1pt] at (Tt) {$t$};
  \draw[decorate,decoration={brace,mirror,amplitude=2.5pt}] (-1.081,-0.22) -- (-0.9,-0.22)
    node[midway,below=2pt,font=\scriptsize] {$\le s_{\mathrm a}2^{-(C+3)k}$};
  \draw[<->,gray] (1.55,0) -- (1.55,1.2) node[midway,right,font=\scriptsize,black] {$\widetilde z_{t,N}\approx c_0\lambda$};
  \draw[dotted,gray] (Z) -- (1.55,1.2);
  \draw[line width=2.2pt,black!60,opacity=0.5] (-1.275,0) -- (0.375,0);
  \node[font=\scriptsize] at (0.55,-0.6) {$I_k=\varphi_1^{-1}(\sigma_\bfi(L'))$};
  \draw[->,gray,thin] (0.2,-0.5) -- (0.0,-0.06);
  \draw[->,thick,black!70] (SZ) to[bend left=22] node[pos=0.5,above,font=\scriptsize,black] {$\exp^{-1}_{x_0}\circ\sigma^{-1}_\bfi$} (Z);
  \draw[->,thick,black!70] (SZ) to[bend right=42] node[pos=0.6,below,font=\scriptsize,black] {$\varphi_1^{-1}$:\ $t=-\lambda\,c_0\,\mathfrak m_\bfj$} (Tt);
\end{tikzpicture}
 \caption{The non-linear displacement (\Cref{rem:geom-input-nonlinear-displacement}).}
\label{fig:nonlinear-displacement}
\end{figure}
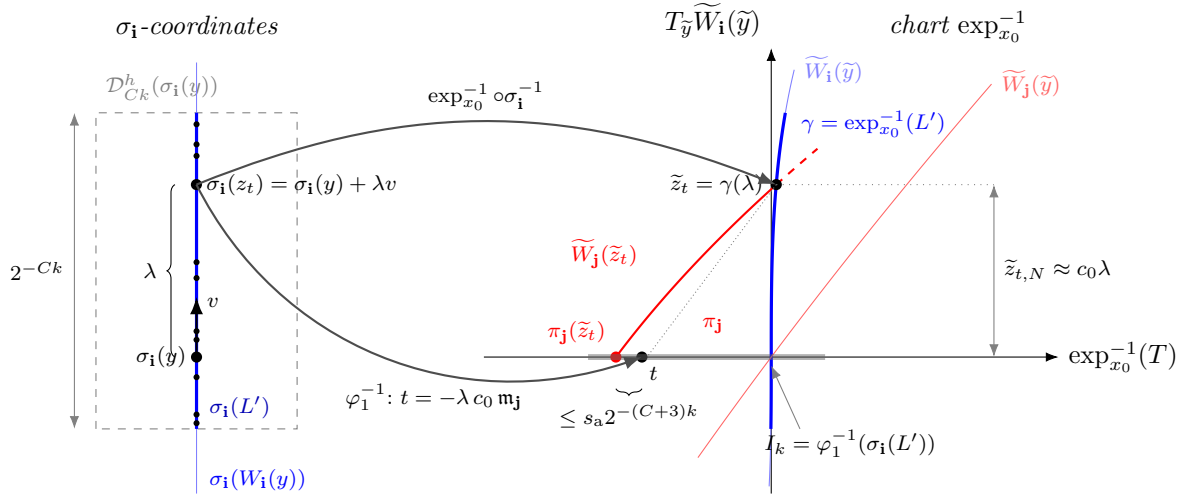

\section{Open dense property of non-integrability}

\begin{lem}[A Markov-coding lemma]
\label{lem:periodic-point-with-two-preimages}
Let \((\Lambda,f)\) be a transitive repeller, and assume that
\(f|_{\Lambda}\) is not one-to-one. Then there exist a periodic point
\(p\in\Lambda\), two distinct points \(q_0,q_1\in\Lambda\), and an open
neighborhood \(V\) of \(q_0\) such that
\[
    f(q_0)=f(q_1)=p,
\]
\(q_1\) is the predecessor of \(p\) on its periodic orbit, \(q_0\) is not on that
periodic orbit, and the following holds. For every \(N_0\geq1\), there exists
\(N\geq N_0\) and two admissible backward orbit segments
\[
    x^0_{-N},\ldots,x^0_{-1}=q_0,x_0=p,
    \qquad
    x^1_{-N},\ldots,x^1_{-1}=q_1,x_0=p,
\]
such that the second segment lies on the periodic orbit of \(p\), and
\[
    x^0_{-j}\notin \overline V
    \qquad\text{for every }2\leq j\leq N .
\]
\end{lem}

\begin{proof}
We use the standard Markov coding of transitive repellers; see
\cite{KatokHasselblatt1995,Przytycki1976}. Since \(f|_\Lambda\) is a local
homeomorphism and is not one-to-one, there is a point with two distinct preimages
in \(\Lambda\). After refining a Markov partition, the corresponding irreducible
Markov graph has a vertex \(v\) with two distinct incoming edges, realized by two
disjoint local inverse branches over the same Markov rectangle:
\[
    e_0:u_0\to v,
    \qquad
    e_1:u_1\to v .
\]

Choose a shortest admissible path from \(v\) to \(u_1\). This path does not contain
\(e_0\). Indeed, if it contained \(e_0\), then after traversing \(e_0\) the path would
return to \(v\), and the remaining tail would give a strictly shorter path from \(v\)
to \(u_1\). Hence this path followed by \(e_1\) gives a periodic cycle \(P\) ending at
\(v\) and not containing \(e_0\).

Similarly, choose a shortest admissible path from \(v\) to \(u_0\). The same argument
shows that this path does not contain \(e_0\). Therefore, by repeating the cycle \(P\)
sufficiently many times, then following this shortest path from \(v\) to \(u_0\), and
finally taking the edge \(e_0:u_0\to v\), we obtain admissible words of arbitrarily
large length in which \(e_0\) occurs only at the last step.

Let \(p\) be the periodic point represented by the cycle \(P\). Let \(q_1\) be its
periodic predecessor corresponding to the edge \(e_1\), and let \(q_0\) be the other
preimage of \(p\) corresponding to the edge \(e_0\). Then
\[
    f(q_0)=f(q_1)=p,
\]
\(q_1\) lies on the periodic orbit of \(p\), and \(q_0\) does not, since the cycle \(P\)
does not contain \(e_0\). Choose a sufficiently small open neighborhood \(V\) of
\(q_0\) whose closure is disjoint from the periodic orbit of \(p\) and contained in the
local inverse branch corresponding to \(e_0\) over the terminal Markov rectangle.
Thus an orbit segment can enter \(V\) only when the symbolic transition \(e_0\) is
used. The admissible words constructed above therefore project to backward orbit
segments
\[
    x^0_{-N},\ldots,x^0_{-1}=q_0,x_0=p,
    \qquad
    x^1_{-N},\ldots,x^1_{-1}=q_1,x_0=p,
\]
where the second segment lies on the periodic orbit of \(p\), and the first segment
satisfies
\[
    x^0_{-j}\notin \overline V
    \qquad\text{for every }2\leq j\leq N .
\]
Since the cycle \(P\) can be repeated arbitrarily many times, \(N\) can be taken
larger than any prescribed \(N_0\).
\end{proof}

\begin{prop}[Open dense property of non-integrability]
\label{appendix:B}
Let \(1<r\leq\infty\), and let \((\Lambda_f,f)\) be a \(C^r\) surface repeller admitting
a dominated splitting in the sense of Definition~\ref{def:dominated splitting}.
Assume that \(f|_{\Lambda_f}\) is non-invertible on every transitive component.
Choosing a \(C^1\)-neighborhood \(\mathcal U_f\) of \(f\) such that every
\(g\in\mathcal U_f\) has a continuation \((\Lambda_g,g)\), the dominated splitting
continues, and every continued transitive component remains non-invertible. 

Then, for every \(1<r\leq\infty\), the set of maps
\[
    g\in \mathcal U_f^r:=\mathcal U_f\cap C^r(M,M)
\]
for which every transitive component of \((\Lambda_g,g)\) is
\(su\)-non-integrable is \(C^1\)-open and \(C^r\)-dense in \(\mathcal U_f^r\).
Consequently, \(su\)-non-integrability is a \(C^1\)-open and \(C^r\)-dense property
among non-invertible surface repellers admitting a dominated splitting.
\end{prop}

\begin{proof}
We use the standard hyperbolic continuation and Markov coding for expanding
repellers, and the standard robustness of dominated cone fields; see
\cite{HirschPughShub1977,KatokHasselblatt1995,Przytycki1976}.

First we prove openness. Suppose that \(g\in\mathcal U_f\) is
\(su\)-non-integrable on a transitive component. Then there exist two histories
\(\hat x^0,\hat x^1\in\hat\Lambda_g\) over the same base point \(x\) such that
\[
    \pi(\hat x^0)=\pi(\hat x^1)=x,
    \qquad
    \angle\bigl(E^{su}_g(\hat x^0),E^{su}_g(\hat x^1)\bigr)>0 .
\]
For every \(h\) sufficiently \(C^1\)-close to \(g\), structural stability gives the base
continuation and hence the induced inverse-limit continuation; the two histories
above continue to two histories over the same continued base point. The dominated
bundles depend continuously on the map and on the history. Hence the above angle
remains positive for all \(h\) sufficiently \(C^1\)-close to \(g\). Since the number of
transitive components is finite, \(su\)-non-integrability on all components is
\(C^1\)-open.

We now prove density. Fix \(g\in\mathcal U_f^r\) and a \(C^r\)-neighborhood
\(\mathcal O\) of \(g\). It is enough to treat one transitive component which is still
\(su\)-integrable; the finitely many components can then be treated one after another
with disjoint supports, and the openness just proved preserves the components
already treated.

Thus assume that \(\Lambda_g\) is transitive, non-invertible, admits a dominated
splitting, and is \(su\)-integrable. Then \(E^{su}_g(\hat x)\) depends only on the base
point \(\pi(\hat x)\); we denote this line by \(E^{su}_g(x)\). Let \(\mathcal C^{su}\) be
a cone field around \(E^{su}_g\), chosen sufficiently narrow but with positive width.
By domination, there exist constants \(K>0\) and \(0<\theta<1\) such that for every
admissible backward orbit segment
\[
    x_{-N},x_{-N+1},\ldots,x_{-1},x_0,
    \qquad
    g(x_{-j})=x_{-j+1},
\]
one has
\begin{equation}\label{eq:generic-cone-contraction}
    \operatorname{diam}_{\mathbb P}
    \left(
        Dg^N_{x_{-N}}\mathcal C^{su}(x_{-N})
    \right)
    \leq K\theta^N .
\end{equation}
Moreover, the cone field is robust under sufficiently small \(C^1\)-perturbations.

Apply Lemma~\ref{lem:periodic-point-with-two-preimages} to this component. We
obtain a periodic point \(p\), two distinct preimages \(q_0,q_1\) of \(p\), and a fixed
open neighborhood \(V\) of \(q_0\), with \(q_1\) lying on the periodic orbit of \(p\) and
\(q_0\) not lying on that orbit. Since \(V\) is fixed, we may choose a small angle
\(\omega>0\) such that \(\omega\) is much smaller than the width of 
$\mathcal C^{su}$ and such that every perturbation supported in \(V\), equal to a
rotation of angle \(\omega\) at the derivative level at \(q_0\), remains in
\(\mathcal O\) and still has the continued dominated splitting. Choose \(N_0\) so large
that
\[
    K\theta^N<\frac{\omega}{10}
    \qquad\text{for every }N\geq N_0,
\]
and choose the two backward orbit segments of some length \(N\geq N_0\) given by
Lemma~\ref{lem:periodic-point-with-two-preimages}:
\[
    x^0_{-N},\ldots,x^0_{-1}=q_0,x_0=p,
    \qquad
    x^1_{-N},\ldots,x^1_{-1}=q_1,x_0=p.
\]
 
Let \(R_\omega:T_pM\to T_pM\) be a linear isometry which rotates the line
\(E^{su}_g(p)\) by projective angle \(\omega\). In local coordinates near \(q_0\) and
\(p=g(q_0)\), choose a \(C^\infty\) bump function \(\chi\) supported in \(V\) and
satisfying \(\chi(q_0)=1\). Define a \(C^r\)-small perturbation \(\tilde g\) of \(g\) by
\[
    \tilde g(u)
    =
    g(u)+\chi(u)(R_\omega-\mathrm{Id})Dg_{q_0}(u-q_0).
\]
For \(r=\infty\), this perturbation is small in any prescribed finite family of
\(C^k\)-seminorms, provided \(\omega\) is chosen sufficiently small. Moreover,
\[
    \tilde g(q_0)=p,
    \qquad
    D\tilde g_{q_0}=R_\omega\circ Dg_{q_0},
\]
and \(\tilde g=g\) outside \(V\). By the choice of \(V\), the two finite backward orbit
segments above are still admissible for \(\tilde g\), and the periodic orbit of \(p\) is
unchanged. Extending the finite segments further to the past gives two histories
\(\hat p^0,\hat p^1\in\hat\Lambda_{\tilde g}\) over \(p\) whose last \(N\) backward
entries are the two finite segments above.

Along the second segment, the derivative cocycle is unchanged. Along the first
segment, only the last derivative is changed, and it is post-composed by
\(R_\omega\). Since the component was \(su\)-integrable for \(g\), both old cone images
are contained in the \(\omega/10\)-neighborhood of \(E^{su}_g(p)\) in the projective
line, by \eqref{eq:generic-cone-contraction}. Hence the first cone image for
\(\tilde g\) is contained in the \(\omega/10\)-neighborhood of
\(R_\omega E^{su}_g(p)\), while the second cone image is contained in the
\(\omega/10\)-neighborhood of \(E^{su}_g(p)\).

By the cone-field characterization of the dominated direction, after taking the
initial cone field sufficiently narrow and the perturbation sufficiently small, we have
\[
    E^{su}_{\tilde g}(\hat p^0)
    \in
    D\tilde g^N_{x^0_{-N}}\mathcal C^{su}(x^0_{-N}),
    \qquad
    E^{su}_{\tilde g}(\hat p^1)
    \in
    D\tilde g^N_{x^1_{-N}}\mathcal C^{su}(x^1_{-N}).
\]
Consequently,
\[
    \angle\bigl(E^{su}_{\tilde g}(\hat p^0),
                E^{su}_{\tilde g}(\hat p^1)\bigr)
    \geq
    \omega-\frac{\omega}{10}-\frac{\omega}{10}
    >0 .
\]
Thus the component is \(su\)-non-integrable after an arbitrarily small
\(C^r\)-perturbation.

Repeating this construction on the finitely many transitive components which
remain \(su\)-integrable, with pairwise disjoint supports and successively smaller
perturbations, gives a map in \(\mathcal O\) for which every transitive component is
\(su\)-non-integrable. This proves \(C^r\)-density. Together with openness, the
proposition follows.
\end{proof}

\bibliography{bibfile}

\end{document}